\documentclass[12pt,reqno]{amsart}

\usepackage[colorlinks=true,urlcolor=blue,
bookmarks=true,bookmarksopen=true,citecolor=blue
]{hyperref}

\usepackage{color,soul}
\setstcolor{red}
\usepackage{pdfsync}
\usepackage[all, cmtip]{xy}
\usepackage{graphicx}

\usepackage{dsfont}
\usepackage{float}
\usepackage{bm}
\usepackage{mathtools}
\usepackage{amssymb}
\usepackage[T1]{fontenc}

\usepackage{epsfig}
\usepackage{amscd}
\usepackage[mathscr]{eucal}
\usepackage{amssymb}
\usepackage{bm}
\usepackage{amsxtra}
\usepackage{amsmath}
\usepackage{latexsym} 
\usepackage[all]{xy}
\usepackage{enumerate}

\usepackage{enumitem}

\usepackage{color}
\usepackage{xcolor, soul}
\sethlcolor{green}

\usepackage{bbm}

\theoremstyle{plain}

\newtheorem{thm}{Theorem}[subsection]

\newtheorem{cor}[thm]{Corollary}

\newtheorem{lem}[thm]{Lemma}
\newtheorem{prop}[thm]{Proposition}

\theoremstyle{definition}
\newtheorem{dfn}[thm]{Definition}

\newtheorem*{claim-nonum}{Claim}

\theoremstyle{remark}
\newtheorem{remark}[thm]{Remark}
\newtheorem{rem}[thm]{Remark}
\newtheorem*{remnonum}{Remark}

\newtheorem{ex}[thm]{Example}

\theoremstyle{plain}

\newtheorem{lemma}[thm]{Lemma}

\def\ep{\epsilon}
\def\Hom{{\rm{Hom}}}

\def\C{\mathcal{C}}
\def\A{\mathcal{A}}

\def\I{\mathbb{I}}
\def\K{\mathcal K}
\def\P{\mathcal P}
\def\D{\mathcal{D}}

\def\k{{\bf k}}
\def\F{\mathcal F}
\def\T{\mathcal T}
\def\G{\mathcal G}
\def\U{\mathbb U}
\def\V{\mathbb V}
\def\W{\mathbb W}

\def\I{{\mathds{1}}}
\def\Mor{{\rm Mor}}

\DeclarePairedDelimiter\ceil{\lceil}{\rceil}

\newcommand{\Qed}{\hfill \qedsymbol \medskip}

\newcommand{\R}{\mathbb{R}}
\newcommand{\Z}{\mathbb{Z}}
\newcommand{\Q}{\mathbb{Q}}
\newcommand{\N}{\mathbb{N}}

\newcommand{\fuk}{\mathcal{F}uk}

\renewcommand{\k}{\mathbf{k}}

\newcommand{\pbnote}[1]{#1}

\newcommand{\jznote}[1]{#1}

\newcommand{\pbaddress}{biran@math.ethz.ch}
\newcommand{\ocaddress}{cornea@dms.umontreal.ca}

\begin{document}

\title{Triangulation and Persistence: Algebra 101}

\date{\today}

\thanks{The second author was supported by an individual NSERC
  Discovery grant. }
%The third author was supported by an individual
 % NSERC Discovery grant, and by the FRQNT start up grant.}

\author{Paul Biran, Octav Cornea and Jun Zhang}

\address{Paul Biran, Department of Mathematics, ETH-Z\"{u}rich,
  R\"{a}mistrasse 101, 8092 Z\"{u}rich, Switzerland}
\email{\pbaddress}
 
\address{Octav Cornea, Department of Mathematics and Statistics,
  University of Montreal, C.P. 6128 Succ.  Centre-Ville Montreal, QC
  H3C 3J7, Canada} \email{\ocaddress}

\address{Jun Zhang, Centre de Recherches Math\'ematiques, University of Montreal, C.P. 6128 Succ.  Centre-Ville Montreal, QC
  H3C 3J7, Canada} \email{jun.zhang.3@umontreal.ca}

\keywords{Triangulated categories, Persistence modules, Tamarkin category,
    Lagrangian cobordism.}
\subjclass[2010]{55N31 53D12 (Primary); 35A27 (Secondary)}

%\bibliographystyle{plain}

%\bibliographystyle{alphanum}

% ----------------------------------------------------------------------
%

% ----------------------------------------------------------------------
%
% Abstract

\begin{abstract}
  This paper lays the foundations of triangulated persistence
  categories (TPC), which brings together persistence modules with the
  theory of triangulated categories. As a result we introduce several
  measurements and metrics on the set of objects of some triangulated
  categories. We also \jznote{provide}  examples of TPC's coming from
  algebra, algebraic topology, \jznote{microlocal sheaf theory} and symplectic topology.
\end{abstract}

\maketitle

% ----------------------------------------------------------------------
%
% Beginning of text
%

\tableofcontents 
% !TEX root = TPC.tex

\section{Introduction} \label{sec-intro}

Persistence theory, invented for the first time in
Zomorodian-Carlsson's pioneering work \cite{ZC05}, is an abstract
framework that emerged from investigations in parts of machine
learning as well as algebraic and differential topology formalizing
the structure and properties of a class of phenomena that are most
easily seen in the homology of a filtered chain complex over a field
$\k$,
$\ldots \subset (C^{\leq \alpha}, d) \subset (C^{\leq \beta}, d)
\subset \ldots \subset (C, d)$ (for
$\alpha < \beta, \alpha,\beta\in \R$).  The homology of such a complex
forms a family $\{H(C^{\leq \alpha})\}_{\alpha\in \R}$ related by maps
$i_{\alpha,\beta}: H(C^{\leq \alpha}) \to H(C^{\leq \beta})$,
$\alpha\leq \beta$ subject to obvious compatibilities and called a
persistence module.  Given two filtered complexes $(C,d)$ and $(D,d)$
that are quasi-isomorphic, it is possible to compare them by the
so-called bottleneck distance.  Its definition is based on the fact
that the linear maps $v: C\to D$ are themselves filtered by their
``shift'': $v$ is of shift $\leq r$ if
$v(C^{\leq \alpha})\subset D^{\leq \alpha +r}$, for all
$\alpha \in \R$.  Using this, given two chain maps $\phi: C\to D$,
$\psi : D\to C$ such that $\psi \circ \phi \simeq id_{C}$, there is a
natural measurement for how far the composition $\psi \circ \phi$ is
from the identity, namely the infimum of the ``shifts'' of chain
homotopies $h: C \to C$ such that $dh+hd=\psi\circ\phi -id$.  The
machinery of persistence modules is much more developed than the few
elements mentioned here. For a survey on this topic and its
applications in various mathematical branches, see research monographs
from Edelsbrunner \cite{Ede14}, Chazal-de Silva-Glisse-Oudot
\cite{CdeSGO16}, and Polterovich-Rosen-Samvelyan-Zhang
\cite{PRSZ20}. In particular, there is a beautiful interpretation of
the bottleneck distance in terms of so-called barcodes (\cite{BL15},
\cite{UZ16}), but we can already formulate the question that we
address in this paper.
 
\
 
{\em How can one use a persistence type structure on the morphisms of
  some (small) category to compare not only (quasi)-isomorphic objects
  but rather define a pseudo-metric on the set of all objects ?}

\

We provide here a solution to this question based on mixing
persistence with triangulation.  Triangulated categories are basic
algebraic structures with applications in algebraic geometry,
algebraic and symplectic topology, mathematical physics and other
fields. They were introduced independently by Puppe \cite{Pup62} and
Verdier \cite{Ver77} in the early '60's. A triangulated category
$\mathcal{D}$ is a category endowed with a \pbnote{translation}
endomorphism $T$ and a class of so called ``distinguished (or exact)
triangles'' - meaning a triple of objects and morphisms
$A \to B \to C \to TA$ - subject to four axioms. These axioms mimic
abstractly the properties of cone attachments in topology, where the
object $C$ is obtained by attaching a cone over a continuous map
$A\to B$. For instance, one axiom claims that any morphism can be
included in a distinguished triangle (corresponding to the fact that
one can attach a cone over any continuous map) and another, the
octahedral axiom, can be viewed as a form of associativity for cone
attachments.

\

Given a triangulated category $\mathcal{D}$ there is a simple notion
of {\em triangular weight} $w$ on $\mathcal{D}$ that we introduce in
\S \ref{subsec:triang-gen}. This associates to each distinguished
triangle $\Delta$ a non-negative number $w(\Delta)$ satisfying a
couple of properties.  The most relevant of them is a weighted variant
of the octahedral axiom (we will give a more precise definition
later).  A basic example of a triangular weight is the flat one: it
associates to each distinguished triangle the value $1$.  The interest
of triangular weights is that they naturally lead to {\em
  fragmentation pseudo-metrics} on $\mathrm{Obj}(\mathcal{D})$ (we
assume here that $\mathcal{D}$ is small) defined roughly as follows
(see \S\ref{subsec:triang-gen} for details).  Such a pseudo-metric
depends on a family of objects $\mathcal{F}$ of $\mathcal{D}$.  With
$\mathcal{F}$ fixed, and up to a certain normalization, the
pseudo-distance $d^{\mathcal{F}}(X,Y)$ between $X$,
$Y\in \mathrm{Obj}(\mathcal{D}))$ is (the symmetrization of) the
infimum of the total weight of distinguished triangles needed to
construct iteratively $X$ out of $Y$ by only attaching cones over
morphisms with domain in $\mathcal{F}$.  The weighted octahedral axiom
implies that this $d^{\mathcal{F}}$ satisfies the \pbnote{triangle}
inequality.  Using such \jznote{pseudo-metrics} one can analyze rigidity properties of
various categories by exploring the induced topology on
$\mathrm{Obj}(\mathcal{D})$, in particular, by identifying cases when
relevant fragmentation pseudo-metrics are non-degenerate.  One can
also study the \jznote{dynamical properties of} endofunctors $\phi$ of $\mathcal{D}$ by estimating
$d^{\mathcal{F}}(X, \phi^{k}(X))$ for $X\in \mathrm{Obj}(\mathcal{D})$
as well as various associated asymptotic invariants.

 \
 
Even the fragmentation pseudo-metrics associated to the flat weight are of interest.  Many qualitative questions concerned with numerical lower bounds  for the complexity of certain 
geometric objects  - classical examples are the Morse inequalities, the Lusternik-Schnirelmann inequality as well as, in symplectic topology, the inequalities predicted by the Arnold conjectures - can be understood by means of inequalities involving such fragmentation pseudo-metrics. Remarkable results based on 
measurements using this flat weight and applied to the study of  endofunctors have appeared recently in work of Orlov \cite{Or} as well as Dimitrov-Haiden-Katzarkov-Kontsevich \cite{Kon} and Fan-Filip \cite{Fa-Si}. On  the other hand, a discrete weight, such as the flat one, is not ideal in a variety of geometric contexts because the associated fragmentation pseudo-metrics vanish on too many pairs of objects. 
 
 \
 
The main aim of the paper is to  show how to use persistence machinery to produce ``exotic'' (non-flat) triangular weights and associated fragmentation pseudo-metrics.  The main tool is a type of refinement of triangulated categories, called {\em Triangulated Persistence Categories} (TPC).   A triangulated persistence category,  $\mathcal{C}$, has two main properties. First, it is a persistence category, a natural notion we introduce in \S\ref{subsec:pers-cat}. This is  a category  $\mathcal{C}$ whose morphisms  $\Mor_{\mathcal{C}}(A,B)$ are persistence modules, \jznote{$\{ \{\Mor^{r}_{\mathcal{C}}(A,B)\}_{r\in \R}\}, \{i_{r,s}\}_{r\leq s}\}$} (\jznote{see \cite{PRSZ20} for a general introduction of the persistence module theory}), and that satisfies obvious relations relative to composition. The second main structural property  of TPC's
is that the objects of $\mathcal{C}$ together with the ``shift'' $0$ morphisms  $\Mor^{0}_{\mathcal{C}}(A,B)$ have the structure of a triangulated category. 
The formal definition of TPC's is given in \S\ref{subsec:TPC}.  The mixture of the persistence and triangulation structures gives rise to a rich algebraic framework. In particular, we will show how to endow such a 
$\mathcal{C}$ with a class of weighted triangles. 
The main result of the paper, in Theorem \ref{thm:main-alg}, shows that the properties of these weighted triangles  lead to a  triangulated structure on the limit category $\mathcal{C}_{\infty}$ - that has the same objects as $\C$
and has as morphisms the $\infty$-limits of the morphisms in $\C$ - as well as to a triangular weight
on $\C_{\infty}$. 

\

Therefore, if a triangulated category $\mathcal{D}$ admits
a TPC refinement - that is a TPC, $\C$, such that $\C_{\infty}=\mathcal{D}$ (as triangulated
categories), then $\mathcal{D}$ carries an exotic triangular weight induced from the persistence structure of 
$\C$.  As a result, this construction provides a technique to build non-discrete  fragmentation pseudo-metrics
on the objects of $\mathcal{D}$. The construction of TPC's is inspired
by recent constructions in symplectic topology and, in particular, by the shadow pseudo-metrics introduced in
\cite{Bi-Co-Sh:LagrSh} and \cite{Bi-Co:LagPict} in the study of Lagrangian cobordism. Remarkably, in that case, rigidity in the sense of 
the non-degeneracy of the fragmentation metrics is implied by symplectic rigidity in the sense of unobstructedness and, conversely, flexibility in the sense of Gromov's h-principle, implies that the relevant fragmentation pseudo-metrics vanish. However, the general set-up of TPC's is independent of symplectic
topology considerations and is potentially of use in a variety of other contexts where both persistence and triangulation are available.

\

This paper only contains the basic algebraic machinery that emerges from
the mixture of triangulation and persistence. Triangulated persistence categories  are natural algebraic structures that appear in a wide variety of mathematical fields. In \S\ref{sec-example} we discuss four classes of examples: algebraic, topological, sheaf-theoretic and symplectic. 

\

\noindent {{\bf Acknowledgement}. The third author is grateful to Mike Usher for useful discussions. We thank Leonid Polterovich for mentioning to us the work of Fan-Filip \cite{Fa-Si}. This work was completed while the third author held a CRM-ISM Postdoctoral Research Fellowship at the {\em Centre de recherches math\'ematiques} in Montr\'eal. He thanks this Institute for its warm hospitality.

% !TEX root = TPC.tex

%\section{Algebra.}

\section{Triangular weights} \label{subsec:triang-gen}
In this subsection we introduce triangular weights associated to a
triangulated category $\mathcal{D}$.  Using such a triangular weight
$w$ on $\mathcal{D}$ we define a class of, so called fragmentation,
pseudo-metrics $d^{\mathcal{F}}_{w}$ on
$\mathrm{Obj}(\mathcal{D})$. All categories used in this paper
($\mathcal{D}$ in particular) are assumed to be small unless otherwise
indicated.

%by the recipe in \S\ref{subsubsec:frag1}.

\begin{dfn}\label{def:triang-cat-w} Let $\mathcal{D}$ be a (small)
  triangulated category and denote by $\mathcal{T}_{\mathcal{D}}$ its
  class of exact triangles. A {\em triangular weight} $w$ on
  $\mathcal{D}$ is a function
  $$w:\mathcal{T}_{\mathcal{D}}\to [0,\infty)$$
  that satisfies properties (i) and (ii) below:

  (i) [Weighted octahedral axiom] Assume that the triangles
  $\Delta_{1}: A\to B\to C\to TA$ and $\Delta_{2}: C\to D\to E\to TC$
  are both exact. There are exact triangles:
  $\Delta_{3}: B\to D\to F\to TB$ and
  $\Delta_{4}: TA\to F\to E\to T^{2}A$ making the diagram below
  commute, except for the right-most bottom square that anti-commutes,
  \[ \xymatrix{
      A \ar[r] \ar[d] & 0 \ar[r] \ar[d] & TA \ar[d]\ar[r]& TA\ar[d] \\
      B \ar[r] \ar[d] & D \ar[r] \ar[d] & F \ar[d]\ar[r]& TB \ar[d]\\
      C \ar[r] \ar[d] & D \ar[r]\ar[d] & E\ar[r]\ar[d] & TC\ar[d] \\
      TA\ar[r] & 0 \ar[r]  & T^{2}A \ar[r]& T^{2} A }
  \]
  and such that 
  \begin{equation}\label{eq:weight:ineq}
    w(\Delta_{3})+w(\Delta_{4})\leq w(\Delta_{1})+w(\Delta_{2})~.~
  \end{equation}

  (ii) [Normalization] There is some $w_{0}\in [0,\infty)$ such that
  \pbnote{$w(\Delta)\geq w_{0}$} for all
  \pbnote{$\Delta\in \mathcal{T}_{\mathcal{D}}$} and
  \pbnote{$w(\Delta')=w_{0}$} for all triangles \pbnote{$\Delta'$} of
  the form $0\to X\xrightarrow{\mathds{1}_{X}}X\to 0$,
  $X\in\mathrm{Obj}(\mathcal{D})$, and their rotations.  Moreover, in
  the diagram at (i) if $B=0$, we may take $\Delta_{3}$ to be
  \begin{equation}\label{eq:simpl-d3}
    \Delta_{3}: 0 \to D\to D\to 0~.~
  \end{equation}
  
\end{dfn}

\begin{remark}\label{rem:gen-weights} (a) Neglecting the weights
  constraints, given the triangles $\Delta_{1}$, $\Delta_{2}$,
  $\Delta_{3}$ as at point (i), the octahedral axiom is easily seen to
  imply the existence of $\Delta_{4}$ making the diagram commutative,
  as in the definition.

(b) The condition  at point (ii), above equation (\ref{eq:simpl-d3}), can be reformulated as a replacement property for exact triangles 
in the following sense: if $\Delta_{2}:C\to D\to E\to TC$ is exact and $C$ is isomorphic to 
$A'\ (=TA)$, then there is an exact triangle $A'\to D\to E\to TA'$ of weight at most $w(\Delta_{2})+w(\Delta_{1})-w_{0}$
where $\Delta_{1}$ is the exact triangle $T^{-1}A'\to 0 \to C\to A'$.
\end{remark}

Given an exact triangle $\Delta : A\to B\xrightarrow{f} C\to TA$ in $\mathcal{D}$ and any $X\in \mathrm{Obj}(\mathcal{D})$
there is an associated exact triangle $ X\oplus\Delta :A\to X\oplus B\xrightarrow{\mathds{1}_{X}\oplus f} X\oplus C\to TA$ and a similar one, $\Delta \oplus X$.
We say that a triangular weight $w$ on $\mathcal{D}$ 
is {\em subadditive} if for any exact triangle $\Delta \in \mathcal{T}_{\mathcal{D}}$ and any object $X$ of $\mathcal{D}$ we have $$w(X\oplus\Delta ) \leq w(\Delta )$$
and similarly for $\Delta \oplus X$.

\

The simplest example of a triangular weight on a triangulated category
$\mathcal{D}$ is the flat one, \pbnote{$w_{fl}(\Delta)=1$,} for all
triangles \pbnote{$\Delta\in \mathcal{T}_{\mathcal{D}}$.} This weight
is obviously sub-additive. A weight that is not proportional to the
flat one is called {\em exotic}.

\

The interest of triangular weight comes from the next definition that provides a measure for the complexity 
of cone-decompositions in $\mathcal{D}$ and this leads in turn to the definition of corresponding pseudo-metrics on the set $\mathrm{Obj}(\mathcal{D})$.

\begin{dfn}\label{def:iterated-coneD-tr} Fix a triangulated category
  $\mathcal{D}$ together with a triangular weight $w$ on
  $\mathcal{D}$.  Let $X$ be an object of $\mathcal{D}$. An {\em
    iterated cone decomposition} $D$ of $X$ with {\em linearization}
  $\ell(D) = (X_1,X_{2}, ..., X_n)$ consists of a family of exact
  triangles in $\mathcal{D}$:
\[ 
\left\{ 
\begin{array}{ll}
\Delta_{1}: \, \, & X_{1}\to 0\to Y_{1}\to  TX_{1}\\
\Delta_2: \,\, & X_2 \to Y_1 \to Y_2 \to  TX_2\\
\Delta_3: \,\, & X_3 \to Y_2 \to Y_3 \to  TX_3\\
&\,\,\,\,\,\,\,\,\,\,\,\,\,\,\,\,\,\,\,\,\vdots\\
\Delta_n: \,\, & X_n \to Y_{n-1} \to X \to  TX_n
\end{array} \right.\]
To accommodate the case $n=1$ we set $Y_0=0$.
The weight of such a cone decomposition is defined by:
\begin{equation}\label{eq:weight-cone} 
w(D)=\sum_{i=1}^{n}w(\Delta_{i})-w_{0}~.~
\end{equation}
\end{dfn}

This weight of cone-decompositions easily leads to a class of
pseudo-metrics on the objects of $\mathcal{D}$, as follows.

\

Let $\mathcal F \subset {\rm Obj}(\mathcal{D})$.
For two objects $X, X'$ of $\mathcal{D}$, define 
\begin{equation} \label{frag-met-0} \delta^{\mathcal F} (X, X') =
  \inf\left\{ w(D) \, \Bigg| \,
    \begin{array}{ll} \mbox{$D$ is an
      iterated cone decomposition} \\
      \mbox{of $X$ with linearization}
      \ \mbox{$(F_1, ..., T^{-1}X', ..., F_k)$}\\
      \mbox { where $F_i \in \mathcal
      F$, $k \geq 0$}
    \end{array}
  \right\}.
\end{equation}
\pbnote{Note that we allow here $k=0$, i.e.~the linearization of $D$
  is allowed to consist of only one element, $T^{-1}X'$, without using
  any elements $F_i$ from the family $\mathcal{F}$.} Fragmentation pseudo-metrics are obtained by symmetrizing
$\delta^{\mathcal{F}}$, as below.  Recall that a topological space is
called an {\it H-space} if there exists a continuous map
$\mu: X \times X \to X$ with an identity element $e$ such that
$\mu(e,x) = \mu(x,e) = x$ for any $x \in X$.

\begin{prop}\label{prop:tr-weights-gen}
  Let $\mathcal{D}$ be a triangulated category and let $w$ be a
  triangular weight on $\mathcal{D}$.  Fix $\F \subset {\rm Obj}(\C)$
  and define
  $$d^{\F}: {\rm Obj}(\C)\times {\rm Obj}(\C)\to [0,\infty)\cup \{ +\infty\}$$ 
  by:
  \[ d^{\F}(X, X') = \max\{\delta^{\F}(X, X'), \delta^{\F}(X',
    X)\}. \]
  \begin{itemize}
  \item[(i)] The \pbnote{map} $d^{\mathcal{F}}$ is a pseudo-metric
    called {\rm the fragmentation pseudo-metric} associated to $w$
    and $\mathcal{F}$.
  \item[(ii)] If $w$ is subadditive, then
    \begin{equation}\label{eq:subad}
      d^{\mathcal{F}}(A\oplus B, A'\oplus B')\leq d^{\mathcal{F}}(A,A') +
      d^{\mathcal{F}}(B,B') + w_0.
    \end{equation} 
    In particular, if $w_{0}=0$, then $\mathrm{Obj}(\mathcal{D})$
    with the operation given by $\oplus$ and the topology induced by
    $d^{\mathcal{F}}$ is an H-space.
  \end{itemize}
\end{prop}

The proof of Proposition \ref{prop:tr-weights-gen} is based on simple
 manipulations with exact triangles.  We will prove a similar
 statement in \S\ref{subsubsec:frag1} in a more complicated setting
 and we will then briefly discuss in \S\ref{subsubsec:proof-prop} how
 the arguments given in that case also imply
 \ref{prop:tr-weights-gen}.

\begin{rem}\label{rem:finite-metr} (a) In case $\mathcal{F}$ is invariant by translation in the sense that $T\mathcal{F}\subset \mathcal{F}$ and, moreover, $\mathcal{F}$ is a family of triangular generators for $\mathcal{C}$, then
 the metrics $d^{\mathcal{F}}$ are finite.  This is not difficult to show by first proving 
 that $\delta^{\mathcal{F}}(0, X)$ is finite for all $X\in \mathrm{Obj}(\C)$. 
 
 (b) It is sometimes useful to view an iterated cone-decomposition as in Definition \ref{def:iterated-coneD-tr}
 as a sequence $\eta$ of spaces and maps forming the successive triangles $\Delta_{i}$ below
 
\begin{equation}\label{eq:iterated-tr}\xymatrixcolsep{1pc}  \xymatrix{
 Y_{0} \ar[rr] &  &  Y_{1}\ar@{-->}[ldd]  \ar[r] &\ldots  \ar[r]& Y_{i} \ar[rr] &  &  Y_{i+1}\ar@{-->}[ldd]  \ar[r] &\ldots \ar[r]&Y_{n-1} \ar[rr] &   &Y_{n} \ar@{-->}[ldd]  &\\
 &         \Delta_{1}                  &  & & &  \Delta_{i+1}                          & &  &  &    \Delta_{n}             \\
  & X_{1}\ar[luu] &  & & &X_{i+1}\ar[luu] &  &  & &X_{n}\ar[luu] }
  \end{equation}
 \end{rem}
where the dotted arrows represent maps $Y_{i} \to TX_{i}$ and, in \ref{def:iterated-coneD-tr}, we have $Y_{0}=0$, $Y_{n}=X$. 

(c) The definition of fragmentation pseudo-metrics is quite flexible and there are a number of possible variants.
One of them will be useful later. Instead of $\delta^{\mathcal{F}}$ as given in (\ref{frag-met-0})
we may use:
\begin{equation}\label{eq:frag-simpl}
\underline{\delta}^{\mathcal{F}}(X,X') =\inf\left\{ \sum_{i=1}^{n}w(\Delta_{i}) \, \Bigg| \, \begin{array}{ll} \mbox{$\Delta_{i}$ are successive exact triangles as in (\ref{eq:iterated-tr})} \\ \mbox{with $Y_{0}=X'$,
$X=Y_{n}$  and  $X_i \in \mathcal{F}$,   $n \in \N$} \end{array} \right\}~.~
\end{equation}
For this to be coherent we need to assume here $0\in \mathcal{F}$.
Comparing with the definition of $\delta^{\mathcal{F}}$ in (\ref{frag-met-0}), $\underline{\delta}^{\mathcal{F}}$ corresponds to only taking into account cone decompositions with linearization $(T^{-1}X', F_{1},\ldots , F_{n})$
and with the first triangle $\Delta_{1}: T^{-1}X_{1}\to 0 \to X_{1}$. There are two advantages of this
expression: the first is that it is trivial to see in this case that $\underline{\delta}^{\mathcal{F}}$ satisfies the 
triangle inequality, this does not even require the weighted octahedral axiom. The other advantage is that one starts the sequence of triangles from $X'$ and thus
the negative translate $T^{-1}X'$ is not needed to define $\underline{\delta}^{\mathcal{F}}$.
There is an associated fragmentation pseudo-metric $\underline{d}^{\mathcal{F}}$ obtained by symmetrizing $\underline{\delta}^{\mathcal{F}}$ and this satisfies a formula similar to (\ref{eq:subad}). Of course, the disadvantage of this fragmentation pseudo-metric is that it is larger than $d^{\mathcal{F}}$ and thus more often infinite.

\section{Persistence categories} \label{subsec:pers-cat}
We introduce in this section the notion of persistence category - a
category whose morphisms are persistence modules and such that
composition respects the persistence structure - and then pursue with
a number of related structures and immediate properties.

\subsection{Basic definitions}\label{subsubsec:pers-cat1}

View the real axis $\R$ as a category with
${\rm Obj}(\R) = \{x \,| \, x \in \R\}$ and for any
$x, y \in {\rm Obj}(\R)$, the hom-set
\[ {\rm Hom}_{\R}(x,y) = \left\{ \begin{array}{lcl} i_{x,y} &
      \mbox{if} & x \leq y \\ \emptyset & \mbox{if} & x>y \end{array}
  \right..\] By definition, for any $x \leq y \leq z$ in $\R$,
$i_{y,z} \circ i_{x,y} = i_{x,z}$. We denote this category by
$(\R, \leq)$. It admits an additive structure. Explicitly, consider
the bifunctor $\oplus: (\R, \leq) \times (\R, \leq) \to (\R, \leq)$
defined by $\oplus(r,s) := r+s$, where $0 \in \R$ is the \pbnote{zero}
object and for any two pairs
$(r,s), (r',s') \in {\rm Obj}((\R, \leq) \times (\R, \leq))$,
\[ {\rm Hom}_{(\R, \leq) \times (\R, \leq)}((r,s), (r',s')) =
  \left\{ \begin{array}{lcl} (i_{r,r'}, i_{s,s'}) & \mbox{if} & r \leq
      r' \,\,\mbox{and} \,\,s \leq s' \\ \emptyset & \mbox{if} &
      \mbox{otherwise} \end{array} \right..\] and further
$\oplus(i_{r,r'}, i_{s,s'}) : = i_{r+s, r'+s'} \in {\rm Hom}_{(\R,
  \leq)}(r+s, r'+s')$.  Fix a ground field $\k$ and denote by
${\rm Vect}_{\k}$ the category of $\k$-vector spaces.

\begin{dfn} \label{dfn-pc} A category $\C$ is called a {\em
    persistence category} if for any $A, B \in {\rm Obj}(\C)$, there
  exists a functor $E_{A,B}: (\R, \leq) \to {\rm Vect}_{\k}$ such that
  the following two conditions are satisfied:
\begin{itemize}
\item[(i)] The hom-set in $\C$ is
  ${\rm Hom}_{\C}(A,B) = \{(f,r) \,| \, f \in E_{A,B}(r)\}$. We denote
  ${\rm Mor}_{\C}^r(A,B): = E_{A,B}(r)$, \jznote{or simply ${\rm Mor}^r(A,B)$ when the ambient category $\C$ is not emphasized}. 
\item[(ii)] The composition
  $\circ: {\rm Mor}_{\C}^r(A,B) \times {\rm Mor}_{\C}^s(B,C) \to {\rm
    Mor}_{\C}^{r+s}(A,C)$ in $\C$ is a natural transformation from
  $E_{A,B} \times E_{B,C}$ to $E_{A,C} \circ \oplus$ (with $\oplus$
  the product $(\R, \leq) \times (\R, \leq)\to (\R,\leq)$).
\end{itemize}
\end{dfn}

\begin{remark} Item (i) means that each hom-set ${\rm Hom}_{\C}(A,B)$
  is a persistence $\k$-module with persistence structure morphisms
  $E_{A,B}(i_{r,s})$ for any $r \leq s$ in $\R$. \jznote{Here, we use the weakest possible definition of a persistence $\k$-module in the sense that no regularities, such as the finiteness of the dimension of ${\rm Mor}_{\C}^r(A,B)$ or the semi-continuity when changing the parameter $r$, are required  (see subsection 1.1 in \cite{PRSZ20})}. Item (ii) implies
  that the following diagram
  % \[ \xymatrix{ E_{A,B}(r) \times E_{B,C}(s) \ar[r]^-{\circ_{(r,s)}}
  %   \ar[d]_-{(E_{A,B} \times E_{B,C})(i_{r,r'},
  %   i_{s,s'})} %& E_{A,C}(\oplus(r,s)) \ar[d]^-{E_{A,C}(\oplus(i_{r,r'}, i_{s,s'}))} \\
  %   E_{A,B}(r') \times E_{B,C}(s') \ar[r]^-{\circ_{(r',s')}} &
  %   E_{A,C}(\oplus(r',s'))} \] for any morphism
  % $(i_{r,r'}, i_{s,s'}): (r,s) \to (r',s')$ where $r \leq r'$ and
  % $s \leq s'$ commutes. By definition, this diagram is the same as
  % the following commutative diagram,
  \begin{equation} \label{comp-nt} \xymatrix{ \Mor_{\C}^r(A,B) \times
      \Mor_{\C}^s(B,C) \ar[r]^-{\circ_{(r,s)}} \ar[d]_-{E_{A,B}(i_{r,r'})
        \times E_{B,C}(i_{s,s'})} & \Mor_{\C}^{r+s}(A,C)
      \ar[d]^-{E_{A,C}(i_{r+s, r'+s'})} \\
      \Mor_{\C}^{r'}(A,B) \times \Mor_{\C}^{s'}(B,C) \ar[r]^-{\circ_{(r',s')}} &
      \Mor_{\C}^{r'+s'}(A,C)}
  \end{equation}
  commutes.
\end{remark}

We will often denote an element in ${\rm Hom}_{\C}(A,B)$, by a single
symbol $\bar{f}$ instead of a pair $(f,r)$. We will use the notation
$\ceil*{\bar{f}~} = r$ to denote the real \pbnote{number} $r$ and
refer to this number as the {\em shift} of $\bar{f}$. For each
$A \in {\rm Obj}(\C)$ the identity
$\bar{\mathds{1}}_A := (\mathds{1}_A, 0) \in \Mor_{\C}^0(A,A)$ is of shift
$0$. For brevity, we will denote from now on the structural morphisms
$E_{A,B}(i_{r,s})$ by $i_{r,s}$.

\

A persistence structure allows us to consider morphism that are identified up to $r$-shift and, similarly, objects that are negligible up to a shift by $r$. 

\begin{dfn} \label{def:acyclics} Fix a persistence category $\C$.
\begin{itemize}
\item[(i)]  For $f,g \in \Mor^{\alpha}_{\C}(A,B)$, we say that $f$ and $g$ are {\em  $r$-equivalent} for some $r \geq0$ if $$i_{\alpha, \alpha+r}(f-g)=0~.~$$ We write $f \simeq_r g$ if $f$ and $g$ are $r$-equivalent. 
\item[(ii)] Two morphisms, $f \in \Mor_{\C}^\alpha(A,B)$ and $g \in \Mor_{\C}^{\beta}(A,B)$,  are {\em $\infty$-equivalent}, written $f\simeq_{\infty} g$,  if there exist $r, r' \geq 0$ with $\alpha+r = \beta+r'$ such that 
$i_{\alpha, \alpha+r}(f) = i_{\beta, \beta+r'}(g)$.
 \item[(iii)] An object $K \in {\rm Obj}(\C)$ is called {\em  $r$-acyclic} for some $r \geq 0$ 
if  $\mathds{1}_K\in \Mor_{\C}^{0}(K,K)$ has the property that $\mathds{1}_K\simeq_{r} 0$.
\end{itemize}
\end{dfn}
Obviously, if $f\simeq_{r} g$ then $f\simeq_{s} g$  for all $s\geq r$. Notice also that $\simeq_{r}$
is indeed an equivalence relation. Indeed, for $r\not=\infty$ this follows immediately from the fact that $i_{\alpha,\beta}: \Mor_{\C}^{\alpha}(A,B)\to \Mor_{\C}^{\beta}(A,B)$ is a linear map and it is an easy exercise for $r=\infty$ .

\

\begin{dfn} \label{dfn-c0-cinf} Given a persistence category $\C$,
  there are two categories naturally associated to it as follows:
\begin{itemize}
\item[(i)] the $0$-level of $\C$, denoted $\C_0$, which is the category with the same objects as $\C$ and, for any $A, B \in {\rm Obj}(\C)$, with $\Mor_{\C_0}(A,B) := \Mor_{\C}^0(A,B)$. 
\item[(ii)] the limit category (or $\infty$-level) of $\C$, denoted
  $\C_{\infty}$, that again has the same objects as $\C$ but for any
  $A,B \in {\rm Obj}(\C)$,
  $\Hom_{\C_{\infty}}(A,B) := \varinjlim_{\alpha \to \infty}
  \Mor^\alpha_{\C}(A,B)$, where the direct limit is taken with respect
  to the morphisms
  $i_{\alpha,\beta}: \Mor_{\C}^\alpha(A,B) \to \Mor_{\C}^\beta(A,B)$ for any
  $\alpha \leq \beta$.
\end{itemize}
\end{dfn}

%\begin{remark} For $\C_0$ defined in Definition \ref{dfn-c0-cinf}, the choice at level $0$ is crucial, otherwise %the resulting category will not be closed under the composition (cf. the item (2) in Definition \ref{dfn-pc}).
%\end{remark}
%\begin{remark} 

\begin{rem}\label{rem:gen-pers} (a) In general, a persistence category is not pre-additive as the hom-sets ${\rm Hom}_{\C}(A,B)$ are not always abelian groups. However, it is easy to see that both $\mathcal{C}_{0}$   and   $\C_{\infty}$ are pre-additive
(the proof is immediate in the first case and a simple exercise in the second).

(b) The limit category $\C_{\infty}$ can be equivalently defined as a quotient category $\C/ \simeq_{\infty}$ which is defined by ${\rm Obj}(\C/ \simeq_{\infty}) = {\rm Obj}(\C)$ and ${\rm Hom}_{\C/ \simeq_{\infty}}(A,B)  ={\rm Hom}_{\C}(A,B)/ \simeq_{\infty}$. 
\end{rem}
 
Two objects $A, B \in {\rm Obj}(\C)$ are  said {\em  $0$-isomorphic}, we write $ A \equiv B$, if they
are isomorphic in the category $\C_{0}$. This is obviously an equivalence relation and it preserves
$r$-acyclics in the sense that if $K \simeq_r 0$ and $K \equiv K'$, then $K' \simeq_r 0$.

\subsection{Persistence functors}
Persistence categories come associated notions of persistence functors
and natural transformations relating them, as described below.

\begin{dfn} \label{dfn-per-functor} Given two persistence categories
  $\C$ and $\C'$, a {\em persistence functor $\F: \C \to \C'$} is a
  functor which is compatible with the persistence structures.
  % in the sense that for any $A, B \in {\rm Obj}(\C)$, $\F$ induces a
  % natural transformation $\F_{A,B}: E_{A,B} \to E_{\F(A), \F(B)}$.
  \pbnote{More explicitly, the action of $\mathcal{F}$ on morphisms
    restricts to maps
    $(\mathcal{F}_{A,B})_r: \Mor_{\mathcal C}^{r}(A,B) \to
    \Mor_{\mathcal C'}^{r}(\F(A),\F(B))$ defined for any
    $A, B \in {\rm Obj}(\C)$ and $r \in \mathbb{R}$. Moreover, for
    every $r \leq s$ we have the following commutative diagram:}
  % for each $i_{r,s}: r \to s$ where
  % $r \leq s$, one has the following commutative diagram,
  \begin{equation} \label{functor-com} \xymatrixcolsep{4pc} \xymatrix{
      \Mor_{\mathcal C}^{r}(A,B) \ar[r]^-{(\F_{A,B})_r}
      \ar[d]_-{i^{\C}_{r,s}}
      & \Mor_{\mathcal C'}^{r}(\F(A),\F(B)) \ar[d]^-{i^{\C'}_{r, s}}\\
      \Mor_{\mathcal C}^{s}(A,B) \ar[r]^-{(\F_{A,B})_s} &
      \Mor_{\mathcal C'}^s(\F(A),\F(B))}
  \end{equation}
  where $i_{r, s}^{\mathcal C}$ and $i_{r, s}^{\mathcal C'}$ are
  persistence structure maps in $\mathcal C$ and $\mathcal C'$,
  respectively. In particular, for each
  $\bar{f} \in {\rm Hom}_{\C}(A,B)$ with $\ceil*{\bar{f}~} = r$, we
  have $\ceil*{\F_{A,B}(\bar{f}~)} = r$.
\end{dfn}

% \begin{ex} \label{ex-pf} Given a persistence category $\C$, fix
%   $\alpha \in \R$ and consider a new persistence category
%   $\C[\alpha]$ defined by ${\rm Obj}(\C[\alpha]) = {\rm Obj}(\C)$
%   and such that for each $A, B \in {\rm Obj}(\C)$,
%   $\Mor^r_{\C[\alpha]}(A,B) = \Mor^{r+\alpha}(A,B)$ with
%   the %persistence
%   structure given by
%   $i^{\C[\alpha]}_{r,s} = i^{\C}_{r+\alpha,s+\alpha}$ for any
%   $r\leq s$.  We call this $\C[\alpha]$ the $\alpha$-shift of $\C$.
%   There exists an obvious persistence functor
%   $\Sigma^{\alpha}: \C \to \C[\alpha]$ which is the identity
%   on %objects and is defined on morphisms by
%   \begin{equation} \label{shift-functor} (\Sigma^{\alpha}_{A,B})_r =
%     i^{\C}_{r, r+\alpha}: \Mor^r_{\C}(A,B) \to
%     \Mor^r_{\C[\alpha]}(A,B)
%   \end{equation}
%   for any $A, B \in {\rm Obj}(\C)$ and $r \in \R$.  Clearly, if
%   $\alpha=0$, we have $\Sigma^{\alpha}=id$.

%   Indeed, this is a natural transformation due the commutativity of
%   the following diagram,
%   \begin{equation*} 
%     \xymatrixcolsep{6pc} \xymatrix{
%     %     \Mor_{\mathcal C}^{r}(A,B)
%     %     \ar[r]^-{(\Sigma_{A,B}^{\alpha})_r = i^{\C}_{r, r+\alpha}}
%     %     \ar[d]_-
%     %     {i^{\C}_{r,s}} & \Mor_{\mathcal C[\alpha]}^{r}(A,B)
%     %     \ar[d]^-{i^{\C[\alpha]}_{r, s} = i^{\C}_{r+
%     %     \alpha, s+\alpha}}\\
%     %     \Mor_{\mathcal C}^{s}(A,B)
%     %     \ar[r]^-{(\Sigma_{A,B})^{\alpha}_s= i^{\C}_{s, s+\alpha}}
%     %     &
%     %     \Mor_{\mathcal C[\alpha]}^s(A,B)}
%   \end{equation*}
%   for any $i_{r,s}: r \to s$ in $(\R, \leq)$ where $r \leq s$. 
% \end{ex}

For any functor $E: (\R, \leq) \to {\rm Vect}_{\k}$ and
$\alpha \in \R$, we denote by
$\Sigma^{\alpha} E: (\R, \leq) \to {\rm Vect}_{\k}$ the $\alpha$-shift
of $E$ defined by $\Sigma^{\alpha}E(r) = E(r+ \alpha)$ and
$\Sigma^{\alpha}E(i_{r,s}) = E(i_{r + \alpha, s+\alpha})$ for any
$i_{r,s}: r \to s$, $r \leq s$.

\begin{dfn} \label{dfn-nt} Given two persistence functors between two
  persistence categories $\F, \G: \C \to \C'$, a {\em persistence
    natural transformation} $\eta: \F \to \G$ is a natural
  transformation for which there exists $r \in \mathbb{R}$ such that
  for any $A \in {\rm Obj}(\C)$, the morphism
  $\eta_A: \F(A) \to \G(A)$ belongs to $\Mor_{\C'}^r(\F(A), \G(A))$. We say
  that $\eta$ is a natural transformation of shift $r$.
%
%  of shift $r$ is a
%  natural transformation such that for any $A \in {\rm Obj}(\C)$,
%  $\eta_A: \F(A) \to \G(A)$ is an element in
%  ${\rm Hom}_{\C'}(\F(A), \G(A))$ with $\ceil*{\eta_A} = r$.
%
%  the following two conditions are
%  satisfied.
%\begin{itemize}
%\item[(i)] For any $A \in {\rm Obj}(\C)$, $\eta_A: \F(A) \to \G(A)$ is
%  an element in ${\rm Hom}_{\C'}(\F(A), \G(A))$ with
%  $\ceil*{\eta_A} = r$.
%\item[(ii)] For each $X \in {\rm Obj}(\C')$ and $A \in {\rm Obj}(\C)$,
%  $\eta$ induces natural tranformations
%  \[ \eta_{X, A}: E_{X, \F(A)} \to \Sigma^r E_{X, \G(A)},
%    \,\,\,\,\mbox{and}\,\,\,\,\, \eta^*_{A,X}: E_{\G(A), X} \to
%    \Sigma^{-r} E_{\F(A), X} \] where $\eta_{X,A} = \eta_A \circ$ and
%  $\eta^*_{A,X} = \circ \eta_A$.
%\end{itemize}
\end{dfn}

\begin{remark} \label{rmk-pf} (a) \pbnote{The morphisms
    $\eta_A: \F(A) \to \G(A)$, $A \in \text{Obj}(\C)$, give rise to
    the following commutative diagrams for all
    $X \in \text{Obj}(\C)$:}
  \begin{equation} \label{nt-dia-2} \xymatrixcolsep{3pc} \xymatrix{
      {\rm Mor}_{\C'}^{\alpha}(X, \F(A)) \ar[r]^-{\eta_A \circ}
      \ar[d]_-{i_{\alpha, \beta}} & {\rm Mor}_{\C'}^{\alpha+r}(X,
      \G(A))
      \ar[d]^-{i_{\alpha+r, \beta+r}} \\
      {\rm Mor}_{\C'}^{\beta}(X, \F(A)) \ar[r]^-{\eta_A \circ} & {\rm
        Mor}_{\C'}^{\beta+r}(X, \G(A))}
  \end{equation}
  and
  \begin{equation} \label{nt-dia-3} \xymatrixcolsep{3pc} \xymatrix{
      {\rm Mor}_{\C'}^{\alpha}(\G(A), X) \ar[r]^-{\circ \eta_A}
      \ar[d]_-{i_{\alpha, \beta}}
      & {\rm Mor}_{\C'}^{\alpha+r}(\F(A), X) \ar[d]^-{i_{\alpha+r, \beta+r}} \\
      {\rm Mor}_{\C'}^{\beta}(\G(A), X) \ar[r]^-{\circ \eta_A} & {\rm
        Mor}_{\C'}^{\beta+r}(\F(A), X)}
  \end{equation}
  commute for any $\alpha \leq \beta \in \R$.

  (b) Given two persistence categories $\C, \C'$, the persistence
  functors themselves form a persistence category denoted by
  $\mathcal{P}{\rm Fun}(\C, \C')$, where
  \[ {\rm Hom}_{\mathcal{P}{\rm Fun}(\C, \C')}(\F, \G) =
    \left\{\left(\eta, \ceil*{\eta}\right) \,\bigg|
      \, \begin{array}{l} \mbox{$\eta$ is a natural transformation} \\
           \mbox{from $\F$ to $\G$ of shift $\ceil*{\eta}$}\end{array}
  \right\}.\] When $\C = \C'$, simply denote
  $\mathcal{P}{\rm Fun}(\C, \C')$ by $\mathcal{P}{\rm End}(\C)$. It is
  easy to verify that $\mathcal{P}{\rm End}(\C)$ admits a strict
  monoidal structure.
  % where the corresponding bifunctor is given by the composition of
  % functors.
\end{remark}

\subsection{Shift functors}
The role of shift functors, to be introduced below, is to allow
morphisms of arbitrary shift (as well as $r$-equivalences) to be
represented as morphisms of shift $0$ with the price of ``shifting''
the domain (or the target) - see Remark \ref{rem:c0repl}.  This turns
out to be very helpful in the study of triangulation for persistence
categories.

\

View the real axis $\R$ as a strict monoidal category $(\R, +)$
induced by the additive group structure of $\R$.  In other words,
${\rm Obj}(\R) = \{x \,| \, x \in \R\}$ and for any
$x, y \in {\rm Obj}(\R)$, ${\rm Hom}_{\R}(x,y) = \{\eta_{x,y}\}$ such
that $\eta_{x,x} = \mathds{1}_x$ and, for any $x, y, z \in \R$,
$\eta_{y,z} \circ \eta_{x,y} = \eta_{x,z}$. In particular,
$\eta_{x,y} \circ \eta_{y,x} = \eta_{y,y}$ and
$\eta_{y,x} \circ \eta_{x,y} = \eta_{x,x}$.  Both $\eta_{x,x}$ and
$\eta_{y,y}$ are the respective identities and therefore each morphism
$\eta_{x,y}$ is an isomorphism.  The monoidal structure is defined by
$\oplus (x,y) := x+y$ on objects and for any two morphisms
$(\eta_{r,r'}, \eta_{s, s'})$ we have
$\oplus (\eta_{r,r'}, \eta_{s, s'}) := \eta_{r+s, r'+s'}$.

\begin{dfn} \label{dfn-shift-functor} Let $\C$ be a persistence
  category. A {\em shift functor} on $\C$ is a strict monoidal functor
  $\Sigma: (\R, +) \to \mathcal{P}{\rm End}(\C)$ such that
  $\Sigma(\eta_{x,y}): \Sigma(x) \to \Sigma(y)$ is a natural
  transformation of shift $\ceil*{\Sigma(\eta_{x,y})} = y-x$ for any
  $x, y \in \R$ and $\eta_{x,y} \in {\rm Mor}_{\R}(x,y)$.
\end{dfn}

For later use, denote
$\Sigma^r := \Sigma(r) \in \mathcal{P}{\rm End}(\C)$ and, for brevity,
we denote $\Sigma(\eta_{r,s})$ by $\eta_{r,s}$ for $r, s\in \R$ and we
let $(\eta_{r,s})_{A}$ be the respective morphism
$\Sigma^rA \to \Sigma^{s}A$.

\begin{remark} \label{rem:shift} Since $\Sigma$ is a strict monoidal
  functor, it preserves the product structure. Therefore,
  $\Sigma^s \circ \Sigma^r = \Sigma^{r+s}$ and
  $\Sigma^{0} = \mathds{1}$. Moreover, since each $\eta_{r,s}$
  \pbnote{is an isomorphism in $(\mathbb{R}, +)$,} the corresponding
  natural transformation $\eta_{r,s}$ is a natural isomorphism. We
  also have $\Sigma^{r}(\eta_{s,s'})_{A}=(\eta_{s+r,s'+r})_{A}$ for
  each object $A$ in $\C$ and all $r,s,s'\in \R$.

  In particular, this implies that for any $Y, A \in {\rm Obj}(\C)$
  and $\alpha \in \R$, we have an isomorphism,
  \begin{equation} \label{nt-iso} \Mor^{\alpha}_{\C}(Y, A)
    \xrightarrow{(\eta_{0,r})_A\circ} \Mor^{\alpha + r}_{\C}(Y,
    \Sigma^r A).
\end{equation}
%that is a consequence of the existence of the following two commutative diagrams:
%\[\xymatrixcolsep{3pc} \xymatrix{ 
%  \Mor^{\alpha}_{\C}(Y, A) \ar[r]^-{(\eta_{0,r})_A \circ}
%  \ar@/_2pc/[rr]^-{\mathds{1}} & %
%  \Mor^{\alpha + r}_{\C}(Y, \Sigma^r A) \ar[r]^-{(\eta_{r,0})_A\circ}
%  & \Mor^{\alpha}_{\C}(Y, A)} \] and
%\[ \xymatrixcolsep{3pc} \xymatrix{ \Mor^{\alpha + r}_{\C}(Y, \Sigma^r
%  A) \ar[r]^-{(\eta_{r,0})_A \circ} \ar@/_2pc/[rr]^-{\mathds{1}}
%  & %\Mor^{\alpha}_{\C}(Y, A) \ar[r]^-{(\eta_{0,r})_A\circ} & \Mor^{\alpha+r}_{\C}(Y, \Sigma^r A)} ~.~
%\] 
Similarly, for any $A, X \in {\rm Obj}(\C)$ and $\alpha \in \R$, we
have an isomorphism
\begin{equation} \label{nt-iso-2} \Mor^{\alpha-r}_{\C}(\Sigma^r A, X)
  \xrightarrow{\circ(\eta_{0,r})_A} \Mor^{\alpha}_{\C}(A, X).
\end{equation}
Further, for any $A, B \in {\rm Obj}(\C)$, the isomorphisms
(\ref{nt-iso}) and (\ref{nt-iso-2}) imply the existence of an
isomorphism
\begin{equation} \label{nt-iso-3} \Mor^{\alpha+s-r}_{\C}(\Sigma^r A,
  \Sigma^sB) \simeq \Mor^{\alpha}_{\C}(A, B).
\end{equation}
% In fact, we have
% $\Mor^{\alpha+s-r}_{\C}(\Sigma^r A, \Sigma^sB) \simeq
% \Mor^{\alpha+s}_{\C}(A, \Sigma^s B) \simeq \Mor^{\alpha}_{\C}(A,
% B)$.
In particular, when $r=s$, we get a canonical isomorphism:
\begin{equation} \label{sigma-mor} \Sigma^r: \Mor^{\alpha}_{\C}(A, B)
  \to \Mor^{\alpha}_{\C}(\Sigma^rA, \Sigma^rB).
\end{equation}
% More explicitly, isomorphism $\Sigma^r$ is the composition of two
% isomorphisms $\circ (\eta_{r,0})_A$ and % $(\eta_{0,r})_B \circ$.
Finally, diagrams (\ref{nt-dia-2}) and (\ref{nt-dia-3}) imply that the
following diagrams obtained by setting $\F = \Sigma^0$ and
$\G = \Sigma^r$
\begin{equation} \label{nt-dia-4} \xymatrixcolsep{3pc} \xymatrix{ {\rm
      Mor}_{\C}^{\alpha}(X, A) \ar[r]^-{(\eta_{0,r})_A \circ}
    \ar[d]_-{i_{\alpha, \beta}} & {\rm Mor}_{\C}^{\alpha+r}(X,
    \Sigma^r A)
    \ar[d]^-{i_{\alpha+r, \beta+r}} \\
    {\rm Mor}_{\C}^{\beta}(X, A) \ar[r]^-{(\eta_{0,r})_A\circ} & {\rm
      Mor}_{\C}^{\beta+r}(X, \Sigma^r A)}
\end{equation}
and
\begin{equation} \label{nt-dia-5} \xymatrixcolsep{3pc} \xymatrix{ {\rm
      Mor}_{\C}^{\alpha}(\Sigma^r A, X) \ar[r]^-{\circ (\eta_{0,r})_A}
    \ar[d]_-{i_{\alpha, \beta}} & {\rm Mor}_{\C}^{\alpha+r}(A, X)
    \ar[d]^-{i_{\alpha+r, \beta+r}} \\
    {\rm Mor}_{\C}^{\beta}(\Sigma^r A, X) \ar[r]^-{\circ
      (\eta_{0,r})_A} & {\rm Mor}_{\C}^{\beta+r}(A, X)}
\end{equation}
are commutative for any $\alpha \leq \beta$. All the horizontal
morphisms in (\ref{nt-dia-4}) and (\ref{nt-dia-5}) are isomorphisms
but the vertical morphisms (which are the persistence structure
morphisms) are not necessarily so. \end{remark}

% \subsection{Acyclics}

% \begin{dfn} An object $K \in {\rm Obj}(\Sigma)$ is called {\em
%   $r$-acyclic} for some $r \geq 0$ if the identity
%   $\mathds{1}_K\in \Mor^{0}(K,K)$ has the property that
%   $i_{0,r}(\mathds{1}_K)\in \Mor^{r}(K,K)$ vanishes.
%\end{dfn}

Assume that $\C$ is a persistence category (with persistence structure
morphisms denoted by $i_{r,s}$) endowed with a shift functor $\Sigma$.
To simplify notation, for $A \in {\rm Obj}(\C)$, \pbnote{$r \geq 0$,}
we consider $(\eta_{r,0})_A \in \Mor_{\C}^{-r}(\Sigma^rA,A)$ and
$(\eta_{0,-r})_A \in \Mor_{\C}^{-r}(A, \Sigma^{-r}A)$ and we will denote
below by $\eta^{A}_{r}$ the \jznote{following maps 
\begin{equation} \label{dfn-eta-map}
\eta^A_r = i_{-r,0}((\eta_{r,0})_A) \,\,\,\mbox{or} \,\,\,  \eta^A_r = i_{-r,0}((\eta_{0,-r})_A).
\end{equation} }
Thus $\eta^{A}_{r}\in \Mor_{\C}^{0}(\Sigma^{r} A, A)$ or
$\eta^{A}_{r}\in\Mor_{\C}^{0}(A,\Sigma^{-r}A)$, depending on the context. \jznote{Note that there is no ambiguity of the notation $\eta_r^A$ due to the canonical identification via $\Sigma^r$ in (\ref{sigma-mor}).} 
% \begin{equation} \label{dfn-eta-A} \eta^A: = i_{-r,0}
%   ((\eta_{r,0})_A) = i_{-r,0}((\eta_{0,-r})_A) \in \Mor^0(\Sigma^rA,
%   A).
%\end{equation}
%Diagrams (\ref{nt-dia-2}) and (\ref{nt-dia-3}) guarantee that
% $\eta^A$ is well-defined.  Let us start from the following
% elementary result.
The notions discussed before, $r$-acyclicity, $r$-equivalence and so
forth, can be reformulated in terms of compositions with appropriate
shift morphisms $\eta^{A}_{r}$.

The next lemma is a characterization of $r$-equivalence that follows
easily from the diagrams (\ref{nt-dia-4}) and (\ref{nt-dia-5}).
\begin{lemma} \label{lemma-comp} Suppose that
  $f \in \Mor^{\alpha}(A,B)$. Then $i_{\alpha, \alpha + r}(f) =0$ for
  some $r \geq 0$ if and only if $f \circ \eta^A_{r}=0$ in
  $\Mor^{\alpha}(\Sigma^r A, B)$ and (equivalently) if and only if
  $\eta^B_{r} \circ f =0$ in $\Mor^{\alpha}(A,
  \Sigma^{-r}B)$.
\end{lemma}

In particular, we easily see that for two morphisms
$f,g\in \Mor^{\alpha}(A,B)$, $f\simeq_{r} g$ if and only if
$f\circ \eta^{A}_{r}=g\circ \eta^{A}_{r}$. Moreover, $r$-equivalence
is preserved under shifts.  Further, it is immediate to check that
$f \in \Mor^\alpha(A,B)$ and $g \in \Mor^{\beta}(A,B)$ are
$\infty$-equivalent if and only if there exist $r, r' \geq 0$ with
$\alpha+r = \beta+r'$ such that
\[ f \circ \eta_r^A = g \circ \eta_{r'}^{A}\,\,\,\,\mbox{in} \,\,
  \Mor^{\alpha+r}(A,B) \] where we identify both
$\Mor^{\alpha}(\Sigma^r A, B) $ and $\Mor^{\beta}(\Sigma^{r'} A, B)$
with $\Mor^{\alpha+r}(A, B)$ through the canonical isomorphisms in
Remark \ref{rem:shift}.

Here is a similar characterization of $r$-acyclicity. 

\begin{lemma} \label{lemma-acyclic} $K \simeq_r 0$ is equivalent to
  each of the following:
  \begin{itemize}
  \item[(i)] \pbnote{$\eta_{r}^{K} = 0$}.
  \item[(ii)]
    $i_{\alpha, \alpha+r}: \Mor^{\alpha}(A,K) \to
    \Mor^{\alpha+r}(A,K)$ vanishes for any $\alpha \in \R$ and $A$.
  \item[(iii)]
    $i_{\alpha, \alpha+r}: \Mor^{\alpha}(K,A) \to \Mor^{\alpha+r}(K,
    A)$ vanishes for any $\alpha \in \R$ and $A$.
\end{itemize}
\end{lemma}

\begin{proof} Point (i) is an immediate consequence of the definition of $r$-acylics in Definition \ref{def:acyclics} and of Lemma \ref{lemma-comp} applied for $A,B=K$, 
$f=\mathds{1}_{K}$. 
We now prove (ii). The proof of (iii) is similar and will be omitted.
It is obvious that (ii) implies $K\simeq_{r} 0$ by specializing to $A=K$, $\alpha=0$
 and applying $i_{0,r}$ to $\mathds{1}_K$. 
To prove the converse, we first use diagram (\ref{nt-dia-4}) to deduce that 
the map $i_{\alpha,\alpha+r}$ factors as below: 
\begin{equation} \label{com-da}  \xymatrixcolsep{5pc} \xymatrix{ 
\Mor^{\alpha}(A, K) \ar[r]^-{(i_{-r,0}(\eta_{0,-r})_K) \circ} \ar@/_2pc/[rr]^-{i_{\alpha, \alpha+r}} & \Mor^{\alpha}(A, \Sigma^{-r}K) \ar[r]^-{(\eta_{-r,0})_K\circ} & \Mor^{\alpha+r}(A, K)} ~.~
\end{equation}
%Indeed, from diagram (\ref{nt-dia-4}) we deduce the following commutative diagram: 
%\begin{equation} \label{com-da} 
%\xymatrixcolsep{3pc} \xymatrix{
%{\rm Mor}^{\alpha}(A, K) \ar[r]^-{(\eta_{0,-r})_K \circ} \ar[d]_-{i_{\alpha, \alpha+r}} & {\rm Mor}%^{\alpha-r}(A, \Sigma^{-r} K) \ar[d]^-{i_{\alpha-r, \alpha}} \\
%{\rm Mor}^{\alpha+r}(A, K) \ar[r]^-{(\eta_{0,-r})_K\circ} & {\rm Mor}^{\alpha}(A, \Sigma^{-r} K)}~.~
%\end{equation}
%Thus, for any $f \in {\rm Mor}^{\alpha}(A, K)$, 
%\begin{align*}
%(\eta_{-r,0})_K \circ i_{-r,0}(\eta_{0,-r})_K \circ f & = (\eta_{-r,0})_K \circ i_{\alpha-r,\alpha}((\eta_{0, %-r})_K \circ f) \\
%& = (\eta_{-r,0})_K \circ (\eta_{0, -r})_K \circ i_{\alpha, \alpha+r}(f) = i_{\alpha, \alpha+r}(f).
%\end{align*}
Therefore, since $(\eta_{-r,0})_K\circ$ is an isomorphism, for any $f \in \Mor^{\alpha}(A, K)$
we have $i_{\alpha, \alpha+r}(f) = 0$ if and only if $i_{-r,0}(\eta_{0,-r})_K \circ f =0$. From point 
(i) we know that this relation is true for $f=\mathds{1}_{K}$. 
%We now set $A = K$ in (\ref{com-da}) 
%and then for any $f \in {\rm Mor}^{\alpha}(K, K)$, we have $i_{-r,0}(\eta_{0,-r} \circ f) = (\eta_{0,-r})_K %\circ i_{0,r}(f)$. 
%and given that $K \simeq_r 0$, by Lemma \ref{lemma-comp},
%we have $i_{-r,0}(\eta_{0,-r}) \circ \mathds{1}_K=0$.
% and we deduce:
%\begin{equation} \label{computation-identity}
%i_{-r,0}(\eta_{0,-r} \circ \mathds{1}_K) = (\eta_{0,-r})_K \circ i_{0,r}(\mathds{1}_K)=0. 
%\end{equation}
Now, for any $f\in \Mor^{\alpha}(A,K)$, we write $f = \mathds{1}_K \circ f$ and conclude 
 $i_{-r,0}(\eta_{0,-r})_K \circ f =i_{-r,0}(\eta_{0,-r} \circ \mathds{1}_K)\circ f=0$. \end{proof}

%\begin{remark} The $r$-acyclic notation can be naturally generalized to a filtered dg-category $\A$ (one definition will be given in Remark \ref{rmk-dg-cat}) by requiring $(i_{0,r})_*(\mathds{1}_K) = 0$ (cf. Remark \ref{rmk-ac}) in the associated homology category ${\rm H}(\A)$, where $(i_{0,r})_*$ the induced persistence structure maps. In other words, $i_{0,r}(\mathds{1}_K) \in \Mor^r(K,K)$ is a boundary. Similarly, the notation of being $r$-equivalent, $f\simeq_r g$, can also be generalized to a filtered dg-category in the same way. \end{remark} 

In particular, we see that $K$ is $r$-acyclic if and only if any of
its shifts $\Sigma^{\alpha} K$ is so.

\begin{rem} \label{rem:c0repl} (a) Assume that $\C$ is a persistence
  category endowed with a shift functor $\Sigma$ and that
  $f_{1},f_{2}\in \Mor_{\C}^{\alpha}(A,B)$, then, for all practical
  purposes, we may replace $f_{i}$ with the $\C_{0}$ morphisms
  $\tilde{f}_{i}\in \Mor_{\C}^{0}(\Sigma^{\alpha}A,B)$ where
  $\tilde{f}_{i}=f_{i}\circ\eta_{0,-\alpha}$. The property
  $f_{1}\simeq _{r} f_{2}$ is equivalent to
  $$\tilde{f}_{1}\circ \eta_{r}^{\Sigma^{\alpha}A} \simeq_{0}  \tilde{f}_{2}
  \circ \eta_{r}^{\Sigma^{\alpha}A}$$ which is a relation in $\C_{0}$.

  (b) For a given persistence category $\C$, the existence of a shift
  functor $\Sigma$ on $\C$ \jznote{sometimes} is a non-trivial additional structure. Such
  a structure often happens to be available in geometric examples. At
  the same time, there is an obvious way to formally complete any
  persistence category $\C$ to a larger persistence category
  $\tilde{\C}$ that is endowed with a canonical shift functor. This is
  achieved by formally adding objects $\Sigma^{r}X$ for each $r\in \R$
  and $X\in \mathrm{Obj}(\C)$ and defining morphisms such that the
  relations in Remark \ref{rem:shift} are satisfied.
\end{rem}

\subsection{An example of a persistence category} \label{ssec-pm}

We give an example of a persistence category that is constructed from persistence $\k$-modules. To some extent, this is the motivation of the definition of a persistence category. Recall that for a persistence $\k$-module $\V$, the notation $\V[r]$ denotes another persistence $\k$-module which comes from an $r$-shift from $\V$ in the sense that 
\[ \V[r]_t = \V_{r+t} \,\,\,\,\,\mbox{and}\,\,\,\,\, \iota^{\V[r]}_{s,t} = \iota^{\V}_{s+r, t+r}. \]
A persistence morphism $f: \V \to \W$ is an $\R$-family of morphisms $f = \{f_t\}$ such that it commutes with the persistence structure maps of $\V$ and $\W$, i.e., $f_{t} \circ \iota^{\V}_{s,t} = \iota^{\W}_{s,t} \circ f_s$. Similarly, one can define $r$-shifted persistence morphism $f[r]$ where $(f[r])_t = f_{r+t}$. 

\begin{remark} There are cases of $\V$ and $\W$ such that the only {\it persistence} morphism from $\V$ to $\W$ is the zero morphism. For instance, $\V = \I_{[0, \infty)}$ and $\W =\I_{[1, \infty)}$. Note that $\W = \V[-1]$. On the other hand, there always exists a well-defined persistence morphism from $\V[r]$ to $\V[s]$ for $r \leq s$, which is just $\iota^{\V}_{\cdot +r, \cdot +s}$. \end{remark}

Let $\mathcal P{\rm Mod}_{\k}$ be the category of persistence $\k$-modules, then we claim that $\mathcal P{\rm Mod}_{\k}$ can be enriched to be a persistence category $\C^{\mathcal P{\rm Mod}_{\k}}$. Indeed, let ${\rm Obj}(\C^{\mathcal P{\rm Mod}_{\k}}) = {\rm Obj}(\mathcal P{\rm Mod}_{\k})$, and for objects $\V, \W$ in $\mathcal P{\rm Mod}_{\k}$, define 
\begin{equation} \label{per-mor}
{\rm Mor}_{\C^{\mathcal P{\rm Mod}_{\k}}}(\V, \W) : = \left\{\{{\rm Hom}_{\mathcal P{\rm Mod}_{\k}}(\V, \W[r])\}_{r\in \R}; \{i_{r,s}\}_{r \leq s \in \R} \right\}.
\end{equation}
Here, ${\rm Mor}_{\C^{\mathcal P{\rm Mod}_{\k}}}^r(\V, \W) = {\rm Hom}_{\mathcal P{\rm Mod}_{\k}}(\V, \W[r])$, and ${\rm Hom}_{\mathcal P{\rm Mod}_{\k}}(\cdot, \cdot)$ consists of persistence morphisms. For any $r \leq s$, the well-defined persistence morphism $\iota^{\W}_{\cdot+r,\cdot +s}: \W[r] \to \W[s]$ induces structure maps $i_{r,s} :=  \iota^{\W}_{\cdot+r,\cdot +s} \circ$  in (\ref{per-mor}). Moreover, the composition $\circ_{(r,s)}: \Mor_{\C^{\mathcal P{\rm Mod}_{\k}}}^{r}(\U, \V) \times \Mor_{\C^{\mathcal P{\rm Mod}_{\k}}}^s(\V, \W) \to \Mor_{\C^{\mathcal P{\rm Mod}_{\k}}}^{r+s}(\U, \W)$ is defined by 
\[ (f, g) \mapsto g[r] \circ f \]
where we use the identification ${\rm Hom}_{\mathcal P{\rm Mod}_{\k}}(\V, \W)[r] ={\rm Hom}_{\mathcal P{\rm Mod}_{\k}}(\V[r], \W[r])$ for any $r \in \R$. Moreover, for the following diagram where $r \leq r'$ and $s \leq s'$, 
\begin{equation} \label{ex-per-dia}
\xymatrix{
{\rm Hom}_{\mathcal P{\rm Mod}_{\k}}(\U, \V[r]) \times {\rm Hom}_{\mathcal P{\rm Mod}_{\k}}(\V,\W[s]) \ar[r]^-{\circ_{(r,s)}} \ar[d]_-{i_{r,r'} \times i_{s,s'}} & {\rm Hom}_{\P}(\U,\W[r+s]) \ar[d]^-{i_{r+s, r'+s'}} \\
{\rm Hom}_{\mathcal P{\rm Mod}_{\k}}(\U, \V[r']) \times {\rm Hom}_{\mathcal P{\rm Mod}_{\k}}(\V,\W[s']) \ar[r]^-{\circ_{(r',s')}} & {\rm Hom}_{\P}(\U,\W[r'+s'])}
\end{equation}
we have 
\begin{align*}
(\circ_{(r',s')} \circ (i_{r,r'} \times i_{s,s'}))(f,g) & = \circ_{(r',s')}(i_{r,r'}(f), i_{s,s'}(g)) \\
& = i_{s,s'}(g)[r'] \circ i_{r,r'}(f) \\
& = (\iota^{\W}_{\cdot +s, \cdot + s'} \circ g)[r'] \circ (\iota^{\W}_{\cdot + r, \cdot + r'} \circ f) \\
& = \iota^{\W}_{\cdot + r'+s, \cdot + r' + s'} \circ g[r'] \circ \iota^{\V}_{\cdot + r, \cdot + r'} \circ f \\
& = \iota^{\W}_{\cdot + r' + s, \cdot + r'+s'} \circ \iota^{\W}_{\cdot + r + s, \cdot + r' + s} \circ g[r] \circ f\\
& = \iota^{\W}_{\cdot + r+s, \cdot + r'+s'} \circ (g[r] \circ f) = (\iota_{r+s, r'+s'} \circ \circ_{(s,t)})(f,g)
\end{align*}
where the fifth equality is due to the fact that $g$ is a persistence morphism (so, in particular, it commutes with the persistence structure maps). Therefore, the diagram (\ref{ex-per-dia}) is commutative and $\C^{\mathcal P{\rm Mod}_{\k}}$ is a persistence category by Definition \ref{dfn-pc}. 

\medskip

Since ${\rm Mor}_{\C^{\mathcal P{\rm Mod}_{\k}}}^0(\V, \W) = {\rm Hom}_{\mathcal P{\rm Mod}_{\k}}(\V, \W)$, we know $\C^{\mathcal P{\rm Mod}_{\k}}_0 = \mathcal P{\rm Mod}_{\k}$.  An example of a persistence endofunctor on $\C^{\mathcal P{\rm Mod}_{\k}}$, denoted by $\Sigma^{\alpha}: \C^{\mathcal P{\rm Mod}_{\k}} \to \C^{\mathcal P{\rm Mod}_{\k}}$, is defined by 
\begin{equation} \label{ex-per-fun} 
\Sigma^{\alpha}(\V) : = \V[-{\alpha}] \,\,\,\,\mbox{and}\,\,\,\, \Sigma^{\alpha}(f) = f[-\alpha]
\end{equation}
for any $\alpha \in \R$. By Definition \ref{dfn-per-functor}, we need to confirm the commutativity of the following diagram where $r\leq s$ in $\R$, 
\begin{equation*} 
\xymatrixcolsep{4pc} \xymatrix{
\Mor_{\mathcal C^{\mathcal P{\rm Mod}_{\k}}}^{r}(\V,\W)  \ar[r]^-{\Sigma^{\alpha}} \ar[d]_-{i_{r,s}} & \Mor_{\mathcal C^{\mathcal P{\rm Mod}_{\k}}}^{r}(\Sigma^{\alpha} \V, \Sigma^{\alpha} \W) \ar[d]^-{i_{r,s}}\\
\Mor_{\mathcal C^{\mathcal P{\rm Mod}_{\k}}}^{s}(\V,\W) \ar[r]^-{\Sigma^{\alpha}} & \Mor_{\mathcal C^{\mathcal P{\rm Mod}_{\k}}}^s(\Sigma^{\alpha} \V,\Sigma^{\alpha} \W)} 
\end{equation*}
Indeed, for any $f \in \Mor_{\mathcal C^{\mathcal P{\rm Mod}_{\k}}}^{r}(\V,\W)( = {\rm Hom}_{\mathcal P{\rm Mod}_{\k}}(\V, \W[r])$), we have 
\begin{align*}
(\Sigma^{\alpha} \circ i_{r,s})(f) & = \Sigma^{\alpha}(\iota^{\W}_{\cdot +r, \cdot +s} \circ f) \\
& = (\iota^{\W}_{\cdot +r, \cdot +s} \circ f)[- \alpha] \\
& = (\iota^{\W}_{\cdot + r-\alpha, \cdot + s-\alpha}) \circ f[-\alpha] = (i_{r,s} \circ \Sigma^{\alpha})(f) \in \Mor_{\mathcal C^{\mathcal P{\rm Mod}_{\k}}}^s(\Sigma^{\alpha} \V,\Sigma^{\alpha} \W).
\end{align*}
Therefore, $\Sigma^{\alpha}$ is a persistence endofunctor on $\C^{\P{\rm Mod}_{\k}}$ for any $\alpha \in \R$. 
\medskip

Consider a canonical shift functor on $\C^{\mathcal P{\rm Mod}_{\k}}$, denoted by $\Sigma: (\R, +) \to \mathcal P{\rm Fun}(\C^{\mathcal P{\rm Mod}_{\k}})$, defined by 
\[ \Sigma(\alpha) : = \Sigma^{\alpha} \,\,\mbox{defined in (\ref{ex-per-fun})} \,\,\,\,\,\,\,\mbox{and} \,\,\,\,\,\,\, \eta_{\alpha,\beta} = \mathds{1}_{\, \cdot[-\alpha]} \]
for any $\alpha, \beta \in \R$. Indeed, evaluate $\eta_{\alpha,\beta}$ on any object $\V$, 
\begin{align*}
(\eta_{\alpha,\beta})_{\V} = \mathds{1}_{\V[-\alpha]} & \in {\rm Hom}_{\mathcal P{\rm Mod}_{\k}}(\V[-\alpha], \V[-\alpha])\\
& = {\rm Hom}_{\mathcal P{\rm Mod}_{\k}}(\V[-\alpha], \V[-\beta + \beta-\alpha])\\
& = {\rm Mor}_{\C^{\mathcal P{\rm Mod}_{\k}}}^{\beta-\alpha}(\V[-\alpha], \V[-\beta])\\
& = {\rm Mor}_{\C^{\mathcal P{\rm Mod}_{\k}}}^{\beta-\alpha}(\Sigma^\alpha \V, \Sigma^\beta \V). 
\end{align*}
%Moreover, for any $\alpha, \beta \in \R$, and $r \leq s$ in $\R$, the following diagram 
%\begin{equation} \label{nt-dia-2}
%\xymatrixcolsep{3pc} \xymatrix{
%{\rm Mor}_{\C}^{r}(X, \Sigma^{\alpha}\V) \ar[r]^-{(\eta_{\alpha, \beta})_{\V} \circ} \ar[d]_-{i_{r, s}} & {\rm Mor}_{\C}^{r + \beta - \alpha}(X, \Sigma^{\beta}\V) \ar[d]^-{i_{r+\beta-\alpha, s+\beta-\alpha}} \\
%{\rm Mor}_{\C}^{s}(X, \Sigma^{\alpha}\V)) \ar[r]^-{(\eta_{\alpha, \beta})_{\V} \circ} & {\rm Mor}_{\C}^{s+\beta-\alpha}(X, \Sigma^{\beta}\V)}
%\end{equation}
In other words, $\eta_{\alpha, \beta}$ is a persistence natural transformation of shift $\beta - \alpha$ by Definition \ref{dfn-nt}. Therefore, $\Sigma$ defines a shift functor on $\C^{\mathcal P{\rm Mod}_{\k}}$.

\medskip

For $r \geq 0$, recall that the notation $\eta_r^\V$ in (\ref{dfn-eta-map}) denotes the composition $i_{-r,0} \circ (\eta_{r,0})_\V$. In particular, by definition above, $\eta_r^{\V} \in \Mor_{\C^{\mathcal P{\rm Mod}_{\k}}}^0(\Sigma^{r} \V, \V) = {\rm Hom}_{\mathcal P{\rm Mod}_{\k}}(\V[-r], \V)$, and it equals the following composition 
\[ \V[-r] \xrightarrow{\mathds{1}_{\V[-r]}} \V[-r] \xrightarrow{\iota^{\V}_{\cdot+r,\cdot } \circ} \V\]
which is just $\iota^{\V}_{\cdot-r, \cdot}$, the persistence structure maps of $\V$. Assume that objects in $\mathcal P{\rm Mod}_{\k}$ admit sufficient regularities so that they can be equivalently described via barcodes (see \cite{C-B15}), then, by  (i) in Lemma \ref{lemma-acyclic}, $r$-acyclic objects in $\C^{\mathcal P{\rm Mod}_{\k}}$ are precisely those persistence $\k$-modules with only finite-length bars in their barcodes, and its boundary depth, the length of the longest finite-length bar (see \cite{Ush11,Ush13} and \cite{UZ16}), is no greater than $r$.

\begin{remark} The way that we enrich the category $\mathcal P{\rm Mod}_{\k}$ to $\C^{\mathcal P{\rm Mod}_{\k}}$ as above, in order to admit more structures, is also investigated in the recent work \cite{BM19}, Section 10. In particular, the morphism set defined in (\ref{per-mor}) coincides with the enriched morphism set in its Proposition 10.2, and $\C^{\mathcal P{\rm Mod}_{\k}}$ is similar to $\underline{\rm Mod}_R^P$ in its Proposition 10.3 (when taking $R = \k$ and $P = \R$). \end{remark}

\section{Triangulated persistence categories} \label{subsec:TPC}
This section contains the main algebraic notion introduced in this
paper. Its purpose is to reflect triangulation properties in the
context of persistence categories. The basic properties of this
structure lead to the definition of a class of weighted triangles and
induced fragmentation pseudo-metrics on the set of objects of such a
category.

\subsection{Main definition}

We will use consistently below the characterization of $r$-equivalence
in Lemma \ref{lemma-comp} as well as that of $r$-acyclics in Lemma
\ref{lemma-acyclic}.

\begin{dfn} \label{dfn-tpc} A {\em triangulated persistence category}
  is a persistence category $\C$ endowed with a shift functor $\Sigma$
  such that the following three conditions are satisfied:
  \begin{itemize}
  \item[(i)] The $0$-level category $\C_0$ is triangulated with a
    translation automorphism denoted by $T$. \pbnote{Note that in
      particular $\C_0$ is additive and we further assume that the
      restriction of the persistence structure of $\C$ to $\C_0$ is
      compatible with the additive structure on $\C_0$ in the obvious
      way. Specifically this means that
      $\Mor_{\C}^r(A \oplus B,C) = \Mor_{\C}^r(A,C) \oplus \Mor_{\C}^r(B,C)$ for all
      $r \leq 0$ and the persistence maps $i_{r,s}$, $r \leq s \leq 0$
      are compatible with this splitting. The same holds also for
      $\Mor_{\C}^r(A, B \oplus C)$.}
  \item[(ii)] The restriction of $\Sigma^{r}$ to
    $\mathrm{End}(\C_{0})$ is a triangulated endofunctor of $\C_{0}$
    for each $r\in \R$. \pbnote{Note that each of the functors
      $\Sigma^r$, being a triangulated functor, is also assumed to be
      additive. We further assume that all the natural transformations
      $\eta_{r,s}: \Sigma^r \to \Sigma^s$, $s, r \in \mathbb{R}$, are
      compatible with the additive structure on $\C_0$.}
  \item[(iii)] For any $r \geq 0$ and any $A \in {\rm Obj}(\C)$, the
    morphism $\eta^A_r: \Sigma^rA \to A$ defined in (\ref{dfn-eta-map}) embeds into an exact triangle
    of $\C_0$
    \[ \Sigma^r A \xrightarrow{\eta^A_r} A \to K \to T\Sigma^r A \] such
    that $K$ is $r$-acyclic.
  \end{itemize}
\end{dfn}

\begin{remark} \label{rem:shifts-T} (a) Given that $\C_0$ is
  triangulated, the functors $\Mor_{\C}^s(X,-)$ and $\Mor_{\C}^s(-, X)$ are
  exact for $s=0$. This property together with the relations in
  Remark~\ref{rem:shift} imply that these functors are exact {\em for
    all} $s\in \R$.

  (b) Condition (ii) \pbnote{requires in particular} that $\Sigma$ and
  $T$ commute. Thus, $T\Sigma^{r} X=\Sigma^{r}TX$ for each object $X$
  and for any $f\in \Mor_{\C}^{0}(A,B)$ we have
  $\Sigma^{r} Tf=T\Sigma^{r}f$. Additionally, each $\Sigma^{r}$
  preserves the additive structure of $\C_{0}$ and it takes each exact
  triangle in $\C_{0}$ to an exact triangle. \pbnote{Moreover, the
    assumptions above imply that we have canonical isomorphisms
    $$\Mor_{\C}^r(A \oplus B, C) \cong \Mor_{\C}^r(A,C) \oplus \Mor_{\C}^r(B,C), \ \; \;
    \forall r \in \mathbb{R},$$} \pbnote{and the persistence maps
    $i_{r,s}$ are compatible with these isomorphisms. The same holds
    also for $\Mor_{\C}^r(A, B \oplus C)$. Finally, the maps $\eta_r^{A}$,
    $A \in \text{Obj}(\C_0)$, $r \geq 0$, are compatible with the
    additive structure on $\C_0$.}

  Notice also that the functor $T$ extends from $\C_{0}$ to a functor
  on $\C$. Indeed $T$ is already defined on all the objects of $\C$ as
  well as on all the morphisms of shift $0$.  For $f\in \Mor^{r}(A,B)$
  we define $Tf= T(f\circ (\eta_{r,0})_{A})\circ
  (\eta_{0,r})_{TA}$. It is easily seen that with this definition $T$
  is indeed a functor and it immediately follows that
  $T((\eta_{r,s})_{A})= (\eta_{r,s})_{TA}$ for all objects $A$ in $\C$
  and $r,s\in \R$. Further, by using the identifications in Remark
  \ref{rem:shift}, it also follows that $T$ is a persistence functor.
  In particular, we have $\eta_{r}^{TA}=T\eta_{r}^{A}$ for each object
  $A$ in $\C$.

  (c) Given that $0$-isomorphism preserves $r$-acyclicity - as noted
  in \S\ref{subsubsec:pers-cat1}, property (iii) does not depend on
  the specific extension of $\eta_{r}^{A}$ to an exact triangle.
 
  (d) Given a triangulated category $\mathcal{D}$ together with an
  appropriate shift functor
  $\Sigma : (\R, +)\to \mathrm{End}(\mathcal{D})$ it is possible to
  define a triangulated persistence category $\widetilde{\mathcal{D}}$
  with the same objects as $\mathcal{D}$, that has $\mathcal{D}$ as
  $0$-level and with morphism endowed with a persistence structure in
  the unique way such that (\ref{nt-dia-4}) and (\ref{nt-dia-5}) are
  satisfied with respect to the given shift functor $\Sigma$. We do
  not give the details here but we will see an example
  in~\S\ref{subsec:Tam}.
\end{remark}

\begin{dfn} \label{dfn-r-iso} Let $\C$ be a triangulated persistence category. A map $f \in \Mor_{\C}^0(A,B)$ is said to be an {\em {\boldsymbol $r$}-isomorphism} (from $A$ to $B$) if it embeds into an exact triangle in $\C_0$ \[ A \xrightarrow{f} B \to K \to TA  \]
such that $K \simeq_r 0$. \end{dfn}

We write  $f: A \simeq_r B$.

\begin{remark}(a) If $f$ is an $r$-isomorphism, then $f$ is an $s$-isomorphism for any $s \geq r$. 
It is not difficult to check, and we will see this explicitly in Remark \ref{rem:various-iso}, that for $r=0$ this definition is equivalent to the notion of $0$-isomorphism introduced before  (namely isomorphism in the category $\C_{0}$). 

(b) The relation $T(\eta_{r}^{K})=\eta_{r}^{TK}$ implies that $TK$ is $r$-acyclic if and only if $K$ is $r$-acyclic and, therefore, $f$ is an $r$-isomorphism if and only if $Tf$ is one. 

(c) From Definition \ref{dfn-tpc} (iii) we see that for any $r \geq 0$ and $A \in {\rm Obj}(\C)$ we have
 $$\eta^A_r: \Sigma^r A \simeq_r A~.~$$ 
\end{remark}

\begin{prop} \label{prop-r-iso} Any triangulated persistence category
  $\C$ has the following properties.
\begin{itemize} 
\item[(i)] If $f: A \to B$ is an $r$-isomorphism, then there exist $\phi \in \Mor_{\C}^0(B, \Sigma^{-r}A)$ and $\psi \in \Mor_{\C}^0(\Sigma^r B, A)$ such that 
\[ \phi \circ f = \eta^A_r  \,\,\mbox{in $\Mor_{\C}^0(A, \Sigma^{-r}A)$} \,\,\,\,\, \mbox{and}\,\,\,\,\, f \circ \psi = \eta^B_r \,\,\mbox{in $\Mor_{\C}^0(\Sigma^r B, B)$}. \]
The map $\psi$ is called a right $r$-inverse of $f$ and  $\phi$ is a left $r$-inverse of $f$. 
They satisfy $\Sigma^r \phi \simeq_r \psi$.
\item[(ii)] If $f$ is an $r$-isomorphism,  then any two left  $r$-inverses $\phi, \phi'$ of $f$ are themselves 
$r$-equivalent and the same conclusion holds for right $r$-inverses. 
\item[(iii)] If $f: A \simeq _r B$ and $g: B \simeq_s C$, then $g \circ f: A \simeq_{r+s} C$. 
\end{itemize}
\end{prop}

\begin{proof} (i) We first construct $\phi$. In $\C_0$, the morphism
  $f: A \to B$ embeds into an exact triangle
  $A \xrightarrow{f} B \xrightarrow{g} K \xrightarrow{h} TA$ with
  $K \simeq_r 0$. Using the fact that $\Sigma$ and $T$ commute, the
  following diagram is easily seen to be commutative:
  \begin{equation} \label{com-shift-fun2} \xymatrixcolsep{3pc}
    \xymatrix{
      K \ar[r]^-{h} \ar[d]_-{\eta_r^K} & TA \ar[d]^-{\eta^{TA}_r}\\
      \Sigma^{-r} K \ar[r]^-{\Sigma^{-r}h} & \Sigma^{-r} TA}
  \end{equation}
  Thus $\eta_r^{TA} \circ h = \Sigma^{-r} h \circ \eta_r^K = 0$ since
  $K$ is $r$-acyclic (and so $\eta_r^K =0$). By rotating exact
  triangles in $\C_0$ we obtain a new $\C_{0}$-exact triangle
  \jznote{$K \xrightarrow{h} TA \xrightarrow{Tf} TB \xrightarrow{Tg} TK$} and
  consider the diagram below (in $\C_{0}$):
  \pbnote{
    \[ \xymatrixcolsep{3pc} \xymatrix{ K \ar[d] \ar[r]^{h} & TA
        \ar[d]^-{\eta^{TA}_r} \ar[r]^{Tf}
        & TB \ar@{-->}[d]^-{\tilde{\phi}} \ar[r]^-{Tg} & TK \ar[d] \\
        0 \ar[r] & \Sigma^{-r} TA \ar[r]^-{\mathds{1}} & \Sigma^{-r} TA \ar[r]
        & 0} ~.~\]} The first square on the left commutes, so we
  deduce the existence of a map
  $$\tilde{\phi}\in \Mor_{\C}^{0}(TB,\Sigma^{-r}TA)$$ 
  that makes commutative the middle and right squares. The desired
  left inverse of $f$ is $\phi=T^{-1}\tilde{\phi}$. A similar argument
  leads to the existence of $\psi$.  We postpone the identity
  $\Sigma^{r}\phi\simeq_{r} \psi$ after the proof of (ii).  \medskip

  (ii) If $\phi$ $\phi'$ are two left inverses of $f$ then
  $(\phi - \phi') \circ f =0$. Therefore
  $$(\phi - \phi') \circ \eta^B_r = (\phi - \phi') \circ (f
  \circ \psi) = ((\phi - \phi') \circ f) \circ \psi = 0.$$
%  
%  \[ ((\phi - \phi') \circ f) \circ \psi = (\phi - \phi') \circ (f
%    \circ \psi) = (\phi - \phi') \circ \eta^B_r = 0. \]
  Lemma~\ref{lemma-comp} implies that $\phi \simeq_r \phi'$. The same
  argument works for right inverses. We now return to the identity
  $\Sigma^{r}\phi\simeq_{r} \psi$ (with the notation at point (i)).
  We have the following commutative diagram:
  \[ \xymatrixcolsep{5pc} \xymatrix{ \Sigma^r B \ar[r]^-{\psi}
      \ar@/_1.5pc/[rr]^-{\eta^B_r} & A \ar[r]^-{f}
      \ar@/^1.5pc/[rr]^-{\eta^A_r} & B \ar[r]^-{\phi} & \Sigma^{-r} A
    }~.~\] Therefore, $\eta^A_r \circ \psi = \phi \circ
  {\eta^B_r}$. By the naturality properties of $\eta$ we also have
  $\phi \circ {\eta^B_r}=\eta^A_r \circ \Sigma^r \phi$. Thus, by Lemma
  \ref{lemma-comp}, $\Sigma^r \phi \simeq_r \psi$.  \medskip

  (iii) We will make use of the following lemma.
  \begin{lemma} \label{lemma-exact} If $K \to K'' \to K' \to TK$ is an
    exact triangle in $\C_0$, $K\simeq_r 0$ and $K' \simeq_s 0$, then
    $K'' \simeq_{r+s} 0$.
  \end{lemma}

  \begin{proof} [Proof of Lemma \ref{lemma-exact}] 
    We associate the following commutative diagram to the exact
    triangle in the statement:
    \pbnote{
      \[ \xymatrix{ \Mor^{\alpha}(K'', K) \ar[r] \ar[d] &
          \Mor^{\alpha}(K'', K'')
          \ar[r] \ar[d]^-{v} & \Mor^{\alpha}(K'', K') \ar[d]^-{0} \\
          \Mor^{\alpha+s}(K'', K) \ar[r]^-{k} \ar[d]_-{0} &
          \Mor^{\alpha+s}(K'', K'') \ar[r]^{h} \ar[d]^-{t}
          & \Mor^{\alpha+s}(K'', K')  \\
          \Mor^{\alpha+r+s}(K'', K) \ar[r]^-{n} &
          \Mor^{\alpha+r+s}(K'', K'') & } \]} Here, the vertical
    morphisms are the persistence structure maps. \pbnote{The
      rightmost vertical map and the lower leftmost vertical one are
      both $0$}
    % We have $s=0$ and $w=0$
    due to our hypothesis together with Lemma~\ref{lemma-acyclic}. The
    functor $\Mor^{\alpha}(K'',-)$ is exact which implies that
    $t\circ v=0$ and, again by Lemma \ref{lemma-acyclic}, we deduce
    $K''\simeq_{r+s} 0$.
  \end{proof}

  Returning to the proof of the proposition, point (iii) now follows
  immediately by using the octahedral axiom to construct the following
  commutative diagram in $\C_0$
  \[ \xymatrix{
      A \ar[r]^-{f} \ar[d] & B \ar[r] \ar[d]^-{g} & K \ar[d] \\
      A \ar[r]^{g \circ f} \ar[d] & C \ar[r] \ar[d] & K'' \ar[d]\\
      0 \ar[r] & K' \ar[r] & K' }\] with exact rows and columns and
  applying Lemma~\ref{lemma-exact} to the \pbnote{rightmost column}.
\end{proof}

\begin{remark}\label{rem:various-iso} (a) Points (i) and (ii)
  in Proposition~\ref{prop-r-iso} imply that the notion
  $0$-isomorphism $f: A \to B$, as given by Definition \ref{dfn-r-iso}
  for $r=0$, is equivalent to isomorphism in $\mathcal{C}_{0}$. In
  particular, for $r=0$, $f$ admits a unique inverse \jznote{(up to $0$-isomorphism)} in $\Mor_{\C}^0(B,A)$.

  (b) Point (iii) in Proposition \ref{prop-r-iso} shows that
  $r$-isomorphism can not be expected to be an equivalence relation
  (unless $r=0$).
\end{remark}

Here are several useful corollaries. The first is an immediate consequence of Lemma \ref{lemma-exact} and the octahedral axiom in $\C_{0}$.
\begin{cor} \label{cor-1} If in the following commutative diagram in
  $\C_0$,
  \pbnote{
    \[ \xymatrix{ A \ar[r] \ar[d]_-{u} & B \ar[r] \ar[d]_-{v} & C
        \ar[r] \ar[d]_-{w}
        & TA \ar[d]_-{Tu}\\
        A' \ar[r] & B' \ar[r] & C' \ar[r] & TA'}\]} the two rows are
  exact triangles, $u$ is an $r$-isomorphism and $v$ is an
  $s$-isomorphism, then $w$ is an $(r+s)$-isomorphism.
\end{cor}

\begin{cor} \label{cor-2} If $f: A \to B$ is an $r$-isomorphism, then
  its right inverse $\psi \in \Mor^0(\Sigma^r B, A)$ (given by (i) in
  Proposition \ref{prop-r-iso}) is a $2r$-isomorphism. The same
  conclusion holds for its left inverse.
\end{cor}

\begin{proof} By the octahedral axiom (in $\C_0$), we have the
  following commutative diagram,
  \[ \xymatrix{ \Sigma^r B \ar[r] \ar[d]_-{\psi}
      & \Sigma^r B \ar[r] \ar[d]^-{\eta^B_r} & 0 \ar[d] \\
      A \ar[r] \ar[d] & B \ar[r] \ar[d] & K \ar[d]\\
      K'' \ar[r] & K' \ar[r] & K }\] where $K'' \to K' \to K \to TK''$
  is exact. By (iii) in Definition \ref{dfn-tpc}, $K' \simeq_r
  0$. Therefore, by Lemma \ref{lemma-exact}, $K'' \simeq_{2r} 0$ and
  thus $\psi$ is $2r$-isomorphism. A similar argument applies to the
  left inverse of $f$.
\end{proof} 
The next consequence is immediate but useful so we state it apart.
\begin{cor} \label{cor-3} If $f: A \to B$ is an $r$-isomorphism, then
  for any $u,u' \in \Mor^0(B, C)$ with $u \circ f = u' \circ f$, we
  have $u \simeq_r u'$, i.e., $u$ and $u'$ are $r$-equivalent.
  Similarly, if $v,v'\in \Mor^0(D, A)$ and $f\circ v=f\circ v'$, then
  $v\simeq_{r} v'$.
\end{cor}

\begin{cor} Assume that the following diagram in $\C_0$,
 \jznote{ \[ \xymatrix{ K \ar[r] \ar[d]_{\mathds 1_{K}} & A \ar[r]^-{\phi} \ar[d]_-{f}
      & A' \ar[r] \ar[d]_-{f'} & TK \ar[d]_{\mathds{1}_{TK}}\\
      K \ar[r] & B \ar[r]^{\psi} & B' \ar[r] & TK}\]} is commutative,
  that the two rows are exact and that $K \simeq_r 0$. Then the
  induced morphism $f'$ is unique up to $r$-equivalence.
\end{cor}

\begin{proof} Since $K \simeq_r 0$, by definition, $\phi$ is an
  $r$-isomorphism. For any two induced morphisms
  $f'_1, f'_2 \in \Mor^0(A', B')$, we have
  $f'_1 \circ \phi = f'_2 \circ \phi = \psi\circ f$ and the conclusion
  follows from Corollary \ref{cor-3}.
\end{proof}

\begin{cor} \label{cor5} Let $\phi: A \to A'$ be an
  $r$-isomorphism. Then for any $f \in \Mor_{\C}^0(A,B)$, there exists
  $f' \in \Mor_{\C}^0(A', \Sigma^{-r}B)$ such that the following diagram
  commutes in $\C_0$.
  \[ \xymatrix{
      A \ar[r]^-{\phi} \ar[d]_-{f} & A' \ar[d]^-{f'} \\
      B \ar[r]^-{\eta^B_r} & \Sigma^{-r} B} \]
\end{cor}

\begin{proof} Since $\phi: A \to A'$ is an $r$-isomorphism, there
  exists an left $r$-inverse denoted by $\psi: A' \to \Sigma^{-r} A$
  such that $\psi \circ \phi = \eta^A_r$. Set
  $f' : = \Sigma^{-r}f \circ \psi \in \Mor_{\C}^0(A', \Sigma^{-r} B)$.
\end{proof}
Similar direct arguments lead to the next consequence.
\begin{cor} \label{cor-6} Consider the following commutative diagram
  in $\C_0$,
  \[ \xymatrix{
      A \ar[r]^{\phi} \ar[d]_-{f} & A' \ar[d]^-{f'} \\
      B \ar[r]^-{\phi'} & B'} \] where $f \in \Mor_{\C}^0(A,B)$,
  $f' \in \Mor_{\C}^0(A'.B')$ and $\phi, \phi'$ are $r$-isomorphisms. Let
  $\psi, \psi'$ be any left inverses of $\phi, \phi'$
  respectively. Then the following diagram is $r$-commutative
  \[ \xymatrix{
      A' \ar[r]^-{\psi} \ar[d]_-{f'} & \Sigma^{-r} A \ar[d]^-{\Sigma^{-r} f} \\
      B' \ar[r]^-{\psi'} & \Sigma^{-r} B} \] in the sense that
  $\Sigma^{-r} f \circ \psi \simeq_r \psi' \circ f'$. A similar
  conclusion holds for right inverses.
\end{cor}

\subsection{Weighted exact triangles}\label{subsubsec:weight-tr}
The key feature of a triangulated persistence category $\C$ is that
there is a natural way to associate weights to a class of triangles
larger than the exact triangles in $\C_{0}$.

\begin{dfn} \label{dfn-set} A {\em strict exact triangle $\Delta$ of
    (strict) weight $r\geq 0$ in $\C$} is a diagram
  \begin{equation} \label{dfn-set-1} A \xrightarrow{\bar{u}} B
    \xrightarrow{\bar{v}} C \xrightarrow{\bar{w}} \Sigma^{-r} T A
  \end{equation}
  with $\bar{u}\in \Mor_{\C}^0(A,B)$, $\bar{v} \in \Mor_{\C}^0(B,C)$ and
  $\bar{w} \in \Mor_{\C}^0(C, \Sigma^{-r} TA)$ such that there exists an
  exact triangle
  $A \xrightarrow{u} B \xrightarrow{v} C' \xrightarrow{w} TA$ in
  $\C_0$, an $r$-isomorphism $\phi: C' \to C $ and a right $r$-inverse
  of $\phi$ denoted by $\psi: \Sigma^r C \to C'$ such that in the
  following diagram
  \begin{equation}\label{dfn-set-2} 
    \xymatrix{
      & & \Sigma^r C \ar[d]_-{\psi} \ar[rd]^-{\Sigma^{r} \bar{w}} & \\
      A \ar[r]^-{u} & B \ar[r]^-{v} \ar[rd]_-{\bar{v}}
      & C' \ar[r]^-{w} \ar[d]^-{\phi} & TA\\
      & & C &} 
\end{equation}
we have $\bar{u} = u$, $\bar{v} = \phi \circ v$ and
$\Sigma^r \bar{w} = w \circ \psi$.  The weight of the triangle
$\Delta$ is denoted by $w(\Delta)= r$. \end{dfn}

We will generally only talk about the weight of a strict exact
triangle and not the {\em strict} weight except when the distinction
will be of relevance, later on in the paper.

\begin{remark} \label{rmk- set} (a) Any exact triangle in $\C_0$ is a
  strict exact triangle of weight $0$ \jznote{since $\mathds{1}_A \in {\rm Mor}_{\C}^0(A, A)$ for any $A \in {\rm Obj}(\C)$}.  Conversely, it is a simple
  exercise to see that a strict exact triangle of weight $0$ is exact
  as a triangle in $\C_0$.

  (b) Consider the following diagram
  \[ \xymatrix{
      &\Sigma^{r}B\ar[d]_{\eta_{r}^{B}}\ar[r]^{\Sigma^{r}\bar{v}}
      & \Sigma^r C \ar[d]_-{\psi} \ar[rd]^-{\Sigma^{r} \bar{w}} & \\
      A \ar[r]^-{u} & B \ar[r]^-{v} \ar[rd]_-{\bar{v}}
      & C' \ar[r]^-{w} \ar[d]^-{\phi} & TA \ar[d]^-{\eta^{TA}_r}\\
      & & C \ar[r]^-{\bar{w}} &\Sigma^{-r} TA} \] which is derived
  from the commutative diagram (\ref{dfn-set-2}). The two squares in
  the diagram are not commutative, in general, but they are
  $r$-commutative.  Indeed, since $\phi$ is an $r$-isomorphism let
  $\tilde{\psi}:C\to \Sigma^{-r}C'$ be a left $r$-inverse of $\phi$.
  As $\psi$ is a right \jznote{$r$-inverse} of $\phi$ we deduce from Proposition
  \ref{prop-r-iso} (i) that $\Sigma^{-r}\psi \simeq_{r}
  \tilde{\psi}$. Therefore,
  $\bar{w}\circ \phi=\Sigma^{-r}w\circ \Sigma^{-r}\psi\circ \phi
  \simeq_{r} \Sigma^{-r}w\circ \tilde{\psi}\circ \phi= \Sigma^{-r}w
  \circ \eta_{r}^{C'}=\eta_{r}^{TA}\circ w$.  Using Corollary
  \ref{cor-3}, we also see that
  $\psi \circ \Sigma^{r}\bar{v}\simeq_{r} v\circ \eta_{r}^{B}$ because
  $\phi\circ \psi \circ \Sigma^{r}\bar{v}= \phi \circ v\circ
  \eta_{r}^{B}$.

  (c) Because $T$ commutes with $\Sigma$ and with the natural
  transformations $\eta$, it immediately follows that this functor
  preserves strict exact triangles as well as their weight.
\end{remark}

\begin{ex} \label{ex-w-r-tr} Recall that \pbnote{the map
    $\Sigma^rA \xrightarrow{\eta^A_r} A$ embeds into an exact triangle
    in $\C_0$,}
  $\Sigma^r A \xrightarrow{\eta^A_r} A \xrightarrow{v} K
  \xrightarrow{w} T\Sigma^r A$ \pbnote{, where $K$ is $r$-acyclic.}
  We claim that the diagram
  \[ \Sigma^r A \xrightarrow{\eta^A_r} A \xrightarrow{v} K
    \xrightarrow{0} TA \] is a strict exact triangle of weight
  $r$. Indeed, we have the following commutative diagram,
  \begin{equation*}
    \xymatrix{
      & & \Sigma^r K \ar[d]_-{\eta^K_r} \ar[rd]^-{0} & \\
      \Sigma^r A \ar[r]^-{\eta^A_r} & A \ar[r]^-{v} \ar[rd]_-{v}
      & K \ar[r]^-{w} \ar[d]^-{\mathds 1_K} & T\Sigma^r A\\
      & & K &} 
  \end{equation*}
  where the right upper triangle is commutative since $K$ is
  $r$-acyclic (so $\eta^K_r =0$ \jznote{by Lemma \ref{lemma-acyclic} (i)}). Moreover, $\mathds{1}_K$ is an
  $r$-isomorphism (we recall $T\Sigma^{r} A=\Sigma^{r}TA$).

  \pbnote{Note that also the following diagram
    $$\Sigma^r A  \xrightarrow{\eta^A_r} A \xrightarrow{} 0
    \xrightarrow{} TA$$ is a strict exact triangle of weight $r$.}
\end{ex}

\begin{prop} [Weight invariance] \label{prop-wt-1} Strict exact
  triangles satisfy the following two properties:

  (i) If strict exact triangles $\Delta, \Delta'$ are
  $0$-isomorphic, then $w(\Delta) = r$ if and only if
  $w(\Delta') = r$.

  (ii) If
  $\Delta: A \xrightarrow{\bar{u}} B \xrightarrow{\bar{v}} C
  \xrightarrow{\bar{w}} \Sigma^{-r} TA$ satisfies $w(\Delta) = r$,
  then
  $\Delta' : A \xrightarrow{\bar{u}} B \xrightarrow{\bar{v}} C
  \xrightarrow{\bar{w}'} \Sigma^{-r-s} TA$ satisfies
  $w(\Delta') = r+s$ for $s \geq 0$, where $\bar{w}'$ is the
  composition
  $$C \xrightarrow{\bar{w}} \Sigma^{-r} TA
  \xrightarrow{\eta_{s}^{\Sigma^{-r}TA}} \Sigma^{-r-s}TA~.~$$
\end{prop}

\begin{proof} (i) Two triangles in $\C_{0}$ are $0$-isomorphic if they
  are related through a triangle isomorphism in $\C_{0}$. The property
  claimed here immediately follows from the fact that, within $\C_0$,
  all $0$-isomorphisms admit inverses.  \medskip

  (ii) By definition, there exists a commutative diagram,
  \[ \xymatrix{
      & & \Sigma^r C \ar[d]_-{\psi} \ar[rd]^-{\Sigma^{r} \bar{w}} & \\
      A \ar[r]^-{\bar{u}} & B \ar[r]^-{v} \ar[rd]_-{\bar{v}}
      & C' \ar[r]^-{w} \ar[d]^-{\phi} & TA\\
      & & C &} \] Consider $\psi' \in \Mor^r(\Sigma^{r+s}C, C')$
  defined by $\psi' = \psi \circ \eta^{\Sigma^{r}C}_{s}$. Then $\psi'$
  is an $(r+s)$-right inverse of $\phi$.  Consider the following
  diagram
  \[ \xymatrix{
      & & \Sigma^{r+s} C \ar[d]_-{\psi'} \ar[rd]^-{\Sigma^{r+s} \bar{w}'} & \\
      A \ar[r]^-{\bar{u}} & B \ar[r]^-{v} \ar[rd]_-{\bar{v}}
      & C' \ar[r]^-{w} \ar[d]^-{\phi} & TA\\
      & & C &} \] where $\bar{w}': C\to \Sigma^{-r-s}TA$ is defined by
  $\bar{w}'=\eta^{\Sigma^{-r}TA}_{s}\circ \bar{w}$ and notice that the
  right upper triangle is commutative.
\end{proof}

\begin{prop}[Weighted rotation property] \label{prop-rot} Given a
  strict exact triangle
  $$\Delta: A \xrightarrow{\bar{u}} B \xrightarrow{\bar{v}} C
  \xrightarrow{\bar{w}} \Sigma^{-r} TA$$ satisfying $w(\Delta) = r$,
  there exists a triangle
  \begin{equation} \label{rot-ext-tr} R(\Delta): B
    \xrightarrow{\bar{v}} C \xrightarrow{\bar{w}'} \Sigma^{-r} TA
    \xrightarrow{-\bar{u}'} \Sigma^{-2r} TB
  \end{equation} 
  satisfying \jznote{$w(R(\Delta)) = 2r$}, where $\bar{w}' \simeq_r \bar{w}$
  and $\bar{u}'$ is the composition
  $$\bar{u}':\Sigma^{-r}TA \xrightarrow{\Sigma^{-r} T\bar{u}} \Sigma^{-r} TB 
  \xrightarrow{\eta_{r}^{\Sigma^{-r}TB}} \Sigma^{-2r}TB~.~$$
\end{prop}
We call \jznote{$R(\Delta)$} the (first) positive rotation of $\Delta$.

\begin{proof} By definition, there exists a commutative diagram,
  \[ \xymatrix{
      & & \Sigma^r C \ar[d]_-{\psi} \ar[rd]^-{\Sigma^{r} \bar{w}} & \\
      A \ar[r]^-{\bar{u}} & B \ar[r]^-{v} \ar[rd]_-{\bar{v}}
      & C' \ar[r]^-{w} \ar[d]^-{\phi} & TA\\
      & & C &} \] where
  $A \xrightarrow{\bar{u}} B \xrightarrow{v} C' \xrightarrow{w} TA$ is
  an exact triangle in $\C_0$ \pbnote{and $\psi$ is a right
    $r$-inverse of $\phi$.} By the rotation property of $\C_0$,
  $B \xrightarrow{v} C' \xrightarrow{w}TA \xrightarrow{-T\bar{u}} TB$
  is an exact triangle in $\C_0$. \pbnote{We now construct the
    following diagram in $\C_0$ in which the upper squares will be
    commutative and the lower square $r$-commutative:}
  \pbnote{
    \begin{equation}\label{eq:foursq}\xymatrixcolsep{4pc} \xymatrix{
        B \ar[r]^-{v} \ar[d]_-{\mathds 1_B}
        & C' \ar[r]^-{w} \ar[d]_-{\phi}
        & TA \ar[r]^{-T\bar{u}} \ar[d]_-{\bar{\phi}}
        & TB \ar[d]_-{\mathds 1_{TB}} \\
        B \ar[r]^-{\bar{v}}
        & C \ar[r]^-{w'} \ar[d]_-{\Sigma^{-r} \psi}
        & TA' \ar[r]^-{u'} \ar[d]_-{\bar{\psi}} & TB \\
        & \Sigma^{-r} C' \ar[r]^-{\Sigma^{-r} w} & \Sigma^{-r} TA} 
    \end{equation}
  } \pbnote{Here the second row of maps comes from embedding
    $B \xrightarrow{\bar{v}} C$ into an exact triangle
    $B \xrightarrow{\bar{v}} C \xrightarrow{w'} A'' \to TB$ for some
    $A''$ in $\C_0$ and $A'=T^{-1}A''$. The map $\bar{\phi}$ is then
    induced by the functoriality of triangles in $\C_{0}$ and is an
    $r$-isomorphism by Corollary \ref{cor-1}. So far this gives the
    upper three squares of the diagram and their commutativity. To
    construct the lower square, let $\bar{\psi}$ be a left $r$-inverse
    of $\bar{\phi}$
    (i.e.~$\bar{\psi} \circ \bar{\phi} = \eta^{TA}_r$). By
    Corollary~\ref{cor-2}, $\bar{\psi}$ is a $2r$-isomorphism.}

  \pbnote{We claim that the lower square in diagram~\eqref{eq:foursq}
    is $r$-commutative, and therefore we have
    $\bar{\psi} \circ w'\simeq_{r} \Sigma^{-r}w\circ \Sigma^{-r}\psi =
    \bar{w}$.}

  \pbnote{Indeed, let $\psi'$ be a left $r$-inverse of $\phi$. By
    using the commutativity of the middle upper square in
    diagram~\eqref{eq:foursq}, we deduce
    $\bar{\psi}\circ w'\circ \phi= \Sigma^{-r}w\circ \psi'\circ
    \phi$. As $\phi$ is an $r$-isomorphism we obtain
  \begin{equation} \label{eq:lsquare} \bar{\psi}\circ w' \simeq_{r}
    \Sigma^{-r}w\circ \psi'\simeq_{r} \Sigma^{-r}w\circ
    \Sigma^{-r}\psi = \bar{w},
  \end{equation}
  because, by Proposition~\ref{prop-r-iso}, we have
  $\Sigma^{-r}\psi \simeq_{r} \psi'$. This shows the lower square is
  $r$-commutative and the related $r$-identity.}

  We next consider the following diagram
  \[ \xymatrix{ & & \Sigma^r TA \ar[d]_-{\bar{\phi} \circ \eta^{TA}_r}
      \ar[rd]^-{\Sigma^{2r} \bar{u}'} & \\
      B \ar[r]^-{\bar{v}} & C \ar[r]^-{w'} \ar[rd]_-{\bar{w}'}
      & TA' \ar[r]^-{u'} \ar[d]_{\bar{\psi}} & TB\\
      & & \Sigma^{-r} TA &} \] where
  $\bar{u}'= \Sigma^{-2r}(u'\circ \bar{\phi}\circ \eta_{r}^{TA})$ and
  $\bar{w}'=\bar{\psi}\circ w'$.  Given that $\bar{\psi}$ is a
  $2r$-isomorphism, this means that we have a strict exact triangle of
  weight $2r$ of the form:
  $$B\xrightarrow{\bar{v}} C \xrightarrow{\bar{w}'}
  TA'\xrightarrow{\bar{u}'} \Sigma^{-2r}TB$$ We already know
  $\bar{w}'=\bar{\psi}\circ w'\simeq_{r} \bar{w}$. On the other hand,
  \[ \Sigma^{2r} \bar{u}'=u' \circ (\bar{\phi} \circ \eta^{TA}_r) =
    -T\bar{u} \circ \eta^{TA}_r = -\eta^{\Sigma^{r}TB}_r \circ
    \Sigma^r T\bar{u} \] which concludes the proof.
\end{proof}
 
\begin{rem}\label{rem:neg-rot} A perfectly similar argument also
  shows that there exists a strict exact triangle of strict weight
  $2r$ and of the form:
  $$R^{-1}(\Delta): T^{-1}\Sigma^{r}C\to A\to B \to \Sigma^{-r}C$$
  which is the \jznote{(first)} negative rotation of $\Delta$. \jznote{Note that $R^{-1} (R(\Delta)) \neq R(R^{-1}(\Delta)) \neq \Delta$.} 
\end{rem}

\begin{rem} \label{r:rotation-weight}
  \pbnote{Proposition~\ref{prop-rot} describes a rotation of weighted
    exact triangles that does {\em not} preserve weights. Indeed, the
    rotation of the weight $r$ triangle $\Delta$ from
    Proposition~\ref{prop-rot} has weight $2r$. It is not clear to
    what extent one can improve this. Ideally, one would like to be
    able to rotate $\Delta$ into a weighted exact triangle
    $B \xrightarrow{} C \xrightarrow{} \Sigma^{-r} TA \xrightarrow{}
    \Sigma^{-r} TB$ of the {\em same} weight $r$. There is some
    evidence, coming from symplectic topology, indicating that in
    certain circumstances this might be possible
    (see~\S\ref{sbsb:cobs-rot}). However, the algebraic setting in
    this paper, in particular the definition of weighted exact
    triangles, might be too general to render this feasible, at least
    without additional assumptions on $\Delta$.}
\end{rem}

\begin{prop} [Weighted octahedral formula] \label{prop-w-oct} Given
  two strict exact triangles
  $$\Delta_1: E \xrightarrow{\beta} F \xrightarrow{\alpha} X
  \xrightarrow{k} \Sigma^{-r} TE$$ and
  $$\Delta_2: X \xrightarrow{u} A \xrightarrow{\gamma} B
  \xrightarrow{b} \Sigma^{-s}TX$$ with $w(\Delta_1) = r$ and
  $w(\Delta_2) = s$, there exists diagram
  \begin{equation} \label{w-oct} \xymatrix{ E \ar[d]_{\beta}\ar[r] & 0
      \ar[r] \ar[d]
      & TE \ar[d] \ar[r]& TE \ar[d]^{T\beta}\\
      F \ar[d]_{\alpha} \ar[r]^{u\circ \alpha} & A \ar[r] \ar[d]_{id}
      & C \ar[d] \ar[r] & TF \ar[d]^{(\Sigma^{-s} T\alpha)\circ \eta^{TF}_{s}} \\
      X \ar[d]_{k} \ar[r]^{u} & A \ar[r]^{\gamma}\ar[d]
      & B \ar[r]^{b} \ar[d] & \Sigma^{-s} TX \ar[d]^{\Sigma^{-s}(Tk) } \\
      \Sigma^{-r} TE \ar[r] &0 \ar[r] &\Sigma^{-r-s} T^{2}E \ar[r] &
      \Sigma^{-r-s} T^{2}E}
  \end{equation}
  with all squares commutative except for the right bottom one that is
  $r$-anti-commutative, such that the triangles
  $\Delta_3: F \to A \to C \to TF$ and
  $\Delta_4: TE \to C \to B \to \Sigma^{-r-s} T^{2}E$ are strict exact
  with $w(\Delta_3) =0$ and $w(\Delta_4) = r+s$.
\end{prop}
By forgetting the $\Sigma$'s (or assuming that $r=s=0$) this is
equivalent to the usual octahedral axiom in a triangulated category
(namely $\C_{0}$) and the right bottom square is commutative up to
sign (or anti-commutative).
\begin{proof} [Proof of Proposition \ref{prop-w-oct}]
  By definition, there are two commutative diagrams,
  \begin{equation} \label{prop-w-oct-1} \xymatrixcolsep{4pc}
    \xymatrix{
      & E \ar[d] & \\
      & F \ar[d]_{\alpha'} \ar[rd]^{\alpha} & \\
      \Sigma^r X \ar[r]^-{\psi_1} \ar[rd]_-{\Sigma^r k}
      & X' \ar[d]_-{\delta} \ar[r]^-{\phi} & X \\
      & TE &} \,\,\,\,\,\,\,\, \xymatrixcolsep{4pc} \xymatrix{
      & & \Sigma^s B \ar[d]_-{\psi_3} \ar[rd]^-{\Sigma^s b} & \\
      X \ar[r]^-{u} & A \ar[r]^-{v'} \ar[rd]_{\gamma}
      & B' \ar[r]^-{g} \ar[d]_-{\bar{\phi}} & TX \\
      & & B &}
  \end{equation}
  with $\phi$ an $r$-isomorphism and $\bar{\phi}$ an $s$-isomorphism
  and $\psi_1$ and $\psi_{3}$ are, their right $r$ and $s$-inverses,
  respectively. By the octahedral axiom in $\C_0$, we construct the
  following diagram commutative except for the right bottom square
  that is anti-commutative:
  \begin{equation} \label{w-oct-2} \xymatrixcolsep{3pc} \xymatrix{
      E \ar[r]\ar[d] & 0 \ar[r] \ar[d] & TE \ar[d] \ar[r]& TE\ar[d] \\
      F \ar[r]^{\alpha''} \ar[d]_{\alpha'} & A \ar[r] \ar[d]_{\mathds{1}_A}
      & C \ar[r] \ar[d]_-{w} & TF\ar[d]_{T\alpha'}\\
      X' \ar[r]^-{u \circ \phi} \ar[d]_-{\delta} & A\ar[d] \ar[r]^-{v}
      & B'' \ar[r]^-{\theta} \ar[d]_-{t}
      & TX' \ar[d]_-{T\delta} \\
      TE \ar[r] & 0 \ar[r]& T^{2}E \ar[r]^{\mathds{1}_{T^2E}} & T^{2}E}
  \end{equation}
  Thus \jznote{$\alpha''=u\circ \phi\circ \alpha'$}, $t=-T\delta\circ \theta$.
  We denote by $$\Delta_3: F \xrightarrow{\alpha''} A \to C \to TF$$
  the respective exact triangle in $\C_0$ so that, as in Remark
  \ref{rmk- set}, $w(\Delta_3) =0$. The map $w\in \Mor^{0}(C,B'')$ is
  induced from the commutativity of the middle, left triangle. We now
  consider the following diagram.
  \begin{equation} \label{w-oct-3} \xymatrixcolsep{3pc} \xymatrix{ F
      \ar[d]_{\alpha'}\ar[r] &A \ar[d]_{\mathds{1}_A}\ar[r]& C \ar[d]_{w}\ar[r]
      & TF \ar[d]_{T\alpha'}\\
      X' \ar[r]^-{u \circ \phi} \ar[d]_-{\phi} & A \ar[r]^-{v}
      \ar[d]_{\mathds{1}_A} & B'' \ar[r]^-{\theta} \ar[d]_-{\phi'}
      & TX' \ar[d]_-{T\phi} \\
      X \ar[r]^-{u} & A\ar[d]_{\mathds{1}_A} \ar[r]^-{v'}
      & B' \ar[r]^-{g}\ar[d]_{\bar{\phi}} & TX\\
      & A\ar[r]^{\gamma} & B \ar[r]^{b}& \Sigma^{-s} TX }
  \end{equation}
  The three long rows are exact triangles in $\C_{0}$ and we deduce
  the existence of $\phi' \in \Mor^0(B'', B')$ making commutative the
  adjacent squares.  This is an $r$-isomorphism by Corollary
  \ref{cor-1}.  We fix a right $r$-inverse
  $\psi_{2}\in \Mor^{0}(\Sigma^{r}B',B'')$ of $\phi'$.  The
  composition $\phi'' = \bar{\phi} \circ \phi' \in \Mor^0(B'', B)$ is
  an $(r+s)$-isomorphism by Proposition \ref{prop-r-iso} (iii).  Let
  \jznote{$\psi''=\psi_{2}\circ \Sigma^{r}\psi_{3}$} (recall $\psi_{3}$ from
  (\ref{prop-w-oct-1})) and notice that $\psi''$ is a right
  $(r+s)$-inverse of $\phi''$.

  We are now able to define the triangle $\Delta_{4}$:
  \[ \Delta_4: TE \to C \xrightarrow{\phi'' \circ w} B
    \xrightarrow{\Sigma^{-r-s}(t \circ \psi'')} \Sigma^{-r-s}T^{2}E
    ~.~\] The following commutative diagram shows that $\Delta_{4}$ is
  strict exact and $w(\Delta_4) = r+s$.
  \[ \xymatrix{
      & & \Sigma^{r+s} B \ar[d]_-{\psi''} \ar[rd]^-{t \circ \psi''} & \\
      TE \ar[r] & C \ar[r]^-{w} \ar[rd]_-{\phi'' \circ w}
      & B'' \ar[r]^-{t} \ar[d]^-{\phi''} & T^{2}E\\
      & & B &} \] It is easy to check that all the squares in
  (\ref{w-oct}), except the right bottom one, are commutative.

  We now check the $r$-anti-commutativity of the right bottom
  square. We need to show
  $\Sigma^{-s} (Tk) \circ b \simeq_r - \Sigma^{-r-s}(t \circ \psi'')$,
  which is equivalent to
  $\Sigma^r (Tk) \circ \Sigma^{r+s} b \simeq_r - t \circ
  \psi''$. Given that the $B''(TX')(TX)B'$ square in (\ref{w-oct-3})
  commutes and using Corollary \ref{cor-6}, we have the following
  $r$-commutative diagram
  \[ \xymatrixcolsep{3pc} \xymatrix{ \Sigma^{r} B' \ar[r]^-{\Sigma^r
        g} \ar[d]_-{\psi_2}
      & \Sigma^r TX \ar[d]_-{T\psi_1} \\
      B'' \ar[r]^-{\theta} & TX'} ~.~\]

  Now consider the following diagram, commutative except the middle
  square being $r$-commutative,
  \[ \xymatrixcolsep{3pc} \xymatrix{ \Sigma^{r+s} B \ar[d]_-{\Sigma^r
        \psi_3}
      \ar[rd]^-{\Sigma^{r+s} b} \ar@/_4pc/[dd]_-{\psi''} & & \\
      \Sigma^r B' \ar[r]^-{\Sigma^r g} \ar[d]_-{\psi_2}
      & \Sigma^r TX \ar[d]_-{T\psi_1} \ar[rd]^-{\Sigma^r (Tk)} & \\
      B'' \ar[r]^-{\theta} \ar@/_1.5pc/[rr]^-{-t} & TX'
      \ar[r]^{T\delta} & T^{2}E } \] and write
  \[- t \circ \psi'' = (T\delta) \circ (\theta \circ \psi_2) \circ
    \Sigma^r \psi_3 \simeq_r (T\delta) \circ ((T\psi_1) \circ
    \Sigma^rg) \circ \Sigma^r \psi_3 = \Sigma^r (Tk) \circ
    \Sigma^{r+s}b \] which completes the proof.
\end{proof}

Given a triple of maps
$\Delta: A \xrightarrow{u} B \xrightarrow{v} C \xrightarrow{w} D$ with
shifts $\ceil*{u}, \ceil*{v}, \ceil*{w} \in \R$ it is useful to
introduce a special notation for an associated triple in $\C_{0}$,
denoted by $\Sigma^{s_1, s_2, s_3, s_4} \Delta$, for
$s_1, s_2, s_3, s_4 \in \R$ satisfying the following relations
\begin{align*}
  -s_1 + s_2 + \ceil*{u} & \leq 0 \\
  -s_2 + s_3 + \ceil*{v} & \leq 0\\
  -s_3 + s_4 + \ceil*{w} & \leq 0.
\end{align*}
The triple $\Sigma^{s_1, s_2, s_3, s_4} \Delta$ has the form
\begin{equation}\label{eq:shift-tr} \Sigma^{s_1} A
  \xrightarrow{\bar{u}} \Sigma^{s_2}B \xrightarrow{\bar{v}}
  \Sigma^{s_3} \xrightarrow{\bar{w}} \Sigma^{s_4} D
\end{equation}
where $\bar{u}$ is the composition of the composition
$\Sigma^{s_1} A \xrightarrow{(\eta_{s_1, 0})_A} A \xrightarrow{u} B
\xrightarrow{(\eta_{0, s_2})_B} \Sigma^{s_2} B$ and the persistence
structure map $i_{-s_1 + s_2, 0}$, i.e.,
\begin{equation} \label{shift-u} \bar{u} = i_{-s_1 + s_2 + \ceil*{u},
    0}\left((\eta_{0, s_2})_B \circ u \circ(\eta_{s_1, 0})_A\right).
\end{equation}
The definitions of $\bar{v}$ and $\bar{w}$ are similar and, in
particular, $\ceil*{\bar{u}}=\ceil*{\bar{v}}= \ceil*{\bar{w}}=0$. The
inequalities above ensure that the resulting triangle
(\ref{eq:shift-tr}) has all morphisms in $\C_{0}$.

For $s_1 = s_2 = s_3 = s_4 =k$ (which implies that
$\ceil*{u}, \ceil*{v}, \ceil*{w} \leq 0$) we denote, for
brevity, $$\Sigma^k \Delta : = \Sigma^{s_1, s_2, s_3, s_4} \Delta~.~$$

\begin{remark} \label{rmk-shift-notation} Assume that
  $\Delta: A \xrightarrow{u} B \xrightarrow{v} C \xrightarrow{w}
  \Sigma^{-r} TA$ is strict exact of weight $w(\Delta) = r$

  (a) It is a simple exercise to show that the triangle
  $\Sigma^k \Delta: \Sigma^kA \to \Sigma^k B \to \Sigma^k C \to
  \Sigma^{-r +k} TA$ is strict exact and
  $w\left(\Sigma^k \Delta\right) = r$.

  (b) For $s\geq 0$, Proposition \ref{prop-wt-1} (ii) claims that
  $\Sigma^{0,0,0,-s}\Delta$ is strict exact of weight $r+s$. It is
  again an easy exercise to see that $\Sigma^{0,0,-s,-s}\Delta$ is
  strict exact of weight $r+s$.
\end{remark}

\begin{prop}[Functoriality of triangles] \label{prop-ind} Consider two
  strict exact triangles as below with $f \in \Mor^0(A_1, A_2)$ and
  $g \in \Mor^0(B_1, B_2)$
  \[ \xymatrix{ \Delta_1: & A_1 \ar[r]^-{\bar{u}_1} \ar[d]_-{f} & B_1
      \ar[r]^-{\bar{v}_1} \ar[d]_-{g}
      & C_1 \ar[r]^-{\bar{w}_1} & \Sigma^{-r} TA_1 \\
      \Delta_2: & A_2 \ar[r]^-{\bar{u}_2} & B_2 \ar[r]^-{\bar{v}_2} &
      C_2 \ar[r]^-{\bar{w}_2} & \Sigma^{-s} TA_2} \] and
  $w(\Delta_1) = r$, $w(\Delta_2) = s$. Then there exists a morphism
  $h \in \Mor^0(C_1, \Sigma^{-r} C_2)$ inducing maps relating the
  triangles $\Delta_1 \to \Sigma^{0,0,-r,-r}\Delta_2$ as in the
  following diagram
  \[ \xymatrix{ A_1 \ar[r] \ar[d]_-{f} & B_1 \ar[r] \ar[d]_-{g} & C_1
      \ar[r] \ar[d]_-{h} & \Sigma^{-r} TA_1 \ar[d]^-{\eta^{TA_2}_s
        \circ \Sigma^{-r} Tf} \\
      A_2 \ar[r] & B_2 \ar[r] & \Sigma^{-r} C_2 \ar[r] & \Sigma^{-r-s}
      TA_2} \] where the middle square is $r$-commutative and the
  right square is $s$-commutative. \end{prop}

The proof is left as exercise.

\begin{prop}\label{prop:tr-isos} Let
  $\Delta: A \xrightarrow{\bar{u}} B \xrightarrow{\bar{v}}
  C\xrightarrow{\bar{w}} \Sigma^{-r} TA$ be a strict exact of weight
  $r$ in $\C$ and let
  $\Delta': A\xrightarrow{\bar{u}} B\xrightarrow{v} C'\xrightarrow{w}
  TA$ be the exact triangle in $\C_{0}$ associated to $\Delta$ as in
  Definition \ref{dfn-set}. For any $s\geq 2r$ there are morphisms of
  triangles $h: \Sigma^{2r}\Delta\to \Delta'$ and
  $h':\Delta' \to \Sigma^{-r} \Delta$ such that the compositions
  $h'\circ h$ and $h\circ \Sigma^{3r} h'$ have as vertical maps the
  \jznote{shifted natural transformations $\eta_{3r}^{-}$ defined in (\ref{dfn-eta-map}).}
\end{prop}
\begin{proof} We use the notation in Definition \ref{dfn-set} and
  consider the diagram below:
  \[ \xymatrixcolsep{3pc} \xymatrix{ \Sigma^r A \ar[r]^{\Sigma^r
        \bar{u}} \ar[d]_{\eta^A_r} & \Sigma^r B \ar[r]^{\Sigma^r
        \bar{v}} \ar[d]_{\eta^B_r} & \Sigma^r C \ar[r]^-{\Sigma^r
        \bar{w}} \ar[d]_-{\psi}
      &  TA \ar[d]^-{\mathds{1}_{TA}} \\
      A \ar[r]^-{\bar{u}} \ar[d]_{\mathds 1_A} & B \ar[r]^-{v}
      \ar[d]_-{\mathds 1_B} & C' \ar[r]^-{w} \ar[d]_-{\phi}
      & TA \ar[d]^-{\eta^{TA}_r} \\
      A \ar[r]^-{\bar{u}} & B \ar[r]^-{\bar{v}} & C \ar[r]^{\bar{w}} &
      \Sigma^{-r} TA} ~.~\] Denote
  \jznote{$h_{1}=(\eta_{r}^{A},\eta_{r}^{B},\psi , \mathds{1}_{TA})$ and
  $h_{1}'=(\mathds 1_{A}, \mathds{1}_{B},\phi,\eta_{r}^{TA})$}.  Notice that $h_{1}$ as
  well as $h'_{1}$ are not morphisms of triangles because the bottom
  right-most square is only $r$-commutative, and the same is true for
  the middle top square - as discussed in Remark \ref{rmk- set}
  (b). Let $h=h_{1}\circ \eta_{r}$ and $h'=\eta_{r}\circ h'_{1}$
  (where we view $\eta_{r}$ as a quadruple of morphisms of the form
  $\eta_{r}^{-}$). It follows that both $h$ and $h'$ are morphisms of
  triangles. Moreover, given that $\phi\circ \psi=\eta_{r}^{C}$, it is
  clear that $h'\circ h=\eta_{3r}$.  The other composition
  $h\circ \Sigma^{2r}h'$ has one term of the form
  $\psi\circ \eta_{r}\circ \Sigma^{3r}\phi\circ \eta_{r}$ so, by
  Proposition \ref{prop-r-iso}, this coincides with $\eta_{3r}^{C'}$
  as claimed.
\end{proof}

\begin{rem}\label{rem:tr-iso-0}
  We have seen above that any strict exact triangle of weight $r$ is
  isomorphic, in the precise sense of Proposition \ref{prop:tr-isos},
  to an exact triangle in $\C_{0}$. It is also useful to have a
  converse of this result.  Indeed, it is easy to show that if
  $\Delta :A\to B\to C\to TA$ is a triangle with all morphisms in
  $\C_{0}$ and if there is a morphism of triangles
  $h:\Delta\to \Delta'$ with $\Delta'$ an exact triangle in $\C_{0}$
  and with each component of $h$ an $s$-isomorphism, then there exist
  $k,t\geq 0$ such that the triangle $\Sigma^{0,0,-k,-k-t}\Delta$ is
  strict exact in $\C$.  By combining this remark with the statement
  in Proposition \ref{prop:tr-isos} and using that the natural maps
  $\eta_{r}^{-}$ are $r$-isomorphisms, one can easily see that the
  same statement remains true in the more general case when $\Delta'$
  is strict exact in $\C$.
\end{rem}

\subsection{Fragmentation pseudo-metrics on
  $\mathrm{Obj}(\C)$}\label{subsubsec:frag1}

In a triangulated persistence category there is a natural notion of
iterated-cone decomposition, similar to the corresponding notion in
the triangulated setting from \S\ref{subsec:triang-gen}.

\begin{dfn}\label{def:iterated-coneD} Let $\C$ be a triangulated
  persistence category, and $X \in {\rm Obj}(\C)$. An {\em iterated
    cone decomposition} $D$ of $X$ with {\em linearization}
  $(X_1,X_{2}, ..., X_n)$ where $X_i \in {\rm Obj}(\C)$ consists of a
  family of strict exact triangles in $\C$
  \[ 
    \left\{ 
      \begin{array}{ll}
        \Delta_{1}: \, \, & X_{1}\to 0\to Y_{1}\to \Sigma^{-r_{1}} TX_{1}\\
        \Delta_2: \,\, & X_2 \to Y_1 \to Y_2 \to \Sigma^{-r_2} TX_2\\
        \Delta_3: \,\, & X_3 \to Y_2 \to Y_3 \to \Sigma^{-r_3} TX_3\\
                          &\,\,\,\,\,\,\,\,\,\,\,\,\,\,\,\,\,\,\,\,\vdots\\
        \Delta_n: \,\, & X_n \to Y_{n-1} \to X \to \Sigma^{-r_n} TX_n
      \end{array} \right.\]
  The (strict) weight of such a cone decomposition is defined by
  $$w(D) = \sum_{i=1}^n w(\Delta_i)~.~$$ 
  The linearization of $D$ is denoted by $\ell(D) = (X_1, ..., X_n)$.
\end{dfn}

\begin{prop} \label{prop-cone-ref} Assume that $X$ admits an iterated
  cone decomposition $D$ with linearization $(X_1, ..., X_n)$ and for
  some $i \in \{1, ..., n\}$, $X_i$ admits an iterated cone
  decomposition $D'$ with linearization $(A_1, ..., A_k)$. Then $X$
  admits an iterated cone decomposition $D''$ of linearization
  \begin{equation} \label{cone-ref} (X_1, ..., X_{i-1}, TA_1, ...,
    TA_k, X_{i+1}, ..., X_n).
  \end{equation}
  Moreover, the weights of these cone decompositions satisfy
  $w(D'') = w(D) + w(D')$.
\end{prop}

A cone decomposition $D''$ as in the statement of Proposition
(\ref{cone-ref}) is called a {\em refinement} of the cone
decomposition $D$ with respect to $D'$.

\begin{ex} \label{ex:crefine} A single strict exact triangle
  $A \to B \to X \to \Sigma^{-r} TA$ can be regarded as a cone
  decomposition $D$ of $X$ with linearization $(T^{-1}B,A)$ such that
  $w(D) = r$. Assume that $A$ fits into a second strict exact triangle
  $E\to F\to A\to \Sigma^{-s} TE$ of weight $s$. Thus we have a
  cone-decomposition of $A$, $D'$, with linearization $(T^{-1}F, E)$,
  $w(D')=s$.  Diagram (\ref{w-oct}) from Proposition \ref{prop-w-oct}
  yields the following commutative diagram,
  \[ \xymatrix{
      E \ar[r] \ar[d] & 0 \ar[r] \ar[d] & TE \ar[d] \\
      F \ar[r] \ar[d] & B \ar[r] \ar[d] & Y \ar[d] \\
      A \ar[r] & B \ar[r] & X} \] for some object
  $Y \in {\rm Obj}(\C)$. In particular, we obtain a strict exact
  triangle $TE \to Y \to X \to \Sigma^{-r-s} T^{2}E$ of weight
  $r+s$. Thus, we have a refinement of $D$ with respect to $D'$ as
  follows,
\jznote{  \[ D'' := \left\{
      \begin{array}{l}
        T^{-1}B\to 0 \to B \to B\\
        F  \to B \to Y \to TF \\
        TE \to Y \to X \to \Sigma^{-r-s} TE
      \end{array} \right. \]}
  Moreover, $w(D'') = r+s = w(D) + w(D')$.
\end{ex}

\begin{proof} [Proof of Proposition \ref{prop-cone-ref}] By
  definition, the cone decomposition $D$ consists of a family of
  strict exact triangles in $\C$ as follows,
  \[ 
    \left\{ 
      \begin{array}{ll}
        &\,\,\,\,\,\,\,\,\,\,\,\,\,\,\,\,\,\,\,\,\vdots\\
        \Delta_{i-1}: \,\, & X_{i-1} \to Y_{i-2} \to Y_{i-1}
                             \to \Sigma^{-r_{i-1}} X_{i-1}\\
        \Delta_i: \,\, & X_{i} \,\,\,\,\,\,\to Y_{i-1}
                         \to Y_{i} \,\,\,\,\,\,\to \Sigma^{-r_{i}} X_{i}\\
        \Delta_{i+1}: \,\, & X_{i+1} \to Y_{i} \,\,\,\,\,\,\to Y_{i+1}
                             \to \Sigma^{-r_{i+1}} X_{i+1}\\
        &\,\,\,\,\,\,\,\,\,\,\,\,\,\,\,\,\,\,\,\,\vdots\\
      \end{array} \right.\]
  We aim to replace the triangle $\Delta_i$ by a sequence of strict
  exact triangles
  $$\overline{\Delta}_j: A_j \to B_{j-1} \to B_j \to \Sigma^{-s_j} TA_j$$
  for $j \in \{1, ..., k\}$ with $B_{0} = Y_{i-1}$, $B_k = Y_i$ and
  such that
  \begin{equation} \label{sum-ref} \sum_{j=1}^k w(\overline{\Delta}_j)
    = w(\Delta_i) + w(D').
  \end{equation}
  In this case, the ordered family of strict exact triangles
  $(\Delta_1, ..., \Delta_{i-1}, \overline{\Delta}_1, ...,
  \overline{\Delta}_k, \Delta_{i+1}, ..., \Delta_n)$ form the
  refinement $D''$, and (\ref{sum-ref}) implies that
  \begin{equation}
    w(D'')  = \sum_{j \in \{1, ..., n\} \backslash \{i\}} w(\Delta_j) +
    \sum_{j =1}^{k} w(\overline{\Delta}_j) 
    = \sum_{j=1}^n w(\Delta_j) + w(D') = w(D) + w(D')
  \end{equation}
  as claimed. 

  In order to obtain the desired sequence of strict exact triangles we
  focus on $\Delta_i$ and, to shorten notation, we rename its terms by
  $A = X_i$, $B = Y_{i-1}$, $C = Y_i$ and $r = r_i$ so that, with this
  notation, $\Delta_{i}$ is a strict exact triangle
  $A \to B \to C \to \Sigma^{-r} TA$.

  We now fix notation for the cone decomposition $D'$ of $A=X_{i}$. It
  consists of the following family of strict exact triangles,
  \[ 
    \left\{ 
      \begin{array}{ll}
        \Delta'_1: \,\, & A_1 \to 0 \to Z_1 \to  \Sigma^{-s_{1}}TA_1\\
        \Delta'_2: \,\, & A_2 \to Z_1 \to Z_2 \to \Sigma^{-s_2} TA_2\\
                        &\,\,\,\,\,\,\,\,\,\,\,\,\,\,\,\,\,\,\,\,\vdots\\
        \Delta'_{k-1}: \,\, & A_{k-1} \to Z_{k-2} \to Z_{k-1}
                              \to \Sigma^{-s_{k-1}} TA_{k-1}\\
        \Delta'_k: \,\, & A_k \to Z_{k-1} \to A \to \Sigma^{-s_k} TA_k
      \end{array} \right.
  \]
  We will apply Proposition \ref{prop-w-oct} iteratively. The first
  step is the following commutative diagram obtained from
  (\ref{w-oct}),
  \[ \xymatrix{
      A_k \ar[r] \ar[d] & 0 \ar[r] \ar[d] & TA_k \ar[d] \\
      Z_{k-1} \ar[r] \ar[d] & B \ar[r] \ar[d]_{\mathds{1}_B} & B_{k-1} \ar[d] \\
      A \ar[r] & B \ar[r] & C} \] for some
  $B_{k-1} \in {\rm Obj}(\C)$. Define
  $$\overline{\Delta}_k :TA_{k} \to B_{k-1} \to C \to
  \Sigma^{-r - s_k} T^{2}A_k~.~$$ We have
  \begin{equation} \label{inductive-1} w(\overline{\Delta}_k) = r +
    s_k = w(\Delta_i) + w(\Delta'_k).
  \end{equation}
  We then consider the following commutative diagram again obtained
  from (\ref{w-oct}),
  \begin{equation}\label{eq:diag-del} \xymatrix{
      A_{k-1} \ar[r] \ar[d] & 0 \ar[r] \ar[d] & TA_{k-1} \ar[d] \\
      Z_{k-2} \ar[r] \ar[d] & B \ar[r] \ar[d] & B_{k-2} \ar[d] \\
      Z_{k-1} \ar[r] & B \ar[r] & B_{k-1}} 
  \end{equation}
  for some $B_{k-2} \in {\rm Obj}(\C)$. Define
  $\overline{\Delta}_{k-1}$ to be the strict exact triangle:
  $$\overline{\Delta}_{k-1}: TA_{k-1} \to B_{k-2} \to B_{k-1}
  \to \Sigma^{- s_{k-1}} T^{2}A_{k-1}~.~$$ Then
  \begin{equation} \label{inductive-2} w(\overline{\Delta}_{k-1}) =
    s_{k-1} = w(\Delta'_{k-1}).
  \end{equation}
  Inductively, we obtain $B_i \in {\rm Obj}(\C)$, strict exact
  triangles
  \jznote{
  \begin{equation}\label{eq:it-tr}
    \overline{\Delta}_{i}: TA_i \to B_{i-1} \to B_i \to
    \Sigma^{-s_i} T^{2}A_i
  \end{equation}
} \pbnote{for $2 \leq i \leq k-1$} such that
  \begin{equation} \label{inductive-i}
    w(\overline{\Delta}_i) = s_i  = w(\Delta'_i).
  \end{equation}
  The final step lies in the consideration of the following diagram,
  \begin{equation} \label{eq:final-sq}\xymatrix{
      A_{1} \ar[r] \ar[d] & 0 \ar[r] \ar[d] & TA_{1} \ar[d] \\
      0 \ar[r] \ar[d] & B \ar[r] \ar[d] & B \ar[d] \\
      Z_1 \ar[r] & B \ar[r] & B_1}
  \end{equation}
  for some $B_0 \in {\rm Obj}(\C)$. Define $\overline{\Delta}_1$ to be
  the strict triangle
  $$\overline{\Delta}_{1}: TA_1 \to B\to B_1 \to  TA_1~.~$$ Then 
  \begin{equation} \label{inductive-11} w(\overline{\Delta}_1) =
    w(\Delta'_1)=s_{1}~.~
  \end{equation}
  Together, the ordered family
  $(\overline{\Delta}_1, ..., \overline{\Delta}_k)$ form the desired
  sequence of strict exact triangles. \pbnote{Finally, the
    equalities~\eqref{inductive-1},~\eqref{inductive-i}
    and~\eqref{inductive-11}} yield
  \[ \sum_{j=1}^k w(\overline{\Delta}_j) =s_{1}+ s_2 + ... + s_k + r =
    \sum_{j=1}^k w(\Delta'_j) + w(\Delta_i) = w(D') + w(\Delta_i) \]
  as claimed in (\ref{sum-ref}).
\end{proof}

Let $\mathcal F \subset {\rm Obj}(\C)$ be a family of objects of $\C$.
For two objects $X, X' \in {\rm Obj}(\C)$, define just as in
\S\ref{subsec:triang-gen},
\begin{equation} \label{frag-met} \delta^{\mathcal F} (X, X') =
  \inf\left\{ w(D) \, \Bigg| \, \begin{array}{ll} \mbox{$D$ is an
        iterated cone decomposition} \\ \mbox{of $X$ with
        linearization}\ \mbox{$(F_1, ..., T^{-1}X', ..., F_k)$}\\
      \mbox { where $F_i \in \mathcal F$, $k \in
        \N$} \end{array} \right\}.
\end{equation}

\begin{cor} \label{tri-ineq} With the definition of $\delta^{\F}$ in
  (\ref{frag-met}), we have the following inequality:
  \[ \delta^{\F}(X, X') \leq \delta^{\F}(X, X'') + \delta^{\F}(X'', X') \]
  for any $X, X', X'' \in {\rm Obj}(\C)$.
\end{cor}

\begin{proof} For any $\ep>0$, there are cone decompositions $D$ of
  $X$ and $D'$ of $X''$ respectively such that
  \[ w(D) \leq \delta^{\F}(X, X'')+ \ep\,\,\,\,\,\mbox{and}\,\,\,\,
    w(D') \leq \delta^{\F}(X'', X')+ \ep \] with linearizations
  $\ell(D) = (F_1, ..., T^{-1}X'', ..., F_s)$ and
  $\ell(D') = (F'_1,..., T^{-1}X', ..., F'_k)$, respectively,
  $F_i, F'_j\in \F$.  That means that $T^{-1}X''$ has a corresponding
  cone decomposition $T^{-1}D'$ with linearization
  $\ell(T^{-1}D')=(T^{-1}F'_{1},\ldots, T^{-2} X',\ldots,
  T^{-1}F'_{k})$.  Proposition \ref{prop-cone-ref} implies that there
  exists a cone decomposition $D''$ of $X$ that is a refinement of $D$
  with respect to $T^{-1}D'$ such that
  $\ell(D'')=(F_{1},\ldots, F'_{1},\ldots, T^{-1}X',\ldots
  F'_{k},\ldots, F_{s})$ and
  \[ w(D'') = w(D) + w(D') \leq \delta^{\F}(X, X'') + \delta^{\F}(X'',
    X') + 2 \ep\] which implies the claim.
\end{proof}

Finally, there are also fragmentation pseudo-metrics specific to this
situation with properties similar to those in Proposition
\ref{prop:tr-weights-gen}.

\begin{dfn} \label{dfn-frag-met} Let $\C$ be a triangulated
  persistence category and let $\F \subset {\rm Obj}(\C)$.  The {\em
    fragmentation pseudo-metric}
  $$d^{\F}: {\rm Obj}(\C)\times {\rm Obj}(\C)\to [0,\infty)\cup \{
  +\infty\}$$ associated to $\F$ is defined by:
  \[ d^{\F}(X, X') = \max\{\delta^{\F}(X, X'), \delta^{\F}(X',
    X)\}. \]
\end{dfn}

\begin{rem} \label{ex-delta} (a) It is clear from Corollary
  \ref{tri-ineq} that $d^{\F}$ satisfies the triangle inequality and,
  by definition, it is symmetric. It is immediate to see that
  $d^{\F}(X,X)=0$ for all objects $X$ (this is because of the
  existence of the exact triangle in $\C_{0}$, $T^{-1}X\to 0\to X$).
  It is of course possible that this pseudo-metric be degenerate and
  it is also possible that it is not finite.

  (b) \jznote{If $X\in \mathcal F$}, then $\delta^{\F}(TX,0)=0$ because of the exact
  triangle $X\to 0 \to TX$.  On the other hand, $\delta^{\F}(0,TX)$ is
  not generally trivial. However, the exact triangle $TX\to TX\to 0$
  shows that, if $TX\in \F$, then $\delta^{\F}(0,TX)=0$.

  (c) It follows from the previous point that if $\F = {\rm Obj}(\C)$,
  then $d^{\F}(X,X') =0$ (in other words, the pseudo-metric
  $d^{\mathcal{F}}(-,-)$ is completely degenerate).  More generally,
  if the family $\F$ is $T$ invariant (in the sense that if $X\in \F$,
  then $T^{\pm}X\in \F$), then $T$ is an isometry relative to the
  pseudo-metric $d^{\F}$ and $d^{\F}(X,X')=0$ for all $X,X'\in \F$.

  (d) The remark \ref{rem:finite-metr} (c) applies also in this
  setting in the sense that we may define at this triangulated
  persistence level fragmentation pseudo-metrics
  $\underline{d}^{\mathcal{F}}$ given by (the symmetrization of)
  formula (\ref{eq:frag-simpl}) but making use of weighted triangles
  in $\mathcal{C}$ instead of the exact triangles in the triangulated
  category $\mathcal{D}$.
\end{rem}

Recall that by assumption $\C_0$ is triangulated and thus
additive. Therefore, for any two objects
$X,X' \in {\rm Obj}(\C_0) = {\rm Obj}(\C)$, the direct sum
$X \oplus X'$ is a well-defined object in ${\rm Obj}(\C)$.

\begin{prop} \label{prop-frag-sum} For any
  $A, B, A', B' \in {\rm Obj}(\C)$, we have
  \[ d^{\F}(A \oplus B, A' \oplus B') \leq d^{\F}(A, A') + d^{\F}(B,
    B'). \]
\end{prop}

\begin{proof}
  The proof follows easily from the following lemma.

  \begin{lem} \label{claim-frag-sum} Let
    $\Delta: A \to B \to C \to \Sigma^{-r} TA$ and
    $\overline{\Delta}: \overline{A} \to \overline{B} \to \overline{C}
    \to \Sigma^{-s} T\overline{A}$ be two strict exact triangles with
    $w(\Delta) = r$ and $w(\Delta') = s$. Then
    \[ \Delta'': A \oplus \overline{A} \to B \oplus \overline{B} \to C
      \oplus \overline{C} \to \Sigma^{- \max\{r,s\}} TA \oplus
      T\overline{A} \] is a strict exact triangle with
    $w(\Delta'') = \max\{r,s\}$.
  \end{lem}
  \begin{proof}[Proof of Lemma \ref{claim-frag-sum}]
    By definition, there are two commutative diagrams, 
    \[ \xymatrix{
        & & \Sigma^r C \ar[d]_-{\psi} \ar[rd] & \\
        A \ar[r] & B \ar[r] \ar[rd] & C' \ar[r] \ar[d]^-{\phi} & TA\\
        & & C &} \,\,\,\,\, \xymatrix{
        & & \Sigma^s \overline{C} \ar[d]_-{\bar{\psi}} \ar[rd] & \\
        \overline{A} \ar[r] & \overline{B} \ar[r] \ar[rd]
        & \overline{C'} \ar[r] \ar[d]^-{\bar{\phi}} & T\overline{A}\\
        & & \overline{C} &} \] This yields the following commutative
    diagram,
   \jznote{ \[ \xymatrix{ & & \Sigma^{\max\{r,s\}} C \oplus \overline{C}
        \ar[d]_-{\psi \oplus \bar{\psi}} \ar[rd] & \\
        A \oplus \overline{A} \ar[r] & B \oplus \overline{B} \ar[r]
        \ar[rd] & C' \oplus \overline{C'} \ar[r] \ar[d]^-{\phi \oplus
          \bar{\phi}}
        & TA \oplus T\overline{A}\\
        & & C \oplus \overline{C} &} \] }and it is easy to check that
    $\phi \oplus \bar{\phi}$ is a
    $\max\{r,s\}$-isomorphism.
  \end{proof}

  Returning to the proof of the proposition, it suffices to prove
  $\delta^{\F}(A \oplus B, A' \oplus B') \leq \delta^{\F}(A, A') +
  \delta^{\F}(B, B')$. For any $\ep>0$, by definition, there exist
  cone decompositions $D$ and $D'$ of $A$ and $B$ respectively with
  $\ell(D) = (F_1, ..., F_{i-1}, T^{-1}A', F_{i+1},..., F_s)$ and
  $\ell(D') = (F'_1, ..., T^{-1}B', ..., F'_{s'})$ such that
  \[ w(D) \leq \delta^{\F}(A, A') + \ep \,\,\,\,\,\mbox{and}\,\,\,\,\,
    w(D') \leq \delta^{\F}(B, B') + \ep. \]

  The desired cone decomposition of $A \oplus B$ is defined as follows. 
  \begin{equation}\label{eq:sum-dec} D'' : = 
    \left\{
      \begin{array}{l} 
        F_1 \to 0 \to E_1 \to \Sigma^{-r_1} TF_1 \\
        F_2 \to E_1 \to E_2 \to \Sigma^{-r_2} TF_2\\
        \,\,\,\,\, \,\,\,\,\, \,\,\,\,\, \,\,\,\,\, \,\,\,\,\,\vdots\\
        F_s  \to E_{s-1} \to A \to \Sigma^{-r_s} TF_s\\
        F'_{1}\to A\oplus 0 \to A\oplus E'_{1}\to \Sigma^{-r'_{1}}TF'_{1}\\
        F'_{2} \to A\oplus E'_{1} \to A\oplus E'_{2} \to \Sigma^{-r'_{2}}  
        TF'_{s+2-i}\\
        \,\,\,\,\, \,\,\,\,\, \,\,\,\,\, \,\,\,\,\, \,\,\,\,\,\vdots\\
        T^{-1}B' \to A \oplus E'_j \to A \oplus E'_{j+1} \to
        \Sigma^{- r'_{B'}} B' \\
        \,\,\,\,\, \,\,\,\,\, \,\,\,\,\, \,\,\,\,\, \,\,\,\,\,\vdots\\
        F'_{s'} \to A \oplus E'_{s'-1} \to A \oplus B \to
        \Sigma^{-r'_{s'}} TF'_{s'}
      \end{array} \right. 
  \end{equation}
  Here we identify $0\oplus F'_{i}=F'_{i}$. The first $s$-triangles
  come from the decomposition of $A$ and the following $s'$ triangles
  are associated, using Lemma \ref{claim-frag-sum}, to the respective
  triangles in the decomposition of $B$ and to the triangle
  $0\to A\to A$ (of weight $0$).  It is obvious that
  $w(D'')=w(D)+w(D')$ and thus
  $\delta^{\F}(A \oplus B, A' \oplus B') \leq \delta^{\F}(A, A') +
  \delta^{\F}(B,B') $.
\end{proof}

The next statement
is an immediate consequence of Proposition \ref{prop-frag-sum}.

\begin{cor} \label{cor:Hsp} The set ${\rm Obj}(\C)$ with the topology induced by the fragmentation pseudo-metric $d^{\F}$ is an $H$-space relative to the operation $$(A,B)\in {\rm Obj}(\C) \times {\rm Obj}(\C)
\to A\oplus B\in  {\rm Obj}(\C) ~.~$$\end{cor}

\subsubsection{Proof of Proposition
  \ref{prop:tr-weights-gen}}\label{subsubsec:proof-prop}
We now return to the setting in \S\ref{subsec:triang-gen}. Thus,
$\mathcal{D}$ is triangulated category, $w$ is a triangular weight on
$\mathcal{D}$ in the sense of Definition \ref{def:triang-cat-w}, and
the quantities $w(D)$ (associated to an iterated cone-decomposition
$D$), $\delta^{\mathcal{F}}$, $d^{\mathcal{F}}$ are defined as in
\S\ref{subsec:triang-gen}.

The first (and main) step is to establish a result similar to
Proposition \ref{prop-cone-ref}. Namely, if $X$ admits an iterated
cone decomposition $D$ with linearization $(X_1, ..., X_n)$ and some
$X_i$ admits a decomposition $D'$ with linearization
$(A_1, ..., A_k)$, then $X$ admits an iterated cone decomposition
$D''$ with linearization
$(X_1, ..., X_{i-1}, TA_1, ..., TA_k, X_{i+1}, ..., X_n)$ and
\begin{equation}\label{eq:weight-ineq2}
  w(D'') \leq w(D) + w(D')
\end{equation}
For convenience, recall from \S\ref{subsec:triang-gen} that the
expression of $w(D)$ is:
\begin{equation}\label{eq:weight-again}
 w(D)= \sum_{i=1}^{n}w(\Delta_{i})-w_{0}
\end{equation}
where $w_{0}=w(0\to X\xrightarrow{id} X\to 0)$ (for all $X$).  To show
(\ref{eq:weight-ineq2}) we go through exactly the same construction of
the refinement $D''$ of the decomposition $D$ with respect to $D'$, as
in the proof of \ref{prop-cone-ref}, assuming now that all shifts are
trivial along the way. The analogue of diagram (\ref{eq:final-sq})
that appears in the last step of the construction of $D''$ remains
possible in this context due to the point (ii) of Definition
\ref{def:triang-cat-w}.  By tracking the respective weights along the
construction and using Remark \ref{rem:gen-weights} (b) to estimate
the weight of $TA_{1}\to B\to B_{1}\to T^{2}A_{1}$ from
(\ref{eq:final-sq}) we deduce (with the notation in the proof of
Proposition \ref{prop-cone-ref})
\jznote{$$\sum_{j=1}w(\bar{\Delta}_{j}) \leq w(D')+w(\Delta_{i})-w_{0}$$}
which implies (\ref{eq:weight-ineq2}).  Once formula
(\ref{eq:weight-ineq2}) established, it immediately follows that
$\delta^{\mathcal{F}}(-,-)$ satisfies the triangle
inequality. Further, because the weight of a cone-decomposition is
given by (\ref{eq:weight-again}), it follows that the
cone-decomposition $D$ of $X$ with linearization $(T^{-1}X)$ given by
the single exact triangle $T^{-1}X\to 0\to X\to X$ is of weight
$w(D)=0$. As a consequence, $\delta^{\mathcal{F}}( X,X)=0$. It follows
that $d^{\mathcal{F}}(-,-)$ is a pseudo-metric as claimed at the point
(i) of Proposition \ref{prop:tr-weights-gen}.

Assuming now that $w$ is subadditive, the same type of decomposition
as in equation (\ref{eq:sum-dec}) can be constructed to show that
$\delta^{\mathcal{F}}(A\oplus B, A'\oplus B')\leq
\delta^{\mathcal{F}}(A,A')+\delta^{\mathcal{F}}(B,B')+w_{0}$ which
implies the claim. \Qed

\section{Triangulated structure of $\C_{\infty}$ and triangular weights}\label{subsec:trstr-weights}
The purpose of this section is to further explore the structure of 
the limit category $\C_{\infty}$ associated to a triangulated persistence category $\C$.
We will see that $\C_{\infty}$ is triangulated and that it carries a triangular weight, in the sense of \S\ref{subsec:triang-gen},  induced by the persistence structure.

\subsection{Exact triangles in  $\C_{\infty}$ and main algebraic result} \label{subsubsec:ex-cinfty}

Assume that $\C$ is a persistence category and recall its $\infty$-level  $\C_{\infty}$
from Definition \ref{dfn-c0-cinf}: its objects are the same as those of $\C$ and its morphisms
are $\Mor_{\C_{\infty}}(A,B) = \varinjlim_{\alpha \to \infty} \Mor_{\C}^{\alpha}(A,B)$ for any two objects $A, B$ of $\C$. For a morphism $\bar{f}$ in $\C$ we denote by $[\bar{f}]$ the corresponding morphism in $\C_{\infty}$ and if $f=[\bar{f}]$, $f\in\Mor_{\C_{\infty}}$, $\bar{f}\in\Mor_{\C}$, 
we say that $\bar{f}$ represents $f$. We use the same terminology for diagrams (including triangles) 
in $\C$ in relation to corresponding diagrams in $\C_{\infty}$ in the sense that a diagram in $\C$ represents one in $\C_{\infty}$ if  the nodes in the two cases are the same and the morphisms of the diagram in $\C$  represent the corresponding ones in the $\C_{\infty}$ diagram. Clearly, all $r$-commutativities and $r$-isomorphisms in $\C$ become, respectively, commutativities and isomorphisms in $\C_{\infty}$. For instance, if $K$ is $r$-acyclic, then $K$ is isomorphic to $0$ in $\C_{\infty}$. 

The hom-sets of $\C_{\infty}$ admit a natural filtration as follows. 
For any $A, B \in {\rm Obj}(\C_{\infty})$, and $f \in {\rm Hom}_{\C_{\infty}}(A,B)$, let the 
{\em spectral invariant} of $f$ be given by:
\[ \sigma(f) : =\inf \left\{ k\in \R \cup \{-\infty\} \, \bigg| \, \mbox{$f = \left[\bar{f}\right]$ for some $\bar{f} \in \Mor_{\C}^{k}(A,B)$} \right\} \]
and $$\Mor_{\C_{\infty}}^{\leq \alpha}(A,B)=\{f\in \Mor_{\C_{\infty}}(A,B) \, | \, \sigma(f)\leq \alpha\}.$$

\smallskip

Assume from now on that $\C$ is a triangulated persistence category.
It is easy to see that in this case $\C_{\infty}$ is
additive. Moreover, as explained in \S \ref{subsubsec:frag1}, $\C$ is
endowed with a family of strict exact triangles.  The purpose of this
section is to extract from the properties of the strict exact
triangles in $\C$ a triangulated structure on $\C_{\infty}$. This is
not quite immediate because a morphism $\bar{f}:A\to B$ in $\C$ of
strictly positive shift $\ceil*{\bar{f}}>0$ can not be completed to a
strict exact triangle. As a consequence, the definition of the exact
triangles in $\mathcal{C}_{\infty}$ is necessarily more subtle than
just taking the image in $\C_{\infty}$ of the strict exact triangles
in $\C$ because we need to be able to complete to exact triangles
those morphisms $f$ in $\C_{\infty}$ with $\sigma(f)>0$.

\

Before proceeding with the relevant definition we notice that the
shift functor $\Sigma$ associated to $\mathcal{C}$ (see Definition
\ref{dfn-tpc}) induces a similar functor
$\Sigma : (\R, +) \to \mathrm{End}(\C_{\infty})$. We will continue to
use the same notation for the $r$-shifts $\Sigma^{r}$ and the natural
transformations $\eta_{r,s}:\Sigma^{r}\to \Sigma^{s}$. \pbnote{At the
  same time, in contrast to morphisms in $\C$, there is no meaning to
  the ``amount of shift'' for a morphism in $\C_{\infty}$ (though one
  can associate to such a morphism its spectral invariant as above).}
Similarly, the functor $T$ (which is defined as in
Remark~\ref{rem:shifts-T} (b) on all of $\C$) also induces a similar
functor on $\C_{\infty}$.  Given a triple of maps
$\Delta: A\to B\to C\to D$ in $\C$ we will make use of the shifted
triple $\Sigma^{s_1, s_2, s_3, s_4} \Delta$ as defined in
(\ref{eq:shift-tr}) and we will use the same notation for similar
triples in $\C_{\infty}$. Note however that in $\C_{\infty}$ the
inequalities relating the $s_{i}$'s and the shifts of $u,v,w$ are no
longer relevant \pbnote{(in fact, do not make sense)} and the shift
$\Sigma^{s_{1},s_{2},s_{3},s_{4}}$ will be used in $\C_{\infty}$
without these constraints.

\begin{dfn} \label{dfn-extri-inf} A triangle
  $\Delta: A \xrightarrow{u} B \xrightarrow{v} C \xrightarrow{w} TA$
  in $\C_{\infty}$ is called {\em exact} if there exists a
  \pbnote{diagram}
  $\bar{\Delta}: A \xrightarrow{\bar{u}} B \xrightarrow{\bar{v}} C
  \xrightarrow{\bar{w}} TA$ in $\C$ that represents $\Delta$, such
  that the shifts $\ceil*{\bar{u}}, \ceil*{\bar{v}}, \ceil*{\bar{w}}$
  are all non-negative and the shifted triangle
  \begin{equation} \label{extri-shifted} \widetilde{\Delta} =
    \Sigma^{0, -\ceil*{\bar{u}}, -\ceil*{\bar{u}}-\ceil*{\bar{v}},
      -\ceil*{\bar{u}}-\ceil*{\bar{v}}- \ceil*{\bar{w}}} \bar{\Delta}
  \end{equation}
  is strict exact in $\C$.  The {\em unstable weight} of $\Delta$,
  $w_{\infty}(\Delta)$, is the infimum of the strict weights
  $w(\widetilde{\Delta})$ of all the strict exact triangles
  $\widetilde{\Delta}$ as above. Finally, the {\em weight} of
  $\Delta$, $\bar{w}(\Delta)$, is given by:
  $$\bar{w}(\Delta)=\inf\{ \ w_{\infty}(\Sigma^{s,0,0,s}\Delta)
  \ | \ s\geq 0\}~.~$$
\end{dfn}

The convention is that if there does not exist $\widetilde{\Delta}$ as
in the definition, then $w_{\infty}(\Delta) =\infty$. Notice that,
from the definition of the strict weights in $\C$, the strict weight
of $\widetilde{\Delta}$ is
$\ceil*{\bar{u}}+\ceil*{\bar{v}}+\ceil*{\bar{w}}$.  It is not
difficult to show that (see Remark \ref{rem:tr-ex2} (b)) a triangle
$\Delta$ in $\C_{\infty}$ is exact if and only if $\bar{w}(\Delta)$ is
finite. \jznote{Moreover, by definition, $\bar{w}(\Delta) \leq w_{\infty}(\Delta)$, and Example \ref{ex-embed} below shows that this inequality can be strict.} 

\medskip

\jznote{For the weight $\bar{w}$ of exact triangles in $\C_{\infty}$ defined as above, recall that $\bar{w}_0$ denotes the normalization constant in Definition \ref{def:triang-cat-w} (ii).} The main result is the following:

\begin{thm}\label{thm:main-alg}
  Let $\C$ be a triangulated persistence category. The limit category
  $\C_{\infty}$ together with the class of exact triangles defined in
  \ref{dfn-extri-inf} is triangulated and $\bar{w}$ is a subadditive
  triangular weight on $\C_{\infty}$ with $\bar{w}_{0}=0$.
\end{thm}

A persistence refinement of a triangulated category $\mathcal{D}$ is a
TPC, $\widetilde{\mathcal{D}}$, such that
$\widetilde{\mathcal{D}}_{\infty}=\mathcal{D}$. The triangular weight
$\bar{w}$ as in Theorem \ref{thm:main-alg} is called the persistence
weight induced by the respective refinement.The following consequence
of Theorem \ref{thm:main-alg} is immediate from the general
constructions in \S\ref{subsec:triang-gen}.

\begin{cor} If a small triangulated category $\mathcal{D}$ admits a
  TPC \pbnote{refinement} $\widetilde{\mathcal{D}}$, then
  $\mathrm{Obj}(\mathcal{D})$ is endowed with a family of
  fragmentation pseudo-metrics $\bar{d}^{\mathcal{F}}$, defined as in
  \S\ref{subsec:triang-gen}, associated to the persistence weight
  $\bar{w}$ induced by the refinement $\widetilde{\mathcal{D}}$ and it
  has an H-space structure with respect to the induced topology.
\end{cor}

\begin{rem}
  (a) We have seen in \S\ref{subsubsec:frag1}, in particular
  Definition \ref{dfn-frag-met}, that for a TPC $\mathcal{C}$ there
  are fragmentation pseudo-metrics $d^{\mathcal{F}}$ defined on
  $\mathrm{Obj}(\mathcal{C})$.  The metrics $\bar{d}^{\mathcal{F}}$
  associated to the persistence weight $\bar{w}$ on $\C_{\infty}$,
  through the construction in \S\ref{subsec:triang-gen}, are defined
  on the same underlying set, $\mathrm{Obj}(\mathcal{C})$.  The
  relation between them is
  $$\bar{d}^{\mathcal{F}}\leq d^{\mathcal{F}}$$ for any family of
  objects $\mathcal{F}$. The interest to work with
  $\bar{d}^{\mathcal{F}}$ rather than with $d^{\mathcal{F}}$ is that
  if $\mathcal{F}$ is a family of triangular generators of
  $\C_{\infty}$, and is closed to the action of $T$, then
  $\bar{d}^{\mathcal{F}}$ is finite (see Remark
  \ref{rem:finite-metr}).

  (b) As discussed in Remark \ref{eq:iterated-tr} (c) (and also in
  \ref{ex-delta} (d)), in this setting too we may define a simpler
  (but generally larger) fragmentation pseudo-metric of the type
  $\underline{\bar{d}}^{\mathcal{F}}$.
\end{rem}

We postpone the proof of \jznote{Theorem \ref{thm:main-alg} to \S \ref{subsubsec:proofTh1}
and \S \ref{subsubsec:proofTh2}. Here, we pursue} with a few remarks
shedding some light on Definition \ref{dfn-extri-inf}.

\begin{remark} \label{rmk-tilde} The shifted triangle
  (\ref{extri-shifted}) has the form
  \[ \widetilde{\Delta}: A \xrightarrow{u'} \Sigma^{-\ceil*{\bar{u}}}
    B \xrightarrow{v'} \Sigma^{-\ceil*{\bar{u}}-\ceil*{\bar{v}}} C
    \xrightarrow{w'} \Sigma^{-\ceil*{\bar{u}}-\ceil*{\bar{v}}-
      \ceil*{\bar{w}}} TA \] where the morphisms $u', v'$ and $w'$
  are:
  \begin{align*} 
    u' & = (\eta_{0, -\ceil*{\bar{u}}})_B \circ \bar{u} \\
    v' & = (\eta_{0,-\ceil*{\bar{u}}-\ceil*{\bar{v}}})_C
         \circ \bar{v} \circ (\eta_{-\ceil*{\bar{u}}, 0})_B\\
    w' & = (\eta_{0, -\ceil*{\bar{u}}-\ceil*{\bar{v}}- \ceil*{\bar{w}}})_{TA}
         \circ \bar{w} \circ (\eta_{-\ceil*{\bar{u}}-\ceil*{\bar{v}}, 0})_C.
  \end{align*}
  In particular, $\ceil*{u'} = \ceil*{v'} = \ceil*{w'} = 0$.
\end{remark}

\begin{ex} \label{ex-tri-c-c-inf} Assume that
  $\Delta: A \xrightarrow{\bar{u}} B \xrightarrow{\bar{v}} C
  \xrightarrow{\bar{w}} \Sigma^{-r} TA$ is a strict exact triangle of
  weight $r$ in $\C$. Consider the following triangle in
  $\C_{\infty}$,
  \[ \Delta_{\infty}: A \xrightarrow{u} B \xrightarrow{v} C
    \xrightarrow{w} TA \] where $u = [\bar{u}]$, $v = [\bar{v}]$ and
  $w = [C \xrightarrow{\bar{w}} \Sigma^{-r} TA
  \xrightarrow{(\eta_{-r,0})_{TA}} TA]$. We claim that
  $\Delta_{\infty}$ is an exact triangle in $\C_{\infty}$ and
  $w_{\infty}(\Delta_{\infty}) \leq r$.  Indeed, by construction,
  $\Delta_{\infty}$ is represented by
  $\bar{\Delta}_{\infty}: A \xrightarrow{\bar{u}} B
  \xrightarrow{\bar{v}} C \xrightarrow{(\eta_{-r,0})_A \circ \bar{w}}
  TA$ and $\ceil*{\bar{u}} = \ceil*{\bar{v}} =0$,
  $\ceil*{(\eta_{-r,0})_A \circ \bar{w}} = r +0 = r$. The shifted
  triangle
  $\widetilde \Delta_{\infty} = \Sigma^{0,0,0,-r}
  \bar{\Delta}_{\infty}$ obviously equals $\Delta$, the initial strict
  exact triangle of weight $r$. Thus,
  $[\widetilde{\Delta}_{\infty}] = \Delta_{\infty}$ and
  $w(\widetilde{\Delta}_{\infty}) = r$. Therefore, $\Delta_{\infty}$
  is exact in $\C_{\infty}$ and $w_{\infty}(\Delta_{\infty}) \leq r$.
\end{ex}

\begin{rem} \label{rem:ex0tr} A special case of the situation in
  Example \ref{ex-tri-c-c-inf} is worth emphasizing.  Any exact
  triangle in $\C_0$, induces an exact triangle in $\C_{\infty}$ with
  unstable weight equal to $0$.  This implies that any morphism
  $f\in \Mor_{\C_{\infty}}(A,B)$ with $\sigma(f)\leq 0$ can be
  completed to an exact triangle of unstable weight $0$ in
  $\C_{\infty}$. Indeed, we first represent $f$ by a morphism
  $\bar{f}\in \Mor^{\alpha}_{\C}(A,B)$ with $\alpha\leq 0$. If
  $\alpha\not=0$, we shift $\bar{f}$ up using the persistence
  structure maps and denote $\bar{f}'=i_{\alpha,0} (\bar{f})$. We
  obviously have $[\bar{f}']=f$. We then complete $\bar{f}'$ to an
  exact triangle in $\C_{0}$. The image of this triangle in
  $\C_{\infty}$ is exact, of unstable weight $0$, and has $f$ as the
  first morphism in the triple.
\end{rem}

\begin{ex} \label{ex-id-rigid} Consider the strict exact triangle
  $\Delta: A \to 0 \to \Sigma^{-r}TA \xrightarrow{id} \Sigma^{-r} TA$
  in $\C$, which is of weight $r\geq 0$ (see
  Remark~\ref{rmk-shift-notation}~(b)). Let $\Delta_{\infty}$ be the
  following triangle in $\C_{\infty}$
  \[ \Delta_{\infty}: A \to 0 \to \Sigma^{-r} TA
    \xrightarrow{\eta_{-r,0}} TA. \] We claim that if
  $\sigma(id_A) = 0$, then
  $\bar{w}(\Delta_{\infty})=w_{\infty}(\Delta_{\infty}) = r$.  Indeed,
  assume $w_{\infty}(\Delta_{\infty}) = s<r$.  Therefore there exists
  a strict exact triangle in $\C$ of the form
  \[ A \to 0 \to \Sigma^{-r'} TA \xrightarrow{w'} \Sigma^{-s} TA, \]
  with $r'\geq r$, of weight $s$ and such that $\ceil*{w'}=0$,
  $[\eta_{-s,0}\circ w'\circ \eta_{-r,-r'}]=[\eta_{-r,0}]$. Notice
  $\ceil*{(\eta_{-s,0}\circ w'\circ \eta_{-r,-r'})}=r-r'+s$. Thus, as
  $r'\geq r>s$, we deduce $\sigma([\eta_{-r,0}])< r$. By writing
  $id_{A}= \eta_{0,-r}\circ \eta_{-r,0}$ we deduce
  $\sigma (id_{A})<0$. We conclude $w_{\infty}(\Delta_{\infty})=r$. We
  next consider a triangle $\Sigma^{k,0,0,k}\Delta_{\infty}$:
  $$\Sigma^{k}A \to 0 \to \Sigma^{-r} TA \to \Sigma^{k}TA$$
  and rewrite it as
  $$ \Delta':B\to 0 \to \Sigma^{-r-k} TB\to TB$$
  with $B= \Sigma^{k}A$. Suppose that $w_{\infty}(\Delta') < r+k$,
  then $\sigma (id_{B}) <0$ by the previous argument. This implies
  $\sigma (id_{A})<0$ and again contradicts our assumption.  Thus
  $w_{\infty}(\Delta')=r+k$ and $\bar{w}(\Delta_{\infty})=r$.
\end{ex}

\begin{ex} \label{ex-embed} Let $f \in \Mor_{\C_{\infty}}(A,B)$ and
  suppose that $\sigma(f) > 0$.  We will see here that we can extend
  $f$ to an exact triangle in $\C_{\infty}$, of unstable weight
  $\leq \sigma(f)+\epsilon$ (for any $\epsilon >0$) but, at the same
  time, no triangle extending $f$ has unstable \jznote{weight} less than
  $\sigma(f)$. \pbnote{Fix $\sigma(f) < t \leq \sigma(f)+\epsilon$ and
    let $\bar{f}\in \Mor^{t}_{\C}(A,B)$ be a representative of $f$.}
  Consider the composition
  $f': A \xrightarrow{\bar{f}} B \xrightarrow{(\eta_{0,-t})_B}
  \Sigma^{-t} B$. Then $f' \in \Mor^0(A,\Sigma^{-t} B)$. There exists
  an exact triangle in $\C_{0}$
  $$\Delta: A \xrightarrow{f'} \Sigma^{-t} B \xrightarrow{v'}
  C \xrightarrow{w} TA$$ for some $C \in {\rm Obj}(\C_0)$. In
  particular \pbnote{$w \in \Mor_{\C}^0(C,TA)$}. Next, consider the
  following triangle
  \[ \Sigma^{0,0,0,-t} \Delta: A \xrightarrow{f'} \Sigma^{-t} B
    \xrightarrow{v'} C \xrightarrow{w'} \Sigma^{-t} TA,
    \,\,\,\,\mbox{where $w' = \eta_{t}^{TA} \circ w$}. \] By
  \pbnote{Remark~\ref{rmk-shift-notation}(b)},
  $\Sigma^{0,0,0,-t} \Delta$ is a strict exact triangle in $\C$ of
  weight $t$. Finally, consider the following triangle in $\C$,
  obtained by shifting up the last three terms of \jznote{$\Sigma^{0,0,0,-t}\Delta$}:
  \[ \bar{\Delta}: A \xrightarrow{\bar{f}= (\eta_{-t, 0})_B \circ f'}
    B \xrightarrow{v: = \Sigma^{t} v'} \Sigma^t C \xrightarrow{w:
      =\Sigma^{t} w'} TA. \] Its image in $\C_{\infty}$ is the
  triangle
  \begin{equation} \label{eq:triangle4} \Delta_{\infty}: A
    \xrightarrow{f=[\bar{f}]} B \xrightarrow{[v]} \Sigma^t C
    \xrightarrow{[w]} TA
  \end{equation}
  and is exact. In the terminology of Definition \ref{dfn-extri-inf},
  the representative $\tilde{\Delta}$ of $\Delta_{\infty}$ is the
  strict exact triangle $\Sigma^{0,0,0,-t} \Delta$. In particular,
  $w_{\infty}(\Delta_{\infty})\leq t$.  Notice that Definition
  \ref{dfn-extri-inf} immediately implies that any triangle
  $\Delta'': A\xrightarrow{f} B\to D\to TA$ in $\C_{\infty}$ satisfies
  $w_{\infty}( \Delta'')\geq \sigma (f)$ (because, with the notation
  of the definition, the weight of the triangle $\tilde{\Delta}$ in
  that definition is at least $\ceil*{\bar{u}}$).  At the same time
  $\bar{w}(\Delta_{\infty})=0$ because
  $\Sigma^{t,0,0,t}(\Delta_{\infty})$ has as
  representative
  $$\Sigma^{t}\Delta : \Sigma^{t}A\to B\to \Sigma^{t}C\to
  \Sigma^{t}TA$$ which is exact in $\C_{0}$.
\end{ex}

\begin{ex}\label{ex-embed2} Let $\Delta: A\to B\to C\to TA$ be an
  exact triangle in $\C_{0}$. It is clear that, in $\C_{\infty}$ and
  for $s\leq 0$, the corresponding triangles of the form
  $\Delta_{s}:\Sigma^{s}A\to B\to C\to \Sigma^{s}TA$ (defined using
  the (pre)-composition of the maps in $\Delta$ with the appropriate
  $\eta$'s) have the property $\bar{w}(\Delta_{s})=0$. On the other
  hand, if $s>0$ this is no longer the case, in general. Indeed,
  assuming $s>0$ and $\bar{w}(\Delta_{s})=0$ it follows that for some
  possibly even larger $r>0$ we have $w_{\infty}(\Delta_{r})=0$.  That
  means that there is an exact triangle in $\C_{0}$ of the form
  $\Sigma^{r}A\to B\to C\to \Sigma^{r}TA$. This triangle can be
  compared to the initial $\Delta$ and it follows that
  $\eta_{r}^{A}: \Sigma^{r}A\to A$ is a $0$-isomorphism which implies
  that $\sigma(\mathds 1_{A})\leq -r\leq -s$.  Using this remark, we can
  revisit the triangle $\Delta_{\infty}$ from Example \ref{ex-embed}
  and deduce (from the rotation property of exact triangles) that if
  we consider a triangle $\Delta_{\infty}^{s}$ obtained by replacing
  $\Sigma^{t}C$ with $\Sigma^{s}C$ with $s>t$ in $\Delta_{\infty}$ in
  equation (\ref{eq:triangle4}) (and using the appropriate
  pre/compositions with the $\eta$'s for the maps in the triangle),
  then $\bar{w}(\Delta_{\infty}^{s})>0$.
\end{ex}

\begin{ex} \label{ex-other}
  Let $\Delta : A\to 0\to \Sigma^{r} TA \to TA$ be a triangle in
  $\C_{\infty}$ with the last map (the class of) $\eta^{TA}_{r}$ and
  with $r\geq 0$.  Assuming $\sigma(\mathds 1_{A})=0$, we want to notice that
  $w_{\infty}(\Delta) = r$ and $\bar{w}(\Delta)=0$. The fact that
  $\bar{w}(\Delta)=0$ is obvious because
  $\Sigma^{r,0,0,r}\Delta :\Sigma^{r}A\to 0\to \Sigma^{r} TA\to
  \Sigma^{r}TA$ is exact in $\C_{0}$.  Now assume that
  $w_{\infty}(\Delta)=s<r$. Then there is a triangle
  $A\to 0 \to \Sigma^{r-s'} TA\xrightarrow{u} \Sigma^{-s}TA$ which is
  strict exact in $\C$ and with $s\geq s'$. Using now the definition
  of strict exact triangles and the existence of the exact triangle
  $A\to 0 \to TA\to TA$ in $\C_{0}$, we deduce that there exists an
  $s$-isomorphism $\phi: TA \to \Sigma^{r-s'}TA$ with a right inverse
  $\psi:\Sigma^{s+r-s'}TA \to TA$ that coincides with $\eta^{TA}$ in
  $\C_{\infty}$. As a result, $\phi$ coincides with
  $\eta^{TA}_{0,r-s'}$ in $\C_{\infty}$ and thus
  $0=\ceil*{\phi}\geq r-s'$ (because, due to $\sigma (\mathds 1_{A})=0$, we
  have that $\sigma([\eta^{TA}_{0,r-s'}])=r-s'$). Therefore,
  $s'\geq r$ which contradicts $s<r$.
\end{ex}

\begin{rem}\label{rem:tr-ex2}
  (a) The definition of the exact triangles in $\C_{\infty}$ is
  designed precisely to allow for the construction in Example
  \ref{ex-embed}. This is quite different compared to the case when
  the spectral invariant of $f$ is non positive (compare with Remark
  \ref{rem:ex0tr}) because the persistence structure maps can be used
  to ``shift'' up but not down.  It also follows from Example
  \ref{ex-embed} that for $f\in \Mor_{\C_{\infty}}(A,B)$ with
  $\sigma(f)\geq 0$ we have:
  \begin{equation}
    \sigma(f)= \inf \{\  w_{\infty}(\Delta) \ | \ \Delta :
    A\xrightarrow{f} B\to C \to TA \ \}~.~
  \end{equation}

  (b) It is not difficult to see that if
  $w_{\infty}(\Sigma^{s,0,0,s}\Delta)\leq r$ for a triangle $\Delta$
  in $\C_{\infty}$ and $s\geq 0$, then $w_{\infty}(\Delta)\leq
  r+s$. Thus, a triangle in $\C_{\infty}$ is exact if and only if its
  weight $\bar{w}$ is finite.
\end{rem}

\subsection{The category $\C_{\infty}$ is
  triangulated} \label{subsubsec:proofTh1} We have noticed before that
$\C_{\infty}$ is additive and the functor $T$ has been extended to all
of $\C$ (as discussed in Remark \ref{rem:shifts-T} (b)) in a way that
commutes with $\Sigma$ and thus induces an automorphism of
$\C_{\infty}$. We now start to check the axioms of a triangulated
category as listed in Section 10.2 in Weibel \cite{Weib:intro-halg}
for the class of exact triangles in $\C_{\infty}$ as introduced in
Definition \ref{dfn-extri-inf}.

{\bf Axiom TR1}. For any $[f] \in {\rm Hom}_{\C_{\infty}}(A,B)$, there
exists an exact triangle in $\C_{\infty}$ in the form of
$A \xrightarrow{[f]} B \to C \to TA$ for some
$C \in {\rm Obj}(\C_{\infty})$. This is due to Example \ref{ex-embed}.
For any $A \in {\rm Obj}(\C_{\infty})$, the triangle
$A \xrightarrow{\mathds{1}_A} A \to 0 \to TA$ is exact in $\C_{\infty}$ and of
unstable weight $0$. This follows from Remark \ref{rem:ex0tr}.  The
verification of this first axiom is completed by the following
statement.

\begin{lem}\label{lem:iso-oftr}
  Given an exact triangle $\Delta$ in $\C_{\infty}$, any other
  triangle $\Delta'$ in $\C_{\infty}$ such that
  $\Delta' \simeq \Delta$ (in the sense that all the vertical
  morphisms are isomorphisms in $\C_{\infty}$) is also an exact
  triangle in $\C_{\infty}$.
\end{lem}
 
\begin{proof} Pick a representative $\bar{\Delta}'$ in $\C$ of
  $\Delta'$ as well as the associated triangle $\widetilde{\Delta}'$
  in $\C_{0}$, given as in Definition \ref{dfn-extri-inf}.  By
  definition $\Delta$ admits a representative $\bar{\Delta}$ such that
  $\widetilde{\Delta}:\Sigma^{0,s_{1},s_{2},s_{3}}\bar{\Delta}$ is
  strict exact in $\C$.  Because $\Delta'$ and $\Delta$ are isomorphic
  in $C_{\infty}$, there is some very large $s>0$ and a morphism of
  triangles $h:\widetilde{\Delta}'\to \Sigma^{-s}\widetilde{\Delta}$
  such that all vertical components of $h$ are
  $s$-isomorphisms. Clearly, $\Sigma^{-s}\widetilde{\Delta}$ is strict
  exact and, by Remark \ref{rem:tr-iso-0}, it follows that there are
  $k,t>0$ such that $\Sigma^{0,0,-k,-k-t}\widetilde{\Delta}'$ is
  strict exact. But this means that $\Delta'$ is exact in
  $\C_{\infty}$.
\end{proof}

\medskip

{\bf Axiom TR2}. Suppose
$\Delta: A \xrightarrow{u} B \xrightarrow{v} C \xrightarrow{w} TA$ is
an exact triangle in $\C_{\infty}$. Notice that its first positive
rotation
$R(\Delta): B \xrightarrow{v} C \xrightarrow{w} TA \xrightarrow{-Tu}
TB$ is also an exact triangle in $\C_{\infty}$. Indeed, by definition,
there exists a triangle
$\bar{\Delta}: A \xrightarrow{\bar{u}} B \xrightarrow{\bar{v}} C
\xrightarrow{\bar{w}} TA$ representing $\Delta$ such that the shifted
triangle
$\widetilde{\Delta}: A \xrightarrow{u'} \Sigma^{-t_1} B
\xrightarrow{v'} \Sigma^{-t_1-t_2} C \xrightarrow{w'} \Sigma^{-r} TA$
(where we put $t_1 = \ceil*{\bar{u}}$, $t_2 = \ceil*{\bar{v}}$,
$t_3 = \ceil*{\bar{w}}$ and $r = t_1 + t_2 + t_3$) is a strict exact
triangle in $\C$ of weight $r$.  By Proposition \ref{prop-rot}, the
first positive rotation
\[ R(\widetilde{\Delta}): \Sigma^{-t_1} B \xrightarrow{v'}
  \Sigma^{-t_1-t_2} C \xrightarrow{w''} \Sigma^{-r} TA
  \xrightarrow{u''} \Sigma^{-t_1 - 2r} TB \] is a strict exact
triangle in $\C$ of weight $2r$. We have
\begin{align*} 
u'' & = -(\eta_{0, -t_1 - 2r})_B \circ \bar{u} \circ (\eta_{-r, 0})_A\\
v' & = (\eta_{0, -t_1 -t_2})_C \circ \bar{v} \circ (\eta_{-t_1, 0})_B \\
w''&\simeq_r w = (\eta_{0, -r})_A \circ \bar{w} \circ (\eta_{-t_1 - t_2, 0})_C. 
\end{align*}
Consider a new triangle: 
\begin{equation} \label{rot-tilde-tri} \Sigma^{t_1}
  R(\widetilde{\Delta}): B \xrightarrow{\Sigma^{t_1} v'} \Sigma^{-t_2}
  C \xrightarrow{\Sigma^{t_1} w''} \Sigma^{-r+t_1} TA
  \xrightarrow{\Sigma^{t_1} u''} \Sigma^{-2r} TB.
\end{equation}
By Remark \ref{rmk-shift-notation} (1),
$\Sigma^{t_1} R(\widetilde{\Delta})$ is also a strict exact triangle
in $\C$ of weight $2r$. By shifting up this triangle we get to
$B\to C\to TA \to TB$ that represents $R(\Delta)$ and thus
$w_{\infty}(R(\Delta))\leq 2r$. In a similar way, using Remark
\ref{rem:neg-rot}, it follows that the negative rotation of $\Delta$,
$R^{-1}(\Delta)$, is exact.

\medskip

{\bf Axiom TR3}. Completing a commutative square into a morphism of
triangles follows from Proposition \ref{prop-ind}. Namely, consider
the $\C_{\infty}$ diagram:
$$
\xymatrix{ \Delta : & A \ar[r] \ar[d]_{a} & B \ar[r] \ar[d]_{b}
  & C\ar@{-->}[d]_{c} \ar[r] & TA\ar[d] \\
  \Delta' : & A' \ar[r] & B' \ar[r] & C' \ar[r] & TA'}
$$
with rows exact triangles and with the left square commutative. We
need to construct $c:C\to C'$ making commutative the last two squares.
We consider the associated strict exact triangles $\widetilde{\Delta}$
and $\widetilde{\Delta}'$ in $\C$, as in Definition
\ref{dfn-extri-inf}. We use Proposition \ref{prop-ind} to obtain a
morphism of triangles of the form
$h:\widetilde{\Delta}\to \Sigma^{-k, -k,-k,-k-s}\Delta'$ for
$k,s\in \R$ sufficiently big, $h=(\bar{a},\bar{b},\bar{c},
\bar{a}')$. In particular, the morphism $\bar{c}$ has the form
$\bar{c}: C\to \Sigma^{-m} C'$ for a certain large $m\in \R$. We
consider the composition $c'=(\eta_{-m,0})_{C'}\circ\bar{c}$ (where we
recall $(\eta_{-m,0})_{C'}:\Sigma^{-m} C'\to C'$) and we put
$c = [c']$. The needed commutativities follow from the ones in
Proposition~\ref{prop-ind}.  \medskip

\begin{remark}[Five Lemma]\label{rem:iso-cone}
  A further useful consequence of this argument is that, if the
  morphisms $a$ and $b$ are isomorphisms (in $\C_{\infty}$), then so
  is $c$. Indeed, if $a$ and $b$ are isomorphisms, then
  $\bar{a},\bar{b}$ are $t$-isomorphisms for $t$ large and, by
  Corollary \ref{cor-1}, $\bar{c}$ is a $2t$-isomorphism which implies
  that $c$ is an isomorphism. \pbnote{Note that usually the five lemma
    comes as a corollary of a triangulated structure on a category,
    however here we will use it in order to show that $\C_{\infty}$ is
    triangulated.}
\end{remark}

{\bf Axiom TR4}. Given exact triangles
$\Delta_1: U \to V \to Z \to TU$ and
$\Delta_2: U \to V' \to Z' \to TU$, $\Delta_3: V\to V' \to W \to TV$
in $\C_{\infty}$ commuting as in the diagram below, there exists an
exact triangle $\Delta_4: Z \to Z' \to W \to TZ$ in $\C_{\infty}$
completing the diagram (with the bottom right-most square
anti-commutative).
\begin{equation} \label{oct-c-inf} \xymatrix{
    U \ar[d] \ar[r] & V \ar[r] \ar[d] & Z \ar[d] \ar[r]& TU\ar[d] \\
    U \ar[d] \ar[r] & V' \ar[r] \ar[d] & Z' \ar[d] \ar[r] & TU\ar[d] \\
    0 \ar[r] \ar[d]&  W \ar[r] \ar[d]& W\ar[r]\ar[d] & 0\ar[d]\\
    TU\ar[r] & TV \ar[r] & TZ \ar[r] & T^{2}U}
\end{equation}
We will give the proof in the next section where we will also show
that the weight $\bar{w}$ satisfies the weighted octahedral axiom from
Definition \ref{def:triang-cat-w} (i).

\subsection{An exotic triangular weight on
  $\C_{\infty}$} \label{subsubsec:proofTh2}

If an additive category $\mathcal{D}$ together with an automorphism
$T$ and a class of distinguished triangles satisfies the Axioms TR1,
TR2, TR3 and the Five Lemma (in the form in Remark \ref{rem:iso-cone})
and also the property in Definition \ref{def:triang-cat-w} (i),
without reference to the weight $w$ and without the inequality
(\ref{eq:weight:ineq}), then it satisfies TR4 and thus it is
triangulated.  This is a simple exercise in manipulating exact
triangles that we leave to the reader.  We have already seen in
\S\ref{subsubsec:proofTh1} that $\C_{\infty}$ together with the
functor $T$ induced from $\C$ and the class of exact triangles from
Definition \ref{dfn-extri-inf} satisfies TR1, TR2, TR3 and the Five
Lemma.  Thus, if we show that the weighted octahedral axiom is
satisfied by $\bar{w}$ in $\C_{\infty}$ (relative to our class of
exact triangles) we deduce that $\C_{\infty}$ is
triangulated. Moreover, if $\bar{w}$ also satisfies the normalization
in Definition \ref{def:triang-cat-w} (ii), then $\bar{w}$ is a
triangular weight on $\C_{\infty}$. Finally, to complete the claim in
Theorem \ref{thm:main-alg} we also need to show that $\bar{w}$ is
subadditive and $\bar{w}_{0}=0$.

We start below with the weighted octahedral axiom and will end with
the normalization property.

\begin{lem}
  The class of exact triangles in $\C_{\infty}$ and the weight
  $\bar{w}$ as defined in Definition \ref{dfn-extri-inf} satisfies the
  weighted octahedral axiom from Definition \ref{def:triang-cat-w}
  (i).
\end{lem}
\begin{proof} 
  Recall that given the exact triangles
  $\Delta_{1}: A\to B\to C\to TA$ and $\Delta_{2}: C\to D\to E\to TC$
  in $\C_{\infty}$ we need to show that there are exact triangles:
  $\Delta_{3}: B\to D\to F\to TB$ and
  $\Delta_{4}: TA\to F\to E\to T^{2}A$ making the diagram below
  commute, except for the right-most bottom square that anti-commutes,
  \[ \xymatrix{
      A \ar[r] \ar[d] & 0 \ar[r] \ar[d] & TA \ar[d]\ar[r]& TA\ar[d] \\
      B \ar[r] \ar[d] & D \ar[r] \ar[d] & F \ar[d]\ar[r]& TB \ar[d]\\
      C \ar[r] \ar[d] & D \ar[r]\ar[d] & E\ar[r]\ar[d] & TC\ar[d] \\
      TA\ar[r] & 0 \ar[r] & T^{2}A \ar[r]& T^{2} A }
  \]
  and such that
  $\bar{w}(\Delta_{3}) + \bar{w}(\Delta_{4})\leq \bar{w}(\Delta_{1}) +
  \bar{w}(\Delta_{2})$.

  In $\C$, there are triangles $\bar{\Delta}_1: A \to B \to C \to TE$
  with non-negative morphisms shifts $t_1, t_2, t_3$ and
  $\bar{\Delta}_2: C \to D \to E \to TC$ with non-negative morphisms
  shifts $k_1, k_2, k_3$ that represent $\Delta_{1}$ and $\Delta_{2}$
  respectively and such that the associated triangles
  \[ \widetilde{\Delta}_1: A \to \Sigma^{-t_1} B \to \Sigma^{-t_1-t_2}
    C \to \Sigma^{-r} TA \,\,\,\,\,\mbox{where $r = t_1 + t_2 +
      t_3$} \] and
  \[ \widetilde{\Delta}_2: C \to \Sigma^{-k_1} D \to \Sigma^{-k_1
      -k_2} E \to \Sigma^{-s} TC \,\,\,\,\,\mbox{where
      $s = k_1 + k_2 + k_3$} \] are strict exact triangles in $\C$ of
  weight $r$ and $s$, respectively.  Consider
  $\Sigma^{-t_{1}-t_{2}} \widetilde{\Delta}_2$,
  \[ \Sigma^{-t_{1}-t_{2}} \widetilde{\Delta}_2: \Sigma^{-t_{1}-t_{2}}
    C \to \Sigma^{-k_1 -t_{1}-t_{2}} D \to
    \Sigma^{-k_1-k_2-t_{1}-t_{2}} E \to \Sigma^{-s-t_{1}-t_{2}} TC. \]
  The weighted octahedral property for strict exact triangles in $\C$
  in Proposition \ref{prop-w-oct} implies that we can construct the
  following commutative diagram in $\C$ (with the bottom right square
  which is $r$-anti-commutative).
  \begin{equation} \label{oct-c-inf-2} \xymatrix{
      A \ar[d] \ar[r] & 0 \ar[r] \ar[d] & TA \ar[d] \ar[r] & TA\ar[d] \\
      \Sigma^{-t_1} B \ar[d] \ar[r] & \Sigma^{-k_1-t_{1}-t_{2}}
      D\ar[r] \ar[d] & D' \ar[d]\ar[r]
      & \Sigma^{-t_1} TB  \ar[d]\\
      \Sigma^{-t_{1}-t_{2}} C \ar[r]\ar[d] & \Sigma^{-k_1-t_{1}-t_{2}}
      D \ar[r] \ar[d] & \Sigma^{-k_1-k_2-t_{1}-t_{2}} E \ar[r] \ar[d]
      & \Sigma ^{-t_{1}-t_{2}-s}TC \ar[d]\\
      \Sigma^{-r}TA \ar[r]& 0\ar[r] & \Sigma^{-r-s}T^{2}A\ar[r] &
      \Sigma^{-r-s}T^{2} A}
  \end{equation}
  Here the triangle
  $\Delta_{3}' :\Sigma^{-t_1} B \to \Sigma^{-k_1-t_{1}-t_{2}} D \to D'
  \to \Sigma^{-t_1} TB$ is an exact triangle in $\C_{0}$. The triangle
  $\Delta_{3}'' : \Sigma^{k_{1}+t_{2}}B\to D\to
  \Sigma^{k_{1}+t_{1}+t_{2}}D'\to \Sigma^{k_{1}+t_{2}}TB$ obtained by
  shifting up $\Delta_{3}'$ by $k_{1}+t_{1}+t_{2}$ is also exact in
  $\C_{0}$. Let $[\Delta''_{3}]$ be the image of this triangle in
  $\C_{\infty}$. We put $F= \Sigma^{k_{1}+t_{1}+t_{2}}D'$ and take
  $\Delta_{3}$ to be the triangle in $\C_{\infty}$
  $$\Delta_{3}: B\to D\to F\to TB$$ obtained  by applying
  $\Sigma^{-k_{1}-t_{2},0,0,-k_{1}-t_{2}}$ to $[\Delta_{3}'']$. We
  obviously have $w_{\infty}([\Delta_{3}''])=0$ and thus
  $\bar{w}(\Delta_{3})=0$.

  The next step is to identify the triangle $\Delta_{4}$. Proposition
  \ref{prop-w-oct} implies that the third column in
  (\ref{oct-c-inf-2}):
  $$\Delta_{4}' : TA\to \Sigma^{-k_{1}-t_{1}-t_{2}}F
  \to \Sigma^{-k_{1}-k_{2}-t_{1}-t_{2}}E\to \Sigma^{-r-s}T^{2}A$$ is a
  strict exact triangle in $\C$. We let $[\Delta'_{4}]$ be the image
  of this triangle in $\C_{\infty}$ and let $\Delta_{4}$ be given by
  applying $\Sigma^{0, k_{1}+t_{1}+t_{2},k_{1}+k_{2}+t_{1}+t_{2},r+s}$
  to $[\Delta'_{4}]$:
  $$\Delta_{4} : TA \to F \to E\to T^{2}A~.~$$
  We deduce from Definition \ref{dfn-extri-inf} that
  $w_{\infty}(\Delta_{4})\leq r+s$ and thus also
  $\bar{w}(\Delta_{4})\leq r+s$.

  The commutativity required in the statement follows from that
  provided by Proposition \ref{prop-w-oct} for (\ref{oct-c-inf-2}).
\end{proof}

\begin{remark} \label{rem:weak-oct} For the triangle $\Delta_{3}$
  produced in this proof it is easy to see that
  $w_{\infty}(\Delta_{3})\leq k_{1}+t_{2}$.  Therefore we have:
  \[ w_{\infty}(\Delta_3) + w_{\infty}(\Delta_4) \leq (k_1 + t_2) +
    (r+s) \leq 2
    (r+s)=2(w_{\infty}(\Delta_{1})+w_{\infty}(\Delta_{2})). \] Thus
  the weight $w_{\infty}$ satisfies a weak form of the weighted
  octahedral axiom.
\end{remark}

The next step in proving Theorem \ref{thm:main-alg} is to show the
normalization property in Definition \ref{def:triang-cat-w} (ii). This
property is satisfied with the constant $\bar{w}_{0}=0$. Indeed, any
triangle $0\to X\to X\to 0$ and all its rotations are exact in
$\C_{0}$ and thus they are of unstable weight equal to $0$.  The last
verification needed is to see that, if $B=0$ in the diagram of the
weighted octahedral axiom, then the triangle $\Delta_{3}$ -
constructed in the proof of the Lemma - can be of the form:
$\Delta_{3}: 0\to D\to D\to 0$. This is trivially satisfied in our
construction because if $B=0$ we may take $t_{1}=t_{2}=0$ and the
triangle
$\Delta'_{3}: 0 \to \Sigma^{-k_{1}}D\xrightarrow{\mathds{1}}
\Sigma^{-k_{1}}D\to 0$.

\

Finally, to finish the proof of Theorem \ref{thm:main-alg} we need to
show that $\bar{w}$ is sub-additive. Thus, assuming that
$\Delta: A\to B\to C\to TA$ is exact in $\mathcal{C}_{\infty}$ and $X$
is an object in $\C$, then
$\bar{w}(X\oplus \Delta)\leq \bar{w}(\Delta)$ where the triangle
$X\oplus\Delta$ has the form
$X\oplus \Delta: A\to X\oplus B\to X\oplus C\to TA$. We consider the
strict exact triangle in $\C$
$$\widetilde{\Delta}: A\to \Sigma^{-s_{1}}B
\to \Sigma^{-s_{1}-s_{2}}C \to \Sigma^{-s_{1}-s_{2}-s_{3}}TA$$
associated to $\Delta$ as in Definition \ref{dfn-extri-inf} with
$s_{i}\geq 0$, $1\leq i\leq 3$.  Consider the triangle
\jznote{$$\Delta' : 0\to \Sigma^{-s_{1}}X \xrightarrow{\eta_{s_{2}}^{X}}
\Sigma^{-s_{1}-s_{2}}X\to 0~.~$$} This triangle is obtained from the
exact triangle in $\C_{0}$,
$0\to \Sigma^{-s_{1}}X\xrightarrow{\mathds{1}} \Sigma^{-s_{1}}X\to 0$ by
applying $\Sigma^{0,0,-s_{2}, -s_{2}}$ and it is of strict weight
$\leq s_{2}$. By Lemma \ref{claim-frag-sum} we have
$w(\Delta' \oplus \widetilde{\Delta})\leq w(\widetilde{\Delta})$.  We
now notice that $\Delta'\oplus \widetilde{\Delta}$ can be viewed as
obtained from $X\oplus\Delta$ by applying
$\Sigma^{0,-s_{1},-s_{1}-s_{2}, -s_{1}-s_{2}-s_{3}}$ and thus
$w_{\infty}(X\oplus \Delta)\leq w(\widetilde{\Delta})$ which implies
the claim. The proof for $\Delta\oplus X$ is similar.

% !TEX root = TPC.tex

\subsection{Fragmentation pseudo-metrics and their
  non-degeneracy.}\label{subsec:rem-nondeg}
The purpose of this section is to rapidly review some of the
properties of the persistence fragmentation pseudo-metrics.  We start
by recalling the main definitions, we then discuss some algebraic
properties and conclude with a discussion of non-degeneracy.
\subsubsection{Persistence fragmentation pseudo-metrics, review of
  main definitions} Assume that $\C$ is a TPC. We have defined in \S
\ref{subsubsec:frag1} and \S\ref{subsubsec:ex-cinfty} three types of
similarly defined measurements on the objects of $\C$ that, after
symmetrization, define fragmentation pseudo-metrics on
$\mathrm{Obj}(\C)$.  In \S \ref{subsubsec:frag1} this construction
uses directly the weight $w$ of the strict exact triangles in $\C$
(making use of Proposition \ref{prop-w-oct}) and it gives rise to
pseudo-metrics $d^{\mathcal{F}}$ as in Definition \ref{dfn-frag-met}
as well as a simplified version $\underline{d}^{\mathcal{F}}$
mentioned in Remark \ref{ex-delta} (d).

In \S\ref{subsubsec:ex-cinfty} the aim is to consider weights on the
triangles of the category $\C_{\infty}$.  This category is
triangulated - this is the main part of Theorem \ref{thm:main-alg} -
and its exact triangles are endowed with an unstable weight
$w_{\infty}$ as well as with a (smaller) stable weight $\overline{w}$,
as given by Definition \ref{dfn-extri-inf}. Working in the category
$\C_{\infty}$ has a significant advantage compared to the category
$\C$ because, by contrast to $\C$, in $\C_{\infty}$ any morphism can
be completed to an exact triangle of finite weight. Moreover, in
$\C_{\infty}$ exact triangles have the standard form expected in a
triangulated category and do not involve shifts.  As a consequence,
we will focus here on the fragmentation pseudo-metrics defined using
$w_{\infty}$ and, mainly, $\overline{w}$.

An important remark at this point is that the unstable weight
$w_{\infty}$ does not satisfy the weighted octahedral axiom (but only
its weak form as discussed in Remark \ref{rem:weak-oct}) and thus only
the pseudo-metrics of the form $\underline{\bar{d}}^{\mathcal{F}}$ can
be defined using it. By contrast, $\overline{w}$ does satisfy the
weighted octahedral axiom and there is a pseudo-metric
$\bar{d}^{\mathcal{F}}$ associated to it through the construction in
\S\ref{subsec:triang-gen}.  Both $\overline{w}$ and $w_{\infty}$ are
subadditive and satisfy the normalization property in Definition
\ref{def:triang-cat-w} with $w_{0}=0$.

To eliminate possible ambiguities we recall the definitions of the two
relevant pseudo-metrics here. Both of them are based on considering a
sequence of exact triangles in $\C_{\infty}$ as below:
\begin{equation}\label{eq:iterated-tr2}\xymatrixcolsep{1pc} \xymatrix{
 Y_{0} \ar[rr] &  &  Y_{1}\ar@{-->}[ldd]  \ar[r] &\ldots  \ar[r]& Y_{i} \ar[rr] &  &  Y_{i+1}\ar@{-->}[ldd]  \ar[r] &\ldots \ar[r]&Y_{n-1} \ar[rr] &   &Y_{n} \ar@{-->}[ldd]  &\\
 &         \Delta_{1}                  &  & & &  \Delta_{i+1}                          & &  &  &    \Delta_{n}             \\
  & X_{1}\ar[luu] &  & & &X_{i+1}\ar[luu] &  &  & &X_{n}\ar[luu] }
\end{equation}
where the dotted arrows represent maps $Y_{i} \to TX_{i}$. We fix a
family of objects $\mathcal{F}$ in $\C$ with $0\in \mathcal{F}$ and
now:
\begin{equation}
\underline{\bar{\delta}}^{\mathcal{F}}(X,X') =\inf\left\{ \sum_{i=1}^{n}w_{\infty}(\Delta_{i}) \, \Bigg| \, \begin{array}{ll} \mbox{$\Delta_{i}$ are successive exact triangles as in (\ref{eq:iterated-tr2})} \\ \mbox{with $Y_{0}=X'$,
$X=Y_{n}$  and  $X_i \in \mathcal{F}$,   $n \in \N$} \end{array} \right\}
\end{equation}

\begin{equation}
\bar{\delta}^{\mathcal{F}}(X,X') =\inf\left\{ \sum_{i=1}^{n}\bar{w}(\Delta_{i}) \, \Bigg| \, \begin{array}{ll} \mbox{$\Delta_{i}$ are as in (\ref{eq:iterated-tr2}) with $Y_{0}=0$, $X=Y_{n}$, $X_i \in \mathcal{F}$,   } \\ \mbox{$n \in \N$  except for some $j$ such that $X_{j}=T^{-1}X'$} \end{array} \right\}~.~
\end{equation}
Finally, the pseudo-metrics $\underline{\bar{d}}^{\mathcal{F}}$ and $\bar{d}^{\mathcal{F}}$ are obtained 
by symmetrizing $\underline{\bar{\delta}}^{\mathcal{F}}$ and $\bar{\delta}^{\mathcal{F}}$, respectively:
$$\underline{\bar{d}}^{\mathcal{F}}(X,X')=\max\{\underline{\bar{\delta}}^{\mathcal{F}}(X,X'),\underline{\bar{\delta}}^{\mathcal{F}}(X',X) \}, \ \ \bar{d}^{\mathcal{F}}(X,X')=\max\{\bar{\delta}^{\mathcal{F}}(X,X'), \bar{\delta}^{\mathcal{F}}(X',X) \}~.~$$

\subsubsection{Algebraic properties.}
There are many fragmentation pseudo-metrics of persistence type
associated to the same weight, depending on the choices of family
$\mathcal{F}$.  In fact, the choices available are even more abundant
for the following two reasons:
\begin{itemize}
\item[i.] triangular weights themselves can be mixed. For instance, if
  $\C$ is a TPC, there is a triangular weight of the form
  $\overline{w}^{+}=\overline{w}+w_{fl}$ that is defined on
  $\C_{\infty}$ (where $w_{fl}$ is the flat weight \jznote{defined in \S\ref{subsec:triang-gen}}).
\item[ii.] fragmentation metrics themselves can also be mixed. If
  $d^{\mathcal{F}}$ and $d^{\mathcal{F}'}$ are two fragmentation
  pseudo-metrics (whether defined with respect to the same weight or
  not), then the following expressions
  $\alpha ~d^{\mathcal{F}}+\beta~ d^{\mathcal{F}'}$ with
  $\alpha, \beta\geq 0$ as well as
  $\max \{d^{\mathcal{F}}, d^{\mathcal{F}'}\}$ are also
  pseudo-metrics.
\end{itemize}

In essence, while it is not easy to produce interesting sub-additive
triangular weights on a triangulated category, once such a weight is
constructed - as in the case of the persistence weight $\overline{w}$
defined on $\C_{\infty}$ (where $\C$ is a TPC) - one can associate to
it a large class of pseudo-metrics, either by combining the weight
with the flat one and/or by ``mixing'' the pseudo-metrics associated
to different families $\mathcal{F}$.  Another useful (and obvious)
property relating the pseudo-metrics $d^{\mathcal{F}}$ and
$d^{\mathcal{F}'}$ associated to the same triangular weight is that:
 
\begin{itemize}
\item[iii.] If $\mathcal{F}\subset \mathcal{F}'$, then
  $d^{\mathcal{F}'}\leq d^{\mathcal{F}}$.
\end{itemize}

\subsubsection{Vanishing and non-degeneracy of fragmentation
  metrics.}\label{subsubsec:non-deg-v}
We fix here a TPC denoted by $\C$ together with the associated weights
and pseudo-metrics, as above.  We will assume that $0\in
\mathcal{F}$. We will denote by $\bar{d}^{\emptyset}$ the
pseudo-metric associated to the family consisting of only the element
$0$. In view of point iii. above $\bar{d}^{\emptyset}$ is \jznote{an} upper
bound for all the pseudo-metrics $\bar{d}^{\mathcal{F}}$.

It is obvious, as noticed in Remark \ref{ex-delta}, that, in general
$\bar{d}^{\mathcal{F}}$ is degenerate and, for instance, if
$\mathcal{F}=\mathrm{Obj}(\mathcal{C})$ then
$\bar{d}^{\mathcal{F}}\equiv 0$. The rest of Remark \ref{ex-delta}
also continues to apply to $\bar{d}^{\mathcal{F}}$. We list below some
other easily proven properties. We assume for all the objects $X$
involved here that $\sigma(\mathds{1}_{X})=0$ and we will use the calculations
in Examples \ref{ex-id-rigid}, \ref{ex-embed}, \ref{ex-embed2} and
\ref{ex-other}. Recall the notion of $r$-isomorphism from Definition
\ref{dfn-r-iso}, in particular, this is a morphism in
$\mathcal{C}_{0}$. A $0$-isomorphism is simply an isomorphism in the
category $\C_{0}$ and is denoted by $\equiv$.
\begin{itemize}
\item[i.] If $X\equiv X'$, then $\bar{d}^{\mathcal{F}}(X,X')=0=
  \underline{\bar{d}}^{\mathcal{F}}(X,X')$ for any family
  $\mathcal{F}$.
\item[ii.] If $\bar{\delta}^{\emptyset}(X,X')=s$, then
  $s=\inf\{ r \in \R \ |\ \exists \ \phi : \Sigma^{k}X'\to X \ \
  \mbox{$r$-isomorphism}, \ k\geq 0\} $.
\item[iii.] If $\underline{\bar{\delta}}^{\emptyset}(X,X')=s$ , then
  $s=\inf\{ r \in \R \ |\ \exists \ \phi : \Sigma^{k}X'\to X \ \
  \mbox{$r$-isomorphism}, \ r\geq k\geq 0 \}$.
 \item[iv.] We have $\bar{d}^{\emptyset}(X,\Sigma^{r}X)=r$ for any
   $r\in \R$ (this follows from Examples \ref{ex-id-rigid} and
   \ref{ex-other}).
\end{itemize}

Thus $\bar{d}^{\emptyset}$ is finite for objects that are isomorphic
in $\C_{\infty}$ and $\bar{d}^{\emptyset}(X,X')$ is the optimal
upper-bound $r$ such there are $s$-isomorphisms in $\C$ with
$s\leq r$, from some positive shift of $X$ to $X'$ and, similarly,
from some positive shift of $X'$ to $X$. \jznote{To some extent, $\bar{d}^{\emptyset}$ can be viewed as an abstract analogue of the interleaving distance in the persistence module theory (cf.~Section 1.3 in \cite{PRSZ20} and Proposition \ref{prop-1} below).} For
$\underline{\bar{d}}^{\emptyset}$ there is an additional constraint
that the respective shifts \jznote{should} be also bounded by $r$.  As a consequence:
\begin{itemize}
\item[v.] if $\bar{d}^{\emptyset}(X,X')=0$, then $X$ and $X'$ are
  $0$-isomorphic up to shift. Moreover, if $X$ and $X'$ are not
  $0$-isomorphic, they are both periodic in the sense that there exist
  $k$ and $k'$ (not both null) and $0$-isomorphisms
  $\Sigma^{k}X\to X$, $\Sigma^{k'}X'\to X'$.
\item[vi.] if $\underline{\bar{d}}^{\emptyset}(X,X')=0$, then $X\equiv X'$.
\end{itemize}

In summary, this means that the best we can expect from the
fragmentation pseudo-metrics is that they \jznote{should be} non-degenerate on the
space of $0$-isomorphism types. From now on, we will say that a
fragmentation pseudo-metric is non-degenerate if this is the
case. Assuming no periodic objects exist, the metric
$\bar{d}^{\emptyset}$ is non-degenerate in this sense. However, the
distance it measures for two objects that are not isomorphic in
$\C_{\infty}$ is infinite. On the other hand, a metric such as
$\bar{d}^{\mathcal{F}}$ (as well as
$\underline{\bar{d}}^{\mathcal{F}}$) where $\mathcal{F}$ is a family
of triangular generators of $\C_{\infty}$ is finite but is in general
degenerate.

\

The last point we want to raise in this section is that mixing
fragmentation pseudo-metrics can sometimes produce non-degenerate
ones. We will see an example of this sort in the symplectic section \S\ref{subsec-symp}, 
but we end here by describing a more general, abstract argument. Fix
two families $\mathcal{F}_{i}$, $i=1,2$ of generators of
$\mathcal{C}_{\infty}$.  Consider the mixed pseudo-metric defined by
\begin{equation}\label{eq:mixing1}
  \bar{d}^{\mathcal{F}_{1},\mathcal{F}_{2}}=\max\{\bar{d}^{\mathcal{F}_{1}}, \bar{d}^{\mathcal{F}_{2}} \}~.~
\end{equation}
The idea is that if these two families are ``separated'' in a strong
sense, then the mixed metric is non-degenerate. For instance, denote
by $\mathcal{F}_{i}^{\Delta}$ the subcategory of $\C_{0}$ that is
generated by $\mathcal{F}_{i}$.  Now assume that
$\mathrm{Obj}(\mathcal{F}_{1}^{\Delta})\cap
\mathrm{Obj}(\mathcal{F}_{2}^{\Delta})=\{0\}$ (this is of course quite
restrictive).  We now claim that
$\underline{\bar{d}}^{\mathcal{F}_{1},\mathcal{F}_{2}}$ is
non-degenerate and that $\bar{d}^{\mathcal{F}_{1},\mathcal{F}_{2}}$
satisfies a weaker non-degeneracy conditions which is that
$\bar{d}^{\mathcal{F}_{1},\mathcal{F}_{2}}(X,0)=0$ if and only if
$X\equiv 0$. This latter fact follows immediately by noticing that
$\bar{d}^{\mathcal{F}_{i}}(X,0)=0$ means that
$X\in \mathcal{F}_{i}^{\Delta}$ and we leave the former as an exercise.

\section{Examples}\label{sec-example}

\subsection{Filtered dg-categories}\label{subsec:dg}

The key property of dg-categories, introduced in \cite{BK91} (see also
\cite{Dri04}), is that they admit natural, pre-triangulated
closures. The $0$-cohomological category of this closure is
triangulated. We will see here that there is a natural notion of
filtered dg-categories.  Such a category also admits a
pre-triangulated closure, defined using filtered twisted complexes,
following closely \cite{BK91}. Its $0$-cohomological category is a
triangulated persistence category.

\subsubsection{Basic definitions} \label{subsec:basic-def}
Following standard convention we will work in a co-homological setting
and we keep all the sign conventions as in \cite{BK91}.  For our
purposes it is convenient to view a filtered cochain complex over the
field $\k$ as a triple $(X, \partial, \ell)$ consisting of a cochain
complex $(X, \partial)$ and a filtration function
$\ell: X \to \R \cup\{-\infty\}$ such that for any $a, b \in X$ and
$\lambda \in \k \backslash \{0\}$,
$\ell(\lambda a + b) \leq \max\{\ell(a), \ell(b)\}$,
$\ell(a) = -\infty$ if and only if $a = 0$, and
$\ell(\partial a) \leq \ell(a)$.  We denote
$X^{\leq r}= \{x \in X \,|\, \ell(x) \leq r\} \subset X$ the
filtration induced on $X$ by the filtration function $l$. Clearly,
$X^{\leq r}$ is again a filtered cochain complex. The family
$\{X^{\leq r}\}$ determines the function $\ell$. The \jznote{cohomology} of a
filtered cochain complex is a persistence module:
$V^{r}(X)=H(X^{\leq r};\k)$ whose structural maps $i_{r,s}$ are
induced by the inclusions
$\iota_{r,s}:X^{\leq r}\hookrightarrow X^{\leq s}$, $r\leq s$. We have
omitted here the grading, as is customary. In case it needs to be
indicated we write, for instance,
$[V^{r}(X)]^{i}=H^{i}(X^{\leq r};\k)$. We denote this (graded)
persistence module by $\mathbb{V}(X)$,
\begin{equation}\label{eq:pers-v}\mathbb{V}(X)= (V^{r}(X), i_{r,s})~.~
\end{equation}
Given two filtered cochain complexes $X = (X, \partial^X, \ell_X)$ and
$Y = (Y, \partial^Y, \ell_Y)$, their tensor product is a filtered
\jznote{cochain} complex $(X \otimes Y, \partial^{\otimes}, \ell_{\otimes})$
given by $(X\otimes Y)_k = \bigoplus_{i+j = k} (X_i \otimes Y_j)$ and
\begin{equation} \label{dfn-tensor-complex}
  \partial^{\otimes}(x\otimes y) = \partial^X(x) \otimes y +
  (-1)^{|x|}x \otimes \partial^Y(y) \ , \ \ell_{\otimes}(a \otimes b)
  = \ell_X(a) + \ell_Y(b) ~.~
\end{equation}
If $(X, \ell_{X})$ and $(Y,\ell_{Y})$ are filtered vector spaces, we call a  linear map $\phi : X\to Y$
$r$-filtered if $\ell_{Y}(\phi(x))\leq \ell_{X}(x)+r$ for all $x\in X$. A $0$-filtered map is sometimes
called (for brevity) filtered. For more background on this formalism,  see \cite{UZ16}.
%The barcode of $X \otimes Y$ can be computed from barcodes of $X$ and $Y$. 

The next definition is an obvious analogue of the notion of dg-category in \cite{BK91} \S 1.

\begin{dfn} \label{dfn-fdg-cat} A {\em filtered dg-category} is a
  preadditive category $\A$ where
  \begin{itemize}
  \item[(i)] for any $A, B \in {\rm Obj}(\A)$ the hom-set
    ${\rm Hom}_{\A}(A,B)$ is a filtered cochain complex with
    filtrations denoted by ${\rm Hom}_{A}^{\leq r}(A,B)$ such that for
    each identity element we have $\ell(\mathds 1_{A})=0$ and $\mathds 1_{A}$ is
    closed.
  \item[(ii)] the composition is a filtered chain map:
    $${\rm Hom}_{\A}(B,C) \otimes {\rm Hom}_{\A}(A,B)
    \xrightarrow{\circ} {\rm Hom}_{\A}(A,C)~;~$$
  \item[(iii)] for any inclusions $\iota^{AB}_{r,r'}$ and
    $\iota^{BC}_{s,s'}$, the composition morphism satisfies the
    compatibility condition
    $\iota^{BC}_{s,s'}(g) \circ \iota^{AB}_{r,r'}(f) =
    \iota^{AC}_{r+s, r'+s'}(g \circ f)$ for any
    $f \in {\rm Hom}^{\leq r}_{\A}(A,B)$ and
    $g \in {\rm Hom}^{\leq s}_{\A}(B,C)$.
\end{itemize}
% An {\bf internally filtered dg-category} is a filtered dg-category
% with its underlying preadditive category being internally
% filtered. If (i) above is relaxed to have only pseudo-filtration
% function, then $\A$ is called a {\bf {\color{red} weakly filtered
% dg-category}}.
\end{dfn}

\begin{remark} \label{rmk-fil-dgc} 
  % (a) Item (ii) in Definition \ref{dfn-fdg-cat} means that for all
  % $f \in {\rm Hom}_{\A}(A,B)$ and $g \in {\rm Hom}_{\A}(B,C)$, we
  % have
  % \begin{equation} \label{circ}
  %   \ell_{AC}(g \circ f)\leq \ell_{AB}(f) + \ell_{BC}(g)
  %   \,\,\,\,\mbox{and}\,\,\,\, \partial_{AC}(g \circ f) =
  %   \partial_{BC} g \circ f + g\circ \partial_{AB} f
  % \end{equation}
  % where $\ell_{--}$ and $\partial_{--}$ are the respective
  % filtration functions and boundary maps. The second relation in
  % (\ref{circ}) is usually called {\it Leibniz rule}. It implies that
  % $\partial_{AA}(\mathds{1}_A) = 0$ for any $A \in {\rm Obj}(\A)$,
  % i.e., $\mathds{1}_A$ is a closed element.
  A filtered dg-category is trivially a persistence category by
  forgetting the boundary maps on each ${\rm
    Hom}_{\A}(A,B)$. Explicitly, for any $A, B \in {\rm Obj}(\A)$,
  define $E_{AB}: (\R, \leq) \to {\rm Vect}_{\k}$ by
  $E_{AB}(r) = {\rm Hom}_{\A}^{\leq r}(A,B)$ and
  $E_{AB}(i_{r,s}) = \iota_{r,s}: {\rm Hom}_{\A}^{\leq r}(A,B)\to {\rm
    Hom}_{\A}^{\leq s}(A,B)$.
\end{remark}

The (co)homology category of a filtered dg-category $\A$, denoted by
${\rm H}(\A)$, is a category with
$${\rm Obj}({\rm H}(\A)) = {\rm Obj}(\A)$$ and, for any
$A, B \in {\rm Obj}({\rm H}(\A))$,
\pbnote{
$${\rm Hom}_{{\rm H}(\A)}(A,B) := \mathbb V({\rm Hom}_{\A}(A,B))
= \Bigl( \bigl\{ H^*(\Hom_{\A}^{\leq r}(A,B))\bigr\}_{r \in
  \mathbb{R}}, \{i_{r,s}\}_{r \leq s}\Bigr),$$} \pbnote{is the
persistence module as described in~\eqref{eq:pers-v}.}  It is
immediate to see that for any filtered dg-category $\A$, its (co)homology
category ${\rm H}(\A)$ is a (graded) persistence category.

\subsubsection{Twisted complexes} \label{subsubsec:tw-cplxes} It is easy to construct a formal shift-completion
of a dg-category.

%Example \ref{ex-cfcc} and Example \ref{ex-ex-cfcc} demonstrate nice properties of the filtered dg-category $%\mathcal {FK}_{k}$ and its homology category ${\rm H}(\mathcal {FK}_{\k})$. However, in general, a %filtered dg-category might not be as user-friendly as $\mathcal {FK}_{\k}$. In what follows, we will introduce %another filtered dg-category, which imitates $\mathcal {FK}_{\k}$ in a more abstract set-up. 

\begin{dfn} \label{dfn-sus} Let $\A$ be a filtered dg-category. The {\em shift completion} $\Sigma \A$ of $\A$ 
is a  filtered dg-category such that: 
\begin{itemize}
\item[(i)] The objects of $\Sigma \A$ are
\begin{equation} \label{suspend-not}
{\rm Obj}(\Sigma\A) = \left\{ \Sigma^r A[d] \, | \, A \in {\rm Obj}(\A), \, r \in \R \,\,\mbox{and}\,\, d\in \Z\right\}
\end{equation}
such that $\Sigma^0 A  = A$, $\Sigma^s (\Sigma^r A)  = \Sigma^{r+s} A$, $A[0] = A$, $(A[d_1])[d_2] = A[d_1 +d_2]$, $(\Sigma^r A)[d] = \Sigma^r(A[d])$, for any $r, s \in \R$ and $d_1, d_2, d \in \Z$.
\item[(ii)] For any $\Sigma^r A[d_A], \Sigma^s B[d_B] \in {\rm Obj}(\A)$, the hom-set ${\rm Hom}(\Sigma^rA[d_{A}], \Sigma^s B [d_{B}])$ is a filtered \jznote{cochain} complex with the same underlying cochain complex of ${\rm Hom}(A,B)$ but with degree shifted by $d_B -d_A$ and filtration function $\ell_{\Sigma^rA[d_A]\, \Sigma^s B[d_B]} = \ell_{AB}+s -r$.
\end{itemize}
\end{dfn}

\begin{remark} \label{rmk-sus}  It is immediate to check that $\Sigma \A$ as given in Definition \ref{dfn-sus} is still a filtered dg-category.   
\end{remark}

The category $\Sigma \A$ carries a natural functor $\Sigma: (\R, +) \to \mathcal P{\rm End}(\Sigma \A)$ defined on objects by $\Sigma^r(A) = \Sigma^r A$ and with an obvious definition on morphisms such that  $\Sigma^r$  is filtration preserving and for $r \leq s$, the natural transformations $\eta_{r,s}:\Sigma^{r}\to \Sigma^{s}$ are such that $(\eta_{r,s})_A :\Sigma^{r}A\to \Sigma^{s} A$ is
 induced by the identity map $id_A$ for each $A \in {\rm Obj}(\A)$.
In this context we have a natural definition of (one-sided) twisted complexes obtained by adjusting to the filtered case  the Definition 1 in \S 4 \cite{BK91}.

\begin{dfn} \label{dfn-tw-cpx} Let $\A$ be a filtered dg-category. A filtered (one-sided) twisted complex of $\Sigma \A$ is a pair $A = \left(\bigoplus_{i=1}^n \Sigma^{r_i} A_i[d_i], q = (q_{ij})_{1 \leq i, j \leq n} \right)$ such that the following conditions hold.
\begin{itemize}
\item[(i)] $\Sigma^{r_i} A_i[d_i] \in {\rm Obj}(\Sigma \A)$, where $r_i \in \R$ and $d_i \in \Z$.
\item[(ii)] $q_{ij} \in {\rm Hom}_{\Sigma \A}(\Sigma^{r_j}A_j[d_j], \Sigma^{r_i}A_i[d_i])$ is of degree $1$, and $q_{ij} =0$ for $i \geq j$.
\item[(iii)] $d_{\rm Hom} q_{ij} + \sum_{k=1}^n q_{ik} \circ q_{kj} =0 $.
\item[(iv)] For any $q_{ij}$, $\ell_{\Sigma^{r_j}A_j[d_j] \Sigma^{r_i}A_i[d_i]}(q_{ij}) \leq 0$.
\end{itemize}
\end{dfn}

\begin{rem} We will mostly work with \jznote{{\em filtered} one-sided twisted complexes} as defined above but, more generally, the pair $A = \left(\bigoplus_{i=1}^n A_i[r_i], q = (q_{ij})_{1 \leq i, j \leq n} \right)$ subject only to (i),(ii), (iii) is called a one-sided twisted complex. 
\end{rem}

%\begin{remark} \label{rmk-fil-tc} (1) It is convenient to view $q = (q_{ij})_{1 \leq i, j \leq n}$ as a matrix. %Then the condition (iii) in Definition \ref{dfn-tw-cpx} can be concisely written as $d_{\rm Hom}q + q^2 = %0$, where $q^2$ is the matrix multiplication of $q$ with itself. (2) Later we will see how to modify the second %requirement in the condition (ii) in Definition \ref{dfn-tw-cpx} to be $q_{ij} =0$ for $i >j$. This will result in %a minor change of the condition (iii) in Definition \ref{dfn-tw-cpx}. (3) The condition (iv) in Definition 
%\ref{dfn-tw-cpx} is required in order to make sure that the associated $\A^{\rm pre-tr}$ (defined in %Definition \ref{dfn-fil-pre-tr}) is indeed a {\it filtered} dg-category (in particular, its filtration function should %satisfy the condition that the boundary map will not increase the filtration).\end{remark}

It is easy to see that there are at least as many filtered one-sided twisted complexes as one-sided twisted complexes as it follows from the statement below whose proof we leave to the reader.

\begin{lemma} \label{lemma-exist-fil-tc} Given a twisted complex $\left(\bigoplus_{i=1}^n A_i[d_i], q = (q_{ij})\right)$, there exist $(r_i)_{1 \leq i \leq n}$ such that condition (iv) in Definition \ref{dfn-tw-cpx} is satisfied for the filtration shifted twisted complex $\left(\bigoplus_{i=1}^n \Sigma^{r_i}A_i[d_i], q = (q_{ij})\right)$. \end{lemma}

%\begin{proof} Fix a set of (finite) scalars $\{\lambda_{ij}\}_{1\leq i,j \leq n}$ such that $\ell_{A_j[d_j]\,A_i[d_i]}(q_{ij}) \leq \lambda_{ij}$. The numbers $r_i$ are chosen in an inductive way as follows. Fix $r_1=0$, then by the condition (iv) in Definition \ref{dfn-tw-cpx},
%\[ \ell_{\Sigma^{r_i}A_i[d_i] \Sigma^{r_1}A_1[d_1]}(q_{1i}) \leq 0 \,\,\,\,\mbox{which is equivalent to}\,\,\,\, \ell_{A_iA_1}(q_{1i}) \leq r_i. \]
%This implies that any possible choice of $r_i$ must obey $r_i \geq \lambda_{1i}$ for each $i \in \{2, ..., n\}$. In particular, choose $r_2 := \lambda_{12}$. Then again by the condition (iv) in Definition \ref{dfn-tw-cpx}, 
%\[ \ell_{\Sigma^{r_i}A_i[d_i]\Sigma^{r_2}A_2[d_2]}(q_{2i}) \leq 0 \,\,\,\,\mbox{which is equivalent to}\,\,\,\, \ell_{A_iA_2}(q_{2i}) \leq r_i - r_2. \]
%This implies that any possible choice of $r_i$ must obey $r_i \geq \lambda_{2i}+r_2$ for each $i \in \{3, ..., n\}$. In particular, choose $\delta_3 := \max\{\lambda_{13}, \lambda_{23} + r_2\}$. Inductively, for any $k \in \{1, ..., n\}$, choose 
%\[ r_k := \max\{\lambda_{1k}, \lambda_{2k} + r_2, ..., \lambda_{k-1\,k} + r_{k-1} \}. \]
%Then one can easily check that the condition (iv) in Definition \ref{dfn-tw-cpx} is satisfied for $\left(\bigoplus_{i=1}^n \Sigma^{r_i}A_i[d_i], q = (q_{ij})\right)$. \end{proof}

\subsubsection{Pre-triangulated completion.}
We will see next that the filtered twisted complexes over $\A$ form a category that provides a  (pre-)triangulated
closure of $\A$. The $0$-cohomology category of this completion is a triangulated persistence category.

\begin{dfn} \label{dfn-fil-pre-tr} Given a filtered dg-category $\A$,
  define its \jznote{{\em filtered pre-triangulated completion}}, denoted by
  $Tw(\A)$, to be a category with the following properties.
  \begin{itemize}
  \item[(i)] Its objects are,
    \[ {\rm Obj}(Tw(\A)) := \{ \mbox{filtered one-sided twisted
        complex of $\Sigma\A$} \}.\]
  \item[(ii)] For $A = \left(\bigoplus \Sigma^{r_j}A_j[d_j], q\right)$
    and $A' = \left(\bigoplus \Sigma^{r'_i}A'_i[d'_i], q'\right)$ in
    ${\rm Obj}(Tw(\A))$, a morphism $f \in {\rm Hom}_{Tw(\A)}(A, A')$
    is a matrix of morphisms in $\A$ denoted by
    $f = (f_{ij}): A \to A'$, where
    \[ f_{ij} \in {\rm Hom}_{\Sigma\A}\left(\Sigma^{r_j}A_j[d_j],
        \Sigma^{r'_i}A'_i[d'_i]\right). \]
  \item[(iii)] The hom-differential is defined as follows. For any
    $f \in {\rm Hom}_{Tw(\A)}(A, A')$ as in (ii) above, define
    \begin{equation} \label{diff-hull} d_{Tw\A}(f) := (d_{\rm
        Hom}f_{ij}) + q' f - (-1)^l f q\end{equation} where
    ${\rm deg}(f_{ij}) = l$ and the right-hand side is written
    \pbnote{in matrix form}. The composition $f' \circ f$ is given by
    the matrix multiplication.
  \end{itemize}
\end{dfn}

\begin{lemma} \label{lemma-pre-tr-fdg} Given a filtered dg-category
  $\A$, its filtered \jznote{pre-triangulated completion $Tw (\A)$} is a filtered
  dg-category.
\end{lemma}
\begin{proof} The main step is to notice that there exists a
  filtration function on \jznote{${\rm Hom}_{Tw(\A)}(A, A')$} for any
  $A, A' \in {\rm Obj}(Tw(\A))$. For any
  $f = (f_{ij}) \in {\rm Hom}_{Tw(\A)}(A, A')$,
  set
  \begin{equation} \label{fil-fun-pre-tr} \ell_{AA'}(f) = \max_{i,j}
    \left\{\ell_{\Sigma^{r_j}A_j[d_j]
        \,\Sigma^{r'_i}A'_i[d'_i]}(f_{ij})\right\}.
  \end{equation}
  It is easily checked that $\ell_{AA'}$ is a filtration function as
  well as the other required properties.
\end{proof}

The first step towards triangulation is to define an appropriate cone
of a morphism.

\begin{dfn} \label{dfn-fmc} Let $\A$ be a filtered dg-category and
  $Tw(\A)$ be \jznote{its pre-triangulated completion}. Let
  $A = \left(\bigoplus \Sigma^{r_j} A_j[d_j], q = (q_{ij})_{1 \leq i,
      j \leq n} \right)$,
  $A' = \left(\bigoplus \Sigma^{r'_i} A'_i[d'_i], q = (q'_{ij})_{1
      \leq i, j \leq m} \right)$ be two objects of $Tw(\A)$ and let
  $f: A \to A'$ be a closed, degree preserving, morphism. Define the
  {\em $\lambda$-filtered mapping cone} of $f$, where
  $\lambda \geq \ell_{AA'}(f)$, by
  \begin{equation} \label{const-cone} {\rm Cone}^{\lambda}(f) : =
    \left( \bigoplus_{i} \Sigma^{r'_i} A'_i[d'_i] \oplus \bigoplus_{j}
      \Sigma^{r_j + \lambda} A_j[d_j-1], q_{\rm co}
    \right)\,\,\mbox{where} \,\, q_{\rm co} = \begin{pmatrix}
      q' & f  \\
      0 & q\end{pmatrix},
  \end{equation}
  where $q', q,f$ are all block matrices.
\end{dfn}

\begin{remark}\label{thm-pre-tr-cone} (1) The condition
  $\lambda \geq \ell_{AA'}(f)$ guarantees that
  ${\rm Cone}^{\lambda}(f)$ is indeed a filtered one-sided twisted
  complex over $\Sigma \A$. Therefore, $Tw(\A)$ is closed under taking
  degree-shifts, filtration-shifts, and filtered mapping cones of
  (degree preserving) closed morphisms.

  (2) Notice that a $\lambda$-filtered cone can also be written as a
  $0$-filtered cone but for a different map.

  (3) Given a filtered dg-category $\A$ it is easy to see that every
  object in $Tw(\A)$ can be obtained from objects in $\Sigma \A$ by
  taking iterated filtered $0$-filtered mapping cones.
\end{remark}

The $0$-cohomological category associated to a dg-category is a
triangulated category. The next result is the analogue in the filtered
case.

\begin{prop} \label{thm-hol-hull} If $\A$ is a filtered dg-category
  and $Tw(\A)$ is its \jznote{filtered pre-triangulated completion}, then the
  degree-$0$ cohomology category ${\rm H}^0(Tw(\A))$ is a triangulated
  persistence category.
\end{prop}

In view of this result, it is natural to call a filtered dg-category
$\A$ {\em pre-triangulated} if the inclusion
$\A \hookrightarrow Tw(\A)$ is an equivalence of filtered
dg-categories.
\begin{cor}\label{cor:pre-tr} Let $\A$ be a filtered pre-triangulated
  dg-category. Then its degree-0 cohomology category ${\rm H}^0(\A)$
  is a triangulated persistence category.
\end{cor}

\begin{proof}[Proof of Proposition \ref{thm-hol-hull}]
  It is trivial to notice that the category $H^{0}(Tw(\A))$ is a
  persistence category. It is endowed with an obvious shift functor as
  defined in \S\ref{subsubsec:tw-cplxes}. The first thing to check at
  this point is that the $0$-level category $[H^{0}(Tw(\A))]_{0}$ with
  the same objects as $H^{0}(Tw(\A))$ and only with the shift
  $0$-morphisms is triangulated - see Definition \ref{dfn-tpc}. The
  family of triangles that will provide the exact ones are the
  triangles of the form
  $$A\xrightarrow{f}B\xrightarrow{i}
  \mathrm{Cone}^{0}(f)\xrightarrow{\pi}A[-1]$$ associated to the
  $0$-cones, as given in Definition \ref{dfn-fmc}. From this point on
  checking that $[H^{0}(Tw(\A))]_{0}$ is triangulated comes down to
  the usual verifications showing that the $H^{0}$ of a dg-category is
  triangulated, with a bit of care to make sure that the relevant
  homotopies preserve filtration. We leave this verification to the
  reader. It is then automatic that $\Sigma^{r}$ is triangulated when
  restricted to $[H^{0}(Tw(\A))]_{0}$. The last step is to show that
  the morphism $\eta_{r}^{A}:\Sigma^{r}A\to A$ has an $r$-acyclic
  cofiber in $Tw(\A)$. In this context, of filtered dg-categories, an
  object $K$ is $r$-acyclic if the identity
  $\mathds{1}_{K}\in \Hom_{Tw(\A)}(K,K)$ is a boundary of some element
  $\eta\in \Hom_{Tw(\A)}^{\leq r}(K,K)$.

  The map $\eta^A_r \in {\rm Mor}^0_{Tw(\A)}(\Sigma^r A, A)$ is
  induced by the identity.  By definition
  ${\rm Cone}^0(\eta^A_r) = A \oplus \Sigma^r A[-1]$ and
  $$ q_{\rm co} = \begin{pmatrix}
    q & \eta^{A}_{r}  \\
    0 & q'\end{pmatrix}
  $$
  where $q$ is the structural map of the twisted complex $A$ and
  $q'=\Sigma^{r}q$.

  Consider a homotopy
  \[ K = \begin{pmatrix} 
      0  & 0 \\
      (\eta_{0,r})_A&0
    \end{pmatrix}: {\rm Cone}^0(\eta^A_r) \to {\rm
      Cone}^0(\eta^A_r)[-1]. \] Note that $\ell(K) =r$. We have
\jznote{      \begin{align*} 
dK & = \begin{pmatrix} 
      0 & 0   \\
      0 & 0
    \end{pmatrix} + \begin{pmatrix}
      q & \eta^A_r  \\
      0 & q'
    \end{pmatrix} \begin{pmatrix} 
      0  & 0 \\
      (\eta_{0,r})_A & 0 
    \end{pmatrix} + \begin{pmatrix} 
      0  & 0 \\
      (\eta_{0,r})_A & 0 
    \end{pmatrix}  \begin{pmatrix} 
      q & \eta^A_r  \\
      0 & q'
    \end{pmatrix} \\ &= \begin{pmatrix}
      \mathds{1}_A & 0   \\
      q'\circ (\eta_{0,r})_A-(\eta_{0,r})_A\circ q & \mathds{1}_{\Sigma^r A[-1]})
    \end{pmatrix} =\mathds{1}_{{\rm Cone}^0(\eta^A_r)}.
  \end{align*}}
  because $\Sigma^{r}q\circ (\eta_{0,r})_A= (\eta_{0,r})_A\circ q$ and this
  concludes the proof.
\end{proof}

\begin{rem}
  It is easy to see that in the filtered dg-category $Tw(\A)$ we can
  replicate all the constructions in \S\ref{subsec:TPC} at the chain
  level, similarly to the definition of $r$-acyclic objects mentioned
  inside the proof above.
\end{rem}

% !TEX root = TPC.tex

\subsection{Filtered cochain complexes.}\label{subsec:filt-co}

In this section we discuss the main example of a filtered dg-category,
the category of filtered co-chain complexes. As we shall see, this is
pre-triangulated and thus, in view of Corollary \ref{cor:pre-tr}, its
homotopy category is a triangulated persistence category.

We will work over a field $\k$ and will denote the resulting category
by $\mathcal{FK}_{\k}$.  The objects of this category are filtered
cochain complexes $(X,\partial, \ell)$ where $(X,\partial)$ is a
cochain complex and $\ell$ is a filtration function, as in
\S\ref{subsec:basic-def}. Given two filtered \jznote{cochain} complexes
$(X,\partial_{X}, \ell_{X})$ and $(Y,\partial_{Y},\ell_{Y})$ the
morphisms $\Hom_{\mathcal{FK}_{\k}}(X,Y)$ are linear graded maps
$f:X\to Y$ such that the quantity

\begin{equation}\label{eq:filtr-hom}
  \ell(f)= \inf\{ r \in \R \ | \ \ell_{Y}(f(x))
  \leq \ell_{X}(x)+r, \forall x \in X\}
\end{equation}
is finite. The filtration function on $\Hom_{\mathcal{FK}_{k}}(X,Y)$
is then defined through (\ref{eq:filtr-hom}).  The differential on
$\Hom_{\mathcal{FK}_{k}}(X,Y)$ is given, as usual, by
$\partial (f)= \partial _{Y}\circ f - (-1)^{|f|}f\circ\partial_{X}$
and it obviously preserves filtrations.  The composition of morphisms
is also obviously compatible with the filtration and therefore
$\mathcal{FK}_{\k}$ is a filtered dg-category.

There is a natural shift functor on $\mathcal{FK}_{\k}$ defined by
$\Sigma: (\R, +) \to \mathcal P{\rm End}(\mathcal {FK}_k)$ by
\[ \Sigma^r(X, \partial, \ell_{X}) = (X, \partial, \ell_{X}+r),
  \,\,\,\,\, \mbox{and}\,\,\,\,\, \Sigma^r(f) = f \] for any
$f \in {\rm Hom}_{\mathcal {FK}_k}(X,Y)$. Moreover, for $r, s\in \R$,
there is a natural transformation from $\Sigma^r$ to $\Sigma^s$
induced by the identity.

Assume that
$f: (X,\partial_{X}, \ell_{X})\longrightarrow
(Y,\partial_{Y},\ell_{Y})$ is a cochain morphism such that
$\ell(f)\leq 0$. In this case, the usual cone construction
$\mathrm{Cone}(f)=(Y\oplus X[-1], \partial_{\mathrm{co}})$ with
$$\partial_{\mathrm{co}} = \begin{pmatrix} 
  \partial_{Y} & f  \\
  0 & \partial_{X}\end{pmatrix}$$ produces a filtered complex and fits
into a triangle of maps with $\ell \leq 0$:

$$X \xrightarrow{f} Y\xrightarrow{i}
\mathrm{Cone}(f)\xrightarrow{\pi} X[-1]~.~$$

The standard properties of this construction immediately imply that
the dg-category $\mathcal{FK}_{\k}$ is pre-triangulated and thus the
$0$-cohomological category, $H^{0}\mathcal{FK}_{\k}$, is a
\pbnote{triangulated} persistence category.

It is useful to make explicit some of the properties of this category:
\begin{itemize}
\item[i.] The objects of $H^{0}\mathcal{FK}_{\k}$ are filtered cochain
  complexes $(X,\partial_{X},\ell_{X})$.
\item[ii.] The morphisms in $\Mor_{H^{0}\mathcal{FK}_{\k}}^{r}(X,Y)$
  are cochain maps
  $f:(X,\partial_{X},\ell_{X})\to (Y,\partial_{Y},\ell_{Y})$ such that
  $\ell(f)\leq r$ up to chain homotopy $h:f\simeq f'$ with
  $\ell(h)\leq r$.
\item[iii.] A filtered complex $(K,\partial_{K},\ell_{K})$ is
  $r$-acyclic if the identity $\mathds{1}_{K}$ is chain homotopic to $0$
  through a chain homotopy $h:\mathds{1}_{K}\simeq 0$ with $\ell(h)\leq r$.
\item[iv.] The construction of weighted exact triangles as well as
  their properties can be pursued in this context by following closely
  the scheme in \S\ref{subsubsec:weight-tr}.
\item[v.] The limit category $[H^{0}\mathcal{FK}_{\k}]_{\infty}$ has
  as morphisms chain homotopy classes of cochain maps \pbnote{(where
    both the cochain maps and the homotopies are assumed to be of
    bounded shifts)}.  Its objects are still filtered cochain
  complexes. It is triangulated, with translation functor $TX= X[-1]$,
  as expected.
\end{itemize}

\begin{remark} \label{rem:extensions} The example of the dg-category
  $\mathcal{FK}_{\k}$ can be extended in a number of ways and we
  mention a couple of them here.

  (a) Assume that we fix a filtered dg-category $\mathcal{A}$. There
  is a natural notion of filtered (left/right) module $\mathcal{M}$
  over $\mathcal{A}$. Such modules together with filtered maps
  relating them form a new filtered dg-category denoted by
  $\mathrm{Mod}_{\mathcal{A}}$. The $0$-cohomology category associated
  to this filtered dg-category, $H^{0}\mathrm{Mod}_{\mathcal{A}}$, is
  pre-triangulated because the category $\mathrm{Mod}_{\mathcal{A}}$
  is naturally endowed with a shift functor, just like
  $\mathcal{FK}_{\k}$, as well as with an appropriate
  cone-construction over filtered, closed, degree preserving
  morphisms.

  (b) Similarly to (a), we may take $\mathcal{A}$ to be a filtered
  $A_{\infty}$-category and consider the category of filtered modules,
  $\mathrm{Mod}_{\mathcal{A}}$, over $\mathcal{A}$. Again this is a
  filtered dg-category and it is pre-triangulated (the formalism
  required to establish this fact appears in \cite{Bi-Co-Sh:LagrSh},
  in a version dealing with weakly filtered structures).
\end{remark}

As mentioned in the beginning of Introduction \S \ref{sec-intro},
there exists a quantitative comparison between two filtered cochain
complexes $X, Y$, called the bottleneck distance and denoted by
$d_{\rm bot}(X,Y)$. Here, we directly use the barcode language
from~\cite{Bar94} or~\cite{UZ16} to define $d_{\rm
  bot}(X,Y)$. Explicitly, by Proposition 7.4 in~\cite{UZ16}, there is
a filtered isomorphism as follows,
\begin{equation} \label{decomp} X \simeq \bigoplus_{[a,+\infty) \in
    \mathcal B(X)} E_{1}(a) \oplus \bigoplus_{[c,d) \in \mathcal B(X)}
  E_{2}(c,d)
\end{equation} 
where $E_{1}(a), E_{2}(c,d) \in {\rm Obj}(\mathcal{FK}_{\k})$ are
filtered \jznote{cochain} complexes defined by
\[ E_{1}(a) = ((\cdots \to 0 \to \k\left<x\right> \to 0 \to \cdots),
  \ell(x) = a). \] and
\pbnote{
  \[ E_{2}(c,d) = ((\cdots \to 0 \to \k\left<y \right>
    \xrightarrow{\partial} \k\left<x \right> \to 0 \to \cdots), \;
    \ell(y) =c, \ell(x) = d ),\]} where $c \geq d$,
\pbnote{$\partial(y) = \kappa x$ for some $0 \neq \kappa \in
  \k$}. \pbnote{The notation $\mathcal B(X)$ in~\eqref{decomp} stands
  for a collection of finite or semi-infinite intervals or $\R$.} In
what follows, sometimes for brevity, denote by $E_*(I)$ either
$E_{1}(a)$ or $E_{2}(c,d)$ for the corresponding interval
$I = [a, + \infty)$ or $I =[c,d)$ in $\mathcal B(X)$. Importantly,
$\mathcal B(X)$ is well-defined, up to filtered isomorphism, by
Theorem A in \cite{UZ16}. Then $d_{\rm bot}(X, Y)$ is defined as the
infimum $\tau$ satisfying the following conditions: there exist some
subsets consisting of certain ``short intervals''
$\mathcal B(X)_{\rm short} \subset \mathcal B(X)$ and
$\mathcal B(Y)_{\rm short} \subset \mathcal B(Y)$ such that
\begin{itemize}
\item[(i)] each short interval $[c,d)$ satisfies $2(d-c) \leq \tau$;
\item[(ii)] there is a bijection
  $\sigma: \mathcal B(X) \backslash \mathcal B(X)_{\rm short} \to
  \mathcal B(Y) \backslash \mathcal B(Y)_{\rm short}$;
\item[(iii)] if $\sigma([a,b)) = [c,d)$, then $\max\{|a-c|, |b-d|\} \leq \tau$;
\item[(iv)] if $\sigma([a, \infty)) = [c,\infty)$, then $|a-c| \leq \tau$. 
\end{itemize}

In what follows, we assume that the cardinalities of barcodes
$\#|\mathcal B(X)|$ and $\#|\mathcal B(Y)|$ are both finite. The
following result compares the fragmentation pseudo-metric $d^{\F}$
defined in Definition \ref{dfn-frag-met} with the bottleneck distance
$d_{\rm bot}$ defined above.

\begin{prop} \label{prop-1} Let $\C = H^0\mathcal{FK}_{\k}$ and
  $\F \subset {\rm Obj}(\C)$ be a subset containing $0$. Then
  \[ d^{\F}(X,Y) \leq C_{X,Y} d_{\rm bot}(X,Y),\] where
  $C_{X,Y} = 4\{\#|\mathcal B(X)|, \#|\mathcal B(Y)|\} +1$.
\end{prop}

\begin{proof} It suffices to prove the conclusion when $\mathcal B(X)$
  and $\mathcal B(Y)$ have the same cardinality of the infinite-length
  bars (otherwise by definition $d_{\rm bot}(X,Y) = +\infty$ and the
  conclusion holds trivially). Let $\tau : = d_{\rm bot}(X,Y) + \ep$
  for an arbitrarily small $\ep>0$. Since $d^{\F}(\cdot, \cdot)$ is
  invariant under \pbnote{filtered isomorphisms (applied to either of
    its two inputs) then} by~\eqref{decomp} and by reordering summands
  we obtain:
  \begin{align*}
    d^{\F}(X,Y) & \leq  d^{\F}\left(\bigoplus_{I \in \mathcal B(X) \backslash \mathcal B(X)_{\rm short}} E_*(I),  \bigoplus_{\sigma(I) \in \mathcal B(Y) \backslash \mathcal B(Y)_{\rm short}} E_*(\sigma(I))\right)\\
                & + d^{\F} \left(\bigoplus_{J \in \mathcal B(X)_{\rm short}} E_2(J),  \bigoplus_{J' \in \mathcal B(Y)_{\rm short}} E_2(J')\right),
  \end{align*}
  where the inequality is given by the triangle inequality of $d^{\F}$
  with respect to the direct sum, see Proposition
  \ref{prop-frag-sum}. For $d^{\F}$ with short intervals, both
  $\bigoplus_{J \in \mathcal B(X)_{\rm short}} E_2(J)$ and
  $\bigoplus_{J' \in \mathcal B(Y)_{\rm short}} E_2(J')$ are acyclic
  objects in $\C$, therefore by (i) in the definition $d_{\rm bot}$
  above, triangles
  \[ 0 \to 0 \to \bigoplus_{J \in \mathcal B(X)_{\rm short}} E_2(J)
    \to 0 \,\,\,\,\mbox{and}\,\,\,\, 0 \to 0 \to \bigoplus_{J' \in
      \mathcal B(Y)_{\rm short}} E_2(J') \to 0\] are
  weight-$\frac{\tau}{2}$ exact triangles (here we identify
  $\Sigma^{\lambda} 0$ with $0$ for any shift $\lambda \in \R$). Thus,
  \begin{align*} \label{est-2}
d^{\F} \left(\bigoplus_{J \in \mathcal B(X)_{\rm short}} E_2(J),  \bigoplus_{J' \in \mathcal B(X)_{\rm short}} E_2(J')\right) &  \leq d^{\F}\left(\bigoplus_{J \in \mathcal B(X)_{\rm short}} E_2(J), 0\right) \\
& + d^{\F}\left(0, \bigoplus_{J' \in \mathcal B(X)_{\rm short}} E_2(J')\right) \leq \tau.
  \end{align*}

  On the other hand, by Proposition \ref{prop-frag-sum} again, for
  $d^{\F}$ with non-short intervals, we have
  \[ d^{\F}\left(\bigoplus_{I \in \mathcal B(X) \backslash \mathcal
        B(X)_{\rm short}} E_*(I), \bigoplus_{\sigma(I) \in \mathcal
        B(Y)} E_*(\sigma(I) \right) \leq \sum_{I \in \mathcal B(X)
      \backslash \mathcal B(X)_{\rm short}} d^{\F}(E_*(I),
    E_*(\sigma(I)). \] Since $d_{\rm bot}(X,Y)< +\infty$, the bijection
  $\sigma$ will always map a finite interval to a finite interval, a
  semi-infinite interval to a semi-infinite interval, so it suffice to
  consider the following two cases.

  \medskip

  \noindent {Case I}. \emph{Estimate $d^{\F}(E_1(a), E_1(c))$}. We
  need to build a desired cone decomposition. Without loss of
  generality, assume $a\geq c$. Then the identity map
  $\left<x \right>_{E_1(a)} \to \left<x \right>_{E_1(c)}$ (with {\it
    negative} filtration shift) implies that the triangle
  $E_1(a)\to E_1(c)\to K \to E_1(a)$ is weight-0 exact triangle (in
  fact in $\C_0$) where $K$ is the filtered mapping cone. Then in the
  following cone decomposition (with linearization $(0, E_1(c))$),
  \[ 
    \left\{ \begin{array}{l} 0 \to 0 \to K \to 0 \\ E_1(c) \to K \to
        E_1(a) \to E_1(c) \end{array} \right.\] the first triangle is
  a weight-$(c-a)$ exact triangle since it is readily to verify that
  $K$ is $(c-a)$-acyclic. Then
  $\delta^{\F}(E_1(a), E_1(c)) \leq (c-a) + 0 \leq \tau$ by (iv) in
  the definition $d_{\rm bot}$ above. On the other hand, consider the
  following cone decomposition with linearization $(0, E_1(a), 0)$
  (note that by definition $\Sigma^{-(a-c)} E_1(a) = E_1(c)$),
  \begin{equation} \label{1} \left\{
      \begin{array}{l} 0 \to 0 \to 0 \to
        0 \\ E_1(a) \to 0 \to \Sigma^{-(a-c)} E_1(a) \to \Sigma^{-(a-c)}
        E_1(a) \\ 0 \to E_1(c) \to E_1(c) \to 0
      \end{array} \right.
  \end{equation}
  where the second triangle has weight $a-c >0$ by Remark
  \ref{rmk-shift-notation} (b). Therefore,
  $\delta^{\F}(E_1(c), E_1(a)) \leq 0 + (a-c) + 0 \leq \tau$, which
  implies that
  \begin{equation} \label{est-1-1}
    d^{\F}(E_1(a), E_1(c)) \leq \tau. 
  \end{equation}
  
  \medskip

  \noindent {Case II}. \emph{Estimate $d^{\F}(E_2(a,b),
    E_2(c,d))$}. We will carry on the estimation as follows,
  \[ d^{\F}(E_2(a,b), E_2(c,d)) \leq d^{\F}(E_2(a,b), E_2(c,b)) +
    d^{\F}(E_2(c,b), E_2(d,b)). \] Moreover, we will only estimate
  $d^{\F}(E_2(a,b), E_2(c,b))$ with $a \geq c$, and other situations
  can be done in a similar and symmetric way. Similarly to Case I
  above, consider the following cone decomposition
  \[ 
    \left\{ \begin{array}{l} 0 \to 0 \to K \to 0 \\ E_2(c,b) \to K \to
        E_2(a,b) \to E_2(c,b) \end{array} \right.\] where
  $E_2(a,b) \to E_2(c,b)$ is the identity map
  $\left<x\right>_{E_2(a,b)} \to \left<x\right>_{E_2(c,b)}$ (and
  similarly to the generator $y$) with a negative filtration shift and
  $K$ is the cone. Since $K$ is $(a-c)$-acyclic, we have
  $\delta^{\F}(E_2(a,b), E_2(c,b)) \leq 0 + (a-c) \leq \tau$. On the
  other hand,
  \[ \delta^{\F}(E_2(c,b), E_2(a,b)) \leq \delta^{\F}(E_2(c,b),
    \Sigma^{a-c} E_2(c,b)) + \delta^{\F}(E_2(a, b+a-c), E_2(a,b)), \]
  where $\delta^{\F}(E_2(c,b), \Sigma^{a-c} E_2(c,b)) \leq a-c$ by a
  similar cone decomposition as in (\ref{1}). Meanwhile, since
  $b + a-c \geq b$, the identity map from $E_2(a,b+a-c)$ to $E_2(a,b)$
  (with negative filtration shift) yields
  $\delta^{\F}(E_2(a, b+a-c), E_2(a,b)) \leq a-c$. Therefore, together
  we have, by (iii) in the definition $d_{\rm bot}$ above,
  \[ d^{\F}(E_2(a,b), E_2(c,b)) \leq 2 (a-c) \leq 2\tau, \] which
  implies
  \begin{equation} \label{est-1-2} d^{\F}(E_2(a,b), E_2(c,d)) \leq
    4\tau.
  \end{equation}

  Therefore, by (\ref{est-1-1}) and (\ref{est-1-2}) together, we have
  \begin{align*}
    d^{\F}(X,Y) & \leq \#|\mathcal B(X)\backslash \mathcal B(X)_{\rm short}| \cdot 4\tau + \tau\\
                & \leq (4 \#|\mathcal B(X)\backslash \mathcal B(X)_{\rm short}|+1) (d_{\rm bot}(X,Y) + \ep) \\
                & \leq (4 \min\{\#|\mathcal B(X)|, \#|\mathcal B(Y)|\} +1) (d_{\rm bot}(X,Y) + \ep)
  \end{align*}
  where the last inequality holds since $\sigma$ is a bijection by
  (ii) in the definition $d_{\rm bot}$ above.  Let $\ep \to 0$, and we
  complete the proof.
\end{proof}

\subsection{Topological \jznote{spaces} $+$}\label{subsec:top-sp}

There are many topological categories, consisting of topological
spaces endowed with additional structures (indicated by the $+$ in the
title of the subsection), that can be analyzed with the tools
discussed before. We will discuss here two elementary examples. They
both fit the following scheme: we will have a triple consisting of a
(small) category $\mathcal{K}$, an endofunctor
$T_{\mathcal K}:\mathcal{K}\to \mathcal{K}$ and a class of triangles
$\Delta_{\mathcal{K}}$, in $\mathcal{K}$ of the form
$$A\to B \to C \to
T_{\mathcal{K}}A~.~$$ %We will refer to such a quadruple
% $(\mathcal{K}, \Sigma_{\mathcal{K}}, T_{\mathcal{K}},
% \Delta_{\mathcal{K}})$ as a {\em shift category with distinguished
% triangles}. This is left on purpose a bit vague,
In these cases the objects of $\mathcal{K}$ have an underlying
structure as topological spaces and, similarly, the morphisms in
$\mathcal{K}$ are continuous maps, the functor $T_{\mathcal{K}}$
corresponds to the suspension of spaces.

The aim is to define fragmentation pseudo-metrics on the objects of
$\mathcal{K}$ by first associating a weight with some reasonable
properties to the triangles in $\Delta_{\mathcal{K}}$,
$w_{\mathcal{K}}: \Delta_{\mathcal{K}}\to \R$, and then defining
quantities $\delta^{\mathcal{F}}(X,Y)$ and
$\underline{\delta}^{\mathcal{F}}(X,Y)$ as in, respectively,
(\ref{frag-met-0}) and (\ref{eq:frag-simpl}), only taking into account
decompositions appealing to triangles
$\Delta_{i}\in \Delta_{\mathcal{K}}$. Notice that
$\delta^{\mathcal{F}}$ is not generally defined in this setting as its
definition requires to desuspend spaces. On the other hand, as soon as
$w_{\mathcal{K}}$ is defined, $\underline{\delta} ^{\mathcal{F}}$ is
defined as well as the pseudo-metric $\underline{d}^{\mathcal{F}}$
obtained by its symmetrization (see Remark \ref{rem:finite-metr} (c)).
Based on the various constructions discussed earlier in the paper,
there are two approaches to define a weight $\bar{w}_{\mathcal{K}}$
(that is not flat) and they both require some more structure:
\begin{itemize}
\item[A.] The additional structure in this case is a functor
  $\Phi :\mathcal{K}\to \C_{\infty}$ where $\C$ is a TPC, in the
  examples below $\C=H^{0}\mathcal{FK}_{\k}$ - the triangulated
  persistence \pbnote{homotopy category of \jznote{filtered cochain
    complexes}}. We also require that $\Phi$ commutes with $T$ (at
  least up to some natural equivalence) and that for each
  $\Delta\in \Delta_{\mathcal{K}}$ the image $\Phi(\Delta)$ of
  $\Delta$, is exact in $\mathcal{C}_{\infty}$ (and thus
  $\bar{w}(\Phi(\Delta))<\infty$ where $\bar{w}$ is the persistence
  weight introduced in Definition \ref{dfn-extri-inf}).  In this case
  for each $\Delta\in \Delta_{\mathcal{K}}$ we put
$$\bar{w}_{\mathcal{K}}(\Delta)=\bar{w}(\Phi(\Delta))~.~$$  
\item[B.] This second approach requires first that the morphisms
  $\Mor_{\mathcal{K}} (A,B)$ are endowed with a natural increasing
  filtration compatible with the composition. Secondly, there should
  be a shift functor
  $\Sigma_{\mathcal{K}}: (\R,+)\to {\rm End}(\mathcal{K})$ compatible with
  the filtration on morphisms and that commutes with
  $T_{\mathcal{K}}$. Moreover, the triangles in $\Delta_{\mathcal{K}}$
  have to be part of a richer structure such as a model category or a
  Waldhausen category (that is compatible with the functor
  $\Sigma_{\mathcal{K}}$).  In this case, the definition of weighted
  triangles can be pursued following the steps in
  \S\ref{subsubsec:weight-tr}, but at the space level, without moving
  to an algebraic category. This approach goes beyond the scope of
  this paper and will not be pursued here.
\end{itemize}

\begin{rem}\label{rem:shift-et-al}
  Of course, it is also possible to mix in some sense the two
  approaches mentioned before. For instance, in the two examples below
  the category $\mathcal{K}$ carries a shift functor \jznote{$\Sigma_{\K}$ as
  at $B$ but also a functor $\Phi$ as at $A$} such that $\Phi$ commutes
  with the shift functors in the domain and target. In that case we
  can use $\Phi$ to pull back to $\mathcal{K}$ more of the structure
  and weights in $\C$ (of course, this remains less precise than
  constructing weights at the space level).
\end{rem}
% \begin{rem} When pursuing the approach at A above, more refined
%   information can be extracted from a lift
%   $ \bar{\Phi}:\mathcal{K}\to \C$ of $\Phi$, if it exists, and if it
%   commutes with $\Sigma$ and $T$.
%\end{rem}

\subsubsection{Topological spaces with action functionals.}
We will discuss here a category denoted by
$\mathcal{AT}op_{\ast}$. The objects of this category are pairs
$(A,f_{A})$ where $A$ is a pointed topological space and
$f_{A}:A\to \R$ is a continuous function bounded from below by the
value $f_{A}(\ast_{A})$ of $f_{A}$ at the base point $\ast_{A}$ of
$A$. We will refer to $f_{A}$ as the action functional associated to
$A$. The morphisms in this category are pointed continuous maps
$u:A\to B$ such that there exists $r\in \R$ with the property that
$f_{B}(u(x))\leq f_{A}(x)+r\ , \ \forall x\in A$.

We will see that there is a natural \pbnote{{\em contravariant}}
functor
\begin{equation} \label{eq:funtor-spaces}
\Phi:\mathcal{AT}op_{\ast} \to [H^{0}\mathcal{FK}_{\k}]_{\infty}
\end{equation}
inducing a weight $\bar{w}_{\mathcal{AT}op_{\ast}}$ and the associated
pseudo-metrics $\underline{d}^{\mathcal{F}}$ on
$\mathrm{Obj}(\mathcal{AT}op_{\ast})$ along the lines of point A
above.

\begin{rem}\label{rem:constr-top}
  The condition on $f_{A}$ being bounded from below is one possible
  choice in this construction. Its role is to allow the constant map
  $u: (A, f_{A})\to (B,f_{B})$ to be part of the morphisms of
  $\mathcal{AT}op_{\ast}$.
\end{rem}

Before proceeding with the construction of the functor $\Phi$ we
discuss some features of $\mathcal{AT}op_{\ast}$. Notice first that
the morphisms are filtered with the $r$-th stage being
$$\Mor_{\mathcal{AT}op_*}^{\leq r}(A,B) =
\{ u: A\to B \ | \ u\ \mathrm{continuous}, \ \ f_{B}(u(x)) \leq
f_{A}(x)+r \ , \ \forall x\in A\}~.~$$ There is an obvious functor
$\Sigma_{\mathcal{AT}op_{\ast}} : (\R,+)\to \mathcal{AT}op_{\ast}$
defined by $\Sigma^{s}(A,f_{A})=(A,f_{A}+s)$ which is the identity on
morphisms.  The next step is to define the translation functor
$T_{\mathcal{AT}op_{\ast}}$.  At the underlying topological level this
is just the topological suspension but we need to be more precise
about the action functional. Given an object $(A,f_{A})$ we first
define the cone $(CA, f_{CA})$. We take $CA$ to be the reduced cone,
in other words the quotient topological space
$CA=A\times [0,1]/(A\times \{1\}\cup \ast_{A}\times [0,1])$.  To
define $f_{CA}$ we first consider the homotopy
$h_{A}:A\times [0,1]\to \R$,
$$
h_{A}(x,t)= \left\{ \begin{array}{l}
f_{A}(x) \  \ \  \mathrm{if}\ \ \  0\leq t \leq \frac{1}{2}\\
(2-2t) (f_{A}(x) -f_{A}(\ast_{A})) + f_{A}(\ast_{A}) \ \ \  \mathrm{if} \  \  \ \frac{1}{2}\leq t \leq 1 
\end{array}\right. ~.~
$$
The map $f_{CA}: CA\to \R$ is induced by $h_{A}$. We now define the
reduced suspension, $SA= CA/A\times\{0\}$ and take $f_{SA}$ to be the
map induced to the quotient by the homotopy
$h'_{A}:A\times [0,1]\to \R$,

$$
h'_{A}(x,t)= \left\{ \begin{array}{l}
2t (f_{A}(x)-f_{A}(\ast_{A}))+f_{A}(\ast_{A}) \  \ \  \mathrm{if}\ \ \  0\leq t \leq \frac{1}{2}\\
(2-2t) (f_{A}(x) -f_{A}(\ast_{A})) + f_{A}(\ast_{A})\ \ \  \mathrm{if} \  \  \ \frac{1}{2}\leq t \leq 1 
\end{array}\right.~.~
$$
We put $T(A,f_{A})=(SA, f_{SA})$. It is immediate to see that $T$
extends to a functor on $\mathcal{AT}op_{\ast}$ and that it commutes
with $\Sigma$. Moreover, both $\Sigma$ and $T$ so defined commute and
are compatible with the filtration of the morphisms in the sense that
they take $\Mor^{\leq r}$ to $\Mor ^{\leq r}$ for each $r$. Moreover,
composition of morphisms is also compatible with the filtrations in
the sense that it takes
$\Mor^{\leq r_{1}}(B,C) \times \Mor^{\leq r_{2}}(A,B)$ to
$\Mor^{\leq r_{1}+r_{2}}(A,C)$.

We now define the class of exact triangles
$\Delta_{\mathcal{AT}op_{\ast}}$. For this we consider a morphism
$u: (A, f_{A})\to (B,f_{B})$ and we first define its cone
$\mathrm{Cone}(u)$. As a topological space this is, as expected, the
quotient topological space $(B\cup CA)/\sim$ where the equivalence
relation $\sim$ is generated by $f(x)\sim x\times \{0\}$. The base
point of $\mathrm{Cone}(u)$ is the same as that of $B$.  The action
functional $f_{\mathrm{Cone}(u)}$ is induced to the respective
quotient by :
$$
G(x)= \left\{ \begin{array}{l}
f_{B}(x) \  \ \  \mathrm{if}\ \ \  x\in B \\
(1-2t) (f_{B}(u(y)) -f_{B}(\ast_{B})) + 2t( f_{A}(y)-f_{A}(\ast_{A}))+f_{B}(\ast_{B}) \   \mathrm{if} \ \ x = (y,t) \in A\times [0,\frac{1}{2}]\\
(2-2t) (f_{A}(y) -f_{A}(\ast_{A}))+ f_{B}(\ast_{B}) \  \mathrm{if}\ \  x=(y,t) \in A\times [\frac{1}{2}, 1]
\end{array}\right.. 
$$
There is an obvious inclusion $i: (B,f_{B})\to \mathrm{Cone}(u)$ as
well as a projection $p :\mathrm{Cone}(u)\to TA$ (that contracts $B$
to a point). This map belongs to our class of morphisms because the
functional $u$ is bounded from below.  The class
$\Delta_{\mathcal{AT}op_{\ast}}$ consists of triangles $\Delta$:
\begin{equation}\label{eq:extr-top}
  \Delta : A\xrightarrow{u} B\xrightarrow{i}
  \mathrm{Cone}(u)\xrightarrow{p} TA ~.~
\end{equation}
We finally construct the functor
$\Phi: \mathcal{AT}op_{\ast}\to [H^{0}\mathcal{FK}_{\k}]_{\infty}$.
\pbnote{This functor will be contravariant, since the objects of
  $\mathcal{FK}_{\k}$ are \jznote{cochain} complexes (rather than chain
  complexes).}

% There is a slight (formal) complication
% at this point because $\mathcal{FK}_{\k}$ has as objects cochain
% complexes (and not chain complexes) and, as a result, $\Phi$ is
% contravariant.

First we fix some notation: for a pointed topological space $X$ we
denote by $\tilde{C}_{\ast}(X)$ the reduced singular chain complex of
$X$ with coefficients in $\k$ and by $\tilde{C}^{\ast}(X)$ the reduced
singular cochain complex (we denote without $\tilde{(-)}$ the
non-reduced chain/cochain complexes) and if $Y\subset X$ is a pointed
subspace, then $\tilde{C}_{\ast}(X,Y)$ and $\tilde{C}^{\ast}(X,Y)$ are
the relative (co)chains.  Consider an object of \jznote{$\mathcal{AT}op_*$},
$(A,f_{A})$, and let $A^{\leq r}=(f_{A})^{-1}(-\infty, r]$. Notice
that the spaces $A^{\leq r}$ are pointed (if non-void).  There is a
filtration of $C^{\ast}(A)$ defined by:
\jznote{$$\tilde{C}^{\ast}(A)^{\leq r}=
\mathrm{im}[\tilde{C}^{\ast}(A, A^{\leq r})\to \tilde{C}^{\ast}(A)]$$}
(these are the cochains in $A$ that vanish over the singular chains of
$A^{\leq r}\subset A$). It is clear that the cochain differential
preserves this filtration.  Finally, we define the functor $\Phi$.
For each object $(A,f_{A})$ of $\mathcal{AT}op_{\ast}$ we take
$\Phi(A, f_{A})$ to consist of the cochain complex
$\tilde{C}^{\ast}(A)$ together with the filtration
$\{\tilde{C}^{\ast}(A)^{\leq r}\}$ defined above.  For a morphism
$u: (A,f_{A})\to (B, f_{B})$ we take $\Phi(u)=[\tilde{C}^{\ast}(u)]$
where $[-]$ represents the cochain-homotopy class of the respective
cochain morphism.

The definition of the morphisms in $\mathcal{AT}op_{\ast}$ implies
that $\Phi(u)$ is indeed a morphism in
$[H^{0}\mathcal{FK}_{\k}]_{\infty}$. Moreover, because we are using
everywhere reduced cochain complexes (and we work in the pointed
category), we have that $\Phi(\Delta)$ is exact in
$[H^{0}\mathcal{FK}_{\k}]_{\infty}$ for each of the triangles in
$\Delta_{\mathcal{AT}op_{\ast}}$. Further, the functor $\Phi$ also
interchanges the shift functors in the domain and the target.

In all cases, the weight $\bar{w}_{\mathcal{AT}op_{\ast}}$ is well
defined as well as the associated fragmentation pseudo-metrics
$\underline{d}^{\mathcal{F}}(-,-)$ on the objects of
$\mathcal{AT}op_{\ast}$. \jznote{Roughly speaking, these fragmentation pseudo-metrics measure how much ``weight'' we need to obtain a given topological space via successive cone attachments of spaces in $\F$.}
\begin{rem} \label{rem:gen-top} (a) The choice of the class
  $\Delta_{\mathcal{AT}op_{\ast}}$ given above is quite restrictive
  with the consequence that the resulting pseudo-metrics are often
  infinite. One alternative is to enlarge this class to all triangles
  in $\mathcal{AT}op_{\ast}$ that are homotopy equivalent to those in
  the initial class through maps (and homotopies) of filtration $0$.

  (b) From some points of view, working in the {\em pointed} category
  of spaces endowed with an action functional is not natural. Other
  choices are possible, in particular some such that the translation
  functor $T$ more closely imitates dynamical stabilization.

  (c) The restriction of $\Phi$ to compact topological spaces admits
  an obvious lift to $H^{0}\mathcal{FK}_{\k}$. However, without such a
  restriction, such a lift does not seem to be available in full
  generality.
\end{rem}

\subsubsection{Metric spaces} \label{subsubsec:metr} The category
$\mathcal{M}etr_{0}$ that we will consider here has as objects
path-connected metric spaces $(X,d_{X})$ of finite diameter.  The
morphisms are Lipschitz maps. These are maps $\phi :X\to Y$ such that
there exists a constant $c\in [0,\infty)$, called the Lipschitz
constant of $u$, with the property that
$d_{Y}(\phi(x),\phi(y))\leq c~ d_{X}(x,y)$ for all $x,y\in X$.
\begin{rem} \label{rem:constr-metr} The finite diameter condition
  imposed here - indicated by the \pbnote{subscript} $_{0}$ -
  \pbnote{is necessary for some of} the constructions below.  The
  connectivity \pbnote{assumption} is more a matter of convenience.
\end{rem}

We will construct a functor as in (\ref{eq:funtor-spaces}) with one
main modification.  For convenience, we prefer defining a covariant
functor and thus our target category will not be a category of cochain
complexes but rather one of filtered {\em chain complexes} (the
passage from one to the other is formal, \pbnote{replacing $C^*$ by
  $C_{-*}$ and vice-versa}). We will denote the category of filtered
chain complexes over $\k$ by $\mathcal{FK}'_{\k}$. This behaves just
as a usual dg-category except that the differential on the space of
morphisms is of degree $-1$. With this change, we will construct:
\jznote{\begin{equation}\label{eq:functor-metr}
  \Phi': \mathcal{M}etr_{0} \to [H_{0}\mathcal{FK}'_{\k}]_{\infty}~.~
\end{equation}}
as well as related structures on $\mathcal{M}etr_{0}$, as at point A at
the beginning of the section (see also Remark \ref{rem:shift-et-al}).

We start by noting that there is an obvious increasing filtration of
the morphisms in $\mathcal{M}etr_{0}$
with
$$\Mor_{\mathcal{M}etr_{0}}^{\leq r}(X,Y)=\{ u: X\to Y \ | \
\mathrm{the\ Lipschitz\ constant\ of}\ u \ \mathrm{is}\ \leq
e^{r}\}~.~$$ It is immediate to see that this filtration is compatible
with composition. There is also a functor
$\Sigma_{\mathcal{M}etr_{0}}:(\R, +)\to \mathcal{M}etr_{0}$ defined by
rescaling the metric, $\Sigma^{s}(A, d_{A}) = (A, e^{s} d_{A})$, and
which is the identity on morphisms. As in the example in the previous
section, we next will define the translation functor
$T_{\mathcal{M}etr_{0}}$ and the class of triangles
$\Delta_{\mathcal{M}etr_{0}}$. The first step is to construct the
metric cone $C'A$ for an object $(A, d_{A})$ in our
class. Topologically, the cone $C'A$ will be this time the {\em
  unreduced} cone over $A$. Thus it is defined by
$C'A=A\times [0,1]/A\times \{0\}$.  To define the metrics $d_{C'A}$,
first let $D_{A}$ be the diameter of $A$.  We then put
\begin{equation}
d_{C'A}((x,t), (y,t'))= \frac{D_{A}}{2}|t-t'| + \min\{t,t'\}~d_{A}(x,y).
\end{equation}
It is immediate to see that this does indeed define a metric on
$C'A$. A similar construction is available to construct $T(A,
d_{A})$. Topologically, we will define first the - {\em non-reduced} -
suspension, $S'A$, as the topological quotient of
$A\times [-\frac{1}{2},\frac{1}{2}]$ with $A\times \{-1/2\}$
identified to a point $S$ and $A\times \{+1/2\}$ identified to a
different point $N$. We now define $d_{S'A}$ by
\begin{equation}
  d_{S'A}((x,t), (y,t'))=
  \frac{D_{A}}{2}|t-t'| +\min\{\frac{1}{2}-|t|,\frac{1}{2}-|t'|\}~d_{A}(x,y)
\end{equation}
and again it is immediate to see that this defines a metric on
$S'A$. We now put $T(A,d_{A})= (S'A, d_{S'A})$.  The next step is to
define the triangles in $\Delta_{\mathcal{M}etr_{0}}$. For this we
assume $u : (A,d_{A})\to (B,d_{B})$ is a morphism in our category and
we want to define the (non-reduced) cone of $u$,
$\mathrm{Cone}'(u)$. Topologically, this is, as usual,
$B\cup C'A/[\{x\}\times \{0\}\sim u(x) \ | \ x\in A]$. To define a
metric on $\mathrm{Cone}'(u)$ we notice first that given \pbnote{a
  map} $g:X\to Y$ and a pseudo-metric $d_{Y}$ on $Y$, there is a
pull-back pseudo-metric on $X$ given by
$g^{\ast}d_{Y} (a,b)= d_{Y}(g(a),g(b))$.  We now let
$A'= u(A)\subset B$ and we denote by $\bar{u}:A'\hookrightarrow B$ the
inclusion. Notice that $\mathrm{Cone}'(\bar{u})\subset C'B$.  Thus
$\mathrm{Cone}'(\bar{u})$ is endowed with a metric given by the
restriction of the metric $d_{C'B}$ on $C'B$. There are obvious
projections $\pi:\mathrm{Cone}'(u)\to \mathrm{Cone}'(\bar{u})$ and
$p: \mathrm{Cone}'(u)\to S'A$. Here $p$ collapses $B$ to the point $S$
in the suspension and sends $(x,t) \to (x,t-\frac{1}{2})$ for the
points $(x,t)\in C'A$. We now define
$$d_{\mathrm{Cone}'(u)}=\max\{\pi^{\ast}d_{C'B}+
p^{\ast}d_{S'A}\}~.~$$ Notice that, if $u$ is not injective and $B$ is
not a single point, then the two pseudo-metrics in the right term of
the equality are each degenerate. Nonetheless, $d_{\mathrm{Cone}'(u)}$
is non-degenerate. Finally, the class of triangles
$\Delta_{\mathcal{M}etr_{0}}$ consists of triangles:
$$A\xrightarrow{u} B\xrightarrow{i} \mathrm{Cone}'(u)\xrightarrow{p} S'A$$
where $i$ is the inclusion and $p$ is the projection above.

With this preparation, we can now define the functor $\Phi'$ from
(\ref{eq:functor-metr}).  Consider an object $(A,d_{A})$ in our
category and the associated singular complex $C_{\ast} (A)$.  This
chain complex is filtered as follows:
\jznote{$$C^{\leq r}_{k}(A)=\left\{\sum_{i }a_{i}
\sigma_{i} \ \bigg|\ a_{i}\in\k \ , \ \sigma_{i} \ \mathrm{a\ singular\
  simplex\ of\ diameter\ at\ most}\ e^{r}\right\}~.~$$} In other words, in
the expression above, $\sigma_{i}: \Delta^{k}\to A$ is a continuous
map with the standard $k$-simplex as domain and such that
$d_{A}(\sigma_{i}(x),\sigma_{i}(y))\leq e^{r}$ for any
$x,y\in \Delta^{k}$.  Consider the constant map $c : A \to \ast$ this
induces an obvious surjection $C_{\ast}\to C_{\ast}(\ast)$ and we
denote by $\bar{C}_{\ast}(A)$ the kernel of this map (this is
quasi-isomorphic to the reduced singular chain complex of $A$ -
because $A$ is connected - but is independent of the choice of
base-point). There is an induced filtration
$\bar{C}^{\leq r}_{\ast}(A)$.  We now put
$$\Phi'(A,d_{A})= \bar{C}_{\ast}(A)\ \mathrm{with\ the\ filtration}\
\{\bar{C}^{\leq r}_{\ast}(A)\}_{r} ~.~$$ Further, for a morphism
$u: (A,d_{A})\to (B,d_{B})$ we take $\Phi'(u)=[C_{\ast}(u)]$, the
chain homotopy class of the singular chain map $C_{\ast}(u)$
(restricted to the $\bar{C}(-)$ complexes).

It is easy to see that this $\Phi'$ is indeed a functor as desired and
that $\Phi'(\Delta)$ is exact for each triangle $\Delta$ as defined
above and, again, $\Phi'$ interchanges the shift functors in the
domain and target. In summary, the weight
$\overline{w}_{\mathcal{M}etr_{0}}$ is well defined as well as the
quantities $\underline{\delta}^{\mathcal{F}}$ and the pseudo-metrics
associated to them.

\begin{rem}\label{rem:gen-metr}
  (a) Similarly to Remark \ref{rem:gen-top}, the definition of the
  triangles in $\Delta_{\mathcal{M}etr_{0}}$ is highly restrictive
  and, in this case, even the objects in our category are subject to a
  constraint - finiteness of the diameter - that might be a hindrance
  in applications. One way to apply the methods above to study spaces
  of infinite diameter is to consider triangles of the form
  $\Delta : A\to B\to C \to S'A$ where $A$ is of finite diameter such
  that $S'A$ admits a metric as above and analyze when $\Phi'(\Delta)$
  is of finite persistence weight in
  $[H_{0}\mathcal{FK}'_{\k}]_{\infty}$.

  (b) In studying metric spaces of infinite diameter by these methods,
  it is likely that the most appropriate structure that fits with the
  cone construction is that of {\em length structure}, in the sense of
  Gromov, as in Chapter 1, Section A in \cite{Pansu}. We will not
  further pursue this theme here.
\end{rem}

\subsubsection{Further remarks on topological examples.}

\

A. In the topological examples above - for instance in
$\mathcal{AT}op_{\ast}$ - it is natural to see what the quantities
$\underline{\delta}^{\mathcal{F}}(-)$ mean even for the flat weight
$w_{0}$, which associates to each exact triangle the value $1$. Of
course, in this case $\underline{\delta}^{\mathcal{F}}(A, B)$ simply
counts the minimal number of cone-attachments in the category
$\mathcal{AT}op_{\ast}$ that are needed to obtain $A$ out of the space
$B$ by attaching cones over spaces in the family $\mathcal{F}$ using
the family of triangles $\Delta_{\mathcal{AT}op_{\ast}}$. Given that
the weight is flat, the question is independent of filtrations and
shift functors and it reduces to the identical question in the
category of pointed spaces $\mathcal{T}op_{\ast}$. In the examples
below we will focus on this category and on
$\underline{\delta}^{\mathcal{F}}(A,\ast)$ which is one of the most
basic quantities involved.

It is useful to keep in mind that there are two more choices that are
essential in defining $\underline{\delta}^{\mathcal{F}}(A,\ast)$: the
choice of family $\mathcal{F}$ and the choice of the class of exact
triangles $\Delta_{\mathcal{T}op_{\ast}}$ - see also Remark
\ref{rem:gen-top} (a).

\begin{itemize}
\item[(i)] $\mathcal{F}=\{S^{0}, S^{1},\ldots, S^{k},\ldots\}$;
  $\Delta_{\mathcal{T}op_{\ast}}$ are the triangles
  $A\xrightarrow{u} B\to \mathrm{Cone}(u)\to SA$ as in
  (\ref{eq:extr-top}) (but omitting the action functionals).  In this
  case, $\underline{\delta}^{\mathcal{F}}(A,\ast)= k < \infty$ means
  that $A$ has the structure of a finite $CW$- complex with $k$ cells.
\item[(ii)] $\mathcal{F}=\{S^{0}, S^{1},\ldots, S^{k},\ldots\}$; we
  now take $\Delta_{\mathcal{T}op_{\ast}}$ to be the triangles
  homotopy equivalent to the triangles
  $A\xrightarrow{u} B\to \mathrm{Cone}(u)\to SA$ from
  (\ref{eq:extr-top}). In this case,
  $\underline{\delta}^{\mathcal{F}}(A,\ast)= k$ means that $A$ is
  homotopy equivalent to a $CW$-complex with $k$ cells.  This number
  is obviously a homotopy invariant. It is clearly bounded from below
  by the sum of the Betti numbers of $A$.
\item[(iii)] $\mathcal{F}$ are pointed spaces with the homotopy type
  of $CW$-complexes; $\Delta_{\mathcal{T}op_{\ast}}$ are as at
  (ii). In this case, the definition
  $\underline{\delta}^{\mathcal{F}}(A,\ast)$ coincides with that of
  the cone-length, $\mathrm{Cl}(A)$, of $A$ (for a space $A$ with the
  homotopy type of a $CW$-complex).  Cone-length is a homotopical
  invariant which is of interest because it is bigger, but not by more
  than one, than the Lusternik-Schnirelmann category \cite{Cor94} which, in turn,
  provides a lower bound for the minimal number of critical points of
  smooth functions on manifolds. Incidentally, as noted by Smale
  \cite{Sm}, a version of the Lusternik-Schnirelmann category provides
  also a measure for the complexity of algorithms, see
  \cite{Co-Lu-Op-Ta} for more on this subject.
\item[(iv)] At this point we will change the underlying category and
  place ourselves in the pointed category of finite type,
  simply-connected {\em rational} spaces $\mathcal{T}op_{1}^{\Q}$
  (see~\cite{Hal-Fel-Th}). We take $\mathcal{F}$ to consist of finite
  wedges of rational spheres of dimension at least $2$.  The triangles
  $\Delta_{\mathcal{T}op^{\Q}_{\ast}}$ are as at (ii) (in the category
  of rational spaces) but we will also allow in
  $\Delta_{\mathcal{T}op^{\Q}_{\ast}}$ ``formal'', triangles of the
  form $S^{-1}F\to \ast \to F$ where $F\in \mathcal{F}$ (de-suspending
  is not possible in our category but we still want to have for a
  rational $2$-sphere, $S^{2}_{\Q}$,
  $\underline{\delta}^{\mathcal{F}}(S^{2}_{\Q}, \ast)=1$).  In this
  setting, it turns out \cite{Co-tr} that
  $$\underline{\delta}^{\mathcal{F}}(A,\ast)=\mathrm{Cl}(A)=\mathrm{nil}(A)~.~$$ 
  Both equalities here are non-trivial, the first because in the
  definition of $\mathrm{Cl}$ we are using cones over arbitrary
  (rational) spaces while in this example $\mathcal{F}$ consists of
  only wedges of spheres. For the second equality, $\mathrm{nil}(A)$
  is the minimal order of nilpotence of the augmentation ideal
  $\overline{\mathcal{A}}$ of a rational differential graded
  commutative algebra $\mathcal{A}$ representing $A$ (recall that by a
  celebrated result of Sullivan \cite{Sul}, the homotopy category of
  rational simply connected spaces is equivalent to the homotopy
  category of rational differential graded commutative algebras, the
  representative of a given space being given by the so-called $PL$-de
  Rham complex of $A$).
\end{itemize}

\

B. One of the difficulties of extracting a triangulated persistence
category from a topological category such as those considered in this
section is very basic and has to do with the difference between stable
and unstable homotopy. In essence, recall that if $\mathcal{C}$ is a
TPC, then the $0$-level category $\mathcal{C}_{0}$ is required to be
triangulated. However, in unstable settings, homotopy categories of
spaces are not triangulated.
\begin{itemize}
\item[(i)] An instructive example is a variant of our discussion
  concerning the category $\mathcal{M}etr_{0}$. In this case the
  morphisms $\Mor_{\mathcal{M}etr_{0}}(A,B)$ carry an obvious topology
  as well as a filtration, as described in \S\ref{subsubsec:metr}. We
  now can consider a new category, $\widetilde{\mathcal{M}}etr_{0}$,
  with the same objects as $ \mathcal{M}etr_{0}$ but with morphisms
  $\Mor_{\widetilde{\mathcal{M}}etr_{0}}(A,B)=S_{\ast}(\Mor_{\mathcal{M}etr_{0}}(A,B))$
  where $S_{\ast}(-)$ stands for cubical chain complexes. These
  morphisms carry an obvious filtration obtained by applying the
  cubical chains to the filtration of
  $\Mor_{\mathcal{M}etr_{0}}(A,B)$.  The composition in this category
  is given by applying cubical chains to the composition
  $\Mor_{\mathcal{M}etr_{0}}(B,C)\times
  \Mor_{\mathcal{M}etr_{0}}(A,B)\to \Mor_{\mathcal{M}etr_{0}}(A,C)$
  and composing with the map
  $S_{\ast} (\Mor_{\mathcal{M}etr_{0}}(B,C))\otimes
  S_{\ast}(\Mor_{\mathcal{M}etr_{0}}(A,B))\to
  S_{\ast}(\Mor_{\mathcal{M}etr_{0}}(B,C)\times
  \Mor_{\mathcal{M}etr_{0}}(A,B))$ induced by taking products of
  cubes. It follows that $\widetilde{\mathcal{M}}etr_{0}$ is a
  filtered dg-category (in homological formalism).  Thus all the
  machinery in \S\ref{subsec:dg} is applicable in this case. Moreover,
  this category carries an obvious shift functor.  However,
  $[H_{0}\widetilde{\mathcal{M}}etr_{0}]_{\infty}$ is not triangulated
  and thus $\widetilde{\mathcal{M}}etr_{0}$ is not pre-triangulated
  (quite far from it).  Indeed,
  $\Mor_{[H_{0}\widetilde{\mathcal{M}}etr_{0}]_{\infty}}(A,B)$ is the
  free abelian group generated by the homotopy classes of Lipschitz
  maps from $A$ to $B$. As a result, the translation functor (which is
  in our case the topological suspension) is certainly not an
  isomorphism.

\item[(ii)] As mentioned before, at point B at the beginning of
  \S\ref{subsec:top-sp}, a way to bypass these issues is to introduce
  a sort of filtered Waldhausen category or a similar formalism and
  develop a machinery parallel to that of TPC's in this unstable
  context. The structure present in $\mathcal{AT}op_{\ast}$ and
  $\mathcal{M}etr_{0}$ suggests that such a construction is possible
  and will be relevant in these cases.

\item[(iii)] There is yet another approach to associate to each of
  $\mathcal{AT}op_{\ast}$ and $\mathcal{M}etr_{0}$ a triangulated
  persistence category that is more geometric in nature. This is based
  on moving from these categories to stable categories where the
  underlying objects are the spectra obtained by stabilizing the
  objects of the original categories and the morphisms come with an
  appropriate filtration induced from the respective structures
  (action functionals or, respectively, metrics) on the initial
  objects.  This seems likely to work and to directly produce a TPC
  but we will not pursue the details at this time.
\end{itemize}

% !TEX root = TPC.tex

\subsection{Filtrations in Tamarkin's category}\label{subsec:Tam}

This section is devoted to an example of a triangulated persistence
category that comes from the filtration structure present in
Tamarkin's category. This category was originally defined in
\cite{Tam08}, based on singular supports of sheaves, and was used to
prove some non-displaceability results in symplectic geometry, as well
as other more recent results related to Hamiltonian dynamics (see
\cite{GKS12}).

\subsubsection{Background on Tamarkin's category} Let $X$ be a
manifold, and let $\D(\k_X)$ be the derived category of sheaves of
$\k$-modules over $X$. In particular, this is a triangulated
category. For any $A \in {\rm Obj}(\D(\k_X))$, due to microlocal sheaf
theory, as established in \cite{KS90}, one can define the singular
support of $A$, denoted by $SS(A)$, a conical (singular) subset of
$T^*X$. \jznote{We refer to Chapter V in \cite{KS90} for the precise definition of $SS(A)$ and a detailed study of its properties.} Now, let $X = M \times \R$ where $M$ is a closed manifold, and
denote by $\tau$ the co-vector coordinate of $T^*\R$ in
$T^*(M \times \R)$. Consider the following full subcategory of
$\D(\k_{M \times \R})$, denoted by
$\D_{\{\tau \leq 0\}}(\k_{M \times \R})$, where
\[ {\rm Obj}(\D_{\{\tau \leq 0\}}(\k_{M \times \R})) = \left\{A \in {\rm Obj}(\D(\k_{M \times \R})) \, | \, SS(A) \subset \{\tau \leq 0\}\right\}. \]
If $A \to B \to C \to A[1]$ is an exact triangle in $\D(\k_{M \times \R})$, then $SS(C) \subset SS(A) \cup SS(B)$. This implies that $\D_{\{\tau \leq 0\}}(\k_{M \times \R})$ is a triangulated subcategory of $\D_{\k_{M \times \R}}$. Tamarkin's category is defined by 
\begin{equation} \label{dfn-tam-cat}
\T(M) : = \D_{\{\tau \leq 0\}}(\k_{M \times \R})^{\perp, l}
\end{equation} 
where the ${\perp, l}$ denotes the left orthogonal complement of
$\D_{\{\tau \leq 0\}}(\k_{M \times \R})$ in $\D(\k_{M \times
  \R})$. Then $\T(M)$ is also a triangulated subcategory. By
definition, note that
$\T(M) \subset \D_{\{\tau \geq 0\}}(\k_{M \times \R})$. When
$M = \{\rm pt\}$, Tamarkin's category $\T(\{{\rm pt}\})$, together
with a constructibility condition, can be identified with the category
of persistence $\k$-modules (see A.1 in \cite{Zha20}).

\begin{rem} There exists a restricted version of Tamarkin's category
  denoted by $\T_V(M)$ where $V \subset T^*M$ is a closed subset (see
  Section 3.2 in \cite{Zha20}). This restricted Tamarkin category is
  useful to prove the non-displaceability of some subsets in $T^*M$
  (see \cite{AI20}). In this paper, we will only focus on $\T(M)$.
\end{rem}

One way to understand the definition (\ref{dfn-tam-cat}) is that $\T(M)$ is an admissible subcategory (see Definition 1.8 in \cite{Orl06}) in the sense that for any object $A$ in $\D(\k_{M \times \R})$, one can always split $A$ in the form of an exact triangle 
\jznote{
\begin{equation} \label{tamarkin-split}
B \to A \to C \to B[1]
\end{equation}}
in $\D(\k_{M \times \R})$, where $B \in \T(M)$ and $C \in \D_{\{\tau \leq 0\}}(\k_{M \times \R})$. In fact, this splitting can be achieved in a rather concrete manner, which involves an important operator called sheaf convolution on objects in $\D(\k_{M \times \R})$. Explicitly, for any two objects $A, B$ in $\D({\bf k}_{M \times \R})$, the sheaf convolution of $A$ and $B$ is defined by
\begin{equation} \label{dfn-sh-conv}
A \ast B : = \delta^{-1} Rs_! (\pi_1^{-1} A \otimes \pi_2^{-1} B),
\end{equation}
where $\pi_i: (M \times \R)^2 \to M \times \R$ are the projections to each factor of $M \times \R$, $s$ keeps the $M \times M$-part the same but adds up two inputs on the $\R$-factors, and $\delta$ is the diagonal embedding from $M$ to $M \times M$. For instance, $\k_{M \times [0, \infty)} * \k_{M \times [0,\infty)} = \k_{M \times [0, \infty)}$, where $\k_{V}$ for a closed subset $V$ denotes the constant sheaf with its support in $V$. Moreover, this operator is commutative and associative. An important characterization of an object in $\T(M)$ is that (see Proposition 2.1 in \cite{Tam08}), 
\begin{equation} \label{tam-char}
A \in {\rm Obj}(\T(M)) \,\,\,\,\mbox{if and only if} \,\,\,\, A * \k_{M \times [0,\infty)} = A
\end{equation}
which implies that (i) for any object $A$ in $\D(\k_{M \times \R})$, the sheaf convolutions $B: = A * \k_{M \times [0,\infty)}$ and $C: =A * \k_{M \times (0, \infty)}[1]$ provide the desired exact triangle for a splitting of $A$ in (\ref{tamarkin-split}); (ii) sheaf convolution is a well-defined operator on $\T(M)$. 

With the help of the sheaf convolution, the $\R$-component generates a filtration structure in $\T(M)$ in the following way. For any $r \in \R$, consider the map $T_r: M \times \R \to M \times \R$ defined by $(m, a) \mapsto (m, a+r)$. One can show that for any object $A$ in $\T(M)$, the induced object $(T_r)_*A = A * \k_{M \times [r, \infty)}$ (see Lemma 3.2 in \cite{Zha20}). In fact, $\{(T_r)_*\}_{r \in \R}$ defines an $\R$-family of functors on $\T(M)$. Moreover, if $r \leq s$, then by the restriction map $\k_{M \times [r, \infty)} \to \k_{M \times [s, \infty)}$, we have a canonical morphism $\tau_{r,s}(A): (T_{r})_*A \to (T_{s})_*A$. At this point, notice that for $r \leq s$, there does not exist non-zero morphism from $\k_{M \times [s, \infty)}$ to $\k_{M \times [r, \infty)}$, so the canonical map $\tau_{r,s}$ respects the partial order $\leq$ on $\R$. For any $r \leq s$, $\tau_{r,s}$ is viewed as a natural transformation from $(T_r)_*$ to $(T_s)_*$. Finally, we call an object $A$ in $\T(M)$ a $c$-torsion element if $\tau_{0,c}(A): A \to (T_c)_*A$ is zero. For instance, when $M = \{\rm pt\}$, the constant sheaf $\k_{[a,b)} \in \T(\{\rm pt\})$ with a finite interval $[a,b)$ is a $(b-a)$-torsion. 

We will end this subsection by a discussion on the hom-set in $\T(M)$. It is more convenient to consider derived hom, that is, ${\rm RHom}_{\T(M)}(A,B)$ for any two objects $A,B$ in $\T(M)$. Lemma 3.3 in \cite{Zha20} (or (1) in Lemma 3.8 in \cite{Tam08}) provides a more explicit way to express such ${\rm RHom}$, that is, 
\begin{equation} \label{tam-hom-set}
{\rm RHom}_{\T(M)}(A,B) = {\rm RHom}_{\R}(\k_{[0, \infty)}, R\pi_*{\mathcal Hom}^*(A,B)). 
\end{equation}
By taking the cohomology at degree $0$, we obtain ${\rm Hom}_{\T(M)}(A,B)$ as a $\k$-module. Here, $\pi: M \times \R \to \R$ and ${\mathcal Hom}^*(\cdot, \cdot)$ is the right adjoint functor to the sheaf convolution (see Definition 3.1 in \cite{AI20}). The right-hand side of (\ref{tam-hom-set}) is relatively computable since they are all (complexes of) sheaves over $\R$ (cf.~A.2 in \cite{Zha20}). Moreover, by using the adjoint relation between ${\mathcal Hom}^*(\cdot, \cdot)$ and the sheaf convolution, one obtains a shifted version of (\ref{tam-hom-set}), that is,
\begin{equation} \label{tam-hom-set2}
{\rm RHom}_{\T(M)}(A,(T_r)_*B) = {\rm RHom}_{\R}(\k_{[0, \infty)}, (T_r)_*(R\pi_*{\mathcal Hom}^*(A,B))).
\end{equation}
Therefore, for any $r \leq s$, there exists a well-defined morphism 
\begin{equation} \label{tam-hom-transfer}
\iota^{A,B}_{r,s}: {\rm Hom}_{\T(M)}(A,(T_r)_*B) \to {\rm Hom}_{\T(M)}(A,(T_s)_*B),
\end{equation}
which is induced by the morphism $\tau_{r,s}(R\pi_*{\mathcal Hom}^*(A,B))$. Finally, we have a canonical isomorphism, 
\begin{equation} \label{tam-equiv}
T_r: {\rm Hom}_{\T(M)}((T_r)_*A, (T_r)_*B) \simeq {\rm Hom}_{\T(M)}(A,B)
\end{equation}
which is induced by the sheaf convolution with $\k_{M \times [r, \infty)}$. In particular, $T_r$ commutes with the morphism $\iota^{A,B}_{r,s}$ defined in (\ref{tam-hom-transfer}).

\subsubsection{Persistence category from Tamarkin's shift functors}
We have seen before that 
 Tamarkin's category is endowed with a shift functor.   We now discuss the persistence structure induced by this shift functor - see Remark \ref{rem:shifts-T} (d).

\begin{dfn} \label{dfn-P(M)} Given the category $\T(M)$ as before, define an enriched category denoted by ${\mathcal P}(M)$ as follows. The object set of ${\rm Obj}({\mathcal P}(M))$ is the same as ${\rm Obj}(\T(M))$, and the hom-set is defined by
\[ {\rm Hom}_{{\mathcal P}(M)}(A,B) = \left\{\{{\rm Hom}_{\T(M)}(A,(T_{r})_*B)\}_{r \in \R}, \{\iota^{A,B}_{r,s}\}_{r\leq s \in \R}\right\} \]
for any two objects $A,B$ in ${\mathcal P}(M)$, where $\iota^{A,B}_{r,s}$ is the morphism defined in (\ref{tam-hom-transfer}). \end{dfn}

\begin{remark} \jznote{Definition \ref{dfn-P(M)} can be regarded as a generalization of (\ref{per-mor}) in \S\ref{ssec-pm} since when $M = \{\rm pt\}$, Tamarkin's category $\mathcal T(\{\rm pt\})$ can be identified with the category of persistence $\k$-modules.} Also, Definition \ref{dfn-P(M)} fits with geometric examples. Indeed, recall a concrete computation of ${\rm Hom}_{\T(M)}(A,B)$ when both $A$ and $B$ are sheaves coming from generating functions on $M$ (see Section 3.9 in 
\cite{Zha20}). In this case, \jznote{${\rm{Hom}}_{{\mathcal P}(M)}(A,B)$} can be identified to a (Morse) persistence $\k$-module in the classical sense. \end{remark}

\begin{lemma} \label{lem-tamarkin-pc} The category ${\mathcal P}(M)$  from Definition \ref{dfn-P(M)},  is a persistence category. \end{lemma}

\begin{proof} Consider the functor $E_{A,B}: (\R, \leq) \to {\rm Vect}_{\k}$ by 
\[ E_{A,B}(r) (= {\rm Mor}^r(A,B)) =  {\rm Hom}_{\T(M)}(A,(T_{r})_*B), \]
and for the morphism $i_{r,s}$ when $r \leq s$, $E_{A,B}(i_{r,s}) = \iota^{A,B}_{r,s}$. Notice that the composition $E_{A,B}(r) \times E_{B,C}(s) \to E_{A,C}(r+s)$ is well-defined due to (\ref{tam-equiv}). Indeed, for any $f \in E_{A,B}(r)$ and $g \in {E_{B,C}(s)}$, the composition is defined by 
\[ (f, g) \mapsto T_r(g) \circ f \in {\rm Hom}_{\T(M)}(A, (T_{r+s})_*C). \]
Then for any $r \leq r'$, $s\leq s'$, we have 
\[ \left(T_{r'}(\iota^{B,C}_{s,s'}(g))\right) \circ \iota^{A,B}_{r,r'}(f) = \iota^{A,C}_{r+s, r'+s'}(T_r(g) \circ f) \]
which completes the proof that ${\mathcal P}(M)$ is a persistence category. 
\end{proof}

%\begin{remark} In a symmetric way, one can consider ${\rm Hom}_{\T(M)}((T_r)_*A,B)$. However, the by %the contravariant property for ${\rm Hom}$-functor in the first input, we will get an ``reserved'' persistence %category in the sense that the functor $E_{A,B}$ is a well-defined functor from $(\R, \geq) \to {\rm Vect}%_{\k}$ if $E_{A,B}(r) : = {\rm Hom}_{\T(M)}((T_r)_*A,B)$. \end{remark}

We now list  some of the properties of the persistence category ${\mathcal P}(M)$. 
\begin{itemize} 
\item[(a)] The $0$-level category  ${\mathcal P}(M)_0$ has the same objects as ${\mathcal P}(M)$, but 
\[ {\rm Hom}_{{\mathcal P}(M)_0}(A,B) = {\rm Hom}_{\T(M)}(A, B). \]
We use the fact that $(T_0)_*= \mathds{1}$. Thus, ${\mathcal P}(M)_0 = \T(M)$. This category is triangulated as we have seen above.
\item[(b)] The $\infty$-level, ${\mathcal P}(M)_{\infty}$, has the same objects as ${\mathcal P}(M)$, but 
\[ {\rm Hom}_{{\mathcal P}(M)_{\infty}}(A,B) = \varinjlim_{r \to \infty} {\rm Hom}_{\T(M)}(A,(T_{r})_*B) \]
where the direct limit is taken via the map $\iota^{A,B}_{r,s}$. This limit category has been considered in (81) in \cite{GS14}.
\item[(c)] On ${\mathcal P}(M)$, each $(T_r)_*$ is a persistence functor for any $r \in \R$, i.e., $(T_r)_* \in \mathcal P({\rm End}({\mathcal P}(M)))$, since $T_r$ commutes with $\tau_{r,s}$. 
\item[(d)] There exists a natural shift functor on ${\mathcal P}(M)$. Define $\Sigma: (\R, +) \to \mathcal P({\rm End}({\mathcal P}(M)))$ by $\Sigma(r) (= \Sigma^r) = (T_{-r})_*$. For any $r, s \in \R$ and $\eta_{r,s} \in {\rm Mor}_{\R}(r,s)$, define 
\[ \Sigma(\eta_{r,s}) := \mathds{1}_{(T_{-r})_* \cdot}. \]
Then, for any object $A$ in $\mathcal P(M)$, 
\begin{align*}
\Sigma(\eta_{r,s})_A = \mathds{1}_{(T_{-r})_*A} & \in {\rm Hom}_{\T(M)}((T_{-r})_*A, (T_{-r})_*A) \\
& = {\rm Hom}_{\T(M)}((T_{-r})_* A, (T_{s-r})_*((T_{-s})_*A))\\
& = E_{(T_{-r})_* A, (T_{-s})_*A)}(s-r) = {\rm Mor}^{s-r}((T_{-r})_*A, (T_{-s})_*A).
\end{align*}
In other words, $\Sigma(\eta_{r,s})_A$ is a natural transformation of shift $s-r$. In particular, the morphism $\eta^A_r = i_{-r,0}((\eta_{r,0})_A) \in {\rm Mor}^0(\Sigma^rA, A)$ is well-defined for any $r \geq 0$. It is easy to check that $\eta^A_r = (\tau_{-r,0})(A)$. 
\item[(e)] The $r$-acyclic objects in ${\mathcal P}(M)$ are precisely the $r$-torsion elements in $\T(M)$. Indeed, by definition, an object $A$ in ${\mathcal P}(M)$ is $r$-acyclic if and only if $\eta^A_{r} = \tau_{-r,0}(A): (T_{-r})_*A \to A$ is the zero morphism, which coincides with the definition of an $r$-torsion element under the isomorphism (\ref{tam-equiv}). 
\item[(f)] Recall that for each $r$, ${\rm Mor}^r(-, X) = {\rm Hom}_{\T(M)}(-, (T_r)_*X)$. This is an exact functor due to (\ref{tam-hom-set2}) on ${\mathcal P}(M)_0 = \T(M)$.  Similarly, ${\rm Mor}^r(X, -)$ is also an exact functor on ${\mathcal P}(M)_0 = \T(M)$.
\end{itemize}

\begin{lemma} \label{lemma-tor} For any $r \geq 0$ and any object $A$ in ${\mathcal P}(M)$, the morphism $\eta^A_r: (T_{-r})_*A \to A$ embeds into the following exact triangle 
\begin{equation} \label{ext-tor}
(T_{-r})_*A \xrightarrow{\eta^A_r} A \to K \to A 
\end{equation}
in $\T(M) = {\mathcal P}(M)_0$, where $K$ is $r$-acyclic. \end{lemma}

\begin{proof} Since $\T(M)$ is a triangulated category, the morphism $\eta^A_r$ embeds into an exact triangle as (\ref{ext-tor}). By item (e) above, we need to show that $K$ is an $r$-torsion element. By (ii) in Lemma 5.3 in \cite{GS14} which provides a criterion to test an object in an exact triangle to be a torsion element, it suffices to verify that the following diagram is commutative, 
\[ \xymatrixcolsep{4pc} \xymatrix{
(T_{-r})_*A \ar[r]^-{\eta^A_r} \ar[d]_-{\tau_{0,r}((T_{-r})_*A)} & A \ar[d]^-{\tau_{0,r}(A)} \ar[ld]_-{\alpha} \\
A \ar[r]_-{T_r(\eta^A_r)} & (T_r)_*A} \]
for some morphism $\alpha$. Indeed, this is commutative by choosing $\alpha = \mathds{1}_A$ together with the functorial properties of $T_{-r}$. 
\end{proof}

\begin{remark} By the definition of an $r$-isomorphism defined earlier, Lemma \ref{lemma-tor} implies that the morphism $\eta^A_r \in {\rm Mor}^0((T_r)_*A, A)$ is an $r$-isomorphism. On the other hand, Section 3.10 in \cite{Zha20} defines an interleaving relation between two objects in $\T(M)$, which is similar to $r$-isomorphism defined in the sense that $A$ and $(T_r)_*A$ are $r$-interleaved. \end{remark}

\begin{ex} Let $M = \{{\rm pt}\}$ and consider $\T(\{\rm pt\})$. For $A = \k_{[0, \infty)}$, we know that $(T_{-r})_*A = \k_{[-r, \infty)}$ for any $r \geq 0$. Then we have an exact triangle in $\T(\{\rm pt\})$,
\[ \k_{[-r, \infty)} \xrightarrow{\tau_{-r,0}(A)} \k_{[0, \infty)} \to \k_{[-r,0)}[1] \to \k_{[-r, \infty)}[1] \]
where as we have seen that $\k_{[-r,0)}[1]$ is an $r$-torsion element (so $r$-acyclic). Here, by definition, $\tau_{-r,0}(A)$ is the restriction map from $\k_{[-r, \infty)}$ to $\k_{[0, \infty)}$, and the exact triangle is from (2.6.33) in \cite{KS90}. \end{ex}

\jznote{Assembling the properties at the points (a) and (d) in the list
above together with Lemma \ref{lem-tamarkin-pc} and Lemma \ref{lemma-tor}, we deduce the consequence of
main interest to us.}

\begin{cor} \label{thm-tam-tpc} The category ${\mathcal P}(M)$, as
  defined above, is a triangulated persistence category.
\end{cor}

% !TEX root = TPC.tex

\subsection{Symplectic topology} \label{subsec-symp}

The constructions discussed earlier in this paper are inspired by
symplectic rigidity phenomena, in particular the work on Lagrangian
cobordism in \cite{Bi-Co:cob1, Bi-Co:fuk-cob}
and~\cite{Bi-Co-Sh:LagrSh}.  We discuss here one possible
implementation of the TPC machinery in the symplectic setting.

We will focus on a symplectic manifold $(M,\omega)$ that is exact with
$\omega=d\lambda$. We will work with pairs $L=(\bar{L}, f_{L})$ such
that $\bar{L}$ is an exact Lagrangian submanifold $\bar{L} \subset M$,
and $f_{L}:\bar{L}\to \R$ is a primitive of $\lambda|_{\bar{L}}$,
i.e.~$df_{L}=\lambda|_{\bar{L}}$.  We will denote by
$\mathcal{L}ag^{\ast}(M)$ a class of such Lagrangian submanifolds of
$M$ with fixed primitives. Unless otherwise specified, the Lagrangians
are assumed embedded.  We write \jznote{$L\equiv_{\mathcal L} L'$} if the underlying
Lagrangians of $L$ and $L'$ coincide and only the respective
primitives \pbnote{possibly} differ. The set $\bar{\mathcal{L}}$ is
the family of exact Lagrangians in $\mathcal{L}ag^{\ast}(M)$ without
reference to the choice of primitive. In other words it is the
quotient of $\mathcal{L}ag^{\ast}(M)$ by the equivalence relation
\jznote{$\equiv_{\mathcal L}$}.  We will only outline the structures relevant in this
setting and, to simplify matters, we will work \pbnote{with Floer
  theory with $\Z_{2}$-coefficients}. Moreover, the various
homological invariants that will appear below will not be graded.  For
instance, a chain complex is simply a $\Z_{2}$-vector space $V$
together with a linear map $\partial:V\to V$ with $\partial^{2}=0$.

\subsubsection{Fukaya categories}
All our constructions and notation will follow
closely~\cite[\S3]{Bi-Co-Sh:LagrSh}, itself following Seidel's
book~\cite{Seidel}. \pbnote{(But unlike~\cite{Bi-Co-Sh:LagrSh} here we
  work over $\Z_{2}$ and not over the Novikov ring. This is possible
  due to our exactness assumptions.)} In our context, the relevant
Fukaya category $\fuk^{\ast}(M)$ is an $A_{\infty}$-category
associated to the class $\mathcal{L}ag^{\ast}(M)$ in roughly the
following way.  The objects of $\fuk^{\ast}(M)$ are the Lagrangians in
$\mathcal{L}ag^{\ast}(M)$.  The morphisms
\pbnote{$\Hom_{\fuk^{\ast}(M)}(L,L')$} are defined by fixing some
additional so-called Floer data $\mathcal{D}_{L,L'}$.  This is a
couple $\mathcal{D}_{L,L'}=(H_{L,L'}, J_{L,L'})$ where
$H_{L,L'}:[0,1]\times M\to \R$ is a Hamiltonian and
\pbnote{$J_{L,L'} = \{J^t_{L,L'}\}_{t \in [0.1]}$ is a time-dependent}
almost complex structure on $M$ compatible with $\omega$.  Assuming
$\mathcal{D}_{L,L'}$ generic, one may define the Floer chain complex
$CF(L,L';\mathcal{D}_{L,L'})$. \pbnote{As a $\Z_{2}$-vector space this
  chain complex is generated by the Hamiltonian orbits of $H_{L,L'}$
  going from $L$ to $L'$, namely orbits $x:[0,1]\to M$ of the
  time-dependent Hamiltonian vector field $X_t^{H_{L,L'}}$, associated
  to $H_{L,L'}$, with $x(0) \in L$, $x(1) \in L'$. The differential
  $\partial$ of this complex counts (non-stationary) Floer
  trajectories connecting pairs of Hamiltonian orbits, in the sense
  that}
\begin{equation} \label{eq:floer-diff}
  \langle \partial x,y \rangle =
  \#_{2} \mathcal{M}^*(x,y : \mathcal{D}_{L,L'}),
\end{equation}
\pbnote{where the space
  $\mathcal{M}^*(x,y : \mathcal{D}_{L,L'}) = \mathcal{M}(x,y :
  \mathcal{D}_{L,L'}) /\mathbb{R}$ consists of solutions of the Floer
  equation going from $x$ to $y$, up to translation. Specifically, let
  $\mathcal{M}(x,y : \mathcal{D}_{L,L'})$ be the space of
  non-stationary (in case $x=y$) solutions $u:\R\times [0,1]\to M$ of
  the Floer equation} \pbnote{
  $$\frac{\partial u}{\partial s} + J^t_{L,L'} \frac{\partial u}{\partial t}
  - J^t_{L,L'} X_t^{H_{L,L'}}(u)=0,$$} \pbnote{with
  $u(\R\times \{0\})\subset L$, $u(\R\times\{1\})\subset L'$ and
  $\lim_{s\to -\infty}u(s,t)=x(t)$, $\lim_{s\to +\infty}u(s,t)=y(t)$.
  The relevant space for~\eqref{eq:floer-diff} is
  $\mathcal{M}^{\ast}(x,y) = \mathcal{M}(x,y) / \mathbb{R}$, where
  $\mathbb{R}$ acts on $u \in \mathcal{M}^{\ast}(x,y)$ by translation
  along the $s$-variable. Of course, only the $0$-dimensional
  components of $\mathcal{M}^{\ast}(x,y)$ are counted
  in~\eqref{eq:floer-diff} and the counting itself is done modulo
  $2$.}

In case $L$ and $L'$ intersect transversely, then $H_{L,L'}$ may be
taken to be $0$ and the generators of $CF(L,L';\mathcal{D}_{L,L'})$
are \pbnote{the constant orbits at the} intersection points
$L\cap L'$.  We put:
\pbnote{$$\Hom_{\fuk^{\ast}(M)}(L,L')=CF(L,L';\mathcal{D}_{L,L'})$$}
and the first $A_{\infty}$-operation $\mu_{1}$ is the differential of
this complex.  The higher operations $\mu_{n}$ (in homological
notation) are :
$$\mu_{n}:CF(L_{1},L_{2})\otimes CF(L_{2},L_{3})\otimes\ldots
\otimes CF(L_{n},L_{n+1})\to CF(L_{1},L_{n+1})$$ and they are defined
by making use of perturbation data
$\mathcal{D}_{L_{1},\ldots, L_{n+1}}$ that is used to define perturbed
$J$-holomorphic polygons
$$u: D^{2}\backslash\{x_{0}, x_{1},\ldots, x_{n}\}\to M$$
with the $x_{i}$'s in \pbnote{clockwise} order around the circle, the
(open) arc $C_{i}$, $0< i\leq n$, of ends $x_{i-1}x_{i}$, has the
property that $u(C_{i})\subset L_{i}$, the arc $C_{n+1}$ from $x_{n}$
to $x_{0}$ satisfies $u(C_{n+1})\subset L_{n+1}$; the puncture
$x_{i}$, $i\geq 1$, is sent asymptotically to a generator
$\in CF(L_{i},L_{i+1})$, and $x_{0}$ is sent to a generator of
$CF(L_{1},L_{n+1})$.  Moreover, $u$ satisfies an equation of perturbed
Cauchy-Riemann type, with the perturbation prescribed by
$\mathcal{D}_{L_{1},\ldots, L_{n+1}}$.  There is also a system of
strip-like ends around each puncture $x_{i}$ and on this region the
perturbation data coincides with the Floer data
$\mathcal{D}_{L_{i},L_{i+1}}$ when $i\geq 1$ and, respectively, with
$\mathcal{D}_{L_{0},L_{n+1}}$ for $x_{0}$.  While somewhat complex,
this construction is standard today. There are choices of
perturbations such that, when only considering moduli spaces of
dimension $0$, the operations $\mu_{n}$ are well-defined and they
satisfy the $A_{\infty}$-relations. Obviously, the resulting
$A_{\infty}$-category, $\fuk^{\ast}(M)$, depends on the choice of
perturbation data but any two choices yield quasi-equivalent
categories.
 
There is a category of (left) $A_{\infty}$-modules,
$mod(\fuk^{\ast}(M))$, over $\fuk^{\ast}(M)$. This is a dg-category
and is pre-triangulated (in the sense of the current paper) or,
equivalently, $A_{\infty}$-triangulated in the sense of Seidel's
book~\cite{Seidel}. There is a so-called Yoneda embedding
$$\mathcal{Y}:\fuk^{\ast}(M)\to mod(\fuk^{\ast}(M))$$ which 
associates to each object the Yoneda module of $L$, \jznote{denoted by $\mathcal{Y}_{L}$},
such that $\mathcal{Y}_{L}(L')=CF(L',L;\mathcal{D}_{L',L})$.  We
consider the triangulated closure
$\mathrm{Im}(\mathcal{Y})^{\nabla}\subset mod(\fuk^{\ast}(M))$ of the
image of $\fuk^{\ast}(M)$ through the Yoneda functor inside
$mod(\fuk^{\ast}(M))$.  Finally, the derived Fukaya category,
$D\fuk^{\ast}(M)$, is the homological category
$D\fuk^{\ast}(M)= H(\mathrm{Im}(\mathcal{Y})^{\nabla})$. As a result
of this construction, $D\fuk^{\ast}(M)$ is triangulated. Its objects
are iterated cones of Lagrangians from the class
$\mathcal{L}ag^{\ast}(M)$.
  
\subsubsection{Action filtrations and weakly filtered
  $A_{\infty}$-structures} \label{subsubsec:act-filtr} A key property
of the setting outlined above is that the complexes
$CF(L,L';\mathcal{D}_{L,L'})$ admit a filtration function, generally
called an action functional, defined on a generator
$x\in CF(L,L';\mathcal{D}_{L,L'})$, by
\pbnote{$$\mathcal{A}(x)=f_{L'}(x(1))-f_{L}(x(0))+\int_{0}^{1}H_{L,L'}(t,x(t))dt
  - \int_0^1 \lambda(\dot{x}(t))dt,$$} where
$\mathcal{D}_{L,L'}=(H_{L,L'}, J_{L,L'})$ and $x:[0,1]\to M$ is an
orbit of the Hamiltonian flow of $H_{L,L'}$ with $x(0)\in L$,
$x(1)\in L'$ and $f_{L}$, $f_{L'}$ are the fixed primitives associated
respectively to $L$ and to $L'$. This expression is particularly
simple if $L$ and $L'$ are intersecting transversely,
$H_{L,L'}\equiv 0$ and $x\in L \cap L'$. In this case,
$\mathcal{A}(x)=f_{L'}(x)-f_{L}(x)$.
  
It is a central property of Floer theory that the Floer differential
\pbnote{decreases} the action.  On the other hand, things are more
delicate concerning the higher $A_{\infty}$-operations. It is shown
in~\cite{Bi-Co-Sh:LagrSh} that there is a notion of {\em weakly
  filtered} $A_{\infty}$-categories and modules and that
$\fuk^{\ast}(M)$ together with $mod(\fuk^{\ast}(M))$ can be
constructed (with some care in the choice of the perturbation data)
\pbnote{in such a way that they become weakly filtered}. We will
denote by $\widehat{\fuk}^{\ast}(M)$ the weakly filtered Fukaya
category constructed in \cite{Bi-Co-Sh:LagrSh}. The weakly-filtered
modules over this category will be denoted by
$mod(\widehat{\fuk}^{\ast}(M))$.  There is an obvious forgetful
functor $\Theta : \widehat{\fuk}^{\ast}(M) \to \fuk^{\ast}(M)$ and
similarly for the modules. The weakly filtered property differs from
the stronger, {\em filtered}, property by certain discrepancies. For
instance, the morphisms spaces $CF(L, L')$ are filtered with a
filtration denoted by $C^{\leq r}(L, L')$ but the multiplication
$\mu_{2}:C(L_{1},L_{2})\otimes C(L_{2}, L_{3})\to CF(L_{1}, L_{3})$
does not exactly respects the filtration - as required by condition
(ii) in Definition \ref{dfn-fdg-cat} - rather it might increase the
filtration by a discrepancy $\epsilon_{2}$ in the sense that it takes
$CF^{\leq r}(L_{1},L_{2})\otimes CF^{\leq s}(L_{2},L_{3})$ to
$CF^{\leq r+s+\epsilon_{2}}(L_{1},L_{3})$.  There are similar
discrepancies $\epsilon_{n}$ appearing for all the higher operations
$\mu_{n}$.  These discrepancies depend on all the choices of
perturbation data, see \S 2.1 in \cite{Bi-Co-Sh:LagrSh}.  Technically,
the $\epsilon_{n}$ appear in the construction because of the
Hamiltonian perturbation data, inside the polygons that define the
operations $\mu_{n}$, for $n\geq 2$. Obviously, the various algebraic
structures are filtered if all the discrepancies vanish.

The category of weakly filtered modules is also pre-triangulated in an
appropriate sense and we could develop all the machinery in \S
\ref{subsec:dg} in this setting. Passing to the homological category,
this leads to a setting a bit more general than that of TPC's but yet
again weighted triangles and the associated fragmentation
pseudo-metrics make perfect sense.  However, we prefer to avoid these
complications here and focus on a situation that is more restrictive
geometrically \pbnote{but fits precisely with the algebraic framework
  developed in the rest of the paper. The considerations below are not
  fully rigorous and should be taken as an outline of a construction
  that is expected to work.}

\

To this aim, we consider a family of Lagrangians
$\mathcal{X}\subset \mathcal{L}ag^{\ast}(M)$ with the property that
any two objects in $\mathcal{X}$ are either transverse or they
coincide geometrically (hence only \pbnote{possibly} differ by the
choice of primitive). We also require that $\mathcal{X}$ be closed
with respect to changes in primitives. In other words, if
$L\in \mathcal{X}$ and $L=(\bar{L},f_{L})$, then
$(\bar{L},f_{L}+r)\in \mathcal{X}$ \pbnote{for all
  $r \in \mathbb{R}$}.

We consider the full $A_{\infty}$-subcategory
$\widehat{\fuk}^{\ast}(M;\mathcal{X})\subset \widehat{\fuk}^{\ast}(M)$
that has as objects the elements in $\mathcal{X}$. A key, non-trivial
remark is that it is possible to pick the perturbation data
$\mathcal{D}$ so that its restriction $\mathcal{D}_{\mathcal{X}}$ to
the elements in $\mathcal{X}$ makes
$\widehat{\fuk}^{\ast}(M;\mathcal{X})$ a filtered
$A_{\infty}$-category and not only a weakly filtered one. To describe
this construction, recall that our Lagrangians $L$ are in fact pairs
consisting of an exact Lagrangian submanifold $\bar{L}$ of $M$
together with a primitive $f_{L}:\bar{L}\to \R$.  Recall that we write
\jznote{$L\equiv_{\mathcal L} L'$} if the two Lagrangians coincide geometrically and only
\pbnote{possibly} differ by their primitives. Denote by
$\bar{\mathcal{X}}$ the set of equivalence classes
\jznote{$\mathcal{X}/\equiv_{\mathcal L}$}, \pbnote{i.e.~the set of underlying Lagrangian
  submanifolds in $\mathcal{X}$}. With this convention, the
construction uses the ``cluster'' or ``pearly'' model \cite{Co-La} to
define the operations $\mu_{n}$ for sequences of Lagrangians
$L_{1}, L_{2},\ldots, L_{n+1},\subset \mathcal{X}$ such that a
sub-family $L_{i},\ldots, L_{i+k}$ has the property that
\jznote{$L_{i}\equiv_{\mathcal L} L_{i+1} \equiv_{\mathcal L} \ldots \equiv_{\mathcal L} L_{i+k}$}. The
starting point in this model is that if \jznote{$L\equiv_{\mathcal L} L'$}, then $CF(L,L')$
is the Morse complex of a Morse function $h:\bar{L}\to \R$ and the
differential only counts Morse trajectories. In other words, the
``Floer'' data $\mathcal{D}_{L,L}$ consists of a Morse-Smale pair
$(h,(\cdot, \cdot))$ on $\bar{L}$ (where $(\cdot, \cdot)$ is a
Riemannian metric). In terms of filtrations, all critical points of
$h$ appear in filtration $f_{L'}(x)-f_{L}(x)$ for some $x\in L$
(\pbnote{this difference does not depend $x$ since $f_{L'}$ and
  $f_{L}$ differ by a constant}).  To define the operations $\mu_{n}$
for $n\geq 1$, one proceeds again by picking coherent systems of
perturbations $\mathcal{D}_{L_{1},L_{2},\ldots, L_{n+1}}$ following
closely Seidel's scheme with the additional constraint that the
restriction of the perturbation system to any family of objects that
are all coinciding geometrically $(L_{i},L_{i+1},\ldots L_{i+k})$
\pbnote{consists of trees together with a families
  $(h_{z}, (\cdot, \cdot)_{z})$ of Morse functions and metrics on that
  Lagrangian, parametrized by a point $z$ running along each edge of
  the tree. The ends of the tree are associated to the pair
  $(h,(\cdot, \cdot))$.}
%
% counts only trees of Morse trajectories with segment-like ends
% (instead of strip-like) associated to the pair $(h,(\cdot, \cdot))$
% and satisfying the gradient equation with respect to perturbed data
% $(h_{z}, (\cdot, \cdot)_{z})$, depending on the point $z$ that moves
% long the edges of the tree.
For the pairs $(L, L')$, $L,L'\in \mathcal{X}$ such that $L$
intersects transversely $L'$ the Floer data has vanishing Hamiltonian
$H_{L,L'}=0$. Finally, the perturbation data inside perturbed
\pbnote{pseudo-holomorphic} polygons is very
small. \pbnote{Geometrically, things would have become much simpler
  without these perturbations, however they are needed in order to
  obtain transversality for the spaces of pseudo-holomorphic
  polygons.}

\pbnote{The $A_{\infty}$-operations are now defined by counting
  ``generalized'' polygons, which are a hybrid consisting of perturbed
  pseudo-holomorphic polygons together with gradient trajectories
  modeled on trees as above, attached to the polygons. (The
  trajectories on each edge of a tree solve the negative gradient
  equation with respect to $(h_{z}, (\cdot, \cdot)_{z})$ as $z$ moves
  along the respective edge.) This type of \pbnote{``hybrid''}
  construction obviously poses non-trivial technical challenges,
  especially on the analytic side of the story, which we are not
  addressing here. At the same time, variants of this construction
  have appeared by now in the literature~\cite{Co-La, Charest,
    Sheridan} so we will not further expand on it here.}

\pbnote{The main advantage of the above construction is that it leads
  to {\em filtered} $A_{\infty}$-category in the sense that the
  $\mu_n$-operations do not increase action. This is so because the
  number of geometric objects in $\bar{\mathcal{X}}$ is finite, hence
  for each given $n$, the number of possible input configurations
  required to define $\mu_n$ is finite.} \pbnote{An energy-action
  estimate (involving curvature terms coming from the perturbations)
  shows that by taking the perturbations small enough the
  $\mu_n$-operation does not increase action. One can also check that
  the compatibility of the perturbation data with respect to splitting
  of polygons can be established in the present setting without
  interfering with the $\mu_n$'s being action preserving. This can be
  done, as usual, by defining the perturbation data inductively in
  $n$, and adjusting them to be small enough at each stage.}

In summary, with the choices of perturbation data as described above,
the $A_{\infty}$-category $\widehat{\fuk}^{\ast}(M;\mathcal{X})$ is in
fact filtered. Moreover, the units in $CF(L,L)$ are strict \pbnote{in
  the $A_{\infty}$-sense (see~\cite[Section~2]{Seidel} for the
  definition)}. Notice also that the perturbation data picked here
only depends on $\bar{\mathcal{X}}$ and that it can be taken to be as
small as desired in $C^{2}$ norm in the sense that the Hamiltonian
part can be as close to $0$ in $C^{2}$ norm as desired. We will
indicate making $\mathcal{D}$ small in this sense by
$||\mathcal{D}||\to 0$ (in this case the Hamiltonian part of the
perturbation data tends to $0$ in the sense before, but the choices of
almost complex structures are allowed to vary).

We denote by
\pbnote{$mod(M;\mathcal{X})=mod(\widehat{\fuk}^{\ast}(M;
  \mathcal{X}))$ the category of {\em filtered} $A_{\infty}$-modules
  over $\widehat{\fuk}^{\ast}(M; \mathcal{X})$. Note that both
  $\widehat{\fuk}^{\ast}(M; \mathcal{X})$ and $mod(M;\mathcal{X})$
  depend on the choice of perturbation data $\mathcal{D}$, but to
  shorten notation we will omit the reference to $\mathcal{D}$
  whenever no confusion might occur (in case we do want to emphasize
  this dependence we will write $mod_{\mathcal{D}}(M;\mathcal{X})$
  etc.).}

The category $mod(M;\mathcal{X})$ is a filtered dg-category in the
sense of Definition \ref{dfn-fdg-cat}. We now remark that this
category is pre-triangulated in the sense of
Corollary~\ref{cor:pre-tr}. This is essentially immediate because the
category $mod(M;\mathcal{X})$ is closed under taking cones over maps
of any ``shift'' (the non-filtered situation is in Seidel's book and
the weakly filtered one is described in~\cite{Bi-Co-Sh:LagrSh}) and is
endowed with a natural shift functor
$\Sigma: (\R,+) \to \mathrm{End}(mod(M;\mathcal{X}))$ \pbnote{which we
  describe next. Given $r \in \mathbb{R}$ and a module
  $\mathcal{M} \in mod(M;\mathcal{X})$ we define the filtered module
  $\Sigma^r \mathcal{M}$ by
  $(\Sigma^r \mathcal{M})^{\leq \alpha}(N) = \mathcal{M}^{\alpha -
    r}(N)$, endowed with the same $\mu_n$-operations as $\mathcal{M}$.
  This definition extends in an obvious way to morphisms in
  $mod(M;\mathcal{X})$, i.e.~to filtered pre-module homomorphisms
  between filtered modules in the sense of~\cite{Bi-Co-Sh:LagrSh}.
  The definition of the natural transformations $\eta_{r,s}$ between
  $\Sigma^r$ and $\Sigma^s$ is done in the obvious way too.}

It follows from Corollary \ref{cor:pre-tr} that the homological
category
$$\mathcal{H}_{\mathcal{D}}(M;\mathcal{X}):=H[mod_{\mathcal{D}}(M;\mathcal{X})]$$
is a TPC. In particular, its $\infty$-level,
$\mathcal{H}_{\mathcal{D}}(M;\mathcal{X})_{\infty}$, carries the
persistence fragmentation pseudo-metrics $\bar{d}^{\mathcal{F}}(-,-)$
as described in \S\ref{subsec:trstr-weights} (see particularly
\S\ref{subsec:rem-nondeg}). We will again omit $\mathcal{D}$ from the
notation when there is no risk for confusion.

\begin{rem}\label{rem:category-x}
  (a) In some cases it is more efficient to use instead of
  $\mathcal{H}(M;\mathcal{X})$ a smaller TPC.  For instance, if
  $\mathcal{X}^{\nabla}$ is the smallest pre-triangulated filtered
  dg-category that contains the family $\mathcal{X}$ and is included
  in $mod(M;\mathcal{X})$, then its homological category,
  $H(\mathcal{X}^{\nabla})$, is again a TPC.  Its infinity level
  $[H(\mathcal{X}^{\nabla})]_{\infty}$ is an obvious sub-category of
  $D\fuk^{\ast}(M)$.

  (b) If $\mathcal{X}$ contains a family of triangular generators for
  $D\fuk^{\ast}(M)$, then $[H(\mathcal{X}^{\nabla})]_{\infty}$ is
  quasi-equivalent to $D\fuk^{\ast}(M)$.

  (c) \pbnote{Here is an alternative way to define shift functors on
    categories of modules over
    $\widehat{\fuk}^{\ast}(M; \mathcal{X})$. This construction is more
    complicated than the one described earlier and seems to work only
    for some categories of modules. However, it might be useful in
    some other algebraic settings than the one discussed here.}

  \pbnote{First define a shift functor on
    $\widehat{\fuk}^{\ast}(M; \mathcal{X})$ as follows. For
    $r \in \mathbb{R}$, define an $A_{\infty}$-functor
    $\Sigma_{\mathcal{X}}^r$ on
    $\widehat{\fuk}^{\ast}(M; \mathcal{X})$ by
    $$\Sigma_{\mathcal{X}}^{r}L=\Sigma_{\mathcal{X}}^{r}(\bar{L}, f_{L})
    := (\bar{L}, f_{L}+r).$$ It is easy to see that
    $\Sigma_{\mathcal{X}}^r$ extends to a filtered
    $A_{\infty}$-functor on $\widehat{\fuk}^{\ast}(M; \mathcal{X})$
    (with trivial terms of order $\geq 2$). Note that we have
    $CF^{\leq \alpha}(\Sigma_{\mathcal{X}}^r L, \Sigma_{\mathcal{X}}^s
    L') = CF^{\leq \alpha+r-s}(L, L')$.} \pbnote{Turning to modules,
    define $\Sigma^r$ on $mod(M;\mathcal{X})$ using the pullback by
    $\Sigma^{-r}_{\mathcal{X}}$
    $$\Sigma^r \mathcal{M} := (\Sigma_{\mathcal{X}}^{-r})^*\mathcal{M}.$$
    Note that we have
    \jznote{$\Sigma^{r}\mathcal{M}(L)=\mathcal{M}(\Sigma_{\mathcal{X}}^{-r}L)$}.}
  \pbnote{The difficulty with this definition has to do with the
    natural (quasi)-isomorphisms $\eta_{r,s}$ between $\Sigma^r$ and
    $\Sigma^s$. More specifically, we need a quasi-isomorphism
    $\eta_{r,s}: \Sigma^r \mathcal{M} \to \Sigma^s \mathcal{M}$ (that
    shift filtrations by a uniform bound and also become isomorphisms
    in the respective persistence homology of the filtered chain
    complex $\Hom(\Sigma^r \mathcal{M}, \Sigma^s \mathcal{M})$). For
    general modules $\mathcal{M}$ it is not clear how to define such
    maps. However, under further assumptions on $\mathcal{M}$ this
    seems possible. For example, if we restrict to the subcategory of
    strictly unital modules, one can try to construct $\eta_{r,s}$ in
    the following way. Let $L \in \mathcal{X}$ and denote by
    $e_{L; r,s} \in CF(\Sigma_{\mathcal{X}}^{-s}L,
    \Sigma_{\mathcal{X}}^{-r}L)$ the element corresponding to the
    (strict) unit $e_{L} \in CF(L,L)$ under the obvious isomorphism
    between the last two Floer chain complexes. The
    $\mu^{\mathcal{M}}_2$-operation gives rise to a chain map
    $$\mu_2^{\mathcal{M}}(e_{L; r,s}, - ): \Sigma^{r}\mathcal{M}(L)
    \longrightarrow \Sigma^{s}\mathcal{M}(L),$$ which is in fact a
    chain isomorphism due to the strict unitality assumptions. It is
    not hard to check that the collection of these isomorphisms (for
    varying $L$'s) extends to an isomorphism of $A_{\infty}$-modules
    $\Sigma^r \mathcal{M} \longrightarrow \Sigma^s \mathcal{M}$.}

  \pbnote{Note that the Yoneda modules inside $mod(M;\mathcal{X})$,
    and more generally iterated cones of such, are strictly unital
    (recall that in our setting
    $\widehat{\fuk}^{\ast}(M; \mathcal{X})$ is strictly unital), hence
    the above construction works for the subcategories of such
    modules. Moreover, one can easily see that in these cases the
    resulting shift functor coincides with the one given before
    Remark~\ref{rem:category-x}.}
    
    \pbnote{Things become more complicated if one tries to weaken the
      assumption of strict unitality of the modules, to homological
      unitality for instance. \Qed}
\end{rem}

\subsubsection{Filtered perturbations of Yoneda
  modules.}\label{subsubec:filt-mod}
From now on, we will assume that the space 
\jznote{\begin{equation}\label{eq:finite}
  \bar{\mathcal{X}}=\mathcal{X}/\equiv_{\mathcal L} \ \mathrm{is\ finite} ~.~
\end{equation}}

Fix $N=(\bar{N}, f_{N})\in \mathcal{L}ag^{\ast}(M)$,
$\bar{N}\not\in \bar{\mathcal{X}}$ and let
$\mathcal{Y}_{N}\in mod(\widehat{\fuk}^{\ast}(M))$ be its
weakly-filtered Yoneda module. We consider the pull-back module
$ \mathcal{Y}'_{N}=j^{\ast}(\mathcal{Y}_{N})$ where $j$ is the
inclusion:
$$\widehat{\fuk}^{\ast}(M;\mathcal{X})\stackrel{j}{\hookrightarrow}
\widehat{\fuk}^{\ast}(M)~.~$$ The module $\mathcal{Y}'_{N}$ is
generally only weakly filtered. However, we will now notice that it is
possible to select the perturbation data $\mathcal{D}$ in the
definition of $\widehat{\fuk}^{\ast}(M)$ and a Hamiltonian
perturbation $N_{\epsilon}$ of $N$ such that the module
$\mathcal{Y}'_{N_{\epsilon}}$, which is quasi-isomorphic to
$\mathcal{Y}'_{N}$, is filtered and thus belongs to
$mod(M,\mathcal{X})$.

We proceed as follows. For the (geometric) Lagrangian
$\bar{N}\in \bar{\mathcal{L}}$ we select a small Hamiltonian
perturbation $\bar{N}_{\epsilon}$ which is the image of the time-one
map $\phi^{G}$ of the Hamiltonian flow $\phi_{t}^{G}$ induced by a
small autonomous Hamiltonian $G=G_{\bar{N},\epsilon} : M\to \R$,
associated to a small Morse function $g_{\bar{N},\epsilon}$, on
$\bar{N}$ extended in an obvious way first to a disk-cotangent bundle
of $\bar{N}$ and then to $M$, and such that $\bar{N}_{\epsilon}$
intersects transversely each element of $\bar{\mathcal{X}}$ (as well
as $\bar{N}$).

We now define $N_{\epsilon}=(\bar{N}_{\epsilon}, f_{N_{\epsilon}})$
where $f_{N_{\epsilon}}$ is the primitive of $\bar{N}_{\epsilon}$
given as follows. We first consider an intersection point
$x_{0}\in \bar{N}\cap \bar{N}_{\epsilon}$ corresponding to a critical
point of $g_{\bar{N},\epsilon}$ and define for
$x\in \bar{N}_{\epsilon}$,
$$f_{N_{\epsilon}}(x)= f_{N}((\phi^{G})^{-1}(x)) - G(x_{0})  +
\int_{0}^{1} G(\phi_{-t}^{G}(x))~dt -
\int_{\gamma^{G}_{-}(x)}\lambda $$ where $\gamma^{G}_{-}$ is the
Hamiltonian path $(x,t)\to \phi_{-t}^{G}(x)$. In other words, this is
the primitive of $\bar{N}_{\epsilon}$ with the property that
$f_{N_{\epsilon}}(x_{0})=f_{N}(x_{0})$.

The next step is to revisit the construction of the filtered
$A_{\infty}$-category $\widehat{\fuk}^{\ast}(M;\mathcal{X})$ and
notice that the choice of perturbation data
$\mathcal{D}_{\mathcal{X}}$ can be extended to a choice of
perturbation data for the larger family
\jznote{$\mathcal{X}_{N}=\mathcal{X}\cup \{L \ | L\equiv_{\mathcal L} N_{\epsilon}\}$}. This
is a triviality as the family $\bar{\mathcal{X}_{N}}$ is of the same
type as $\mathcal{X}$. The inductive construction of the perturbation
data, following the sequences of Lagrangians of the type
$L_{1},L_{2},\ldots, L_{n+1}$ with $n$ increasing, shows that it is
possible to extend the union of the choices of data
$\cup _{N\in \mathcal{L}ag^{\ast}(M)\backslash \mathcal{X}}
\mathcal{D}_{\mathcal{X}_{N}}$ to the perturbation data $\mathcal{D}$
defining the weakly filtered category $\widehat{\fuk}^{\ast}(M)$. The
upshot is that now the modules
$\mathcal{Y}'_{N_{\epsilon}}\in mod (M;\mathcal{X})$ are all filtered
and not only weakly filtered.

\begin{rem}
  (a) Notice that, in general the Yoneda module of some
  $N'_{\epsilon'}$ is not a {\em filtered} $A_{\infty}$-module over
  $\widehat{\fuk}^{\ast}(M; \mathcal{X}_{N})$ for \jznote{$N\not\equiv_{\mathcal L} N'$} (it
  is only weakly-filtered). Nonetheless, it is filtered as a module
  over the smaller category $\widehat{\fuk}^{\ast}(M;\mathcal{X})$.
  Another useful observation is that the functions
  $g_{\bar{N},\epsilon}$ can be taken to be as small as desired in
  $C^{2}$-norm. To make this more precise we will assume that each
  $g_{\bar{N},\epsilon}$ has the property that
  $||g_{\bar{N},\epsilon}||_{C^{2}}\leq \epsilon$. For each
  $\epsilon >0$ there are choices of $g_{\bar{N},\epsilon}$ satisfying
  this constraint that can be used in the construction.

  (b) While the perturbation data $\mathcal{D}_{\mathcal{X}}$ is fixed
  in the construction above, the extension $\mathcal{D}$ depends on
  the choices of perturbations $N_{\epsilon}$. At the same time, for
  any $\epsilon$, the corresponding perturbation data $\mathcal{D}$
  can be assumed to be arbitrarily small .
\end{rem}

\subsubsection{Persistence pseudo-distances on
  $\mathcal{L}ag^{\ast}(M)$}
Under the finiteness assumption (\ref{eq:finite}), we will apply the
machinery described before to define fragmentation-type pseudo-metrics
on $\mathcal{L}ag^{\ast}(M)$. We will then see that, by using some
appropriate mixing - as in \S\ref{subsubsec:non-deg-v}, we obtain a
family of non-degenerate metrics on
\jznote{$\bar{\mathcal{L}}=\mathcal{L}ag^{\ast}(M)/\equiv_{\mathcal L}$}.

Recall that the category
$\mathcal{H}_{\mathcal{D}}(M;\mathcal{X})_{\infty}$ is the
$\infty$-level of the TPC,
$\mathcal{H}_{\mathcal{D}}(M;\mathcal{X})=H[mod_{\mathcal{D}}(M;\mathcal{X})]$,
as constructed in \S\ref{subsubsec:act-filtr}. Thus, given a family
$\tilde{\mathcal{F}}\subset
\mathrm{Obj}(mod_{\mathcal{D}}(M;\mathcal{X}))$, the category 
$\mathcal{H}_{\mathcal{D}}(M;\mathcal{X})_{\infty}$ carries a
fragmentation pseudo-metric $\bar{d}^{\tilde{\mathcal{F}}}$ (recall
its definition from \S\ref{subsec:rem-nondeg}).

Fix now a family $\mathcal{F}\subset \mathcal{X}$ and put
$\mathcal{F}_{\mathcal{D}}=\{\mathcal{Y}'_{F} \ | \ F\in
\mathcal{F}\}\subset mod_{\mathcal{D}}(M;\mathcal{X})$.  For
$L,L'\in \mathcal{L}ag^{\ast}(M)$ define:

\begin{equation}\label{eq:metric-Lag}
  \bar{D}^{\mathcal{F}}(L,L')=
  \limsup_{\epsilon\to 0, ||D||\to 0} ~\bar{d}^{\mathcal{F}_{\mathcal{D}}}(\mathcal{Y}'_{L_{\epsilon}}, \mathcal{Y}'_{L'_{\epsilon}})~.~
\end{equation}

\begin{prop} With the notation above, the
  map
  $$\bar{D}^{\mathcal{F}}:\mathcal{L}ag^{\ast}(M)\times
  \mathcal{L}ag^{\ast}(M)\to [0,\infty)\cup\{\infty\}$$ has the
  following two properties:
  \begin{itemize}
  \item[i.] it is a pseudo-metric.
  \item[ii.] if $\mathcal{F}$ generates (as triangulated category) the
    category $D\fuk^{\ast}(M)$, then $\bar{D}^{\mathcal{F}}$ is finite
    (in the sense that $\bar{D}^{\mathcal{F}}(L,L')< \infty$ for any
    $L$, $L'$).
  \end{itemize}
\end{prop}

The proofs of these statements are omitted here. They are not
difficult but not completely trivial (for the second statement, for
instance, the key issue is that units are strict in categories of type
$\widehat{\fuk}^{\ast}(M;\mathcal{X})$).

\begin{rem} As long as $\mathcal{F}_{\mathcal{D}}$ is canonically
  defined for each choice of perturbation data $\mathcal{D}$, the
  definition of $\bar{D}^{\mathcal{F}}$ by the formula above remains
  possible. The requirement $\mathcal{F}\subset \mathcal{X}$ and the
  resulting definition of $\mathcal{F}_{\mathcal{D}}$ is just a simple
  way to ensure this. \end{rem}

Notice that if $D\fuk^{\ast}(M)$ admits a finite number of generators,
one can easily pick a family $\mathcal{X}$ and a family
$\mathcal{F}\subset \mathcal{X}$ satisfying all the required
assumptions to define a finite $\bar{D}^{\mathcal{F}}$.

\

Recall the space $\bar{\mathcal{L}}$ of exact Lagrangians in $M$
without reference to the choice of primitive,
\jznote{$\bar{\mathcal{L}}=\mathcal{L}ag^{\ast}(M)/\equiv_{\mathcal L}$}. We now add one
more assumption on the family $\mathcal{F}$, namely that it is closed
with respect to changes in primitives (in the same way that
$\mathcal{X}$ is): if $F=(\bar{F}, f_{L}) \in \mathcal{F}$, then
$(\bar{F}, f_{L}+r)\in\mathcal{F}$ too. Under this additional
assumption, we notice that the pseudo-metric $\bar{D}^{\mathcal{F}}$
induces a pseudo-metric $\tilde{D}^{\mathcal{F}}$ on
$\bar{\mathcal{L}}$ by the formula:
$$\tilde{D}^{\mathcal{F}}(N,N')=\inf_{\bar{L}=N,\bar{L}'=N'} \bar{D}^{\mathcal{F}}(L,L')~.~$$
This follows immediately from the fact that
$\bar{D}^{\mathcal{F}}(\Sigma^{r}L, \Sigma^{r}L')=
\bar{D}^{\mathcal{F}}(L, L')$, for all $r\in \R$, itself a consequence
of $\mathcal{F}$ being closed to changes of primitive.

\

The next important property of the pseudo-metrics
$\bar{D}^{\mathcal{F}}$ has to do with mixing in the sense of
\S\ref{subsubsec:non-deg-v}. Assume that
$\mathcal{F}'\subset \mathcal{X}$ is another family as
above. Similarly to (\ref{eq:mixing1}), define the mixed fragmentation
pseudo-metric:

\begin{equation}\label{eq:mixing2}
  \bar{D}^{\mathcal{F},\mathcal{F}'}(L,L')=
  \max ~\{ \bar{D}^{\mathcal{F}}(L, L') \ , \ \bar{D}^{\mathcal{F}'}(L, L') \}
\end{equation}
as well as the corresponding mixed pseudo-metric on
$\bar{\mathcal{L}}$,
$$\tilde{D}^{\mathcal{F},\mathcal{F}'}=
\max ~\{ \tilde{D}^{\mathcal{F}} \, \ \tilde{D}^{\mathcal{F}'}\}~.~$$

Notice that both \jznote{$\bar{\mathcal{F}}=\mathcal{F}/\equiv_{\mathcal L}$ and
$\bar{\mathcal{F}}'=\mathcal{F}'/\equiv_{\mathcal L}$} are finite as they are
included in $\bar{X}$ which is finite.

\begin{prop} Assuming the setting above and if, additionally,
  $\bar{\mathcal{F}}\cap \bar{\mathcal{F}}'=\emptyset$, then the
  pseudo-metric $\tilde{D}^{\mathcal{F},\mathcal{F}'}$ on the space of
  exact Lagrangians $\bar{\mathcal{L}}$ in $M$ is non-degenerate.
\end{prop}
\begin{proof}
  This is a reformulation of the main result in \cite{Bi-Co-Sh:LagrSh}
  and we briefly review the main steps.  The proof starts from the
  remark that it is enough to show \pbnote{
    \begin{equation}\label{eq:non-deg}
      \bar{D}^{\mathcal{F}}(L,L')=0 \ \Rightarrow \ \bar{L}
      \subset \bigl( \cup_{\bar{F} \in \bar{\mathcal{F}}} \bar{F} \bigr)
      \cup \bar{L'} ~.~
    \end{equation}
  } To show (\ref{eq:non-deg}) one assumes that
  \pbnote{$\bar{L}\not\subset \bigl(\cup_{\bar{F}\in
      \bar{\mathcal{F}}} \bar{F} \bigr) \cup \bar{L'}$ and thus there
    is a symplectic embedding $e: B^{2n}(r_{0}) \to (M,\omega)$ of the
    standard $2n$-dimensional ball ($2n = \dim_{\mathbb{R}}M$) of
    radius $r_0$} such that $e^{-1}(\bar{L})=\R^{n}\cap \bar{L}$ and
  \pbnote{$\mathrm{Im}(e)\bigcap \Bigl( \bigl( \cup_{\bar{F}\in
      \bar{\mathcal{F}}} \bar{F} \bigr) \cup \bar{L'}
    \Bigr)=\emptyset$.}

  We fix an embedding $e$ as above and proceed on the algebra
  side. The identity $\bar{D}^{\mathcal{F}}(L,L')=0$ means that for
  any $\delta>0$ there exists a sufficiently small $\epsilon >0$ such
  that the \pbnote{pseudo-metric} $\bar{\delta}^{\mathcal{F}}$
  (see~\S\ref{subsec:rem-nondeg} for its definition) that is defined
  on the objects of $\mathcal{H}_{\mathcal{D}}(M;\mathcal{X})$ has the
  property:
  \begin{equation}\label{eq:non-deg2}
    \bar{\delta}^{\mathcal{F}}(\mathcal{Y}'_{L_{\epsilon}},
    \mathcal{Y}'_{L'_{\epsilon}}) < \delta
  \end{equation}
  for any perturbation data $\mathcal{D}$ with
  $||\mathcal{D}||\leq \epsilon$. It is an easy algebraic exercise in
  manipulating strict exact triangles to show that inequality
  (\ref{eq:non-deg2}) implies that there are iterated exact triangles
  in $mod_{\mathcal{D}}(M,\mathcal{X})$ of the form:
  \begin{equation}\label{eq:iterated-tr3}\xymatrixcolsep{1pc}
    \xymatrix{
      Y_{0} \ar[rr] &  &  Y_{1}\ar@{-->}[ldd]  \ar[r]
      &\ldots  \ar[r]& Y_{i} \ar[rr] &  &  Y_{i+1}\ar@{-->}[ldd]  \ar[r]
      &\ldots \ar[r]&Y_{n-1} \ar[rr] &   &Y_{n} \ar@{-->}[ldd]  &\\
      &         \Delta_{1}                  &  & & &  \Delta_{i+1}                          & &  &  &    \Delta_{n}             \\
      & X_{1}\ar[luu] &  & & &X_{i+1}\ar[luu] &  &  & &X_{n}\ar[luu] }
  \end{equation}
  with the following properties:
  \begin{itemize}
  \item[i.] all $\Delta_{i}$ are exact of strict weight $0$.
  \item[ii.] $Y_{0}=0$ and there is an index $j\leq n$ and
    $r_{j}\in \R$ such that $X_{i}\in \mathcal{F}$ for all $i\not=j$,
    $X_{j}=\Sigma^{r_{j}}\mathcal{Y}'_{L'_{\epsilon}}$.
  \item[iii.] there exists an $r$-isomorphism of filtered modules
    $\varphi_{\epsilon,\mathcal{D}} : \mathcal{Y}_{L_{\epsilon}}\to
    Y_{n}$ with $r\leq 2\delta$.
  \end{itemize}
  The reason that the bound at iii. above is $2\delta$ and not
  $\delta$ is that in constructing this sequence of $0$-weight
  triangles from the decomposition employing triangles of higher
  weight one needs to invert $k$-isomorphisms and the left (or right)
  inverse of such an $k$-isomorphism is only a $2k$-isomorphism in
  general.  For the next step, we consider a left inverse of
  $\psi_{\epsilon, \mathcal{D}}: Y'_{n}\to \mathcal{Y}_{L_{\epsilon}}$
  which is now a $2r$-isomorphism. The module $Y'_{n}$ is a shift of
  $Y_{n}$.  We now consider the module
  $\mathcal{K}= \mathrm{Cone}(\psi_{\epsilon,\mathcal{D}})$. The
  module $\mathcal{K}$ is $2r$-acyclic.  We also pick a point
  $x_{\epsilon}\in \bar{L}_{\epsilon}$ such that when $\epsilon \to 0$
  the points $x_{\epsilon} \to x_{0}=e(0)$. The proof of the main
  theorem in \cite{Bi-Co-Sh:LagrSh} shows that there exists a
  perturbed (with respect to the data $\mathcal{D}$) $J$-holomorphic
  polygon $u_{\epsilon, \mathcal{D}}$ with boundary on
  $\bar{L}_{\epsilon}\cup \{\bar{F}\in \bar{\mathcal{F}}\}\cup
  \bar{L}'_{\epsilon}$ of energy at most $2r$ and that passes through
  the point $x_{\epsilon}$. The argument to show this is delicate: it
  is based on considering $x_{\epsilon}$ as the maximum of a Morse
  function on $\bar{L}_{\epsilon}$ and integrating this in the data
  $\mathcal{D}$; one then considers the chain complex
  $\mathcal{K}(L_{\epsilon})$ which is $2r$-acyclic; the differential
  of the complex $\mathcal{K}(L_{\epsilon})$ has a description that
  \pbnote{can be made quite precise using the techniques
    from\jznote{~\cite[\S2.5]{Bi-Co-Sh:LagrSh}};} this description
  shows that the only way to ``kill'' the fundamental class of
  $CF(L,L)$ (which is represented by $x_{\epsilon}$) in
  $\mathcal{K}(L_{\epsilon})$ is if there exists a polygon as claimed.
  Once the polygons $u_{\epsilon,\mathcal{D}}$ are constructed we make
  $\epsilon \to 0$ and $||\mathcal{D}||\to 0$. There is yet another
  little twist here: recall that the perturbation data $\mathcal{D}$
  consists of two parts, a Hamiltonian perturbation and a choice of (a
  family) of almost complex structures. The expression
  $||\mathcal{D}||\to 0$ simply means that the Hamiltonian terms go to
  $0$. However, we will also want to take a limit for the almost
  complex structure part in such a way that these choices converge to
  an almost complex structure that extends the almost complex
  structure $J=e_{\ast} J_{0}$ where $J_{0}$ is the standard almost
  complex structure in $\mathbb{C}^{n}$. We now apply Gromov
  compactness to the $u_{\epsilon,\mathcal{D}}$ and obtain as a limit
  a curve $u_{0}$ that is $J$-holomorphic, of symplectic area
  $\omega(u_{0})\leq 2r\leq 4\delta$, that passes through the point
  $e(0)$ and only intersects $\mathrm{Im}(e)$ along the boundary of
  the ball and $\bar{L}$. Therefore, by \pbnote{a version of the
    Lelong inequality} \jznote{(see \cite{GH-book} and \cite{Lel57})}, the symplectic area of $u_{0}$ is at least
  $\pi r_{0}^{2}/2$. We conclude
  $4\delta \geq \omega(u_{0})\geq \pi r_{0}^{2}/2 >0$ which
  contradicts that $\delta$ can be taken arbitrarily small and
  concludes the proof.
\end{proof}
 
In summary, we have:

\begin{cor} If the derived Fukaya category of exact Lagrangians in the
  class $\mathcal{L}ag^{\ast}(M)$, $D\fuk^{\ast}(M)$, has a finite
  number of geometric generators, then the space of exact Lagrangians
  $\bar{\mathcal{L}}$ (without fixing the primitives) is endowed with
  a finite, non-degenerate fragmentation metric of the form
  $\tilde{D}^{\mathcal{F},\mathcal{F}'}$ induced by a persistence
  weight (on a TPC associated to a certain filtered dg-category of
  $A_{\infty}$-modules).
\end{cor}

\begin{rem}
  (a) There are many examples of manifolds $M$ and classes of exact
  Lagrangians $\mathcal{L}ag^{\ast}(M)$ in $M$ that satisfy the
  condition in the corollary. Additionally, there are slight
  extensions of the construction that make the various results
  applicable in even more cases.  The simplest such example is when
  $\mathcal{F}$ still consists of a family of exact Lagrangians in $M$
  but is not necessarily included in $\mathcal{L}ag^{\ast}(M)$. All
  the construction apply in this case. The interest of this extension
  is that, sometimes, the Yoneda modules $\mathcal{Y}'_{F}$,
  $F\in \mathcal{F}$, are very efficient generators of
  $D\fuk^{\ast}(M)$. To see a concrete example, consider a co-tangent
  disk bundle $D^{\ast}N$ of some manifold
  $N$. Following~\cite{FSS:ex, Fu-Se-Sm} (see
  also~\cite{Bi-Co:spectr}) we can symplectically embed
  $D^{\ast}N\hookrightarrow (M,\omega)$ where $(M,\omega)$ is the
  total space of a Lefschetz fibration. We can take as class
  $\mathcal{L}ag^{\ast}(M)$ the exact Lagrangians in the interior of
  $D^{\ast}N$. One can consider the full Fukaya category of the exact
  Lagrangians in $M$, $\fuk(M)$. Among its objects are some
  \pbnote{Lagrangian matching spheres} $S\subset M$ with the property
  that the pull-back of the Yoneda module
  $\mathcal{Y}_{S}\subset D\fuk(M)$ to a module over the category
  $\fuk^{\ast}(M)$ generates (in a triangulated sense) the category
  $D\fuk^{\ast}(M)$ (the module $\mathcal{Y}_{S}$ can be seen as
  representing the fiber of $D^{\ast}N$).  In this case, one can take
  $\mathcal{F}$ to be this $S$ together with all its possible
  primitives and $\mathcal{F}'$ a small perturbation $S'$ of $S$,
  again together with all its possible primitives.  The results in
  \cite{Bi-Co:spectr} then show that the resulting metric
  $\tilde{D}^{\mathcal{F},\mathcal{F}'}$ on the exact, (geometric)
  Lagrangians in the interior of $D^{\ast}N$ is of bounded diameter.

  (b) In \cite{Bi-Co-Sh:LagrSh} are introduced some other
  fragmentation pseudo-metrics on the space of Lagrangians in
  particular one based on the shadow of Lagrangian cobordisms. Those
  pseudo-metrics are always upper bounds for the persistence type
  pseudo-metrics described here. One difficulty with the shadow
  pseudo-metrics is that, when only using embedded cobordism, one
  does not get a usable weight for all the exact triangles in
  $D\fuk^{\ast}(M)$. Indeed, a cobordism with multiple ends induces an
  iterated cone-decomposition in $D\fuk^{\ast}(M)$ and this iterated
  cone decomposition can be ``weighted'' by the shadow of the
  respective cobordism, but not all exact triangles in the derived
  Fukaya category are induced by embedded cobordism. A possible way to
  remedy this is to work with immersed Lagrangians as in
  \cite{Bi-Co:LagPict} but this becomes considerably heavier
  technically. The relation to cobordism is discussed in more detail below, in \S\ref{sb:lcob}.
\end{rem}

% !TEX root = TPC.tex

\newcommand{\xx}{\Gamma}

\subsection{Lagrangian cobordism} \label{sb:lcob} The theory of
Lagrangian cobordism exhibits in a geometric way several key notions
from the algebraic theory of TPC's. The purpose of this section is to
give some geometric insight into TPC's coming from the theory of
Lagrangian cobordism. In particular we will show how certain natural
symplectic measurements on Lagrangian cobordism lead to weighted exact
triangles in the persistence Fukaya category.

\subsubsection{Background on Lagrangian
  cobordism} \label{sbsb:cobs-background} Let $(M, \omega = d\lambda)$
be a Liouville manifold, endowed with a given Liouville form
$\lambda$. We endow $\mathbb{R}^2$ with the Liouville form
$\lambda_{\mathbb{R}^2} = x dy$ and its associated standard symplectic
structure
$\omega_{\mathbb{R}^2} = d\lambda_{\mathbb{R}^2} = dx \wedge dy$. Let
$\widetilde{M} := \mathbb{R}^2 \times M$, endowed with the Liouville
form $\widetilde{\lambda} = \lambda_{\mathbb{R}^2} \oplus \lambda$ and
the symplectic structure
$\widetilde{\omega} = d(\widetilde{\lambda}) = \omega_{\mathbb{R}^2}
\oplus \omega$. We denote by
$\pi: \mathbb{R}^2 \times M \longrightarrow M$ the projection.

Below we will assume known the notion of Lagrangian cobordism, as
developed in~\cite{Bi-Co:cob1, Bi-Co:fuk-cob}. For simplicity we will
consider only {\em negative-ended} cobordisms
$V \subset \mathbb{R}^2 \times M$, which means that $V$ has only
negative ends. Moreover, all the cobordisms considered below will be
assumed to be exact with respect to the Liouville form
$\widetilde{\lambda}$ and endowed with a given primitive
$f_{V}: V \longrightarrow \mathbb{R}$ of
$\widetilde{\lambda}|_V$. Denote by $L_1, \ldots, L_k \subset M$ the
Lagrangians corresponding to the ends of $V$ and by
$\ell_1, \ldots, \ell_k \subset \mathbb{R}^2$ the negative horizontal
rays of $V$ so that $V$ coincides at $-\infty$ with
$(\ell_1 \times L_1) \coprod \cdots \coprod (\ell_k \times L_k)$.  We
remark that we adopt here the conventions from~\cite{Bi-Co:fuk-cob}
regarding the ends of $V$, namely we always assume that the $j$'th ray
$\ell_j$ lies on the horizontal line $\{ y = j \}$. Also, we allow
some of the Lagrangians $L_j$ to be void.
% Sometimes a given cobordism $V$ will need to be translated by an
% integer amount in the $y$-direction in order to comply with these
% conventions.

Note that $\lambda_{\mathbb{R}^2}|_{\ell_i} = 0$ hence
$f_{V}|_{\ell_i \times L_i}$ is constant in the $\ell_i$ direction for
all $i$. Therefore the Lagrangians $L_i \subset M$ are $\lambda$-exact
and $f_{V}$ induces well defined primitives
$f_{L_i}: L_i \longrightarrow \mathbb{R}$ of $\lambda|_{L_i}$ for each
$i$, namely $f_{L_i}(p) := f_{V}(z_0,p)$ for every $p \in L_i$, where
$z_0$ is any point on $\ell_i$.

\subsubsection{Filtered and persistence categories associated to
  cobordism} \label{sbsb:cobs-categs} As constructed
in~\cite{Bi-Co-Sh:LagrSh, Bi-Co:spectr}, there is a weakly filtered
Fukaya $A_{\infty}$-category $\fuk(M)$ of $\lambda$-exact Lagrangians
with objects being exact Lagrangians $L \subset M$ endowed with a
primitive $f_L:L \longrightarrow \mathbb{R}$ of
$\lambda|_L$. Similarly, there is also a weakly filtered Fukaya
$A_{\infty}$-category of cobordisms $\fuk(\mathbb{R}^2 \times M)$ with
objects being negative-ended exact cobordisms
$V \subset \mathbb{R}^2 \times M$ endowed with a primitive $f_V$ of
$\widetilde{\lambda}|_V$ as above\footnote{For technical reasons one
  needs to enlarge the set of objects in $\fuk(\mathbb{R}^2 \times M)$
  to contain also objects of the type $\gamma \times L$, where
  $\gamma \subset \mathbb{R}^2$ is a curve which outside of a compact
  set is either vertical or coincides with horizontal ends with
  $y$-value being $l \pm \tfrac{1}{10}$, where $l \in
  \mathbb{Z}$.}. We also have the dg-categories of weakly filtered
$A_{\infty}$-modules over each of the previous Fukaya categories,
which we denote by $\text{mod}_{\fuk(M)}$ and
$\text{mod}_{\fuk(\mathbb{R}^2 \times M)}$ respectively.

The fact that the above categories are only weakly filtered rather
than genuinely filtered will not a play an important role in the
discussion below which is anyway not meant to be fully rigorous. In
fact, below we will mostly concentrate on the chain complexes
associated to various Lagrangians and modules, ignoring the higher
order $A_{\infty}$-operations, and these are genuinely filtered.
Moreover, as explained in~\S\ref{subsec-symp} above, it seems possible
with some technical effort and at the expense of generality, to set up
the theory so that these categories become genuinely filtered.

Let $\mathcal{Y}: \fuk(M) \longrightarrow \text{mod}_{\fuk(M)}$ be the
Yoneda embedding (in the framework of weakly filtered
$A_{\infty}$-categories) and
$\fuk(M)^{\nabla} \subset \text{mod}_{\fuk(M)}$ the triangulated
closure of the image of $\mathcal{Y}$. We denote by
$\mathcal{C} = \mathcal{P}H(\fuk(M)^{\nabla})$ the persistence
homological category associated to $\fuk(M)^{\nabla}$.  Up to the
technical issue concerning ``filtered'' vs.~``weakly filtered'', the
category $\mathcal{C}$ is a TPC. By a slight abuse of notation we will
denote the Yoneda module $\mathcal{Y}(L)$ of a Lagrangian
$L \in \text{Obj}(\fuk(M))$ also by $L$.

There is also a Yoneda embedding
$\fuk(\mathbb{R}^2 \times M) \longrightarrow
\text{mod}_{\fuk(\mathbb{R}^2 \times M)}$ and we will typically denote
the Yoneda modules corresponding to cobordisms by calligraphic
letters, e.g.~the Yoneda module corresponding to
$V \in \text{Obj}(\fuk(\mathbb{R}^2 \times M))$ will be denoted by
$\mathcal{V}$.

Under additional assumptions on $M$, on the Lagrangians taken as the
objects of $\fuk(M)$ and the Lagrangians cobordisms of
$\fuk(\mathbb{R}^2 \times M)$, one can set up a graded theory,
endowing the morphisms in $\fuk(M)$ and $\fuk(\mathbb{R}^2 \times M)$
with a $\mathbb{Z}$-grading and the categories with
grading-translation functors. See~\cite{Seidel} for the case of
$\fuk(M)$ and~\cite{Hensel:stab-cond} for grading in the framework of
cobordisms. In what follows we will not explicitly work in a graded
setting, but whenever possible we will indicate how grading fits in
various constructions.

\subsubsection{Iterated cones associated to
  cobordisms} \label{sbsb:cobs-icone} Let
$\gamma \subset \mathbb{R}^2$ be an oriented\footnote{The orientation
  on $\gamma$ is necessary in order to set up a graded Floer theory,
  and also in order to work with coefficients over rings of
  characteristic $\neq 2$. Here we work with
  $\mathbb{Z}_2$-coefficients, therefore if one wants to ignore the
  grading then the orientation of $\gamma$ becomes irrelevant.}  plane
curve which is the image of a proper embedding of $\mathbb{R}$ into
$\mathbb{R}^2$. Viewing $\gamma \subset \mathbb{R}^2$ as an exact
Lagrangian we fix a primitive $f_{\gamma}$ of
$\lambda_{\mathbb{R}^2}|_{\gamma}$. Given an exact Lagrangian
$L \subset M$, consider the exact Lagrangian
$\gamma \times L \subset \mathbb{R}^2 \times M$ and take
$f_{\gamma, L}: \gamma \times L \longrightarrow \mathbb{R}$,
$f_{\gamma, L}(z, p) = f_{\gamma}(z) + f_{L}(p)$ for the primitive of
$\widetilde{\lambda}|_{\gamma \times L}$. From now on we will make the
following additional assumptions on $\gamma$. The ends of $\gamma$
will be assumed to coincide with a pair of rays $\ell'$, $\ell''$ each
of which is allowed to be either horizontal or vertical. Moreover in
the case of a horizontal ray, the ray is assumed to have
$y$-coordinate which is in $\mathbb{Z} \pm \tfrac{1}{10}$ and in the
case of vertical rays we assume the rays have $x$-coordinate being
$0$.

Below we will mainly work with the following two types of such
curves. The first one is
$\gamma^{\uparrow} = \{ x = 0\} \subset \mathbb{R}^2$ (i.e. the
$x$-axis with its standard orientation) and we take
$f_{\gamma^{\uparrow}} \equiv 0$. Then for every exact Lagrangian
$L \subset M$ we can identify $f_{\gamma^{\uparrow},L}$ with $f_L$ in
the obvious way.

The second type is the curve $\gamma_{i,j}$, where $i\leq j$ are two
integers, depicted in Figure~\ref{f:curves} and oriented by going from
the lower horizontal end to the upper horizontal end. Note that by
taking $\gamma_{i,j}$ close enough to dotted polygonal curve in
Figure~\ref{f:curves} we can assume that
$\lambda_{\mathbb{R}^2}|_{\gamma_{i,j}}$ is very close to $0$. We fix
the primitive $f_{\gamma_{i,j}}$ to be the one that vanishes on the
vertical part of $\gamma_{i,j}$.
\begin{figure}[htbp]
   \begin{center}
     \includegraphics[scale=0.63]{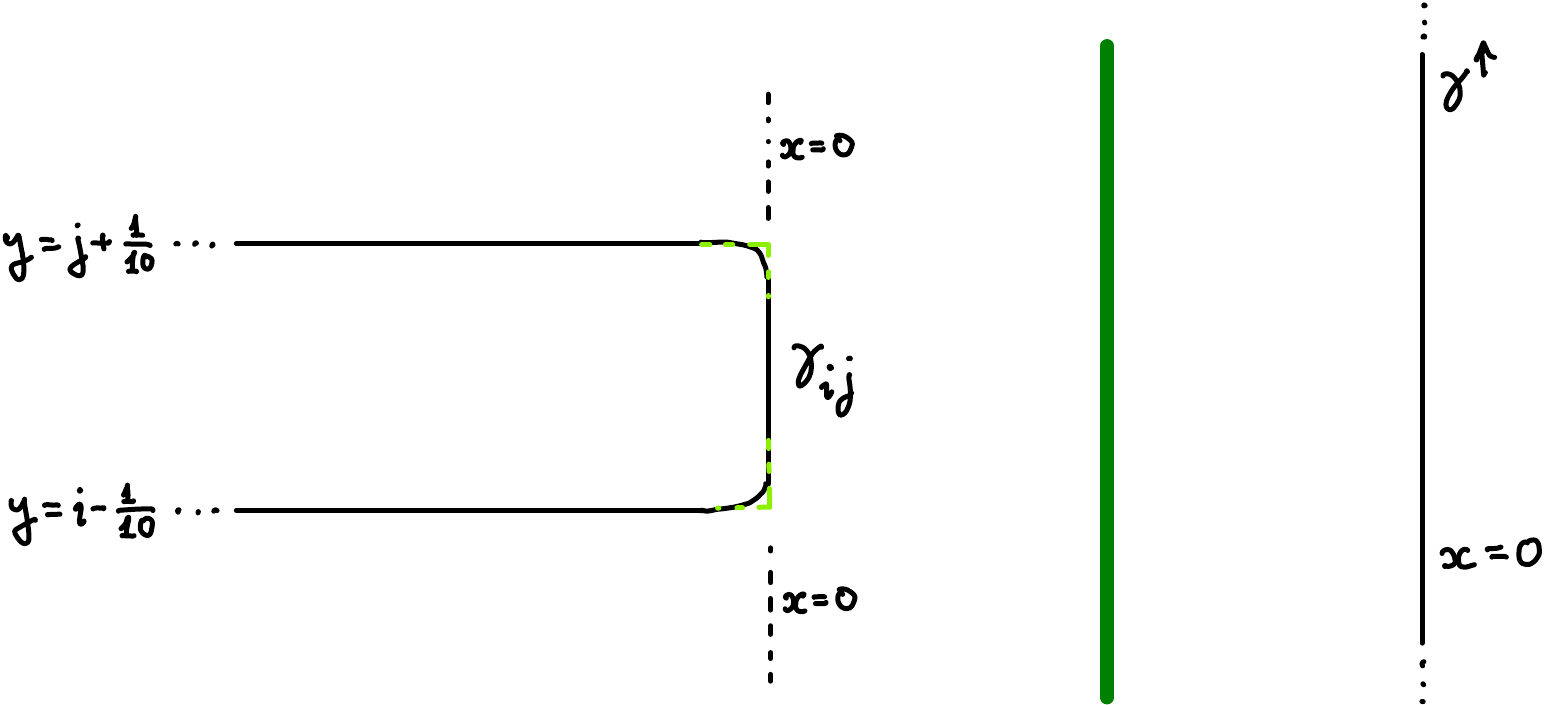}
%     \includegraphics[width=0.70\textwidth,
%     height=0.3\textheight]{eta-r}
   \end{center}
   \caption{The curves $\gamma_{i,j}$ and $\gamma^{\uparrow}$.}
   \label{f:curves}
\end{figure}

Let $\gamma \subset \mathbb{R}^2$ and $f_{\gamma}$ be as above.
Following~\cite{Bi-Co:fuk-cob, Bi-Co:cob1} there is a (weakly)
filtered $A_{\infty}$-functor, called an inclusion functor,
$\mathcal{I}_{\gamma}: \fuk(M) \longrightarrow \fuk(\mathbb{R}^2
\times M)$ which sends the object $L \in \text{Obj}(\fuk(M))$ to
$\gamma \times L \in \text{Obj}(\fuk(\mathbb{R}^2 \times M))$.

Let $V \subset \mathbb{R}^2 \times M$ be a Lagrangian
cobordism. Denote by $\mathcal{V}$ the (weakly) filtered Yoneda module
of $V$ and consider the pull-back module
$\mathcal{I}_{\gamma}^* \mathcal{V}$.  Note that for every exact
Lagrangian $N \subset M$ we have
$$\mathcal{I}_{\gamma}^* \mathcal{V} (N) = CF(\gamma \times N, V)$$
as {\em filtered chain complexes}. (The filtrations are induced by
$f_V$, $f_{\gamma,N}$ and by the Floer data in case it is not
trivial.)

Assume that the ends of $V$ are $L_1, \ldots, L_k$ and moreover that
$V$ is cylindrical over $(-\infty, -\delta] \times \mathbb{R}$ for
some $\delta>0$. (This can always be achieved by a suitable
translation along the $x$-axis.)  Fix $1 \leq i \leq k-1$ and consider
the curve $\gamma_{i,i+1}$ and the pull-back (weakly) filtered module
$\mathcal{I}_{\gamma_{i,i+1}}^*\mathcal{V}$. By the results
of~\cite{Bi-Co:fuk-cob, Bi-Co-Sh:LagrSh} there exists a module
homomorphism
$$\xx_{V, \gamma_{i,i+1}}: L_{i+1} \longrightarrow L_i$$
which preserves action filtrations and such that
$$\mathcal{I}_{\gamma_{i,i+1}}^*\mathcal{V} = T^{d_i}\text{cone}(L_{i+1}
\xrightarrow{\xx_{V, \gamma_{i, i+1}}} L_i)$$ as (weakly) filtered
modules. Here $T$ stand for the grading-translation functor and the
amount $d_i \in \mathbb{Z}$ by which we translate depends only on
$i$. We will be more precise about the values of $d_i$ later on.  The
map $\xx_{V, \gamma_{i,i+1}}$ is canonically defined by $V$ and
$\gamma_{i,i+1}$, up to a boundary in
$\hom^{\leq 0}_{\text{mod}_{\fuk(M)}}(L_{i+1}, L_i)$. Therefore it
gives rise to a well defined morphism in the homological persistence
category
$\mathcal{C}_0 = H\bigl(\hom^{\leq 0}_{\text{mod}_{\fuk(M)}}(L_{i+1},
L_i)\bigr)$ which by abuse of notation we still denote by
$\xx_{V, \gamma_{i,i+1}} \in \hom_{\mathcal{C}_0}(L_{i+1}, L_i)$.

The above can be generalized to several consecutive ends in a row as
follows. Fix $1 \leq i \leq j \leq k$.
\begin{figure}[htbp]
  \begin{center}
    \includegraphics[scale=0.63]{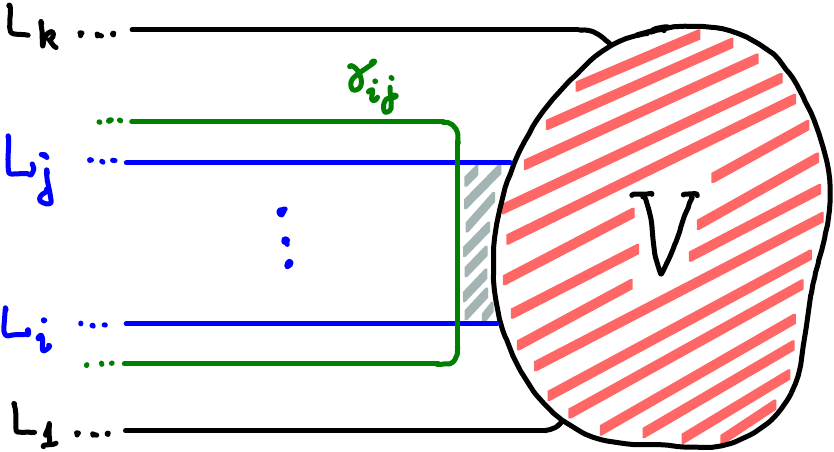}
  \end{center}
  \caption{The module $\mathcal{I}_{\gamma_{i,j}}^* \mathcal{V}$.}
  \label{f:map-ij}
\end{figure}
The pull-back module
$\mathcal{I}_{\gamma_{i,j}}^* \mathcal{V}$ can be identified with an
iterated cone of the type:
\begin{equation} \label{eq:icone-V} \mathcal{I}_{\gamma_{i,j}}^*
  \mathcal{V} = \text{cone}\Bigl(L_j \longrightarrow
  \text{cone}\bigl(L_{j-1} \longrightarrow \ldots \longrightarrow
  \text{cone}(L_{i+1} \longrightarrow L_i) \ldots \bigr)\Bigr),
\end{equation}
where similarly to the case $\xx_{V, \gamma_{i,i+1}}$, all the maps in
the iterated cone are module homomorphism that preserve filtrations.
See figure~\ref{f:map-ij}.

Note that there are some grading-translations in~\eqref{eq:icone-V}
which we have ignored. We will be more precise about this point later
on when we consider iterated cones involving three objects.

\begin{remnonum}
  If $V$ has $k$ ends then for every $i\leq 1$ and $k \leq l$ we have
  $\mathcal{I}_{\gamma_{i,l}}^*\mathcal{V} =
  \mathcal{I}_{\gamma^{\uparrow}}^*\mathcal{V}$.
\end{remnonum}

Finally, fix $1 \leq i \leq l < j \leq k$, and consider the two
modules $\mathcal{I}_{\gamma_{i,l}}^*\mathcal{V}$ and
$\mathcal{I}_{\gamma_{l+1, j}}^*\mathcal {V}$. There is a module
homomorphism
$\xx_{V, \gamma_{i,l}, \gamma_{l+1, j}}:
\mathcal{I}_{\gamma_{l+1,j}}^*\mathcal {V} \longrightarrow
\mathcal{I}_{\gamma_{i,l}}^*\mathcal{V}$ which preserves filtrations.
Note that we have
$$\mathcal{I}^*_{\gamma_{i,j}} \mathcal{V} =
T^{d_{i,l,j}}\text{cone}(\mathcal{I}^*_{\gamma_{l+1,j}}\mathcal{V}
\xrightarrow{\xx_{V, \gamma_{i,l}, \gamma_{l+1,j}}}
\mathcal{I}^*_{\gamma_{i,l}} \mathcal{V}),$$ for some
$d_{i,l,j} \in \mathbb{Z}$. See Figure~\ref{f:maps-ilj}.

\begin{figure}[htbp]
   \begin{center}
     \includegraphics[scale=0.63]{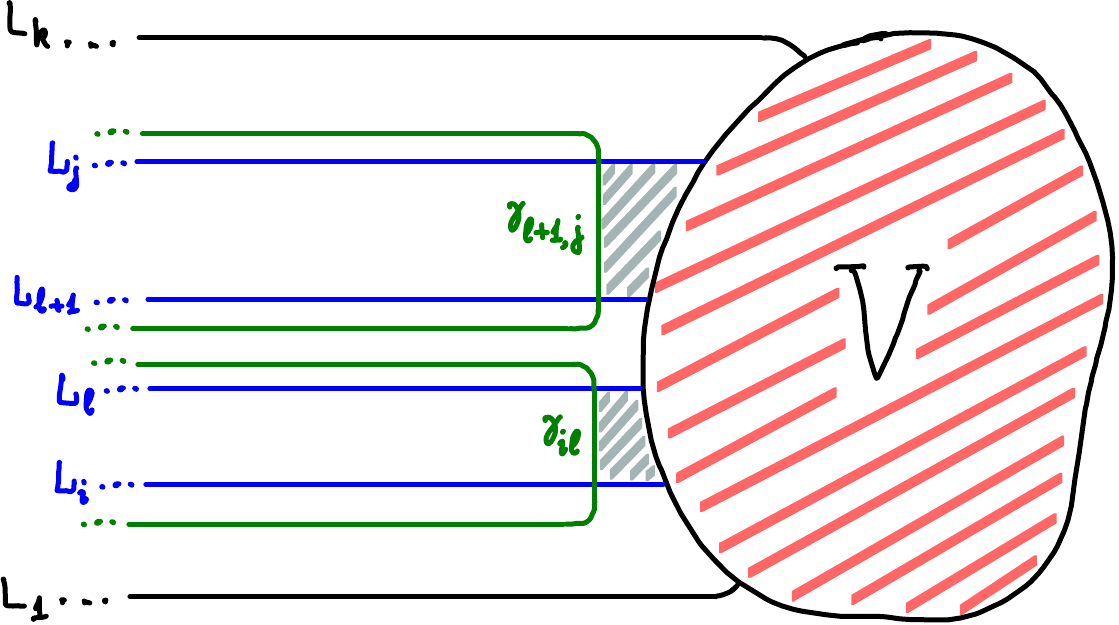}
   \end{center}
   \caption{The modules $\mathcal{I}_{\gamma_{l+1,j}}^* \mathcal{V}$
     and $\mathcal{I}_{\gamma_{i,l}}^* \mathcal{V}$.}
%     and the map $\xx_{V, \gamma_{i,l}, \gamma_{l+1,j}}$ between them.
   \label{f:maps-ilj}
\end{figure}

\subsubsection{Shadow of cobordisms} \label{sbsb:cobs-shadow} Recall
from~\cite{Bi-Co-Sh:LagrSh, Co-Sh} that for each Lagrangian cobordism
$V \subset \mathbb{R}^2 \times M$ we can associate its shadow
$\mathcal{S}(V)$ defined by
\begin{equation} \label{eq:shadow} \mathcal{S}(V) := Area \bigl(
  \mathbb{R}^2 \setminus \mathcal{U} \bigr),
\end{equation}
where $\mathcal{U} \subset \mathbb{R}^2 \setminus \pi(V)$ is the union
of all the {\em unbounded} connected components of
$\mathbb{R}^2 \setminus \pi(V)$.

\subsubsection{$r$-acyclic objects} \label{sbsb:cob-r-acyc} Let
$V \subset \mathbb{R}^2 \times M$ be a cobordism with ends
$L_1, \ldots, L_k$ such that $V$ is cylindrical over
$(-\infty, -\delta] \times \mathbb{R}$ for some $\delta>0$.  Let
$\mathcal{S}(V)$ be the shadow of $V$ and denote by $e$ the area of
the region to the right of $\gamma^{\uparrow}$ enclosed between
$\gamma^{\uparrow}$ and the projection to $\mathbb{R}^2$ of the
non-cylindrical part of $V$. See Figure~\ref{f:cob-1}.

\begin{figure}[htbp]
   \begin{center}
     \includegraphics[scale=0.63]{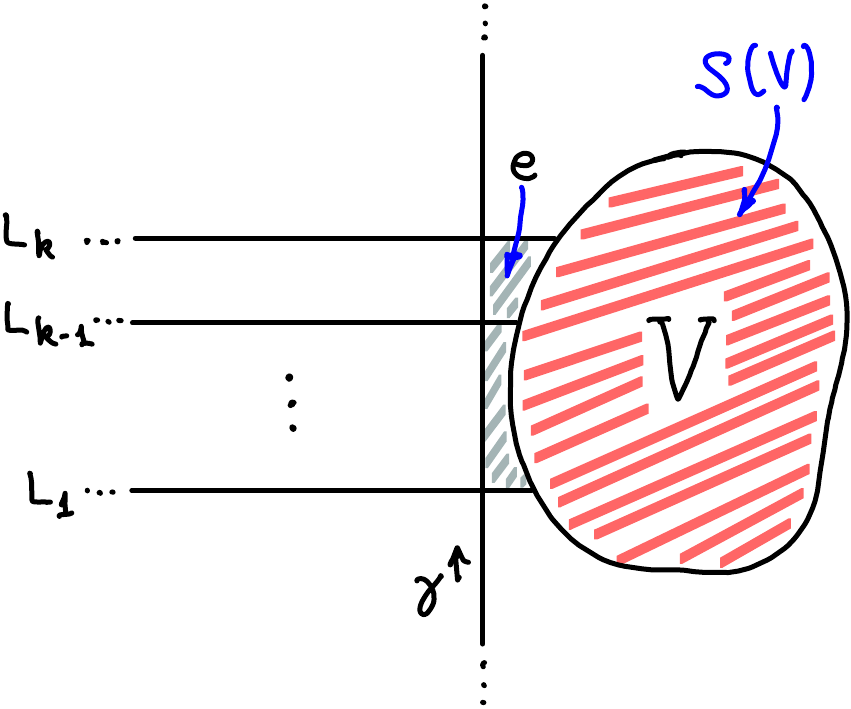}
   \end{center}
   \caption{The projection of a cobordism to $\mathbb{R}^2$, its
     shadow $\mathcal{S}(V)$ and the area $e$.}
   \label{f:cob-1}
\end{figure}

The module
$\mathcal{I}_{\gamma^{\uparrow}}^*\mathcal{V} =
\mathcal{I}_{\gamma_{1,k}}^*\mathcal{V}$ (which also has the
description~\eqref{eq:icone-V} with $i=1$, $j=k$) is $r$-acyclic,
where $r := \mathcal{S}(V) + e$.  This can be easily seen from the
fact that $V$ can be Hamiltonian isotoped to a cobordism $W$ which is
disjoint from $\gamma^{\uparrow} \times M$ via a compactly supported
Hamiltonian isotopy whose Hofer length is $\leq r$. Standard Floer
theory then implies that
$\mathcal{I}_{\gamma^{\uparrow}}^*\mathcal{V}$ is acyclic of boundary
depth $\leq r$. In the terminology used in this paper this means that
the object $\mathcal{I}_{\gamma^{\uparrow}}^*\mathcal{V}$ is
$r$-acyclic.

\begin{remark} \label{r:cobs-e} The area summand $e$ that adds to the
  shadow of $V$ in the quantity $r$ can be made arbitrarily small at
  the expense of applying appropriate shifts to each of the ends $L_i$
  of $V$. One way to do this is to replace the curve
  $\gamma^{\uparrow}$ by a curve $\gamma_{V}^{\uparrow}$ that
  coincides with $\gamma^{\uparrow}$ outside a compact subset and such
  that $\gamma_V^{\uparrow}$ approximates the shape of the projection
  of the non-cylindrical part of $V$ in such a way that the area $e'$
  enclosed between $\gamma^{\uparrow}_V$ and $\pi(V)$ is small. See
  Figure~\ref{f:approx}. One can apply a similar modification to the
  curves $\gamma_{i,j}$. Note that, in contrast to
  $f_{\gamma^{\uparrow}}$, the primitive $f_{\gamma_V^{\uparrow}}$ can
  no longer be assumed to be $0$ (a similar remarks applies to the
  primitives of the modifications of $\gamma_{i,j}$). As a result the
  cone decompositions~\eqref{eq:icone-V} associated to the pullback
  modules $\mathcal{I}^*_{\gamma_{i,j}}$ will have the same shape but
  each of the Lagrangians $L_i, \ldots, L_j$ will gain a different
  shift in action. Note that this will also result in ``tighter''
  weighted exact triangles than the ones we obtain below, in the sense
  of weights and various shifts on the objects forming these
  triangles.

  \begin{figure}[htbp]
   \begin{center}
     \includegraphics[scale=0.63]{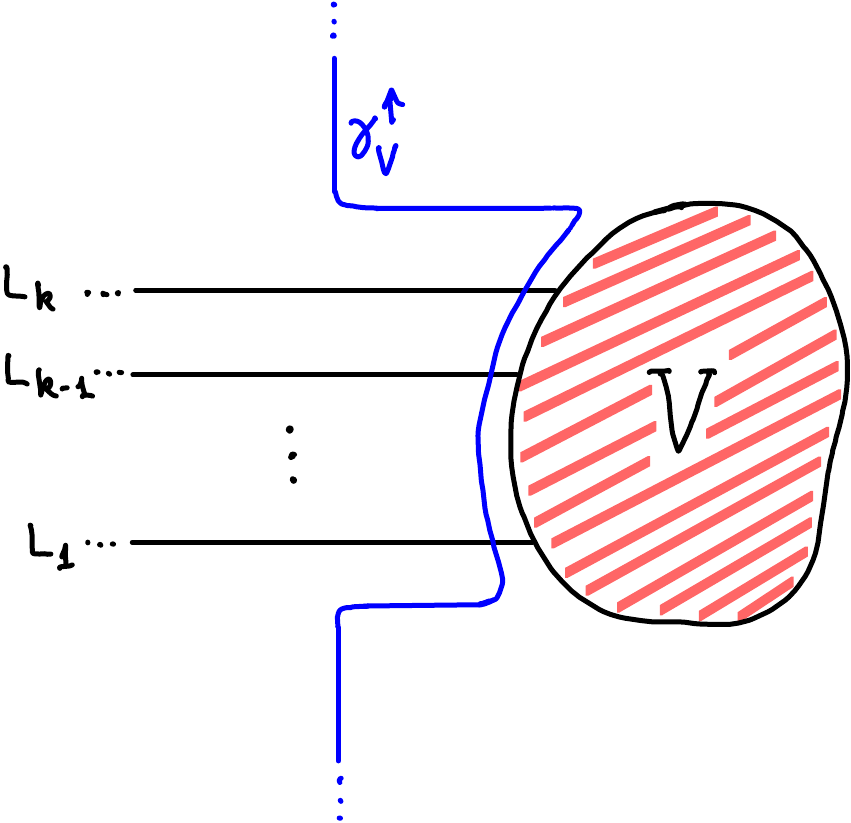}
%     \includegraphics[width=0.70\textwidth,
%     height=0.3\textheight]{eta-r}
   \end{center}
   \caption{Replacing $\gamma^{\uparrow}$ by a curve
     $\gamma_V^{\uparrow}$ better approximating  the shape of $V$.}
   \label{f:approx}
\end{figure}

To simplify the exposition, below we will not make these modifications
and stick to the curves $\gamma^{\uparrow}$ and $\gamma_{i,j}$ as
defined above, at the expense of $e$ not being necessarily small and
the weights of the triangles not being necessarily optimal.
\end{remark}

\subsubsection{$r$-isomorphisms} \label{sbsb:cob-r-iso} We begin by
visualizing the canonical map $\eta^{L}_r: \Sigma^r L \to L$, where
$L \subset M$ is an exact Lagrangian. Consider the curve
$\gamma \subset \mathbb{R}^2$ depicted in Figure~\ref{f:eta-r}, and
let $r$ be the area enclosed between $\gamma$ and
$\gamma^{\uparrow}$. Let $f_{\gamma}$ be the unique primitive of
$\lambda_{\mathbb{R}^2}|_{\gamma}$ that vanishes along the lower end
of $\gamma$. Note that $f_{\gamma} \equiv r$ along the upper end of
$\gamma$.  Let $V = L \times \gamma$ an set $f_{V} := f_{\gamma,L}$.
Therefore, the primitives induced by $V$ on its ends are as follows:
the primitive on the lower end coincides with $f_L$, while the
primitive on the upper end coincides with $f_L + r$. In other words,
the cobordism $V$ has ends $L$ and $\Sigma^r L$. Moreover, the map
$\xx_{V, \gamma_{1,2}}: \Sigma^r L \to L$ induced by $V$ and
$\gamma_{1,2}$ is precisely $\eta^L_r$.

\begin{figure}[htbp]
   \begin{center}
     \includegraphics[scale=0.63]{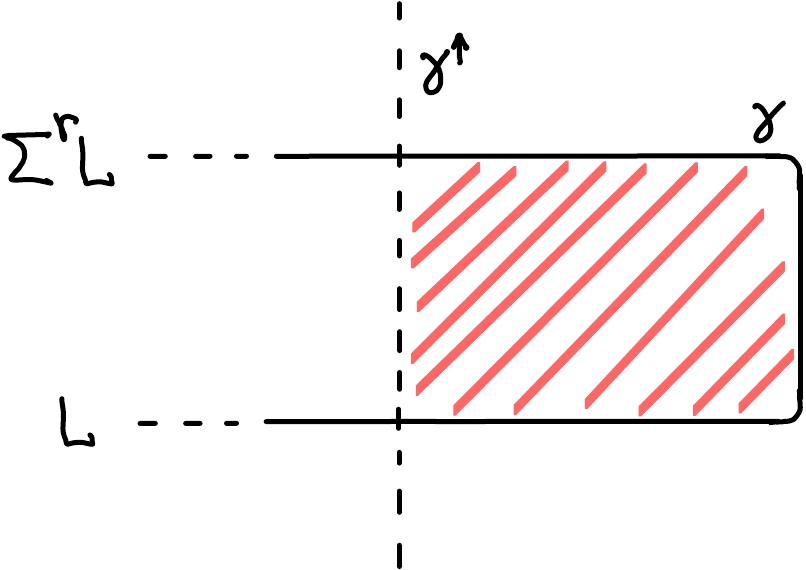}
%     \includegraphics[width=0.70\textwidth,
%     height=0.3\textheight]{eta-r}
   \end{center}
   \caption{The cobordism inducing $\eta^L_r: \Sigma^r L \to L$.}
   \label{f:eta-r}
\end{figure}

Another source of geometric $r$-isomorphisms comes from Hamiltonian
isotopies. Let $\phi_t^{H}$, $t \in [0,1]$ be a Hamiltonian isotopy
and let $L \subset M$ be an exact Lagrangian. The Lagrangian
suspension construction gives rise to an exact Lagrangian cobordism
between $L$ and $\phi_1^H(L)$. After bending the ends of that
cobordism to become negative one obtains a Lagrangian cobordism with
negative ends being $L$ and $\phi_1^H(L)$. See
Figure~\ref{f:suspension}. The primitive $f_V$ on $V$ is uniquely
defined by the requirement that $f_V$ coincides with $f_L$ on the
lower end of $V$. The shadow $\mathcal{S}(V)$ of this cobordism equals
the Hofer length of the isotopy $\{\phi_t^H\}_{t \in [0,1]}$, and we
obtain an $r$-isomorphism $\phi_1^H(L) \to L$.

\begin{figure}[htbp]
   \begin{center}
     \includegraphics[scale=0.63]{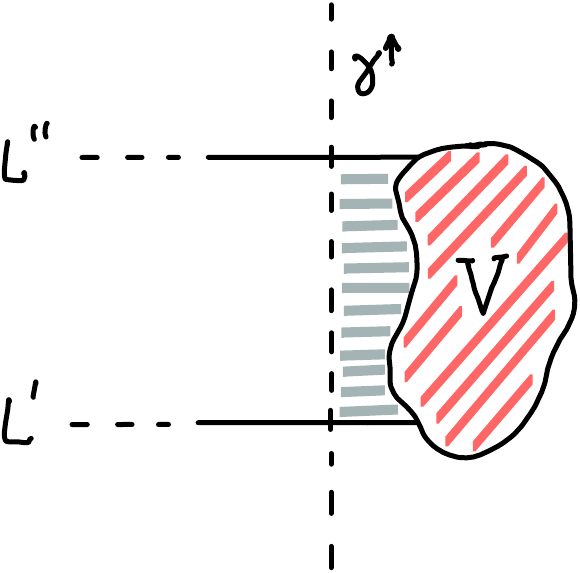}
   \end{center}
   \caption{The Lagrangian suspension of a Hamiltonian isotopy, after
     bending the ends.}
   \label{f:suspension}
\end{figure}

More generally, let $V$ be a Lagrangian cobordism with ends
$L_1, \ldots, L_k$. Let $r = \mathcal{S}(V) + e$ as above.  Fix
$1 \leq l < k$. As explained above we have, up to an overall
translation in grading,
$$\mathcal{I}^*_{\gamma^{\uparrow}} \mathcal{V} = 
\mathcal{I}^*_{\gamma_{1,k}} \mathcal{V} =
\text{cone}(\mathcal{I}^*_{\gamma_{l+1,k}}\mathcal{V}
\xrightarrow{\xx_{V, \gamma_{1,l}, \gamma_{l+1,k}}}
\mathcal{I}^*_{\gamma_{1,l}} \mathcal{V}),$$ and since
$\mathcal{I}^*_{\gamma^{\uparrow}} \mathcal{V}$ is $r$-acyclic the map
$\xx_{V, \gamma_{1,l}, \gamma_{l+1,k}}:
\mathcal{I}^*_{\gamma_{l+1,k}}\mathcal{V} \longrightarrow
\mathcal{I}^*_{\gamma_{1,l}} \mathcal{V}$ is an $r$-isomorphism.

\subsubsection{Weighted exact triangles} \label{sbsb:cobs-exact-t}

Let $V \subset \mathbb{R}^2 \times M$ be a Lagrangian cobordism with
ends $L_1, \ldots, L_k$. Let $V' \subset \mathbb{R}^2 \times M$ be the
cobordism obtained from $V$ by bending the upper end $L_k$ clockwise
around $V$ so that it goes beyond the end $L_1$ as in
Figure~\ref{f:bend-1}. To obtain a cobordism according to our
conventions, we need to further shift $V'$ upwards by one so that its
lower end has $y$-coordinate $1$ (instead of $0$). Clearly $V'$ is
also exact and $\mathcal{S}(V') = \mathcal{S}(V)$. We fix the
primitive $f_{V'}$ for $V'$ to be the unique one that coincides with
$f_V$ on the ends $L_1, \ldots, L_{k-1}$. A simple calculation shows
that $f_{V'}$ induces on the lowest end of $V'$ the primitive
$f_{L_{k}} - r$, where $r = \mathcal{S}(V) + e + \epsilon$. (Here
$\epsilon$ can be assumed to be arbitrarily small. It can be estimated
from above by the area enclosed between the bent end corresponding to
$\ell_k$, the projection to $\mathbb{R}^2$ of the non-cylindrical part
of $V$, $\ell_1$ and $\gamma^{\uparrow}$. See Figure~\ref{f:bend-1}.)

\begin{figure}[htbp]
   \begin{center}
     \includegraphics[scale=0.63]{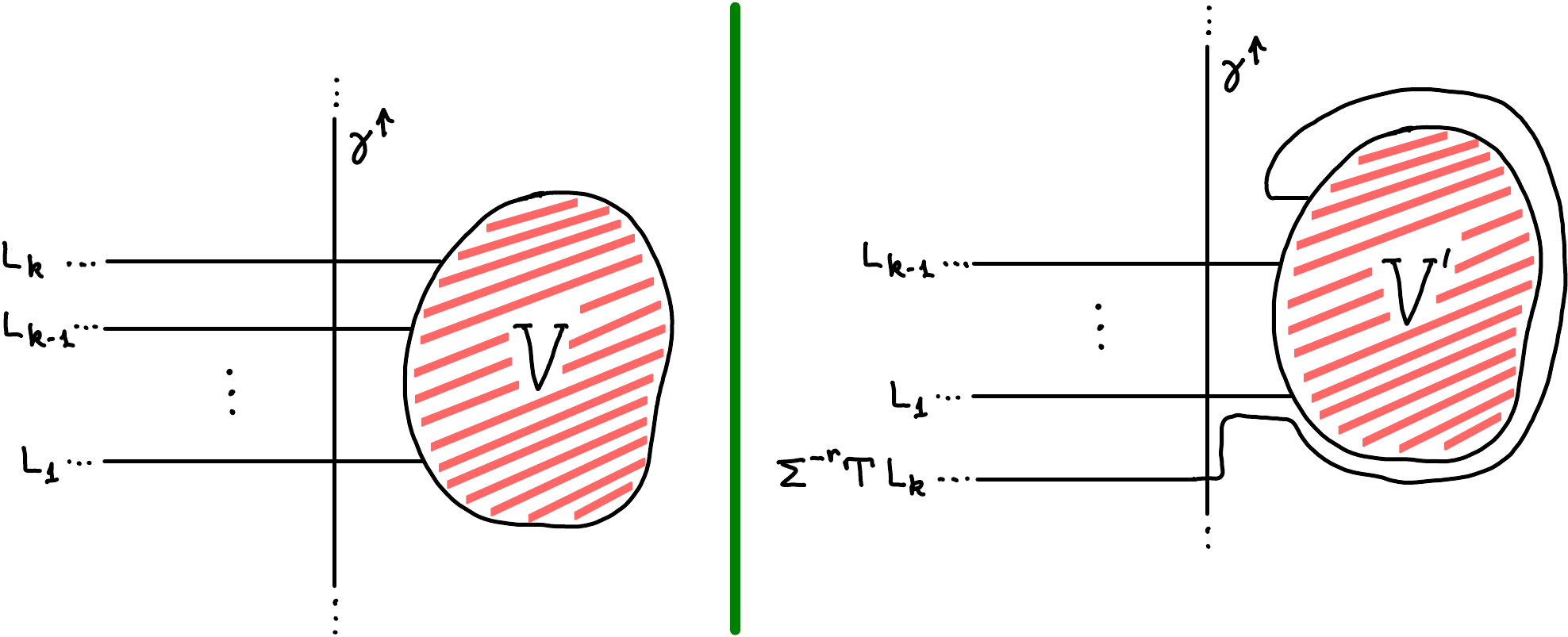}
%     \includegraphics[width=0.70\textwidth,
%     height=0.3\textheight]{eta-r}
   \end{center}
   \caption{Bending the upper end $L_k$ of a cobordism $V$.}
   \label{f:bend-1}
\end{figure}

Taking into account grading (in case $V$ is graded in the sense of
Floer theory), one can easily see that the grading on the lowest end
of $V'$ is translated by $1$ in comparison to $L_k$. Summing up, the
above procedure transforms a cobordism $V$ with ends
$L_1, \ldots, L_k$ into a cobordism $V'$ with ends
$\Sigma^{-r}TL_k, L_1, \ldots, L_{k-1}$.

Similarly, one can take $V$ and bend its lowest end $L_1$
counterclockwise around $V$ and obtain a new cobordism $V''$ with ends
$L_2, \ldots, L_k, \Sigma^r T^{-1}L_1$ and with
$\mathcal{S}(V'') = \mathcal{S}(V)$.

We are now in position to describe geometrically weighted exact
triangles.  Let $V$ be a cobordism with three ends, which for
compatibility with Definition~\ref{dfn-set} we denote by $C, B, A$
(going from the lowest end upward). See Figure~\ref{f:tr-abc}.

\begin{figure}[htbp]
   \begin{center}
     \includegraphics[scale=0.63]{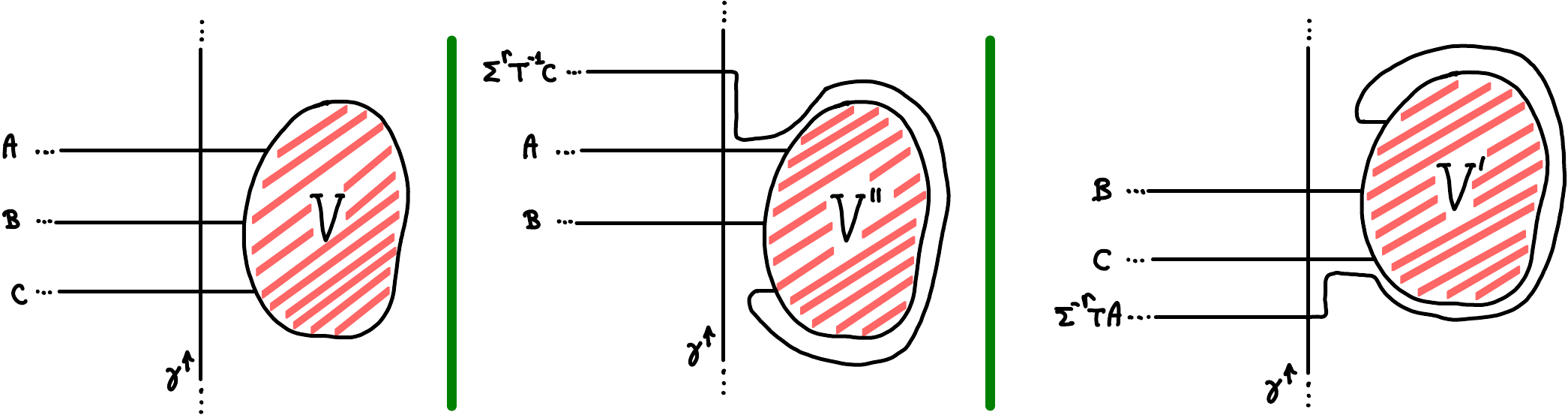}
%     \includegraphics[width=0.70\textwidth,
%     height=0.3\textheight]{eta-r}
   \end{center}
   \caption{A cobordism leading to an exact triangle of weight $r$.}
   \label{f:tr-abc}
\end{figure}

Let $r = \mathcal{S}(V) + e + \epsilon$. Put
$$\bar{u} := \xx_{V, \gamma_{2,3}} : A \to B, \quad \bar{v} :=
\xx_{V, \gamma_{2,1}}: B \to C.$$ Consider also the counterclockwise
rotation $V''$ of $V$ whose ends are $B, A, \Sigma^r T^{-1}C$. Let 
$$\bar{w} := \Sigma^{-r} T \xx_{V'', \gamma_{2,3}} =
\xx_{V', \gamma_{1,2}} : C \to \Sigma^{-r}T A.$$ We claim that
\begin{equation} \label{eq:ex-tr-cob-1} A \xrightarrow{\bar{u}} B
  \xrightarrow{\bar{v}} C \xrightarrow{\bar{w}} \Sigma^{-r}TA
\end{equation}
is a strict exact triangle of weight $r$.  This triangle is based on
the genuinely exact triangle from $\C_0$:
\begin{equation} \label{eq:ex-tr-cob-2} A \xrightarrow{u} B
  \xrightarrow{v} C' \xrightarrow{w} TA,
\end{equation}
where $u = \bar{u}$,
$C' = \mathcal{I}^*_{\gamma_{2,3}}V = \text{cone}(A \xrightarrow{u}
B)$, $v: B \to C'$ is the standard inclusion and $w: C' \to TA$ the
standard projection. The $r$-isomorphism $\phi: C' \to C$ and its
right $r$-inverse $\psi: \Sigma^rC \to C'$ are as follows:
$$\phi = \xx_{V, \gamma_{1,1}, \gamma_{2,3}}, \quad
\psi = T\xx_{V'', \gamma_{1,2}, \gamma_{3,3}}: \Sigma^r C\to T
\mathcal{I}_{\gamma_{1,2}}^*V'' = C'.$$ Note that
$\mathcal{I}_{\gamma_{1,2}}^*V'' = T^{-1}\mathcal{I}_{\gamma_{2,3}}^*V
= T^{-1}C' = T^{-1}\text{cone}(A \xrightarrow{u} B)$.

The fact that $\psi$ is a right $r$-inverse to $\phi$ and that these
maps fit into the diagram~\eqref{dfn-set-2} follows from standard
arguments in Floer theory. Note that these statements do not hold on
the chain level, but only in $\C_0$.
\subsubsection{Rotation of triangles} \label{sbsb:cobs-rot} Let $V$ be
a cobordism with three ends $C, B, A$ as in~\S\ref{sbsb:cobs-exact-t}
and consider the exact triangle~\eqref{eq:ex-tr-cob-1} of weight
$r$. Let $V'$ be the clockwise rotation of $V$, with ends
$B, C, \Sigma^{-r}TA$. The exact triangle associated to $V'$ is
\begin{equation} \label{eq:ex-tr-cob-3} B \xrightarrow{\bar{v}} C
  \xrightarrow{\bar{w}} \Sigma^{-r}TA \xrightarrow{\bar{u}'}
  \Sigma^{-r-\epsilon'} TB, 
\end{equation}
where $\epsilon'$ can be assumed to be arbitrarily small.  It is not
hard to see that in $\C_{\infty}$ (up to identifying objects with
their shifts and ignoring signs in the maps) the exact
triangle~\eqref{eq:ex-tr-cob-3} is precisely the rotation of the exact
triangle corresponding to~\eqref{eq:ex-tr-cob-1} in $\C_{\infty}$.

The above shows that rotation of weighted exact triangles coming from
cobordisms with three ends preserves weights (up to an arbitrarily
small error). Interestingly this is sharper than the case in a general
TPC, described in Proposition~\ref{prop-rot}, where the weight of a
rotated triangle doubles. See also Remark~\ref{r:rotation-weight}.

\subsubsection{Weighted octahedral property} \label{sbsb:cobs-oct}

The weighted octahedral formula from Proposition~\ref{prop-w-oct}
admits too a geometric interpretation in the realm of cobordisms. We
will not give the details of this construction here. Instead we will
briefly explain the cobordism counterpart of cone refinement and why
it behaves additively with respect to weights, as described
algebraically in Proposition~\ref{prop-cone-ref}. Note that weighted
cone-refinement is one of the main corollaries of the weighted
octahedral property.

For simplicity we focus here on the case described in
Example~\ref{ex:crefine} and ignore the grading-translation $T$.
Assume that we have two cobordisms $V$ with ends $X, B, A$ and $U$
with ends $A, F, E$.

These cobordisms induce two weighted exact triangles of weights
$r = \mathcal{S}(V) + e_V + \epsilon$ and
$s = \mathcal{S}(U) + e_U + \epsilon$. By gluing the two cobordisms
along the ends corresponding to $A$ we obtain a new cobordism $W$ with
four ends $X, B, F, E$. See Figure~\ref{f:glue-cob}. By the previous
discussion this exhibits $X$ as an iterated cone with linearization
$(B, F, E)$ which corresponds precisely to the algebraic cone
refinement of $X$ with linearization $(B,A)$ by $A$ with linearization
$(F, E)$.

Clearly the shadow $\mathcal{S}(W)$ of $W$ equals to
$\mathcal{S}(U) + \mathcal{S}(V)$, and the total weight of the cone
decomposition of $X$ associated to $W$ is $r+s$. Note again, that by
the procedure explained in Remark~\ref{r:cobs-e} one can make the
errors $e_V$, $e_U$ small at the expense of applying some shifts to
the elements of the linearization.

\begin{figure}[htbp]
   \begin{center}
     \includegraphics[scale=0.63]{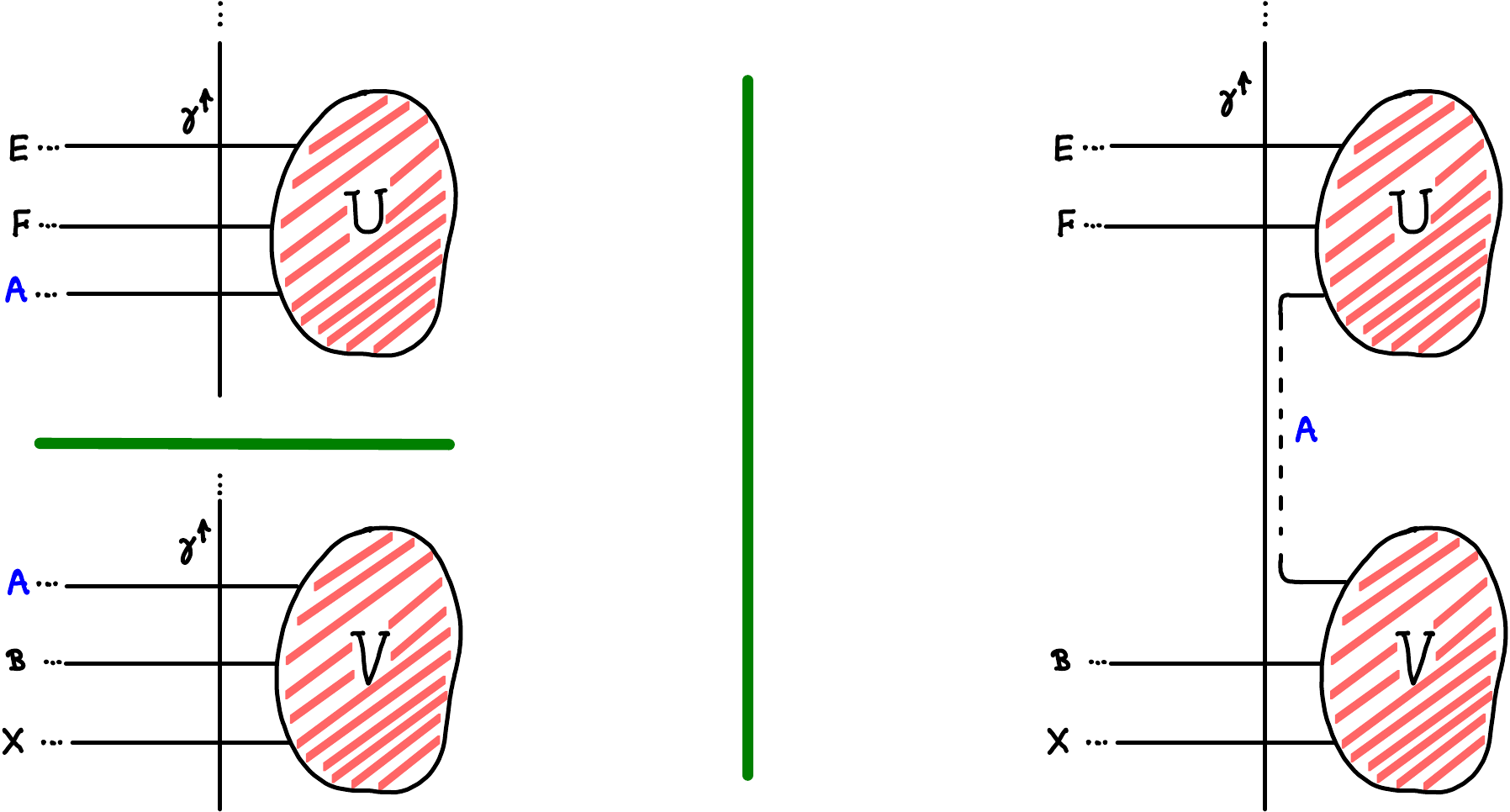}
   \end{center}
   \caption{Cone refinement via gluing cobordisms along two ends.}
   \label{f:glue-cob}
\end{figure}

\subsection{Work in progress (towards algebra 102)}
There are a few immediate questions and properties that we are currently working on  but 
have decided not to include in this paper.

\begin{itemize}
\item[(a)] We have not discussed idempotents and split completion in the triangulated persistence category setting. It turns out that the notion natural in this setting is a sort of refinement of the usual Karoubi completion. Given a TPC $\mathcal{C}$, it is obvious that the starting point in the construction is to
consider the split completion of the $0$-level, $\mathcal{C}_{0}$ in the usual triangulated category sense. However, it is immediate to see that this will not lead to an $\infty$-level that is split complete. One needs to also complete with respect to  pseudo-idempotents $e:\Sigma^{r} X\to X$ that are such that $e\circ \Sigma^{r}e=e\circ \eta_{r}$. This turns out to be possible but a bit more complicated compared to Karoubi completion.
\item[(b)] Given a TPC, $\C$, there are two natural associated Grothendieck groups: $K_{0}^{0}= K_{0}(\mathcal{C}_{0})$ and $K_{0}^{\infty}=K_{0}(\C_{\infty})$. Assume that $d^{\mathcal{F}}$ is 
one of the fragmentation pseudo-metrics on $\C$ as in Definition \ref{dfn-frag-met} (one could also use instead the pseudo-metrics such as $\overline{d}^{\mathcal{F}}$ on $\C_{\infty}$ as reviewed in \S\ref{subsec:rem-nondeg}) associated to a family $\mathcal{F}$ that is invariant with respect to shift (in the sense that if $F\in \mathcal{F}$, then $\Sigma^{r}F\in \mathcal{F}$ for all $r\in \R$). It is not difficult to see that such a pseudo-metric induces a pseudo-metric $d^{\mathcal{F}}$ on
$K_{0}^{0}$. This pseudo-metrics is translation invariant and thus characterized by 
a group-pseudo-norm defined by $\nu^{\mathcal{F}}(a)=d^{\mathcal{F}}(a,0)$. 
Moreover, the shift functor $\Sigma^{r}$ induces a flow on $K_{0}^{0}$ and $\nu^{\mathcal{F}}$ is constant along the orbits of this flow. In short, $K_{0}^{0}$ is a topological abelian group with an interesting
structure that can be used to study the action of various groups of automorphisms of the initial category.
For instance, in the symplectic examples, the Hamiltonian diffeomorphism group, $\mathrm{Ham}(M)$, of the underlying symplectic manifold $(M,\omega)$ admits a natural, continuous representation in $\mathrm{Aut}(K_{0}^{0})$ (where the relevant
TPC is one of the filtered Fukaya categories discussed above). Here $\mathrm{Aut}(K_{0}^{0})$ are the group
automorphisms endowed with a topology induced by the group pseudo-norm defined by $\overline{\nu}^{\mathcal{F}}(\phi)=\sup_{a} \nu^{\mathcal{F}}(\phi(a)-a)$. The topology on $\mathrm{Ham}(M)$
is the one induced by the Hofer norm. Given the properties of the  Hamiltonian diffeomorphism group (for instance, the fact that it is simple for closed manifolds), it seems to be an interesting question to study when this
representation is non-trivial. There are similar properties for the group
$K_{0}^{\infty}$, except that this group is much smaller.
\item[(c)] Another natural direction of investigation is localization and derived functors in the setting of TPCs.
There are two natural classes with respect to which one can localize: the $0$-isomomorphism and
the $r$-isomorphisms with  $r\in (0,\infty)$ variable. In both cases, one can construct filtered versions of the localization of $\mathcal{C}$
with respect to the respective class of morphisms and one can verify that both localizations are again TPCs.
It is also possible to consider persitence derived categories associated to persistence categories as well as 
persistence derived functors associated to persistence functors.
\item[(d)] In \cite{Bi-Co:LagPict} are introduced certain cobordism categories called {\em with surgery models} and it is shown that they give rise to triangulated categories with objects immersed Lagrangian submanifolds and with morphisms certain classes of immersed cobordisms relating them. It is expected that by taking 
into account the shadows of these cobordisms - in the sense in \S\ref{sbsb:cobs-shadow} -  one can extract a TPC refinement of that category.
\end{itemize}
Of course, there are also a variety of other algebraic and geometric questions, such as reformulating parts of this machinery in a Waldhausen type category setting (as mentioned in \S\ref{subsec:top-sp})  as well as symplectic applications deduced by pursuing the properties of the  TPCs relevant in that context.

%\input{Add.tex}
%\bibliography{/home/biran/latex/general/bibliography}

%begin{thebibliography}{10}
%\bibliography {biblio-tpc}
%\end{thebibliography}
%

\bibliographystyle{plain}
\bibliography{biblio_tpc}

\end{document}